\documentclass[11pt,letterpaper]{amsart}
\usepackage{amsmath,amssymb,amsthm,amsfonts, amscd, color, mathdots}
\usepackage{newtxtext,newtxmath}
\usepackage{mathrsfs}
\usepackage{verbatim}
\numberwithin{equation}{subsection}
\newtheorem{lemma}{Lemma}[section]

\newtheorem{proposition}{Proposition}[section]

\newtheorem{theorem}{Theorem}[section]

\newtheorem{corollary}{Corollary}[section]

\newtheorem{Definition}{Definition}[section]
\newtheorem{Remark}{Remark}[section]

\newcommand{\mQ}{\mathbb{Q}}
\newcommand{\mR}{\mathbb{R}}

\newcommand{\mZ}{\mathbb{Z}}
\newcommand{\mA}{\mathbb{A}}
\newcommand{\mC}{\mathbb{C}}

\newcommand{\gl}{\mathrm{GL}}

\begin{document}

\title[Gross-Prasad conjecture
and B\"ocherer conjecture]
  {On the Gross-Prasad conjecture with its refinement
   for $\left(\mathrm{SO}\left(5\right), \mathrm{SO}\left(2\right)\right)$
   and the generalized B\"ocherer conjecture
   } 
%
%
%
\author{Masaaki Furusawa}
\address[Masaaki Furusawa]{
Department of Mathematics, Graduate School of Science,
              Osaka Metropolitan University,
         Sugimoto 3-3-138, Sumiyoshi-ku, Osaka 558-8585, Japan
}
\email{furusawa@omu.ac.jp}
\thanks{The research of the first author was supported in part by 
JSPS KAKENHI Grant Number
19K03407, 
22K03235
and by Osaka Central  Advanced Mathematical
Institute (MEXT Promotion of Distinctive Joint Research Center Program JPMXP0723833165).
The first author would like to dedicate this paper to the memory 
of his brother, Akira Furusawa (1952--2020).}
%
%
%
\author{Kazuki Morimoto}
\address[Kazuki Morimoto]{
Department of Mathematics, Graduate School of Science, 
Kobe University, 1-1 Rokkodai-cho, Nada-ku, Kobe 657-8501, Japan
}
\email{morimoto@math.kobe-u.ac.jp}
\thanks{The research of the second author was supported in part by
JSPS KAKENHI Grant Number 17K14166, 21K03164
and by
Kobe University Long Term Overseas Visiting Program for Young Researchers Fund.}
%
%
%
\begin{abstract}
We investigate the Gross-Prasad conjecture and its refinement
for the  
Bessel periods
in the case of $(\mathrm{SO}(5), \mathrm{SO}(2))$.
In particular, by combining several theta correspondences, 
we prove the Ichino-Ikeda type formula for any 
tempered irreducible 
cuspidal automorphic representation.
As a corollary of our formula, we prove 
an explicit formula relating  
certain weighted averages of Fourier coefficients of 
 holomorphic Siegel cusp forms of degree two
which are Hecke eigenforms
to central special values of $L$-functions.
The formula is regarded as a natural generalization of
the B\"{o}cherer conjecture to the non-trivial toroidal character case.
\end{abstract}
%
%
%
\maketitle
%
%
%
%
\section{Introduction}
%
To investigate relations between periods of automorphic forms
and special values of $L$-functions is one of the focal research subjects
in number theory.
The central special values are  of keen interest
in light of the Birch and Swinnerton-Dyer conjecture and its generalizations.
%

Gross and Prasad~\cite{GP1,GP2} proclaimed a global
conjecture relating non-vanishing
of certain period integrals on special orthogonal groups to non-vanishing
of central special values of certain tensor product $L$-functions,
together with the local counterpart conjecture in the early 1990s.
Later with Gan~\cite{GGP}, they extended the conjecture to classical groups
and metaplectic groups.
Meanwhile a refinement of the Gross-Prasad conjecture,
which is a precise formula for the central special values of the tensor
product
$L$-functions for tempered cuspidal automorphic representations,
was formulated by Ichino and Ikeda~\cite{II}
in the co-dimension one special orthogonal case.
Subsequently Harris~\cite{Ha} formulated a refinement of 
the Gan-Gross-Prasad conjecture in the co-dimension one unitary case.
Later an  extension of  the work of Ichino-Ikeda and Harris 
to  the general  Bessel period case was formulated by Liu~\cite{Liu2}
and the one to the general Fourier-Jacobi period case for symplectic-metaplectic groups was formulated by
Xue~\cite{Xue2}.
%

In \cite{FM1} we investigated the Gross-Prasad conjecture
for Bessel periods for $\mathrm{SO}\left(2n+1\right)\times
\mathrm{SO}\left(2\right)$ when the character on $\mathrm{SO}\left(2\right)$
is trivial, i.e. the special Bessel periods case
and then, in the sequel~\cite{FM2},
we proved its refinement, i.e
the Ichino-Ikeda type precise $L$-value formula
 under the condition that the base field is totally real
and all components at archimedean places are discrete series representations.
As a corollary of our special value formula in \cite{FM2}, we obtained a proof of the long-standing
conjecture by B\"ocherer in \cite{Bo}, concerning
central critical values of imaginary quadratic twists
of spinor $L$-functions for holomorphic Siegel cusp forms
of degree two which are Hecke eigenforms, thanks to
the explicit calculations of the local integrals by
Dickson, Pitale, Saha and Schmidt~\cite{DPSS}.

In this paper, 
for $\left(\mathrm{SO}(5), \mathrm{SO}(2)\right)$,
we  vastly generalize the main results in \cite{FM1} and \cite{FM2}.
Namely we prove the Gross-Prasad conjecture and its refinement
for any Bessel period in the case of $(\mathrm{SO}(5), \mathrm{SO}(2))$.
As a corollary, we prove the generalized B\"ocherer conjecture
in the square-free case formulated in \cite{DPSS}.

Let us introduce some notation
and then state our main results precisely.
%
\subsection{Notation}\label{ss: notation}
Let $F$ be a number field. 
We denote its ring of adeles by $\mA_F$, which is mostly abbreviated as $\mA$ for simplicity. 
Let $\psi$ be a non-trivial character of $\mA \slash F$. 
For $a \in F^\times$, we write by $\psi^a$ the character of $\mA \slash F$ defined 
by $\psi^a(x) = \psi(ax)$. 
For a place $v$ of $F$, we denote by $F_v$ the completion of $F$ at $v$.
When $v$ is non-archimedean, we write 
by $\varpi_v$ and $q_v$ an uniformizer of $F_v$ 
and the cardinality of the residue field of $F_v$, respectively.
%

Let $E$ be a quadratic extension of $F$ and $\mA_E$ be its ring of adeles.
We denote by $x\mapsto x^\sigma$ the unique non-trivial automorphism of $E$ over $F$.
Let us denote by $\mathrm{N}_{E \slash F}$ the norm map from $E$ to $F$.
We choose $\eta\in E^\times$ such that $\eta^\sigma=-\eta$ and fix.
Let $d=\eta^2$.
We denote by  $\chi_E$ the quadratic character of $\mA^\times$ corresponding to the quadratic extension $E \slash F$.
We fix a character $\Lambda$ of $\mA_E^\times \slash E^\times$ whose restriction 
to $\mA^\times$ is trivial once and for all.
%
\subsection{Measures}\label{ss: measures}
Throughout the paper, for an algebraic group $\mathbf G$ 
defined over $F$, we write $\mathbf G_v$ for $\mathbf G\left(F_v\right)$,
the group of rational points of $\mathbf G$ over $F_v$,
and
we always take the measure $dg$ on $\mathbf G\left(\mA\right)$ to
be the Tamagawa measure
unless specified otherwise.
For each $v$, we take the self-dual measure with respect to $\psi_v$
on $F_v$.
Then recall that the product measure on $\mA$ is the self-dual measure
with respect to $\psi$ and  is also the Tamagawa measure since 
$\mathrm{Vol}\left(\mA\slash F\right)=1$.
For a unipotent algebraic group $\mathbf U$ defined over $F$,
we also specify the local measure $du_v$ on $\mathbf{U}(F_v)$
to be the measure corresponding to the gauge form defined over $F$, 
together with our choice of the measure on $F_v$,
at each place $v$ of $F$.
Thus in particular we have 
\[
du=\prod_v\, du_v\qquad\text{and}\qquad
\mathrm{Vol}\left(\mathbf U\left(F\right)\backslash
\mathbf U\left(\mA\right), du\right)=1.
\]
%
%
%
\subsection{Similitudes}
Various similitude groups appear in this article.
Unless there exists a fear of confusion,
we denote by  $\lambda\left(g\right)$
the similitude of an element $g$ of a similitude group
for simplicity.
%
%
%
%
%
%
%
%
%
%
%
%
%
%
\subsection{Bessel periods}\label{ss: bessel periods}
First we recall that when $V$ is a five dimensional vector space over $F$ equipped
with a non-degenerate symmetric bilinear form whose Witt index is 
at least one, 
there exists a quaternion algebra $D$ over $F$ such that
\begin{equation}\label{e: PG_D}
\mathrm{SO}\left(V\right)=\mathbb G_D
\end{equation}
where $\mathbb G_D=G_D\slash Z_D$,
$G_D$ is a similitude quaternionic unitary group  over $F$ defined by
\begin{equation}\label{e: G_D}
G_D(F): = \left\{g \in \mathrm{GL}_2(D) : {}^{t}\overline{g} \,\begin{pmatrix}0&1\\ 1& 0\end{pmatrix}\, g= \lambda(g ) \begin{pmatrix}0&1\\ 1&0 \end{pmatrix} , 
\, \lambda(g) \in F^\times \right\}
\end{equation}
and $Z_D$ is the center of $G_D$.
Here
\[
\overline{g}:=\begin{pmatrix}\overline{t}&\overline{u}\\
\overline{w}&\overline{v}\end{pmatrix}
\quad\text{for}\quad
g=\begin{pmatrix}t&u\\ w&v\end{pmatrix}
\in\mathrm{GL}_2\left(D\right)
\]
where denoted by
$x\mapsto\overline{x}$ for $x\in D$ is 
the canonical involution of $D$.
Also, we define a quaternionic unitary group  $G_D^1$ over $F$ by 
\[
\label{G_D^1}
G_D^1 := \left\{ g \in G_D : \lambda(g) = 1\right\}.
\]
Let 
\[
D^-:=\left\{x\in D: \mathrm{tr}_D\left(x\right)=0\right\}
\]
where $\mathrm{tr}_D$ denotes the reduced trace of $D$ over $F$.
We recall that when $D\simeq \mathrm{Mat}_{2\times 2}\left(F\right)$, 
$G_D$ is isomorphic to the similitude symplectic group $\mathrm{GSp}_2$ which we denote by $G$,
i.e.
\begin{equation}\label{gsp}
G\left(F\right):=\left\{g\in\mathrm{GL}_4\left(F\right):
{}^tg\begin{pmatrix}0&1_2\\-1_2&0\end{pmatrix}g=
\lambda\left(g\right)\begin{pmatrix}0&1_2\\-1_2&0\end{pmatrix},\,
\lambda\left(g\right)\in F^\times\right\}.
\end{equation}
Also, we define the symplectic group $\mathrm{Sp}_2$, which we denote by $G^1$, as
\[
\label{G^1}
G^1 := \left\{ g \in G : \lambda(g) =1 \right\}.
\] 
We denote $\mathrm{PGSp}_2 = G \slash Z_G$ by $\mathbb G$, where $Z_G$ denotes the center of $G$.
Thus when $D$ is split, $G_D\simeq G=\mathrm{GSp}_2$, $G_D^1 \simeq G^1 = \mathrm{Sp}_2$ and $\mathbb G_D\simeq \mathbb G=\mathrm{PGSp}_2$.

The Siegel parabolic subgroup $P_D$ of $G_D$ has
the Levi decomposition $P_D = M_{D} N_{D}$
where 
\[
\label{d:N_D}
M_{D}(F): = \left\{ \begin{pmatrix} x&0\\ 0&\mu \cdot x\end{pmatrix} : x \in D^\times, \mu \in F^\times \right\},
\quad 
N_{D}(F): =  \left\{ \begin{pmatrix} 1&u\\ 0&1\end{pmatrix} : u \in D^-\right\}.
\]

For $\xi\in D^-\left(F\right)$, let us define a character $\psi_\xi$ on $N_D\left(\mathbb A\right)$
by
\begin{equation}\label{e: psi_xi}
\psi_\xi\begin{pmatrix}1&u\\0&1\end{pmatrix}:=\psi\left(\mathrm{tr}_D\left(\xi u\right)
\right).
\end{equation}
We note that for $\begin{pmatrix}x&0\\0&\mu\cdot x\end{pmatrix}\in M_D\left(F\right)$, we have
\begin{equation}\label{e: psi-conjugation}
\psi_\xi
\left[
\begin{pmatrix}x&0\\0&\mu\cdot x\end{pmatrix}
\begin{pmatrix}1&u\\0&1\end{pmatrix}
\begin{pmatrix}x&0\\0&\mu\cdot x\end{pmatrix}^{-1}\right]
=\psi_{\mu^{-1}\cdot x^{-1}\xi x}\begin{pmatrix}1&u\\0&1\end{pmatrix}.
\end{equation}
%

Suppose that  
$F\left(\xi\right)\simeq E$.
Let us define a subgroup $T_\xi$ of $D^\times$ by
\begin{equation}\label{e: T_xi}
T_\xi:=\left\{x\in D^\times: x\,\xi \,x^{-1}=\xi\right\}.
\end{equation}
Then since $F\left(\xi\right)$ is a maximal commutative subfield of $D$, we have
\begin{equation}\label{e: maximal}
T_\xi(F)=F\left(\xi\right)^\times\simeq E^\times.
\end{equation}
We identify $T_\xi$ with the subgroup of $M_D$ given by
\begin{equation}\label{e: identify}
\left\{\begin{pmatrix}x&0\\0&x\end{pmatrix}:
x\in T_\xi\right\}.
\end{equation}
We note that by \eqref{e: psi-conjugation}, we have
\[
\psi_\xi\left(tnt^{-1}\right)=\psi_\xi\left(n\right)\quad
\text{for $t\in T_\xi\left(\mathbb A\right)$ and $n\in N_D\left(\mathbb A\right)$}.
\]
We define the Bessel subgroup $R_{\xi}$ of $G_D$ by
\begin{equation}\label{d: Bessel}
R_\xi:=T_\xi N_D.
\end{equation}
Then the Bessel periods defined below are indeed the periods in question in
the Gross-Prasad conjecture for $\left(\mathrm{SO}\left(5\right),
\mathrm{SO}\left(2\right)\right)$.
%
%
%
\begin{Definition}\label{d: bessel model}
Let $\pi$ be an irreducible cuspidal automorphic representation of 
$G_D\left(\mathbb A\right)$ whose central character is trivial
and $V_\pi$ its space of automorphic forms.
Let $\Lambda$ be a character of $\mathbb A_E^\times\slash E^\times$ whose restriction to $\mA^\times$ is trivial.
Let $\xi\in D^-\left(F\right)$ 
such that $F\left(\xi\right)\simeq E$.
Fix an $F$-isomorphism $T_\xi\simeq E^\times$
and regard $\Lambda$ as a character of $T_\xi\left(\mathbb A\right)
\slash T_\xi\left(F\right)$.
We define a character $\chi^{\xi,\Lambda}$ on 
$R_\xi\left(\mathbb A\right)$ by
\begin{equation}\label{d: Bessel char}
\chi^{\xi,\Lambda}
\left(tn\right):
=
\Lambda\left(t\right)\psi_\xi\left(n\right)
\quad
\text{for $t\in T_\xi\left(\mathbb A\right)$
and $n\in N_D\left(\mathbb A\right)$}.
\end{equation}

Then for $f\in V_\pi$,
we define $B_{\xi, \Lambda,\psi}\left(f\right)$,
the $\left(\xi, \Lambda,\psi\right)$-Bessel period of $f$, by
\begin{equation}\label{e: def of bessel period}
B_{\xi, \Lambda,\psi}\left(f\right):=
\int_{\mathbb A^\times R_\xi\left(F\right)
\backslash R_\xi\left(\mathbb A\right)}
f\left(r\right)\chi^{\xi,\Lambda}\left(r\right)^{-1}\,
dr.
\end{equation}

We say that $\pi$ has the $\left(\xi, \Lambda, \psi\right)$-Bessel 
period when the linear form $B_{\xi,\Lambda, \psi}$ is not 
identically zero on $V_\pi$.
\end{Definition}
%
\begin{Remark}\label{r: dependency}
Here we record the dependency of $B_{\xi,\Lambda,\psi}$
on the choices of $\xi$ and $\psi$.
First we note that 
for $\xi^\prime\in D^-\left(F\right)$, 
we have $F\left(\xi^\prime\right)\simeq E$ if and only if
\begin{equation}\label{e: s-n}
\xi^\prime=\mu\cdot  \alpha^{-1} \xi\alpha
\quad\text{for some  $\alpha\in D^\times\left(F\right)$ and $\mu\in F^\times$}
\end{equation}
by the Skolem-Noether theorem.
Suppose that $\xi^\prime\in D^-\left(F\right)$ satisfies \eqref{e: s-n} and
$\psi^\prime=\psi^a$ where $a\in F^\times$.
Let $m_0=\begin{pmatrix}\alpha&0\\ 0&a^{-1}\mu\cdot \alpha\end{pmatrix}
\in M_D\left(F\right)$.
Then by \eqref{e: psi-conjugation}, we have
\begin{align}\label{e: dependency}
B_{\xi, \Lambda, \psi}\left(\pi\left(m_0\right)f\right)
&=\int_{\mathbb A^\times
T_{\xi^\prime}\left(F\right)\backslash T_{\xi^\prime}\left(\mathbb A\right)}
\int_{N_D\left(F\right)\backslash N_D\left(\mathbb A\right)}
\\
\notag
&\qquad\qquad\qquad
f\left(t^\prime n^\prime\right)
\Lambda\left(t^{\prime}\right)^{-1}
\psi^\prime_{\xi^\prime}\left(n^\prime\right)\, dt^\prime\, dn^\prime
\\
\notag
&=
B_{\xi^\prime,\Lambda, \psi^\prime}\left(f\right)
\end{align}
where we identify
$T_{\xi^\prime}(F)$ with $E^\times$ via the $F$-isomorphism 
$F\left(\xi^\prime\right)\ni x\mapsto \alpha x\alpha^{-1}\in F\left(\xi\right)
\simeq E$.
\end{Remark}
%
\begin{Definition}\label{def of E,Lambda-Bessel period}
Let $\left(\pi, V_\pi\right)$ be an irreducible cuspidal automorphic representation of $G_D(\mA)$ whose central character is trivial.
Let $\Lambda$ be a character of $\mathbb A_E^\times\slash E^\times$
whose restriction to $\mathbb A^\times$ is trivial.
Then we say that $\pi$ has  the 
$\left(E,\Lambda\right)$-Bessel period
if there exist $\xi\in D^-\left(F\right)$
such that $F\left(\xi\right)\simeq E$
and a non-trivial character $\psi$ of $\mathbb A\slash F$
so that $\pi$ has the $\left(\xi,\Lambda,\psi\right)$-Bessel period.
This terminology is well-defined because of 
the relation \eqref{e: dependency}.
\end{Definition}
%
%
\subsection{Gross-Prasad conjecture}
\label{s:Gross-Prasad conjecture}
First we introduce the following definition which is inspired by
the notion of \emph{locally $G$-equivalence} in
Hiraga and Saito~\cite[p.23]{HiSa}.
%
\begin{Definition}
Let $\left(\pi, V_\pi\right)$ be an irreducible cuspidal automorphic representation of $G_D(\mA)$ whose central character is trivial.
Let $D^\prime$ be a quaternion algebra over $F$ and $(\pi^\prime, V_{\pi^\prime})$ an irreducible cuspidal automorphic representation 
of $G_{D^\prime}(\mA)$.
Then we say that $\pi$ is locally $G^+$-equivalent to $\pi^\prime$
if at almost all places  $v$ of $F$ where $D\left(F_v\right)\simeq D^\prime
\left(F_v\right)$, 
there exists a character $\chi_v$ of $G_D\left(F_v\right) \slash G_D\left(F_v\right)^+$ 
such that $\pi_v \otimes \chi_v \simeq \pi_v^\prime$.
Here 
\begin{equation}\label{e: G^+_D}
G_D\left(F\right)^+:=\left\{g\in G_D\left(F\right)
: \lambda\left(g\right)\in \mathrm{N}_{E\slash F}\left(E^\times\right)\right\}.
\end{equation}
\end{Definition}
%
\begin{Remark}
\label{rem loc ne}
When $\pi$ and $\pi^\prime$ have weak functorial lifts
to $\mathrm{GL}_4\left(\mA\right)$, say $\Pi$ and $\Pi^\prime$,
respectively,  the notion of locally $G^+$-equivalence is 
described simply
as the following.
Suppose that $\pi$ and $\pi^\prime$ are locally $G^+$-equivalent.
Then there exists a character $\omega$ of $G_D\left(\mA\right)$
such that $\pi\otimes\omega$ is nearly equivalent to $\pi^\prime$,
where $\omega$ may not be automorphic.
Since $\omega_v$ is either $\chi_{E_v}$ or trivial at almost all places $v$ of $F$,
we have $\mathrm{BC}_{E\slash F}\left(\Pi\right)\simeq
\mathrm{BC}_{E\slash F}\left(\Pi^\prime\right)$ where
$\mathrm{BC}_{E\slash F}$ denotes the base change lift to
$\mathrm{GL}_4\left(\mathbb A_E\right)$.
Then by Arthur-Clozel~\cite[Theorem~3.1]{AC}, we have
$\Pi\simeq \Pi^\prime$ or $\Pi^\prime\otimes\chi_E$.
Hence $\pi$ is nearly  equivalent to either $\pi^\prime$ or $\pi^\prime\otimes\chi_E$.
The converse is clear.
\end{Remark}
%
Then our first   main result is on the 
Gross-Prasad conjecture for $(\mathrm{SO}(5), \mathrm{SO}(2))$.
%
%
%
%
\begin{theorem}
\label{ggp SO}
Let $E$ be a quadratic extension of $F$.
Let $\left(\pi, V_\pi\right)$ be an irreducible cuspidal automorphic representation of $G_D\left(\mA\right)$ with a
trivial central character
and $\Lambda$ a character of $\mathbb A_E^\times\slash E^\times$
whose restriction to $\mathbb A^\times$ is trivial.
%
\begin{enumerate}
\item\label{theorem1-1-(1)}
Suppose that $\pi$ has the $\left(E,\Lambda\right)$-Bessel period.
Moreover assume that:
\begin{multline}\label{genericity}
\text{there exists a finite place $w$ of $F$ such that}
\\
\text{
$\pi_w$ and its local theta lift 
to $\mathrm{GSO}_{4,2}\left(F_w\right)$
are generic.}
\end{multline}
Here $\mathrm{GSO}_{4,2}$ denote the identity component of 
$\mathrm{GO}_{4,2}$, 
the similitude orthogonal group associated to
 the six dimensional orthogonal space 
 $(E, \mathrm{N}_{E \slash F}) \oplus \mathbb{H}^2$ over $F$
where $\mathbb H$ denotes the hyperbolic plane  over $F$.

Then  there exists a finite set $S_0$ of places of $F$ containing all
archimedean places of $F$ such that the partial $L$-function
\begin{equation}\label{e: partial non-vanishing}
L^S \left(\frac{1}{2}, \pi \times \mathcal{AI} \left(\Lambda \right) \right) \ne 0
\end{equation}
for any finite set $S$ of places of $F$ with $S\supset S_0$.
Here, $\mathcal{AI}\left(\Lambda \right)$ denotes the automorphic induction of $\Lambda$ from $\mathrm{GL}_1(\mA_E)$ to 
$\mathrm{GL}_2(\mA)$.
Moreover there exists a globally generic irreducible 
cuspidal automorphic representation $\pi^\circ$
of $G(\mA)$ which is locally $G^+$-equivalent to $\pi$.
%
\item\label{theorem1-1-(2)}
Assume that:
\begin{multline}\label{arthur classification}
\text{
the endoscopic classification of Arthur,}
\\
\text{ i.e. \cite[Conjecture~9.4.2, Conjecture~9.5.4]{Ar}
holds for $\mathbb G_{D_\circ}$ .
}
\end{multline}
Here $D_\circ$ denotes
an arbitrary quaternion algebra over $F$.

Suppose that $\pi$ has a generic Arthur parameter, namely the parameter 
is of the form $\Pi_0$ or $\Pi_1 \boxplus \Pi_2$
where 
$\Pi_i$ is an irreducible cuspidal automorphic representation 
of $\mathrm{GL}_4\left(\mA\right)$ for $i=0$
and of $\mathrm{GL}_2\left(\mA\right)$ for $i=1,2$, respectively,
such that $L(s, \Pi_i, \wedge^2)$ has a pole at $s=1$.

Then we have 
\begin{equation}
\label{e:L non-zero}
L \left(\frac{1}{2}, \pi \times \mathcal{AI} \left(\Lambda \right) \right) \ne 0
\end{equation}
if and only if there exists a pair $\left(D^\prime, \pi^\prime\right)$
where $D^\prime$ is a
quaternion algebra over $F$ containing $E$ and $\pi^\prime$
an irreducible cuspidal automorphic representation of $G_{D^\prime}$
which is nearly equivalent to $\pi$ such that 
$\pi^\prime$ has the $\left(E,\Lambda\right)$-Bessel period.

Moreover, when $\pi$ is tempered, the pair $\left(D^\prime, \pi^\prime\right)$
is uniquely determined.
\end{enumerate}
\end{theorem}
%
%
%
\begin{Remark}
\label{L-fct def rem}
In \eqref{e:L non-zero}, 
$L \left(s, \pi \times \mathcal{AI} \left(\Lambda\right) \right)$
 denotes the complete 
 $L$-function defined as the following.

 When $\mathcal{AI} \left(\Lambda \right)$ is not cuspidal, i.e. $\Lambda=\Lambda_0 \circ \mathrm{N}_{E \slash F}$ for a character $\Lambda_0$ of $\mA^\times \slash F^\times$,
we define 
\[
L\left(s,\pi \times \mathcal{AI} \left(\Lambda \right) \right) : = L\left(s,\pi \times \Lambda_0 \right) L\left(s,\pi \times \Lambda_0 \chi_E \right)
\]
where each factor on the right hand side is 
defined by the doubling method as in Lapid-Rallis~\cite{LR} or Yamana  \cite{Yam}.

When $\mathcal{AI} \left(\Lambda \right)$ is cuspidal, 
the partial $L$-function $L^S\left(s,\pi\times\mathcal{AI}\left(\Lambda \right)\right)$ may be defined by
Theorem~\ref{thm mero} in Appendix~\ref{appendix c} for a finite set $S$ of places of $F$ such that 
$\pi_v$ and $\Pi \left(\Lambda \right)_v$ are unramified at $v \not \in S$.
Further, we define the local $L$-factor at each place $v \in S$ by the 
local Langlands parameters for $\pi_v$ and $\Pi \left(\Lambda \right)_v$, where 
the local Langlands parameters are given by Gan-Takeda~\cite{GT11} for $\mathbb{G}(F_v)$ (also Arthur~\cite{Ar}), 
Gan-Tantono~\cite{GaTan} for $\mathbb{G}_D(F_v)$ 
and  Kutzko~\cite{Kutz} for $\mathrm{GL}_2(F_v)$ at finite places,  and,
 by Langlands~\cite{Lan} at archimedean places.

We note that the condition~\eqref{e: partial non-vanishing}
and   the condition~\eqref{e:L non-zero} are equivalent
from the definition of local $L$-factors when $\pi$ is tempered.
\end{Remark}
%
%
%
\begin{Remark}
Suppose that at a finite place $w$ of $F$,
the group $G_D\left(F_w\right)$ is split
and the representation $\pi_w$ is generic and tempered.
Then by Gan and Ichino~\cite[Proposition~C.4]{GI1},
the big theta lift of $\pi_w$ and the local theta lift of $\pi_w$ coincide.
Thus the genericity of the local theta lift of $\pi_w$
follows from Gan and Takeda~\cite[Corollary~4.4]{GT0}
for the dual pair $\left(G, \mathrm{GSO}_{3,3}\right)$
and from a local analogue of the computations in \cite[Section~3.1]{Mo}
for the dual pair $\left(G^+,\mathrm{GSO}_{4,2}\right)$, respectively.
Here 
\begin{equation}\label{e: G+}
G\left(F\right)^+:=\left\{g\in G: \lambda\left(g\right)\in
\mathrm{N}_{E\slash F}\left(E^\times\right)\right\}.
\end{equation}

When a local representation $\pi_w$ is unramified and tempered, $\pi_w$ is
generic as remarked in \cite[Remark~2]{FM1}. 
Hence the assumption~\eqref{genericity} is fulfilled  when 
$\pi$ is tempered.
\end{Remark}
In our previous paper~\cite{FM1}, 
Theorem~\ref{ggp SO} for the pair
$\left(\mathrm{SO}\left(2n+1\right),\mathrm{SO}\left(2\right)
\right)$ was proved when $\Lambda$ is trivial.
Meanwhile Jiang and Zhang~\cite{JZ} studied the Gross-Prasad
conjecture in a very general setting assuming 
the endoscopic classification of Arthur in general
by using  the twisted automorphic descent.
Though Theorem~\ref{ggp SO} is subsumed in \cite{JZ} as a special case,
we believe that our method, which is  different from theirs, has its own merits
because of its concreteness. 
We also  note that because of the temperedness of $\pi$, the uniqueness of the pair 
$\left(D^\prime,\pi^\prime\right)$ in Theorem~\ref{ggp SO} (2)
follows from the local Gan-Gross-Prasad conjecture for $(\mathrm{SO}(5), \mathrm{SO}(2))$ by 
Prasad-Takloo--Bighash~\cite[Theorem~2]{PT} (see also Waldspurger~\cite{Wal} in general case) at finite places and 
by Luo~\cite{Luo} at archimedean places. 
We shall give another proof of this uniqueness by reducing it 
to a similar assertion  in the unitary group case.
%
%
%
%
\subsection{Refined Gross-Prasad conjecture}
Let $(\pi, V_\pi)$ be an irreducible cuspidal tempered automorphic representation of $G_D\left(\mA\right)$ with trivial central character.
For $\phi_1, \phi_2 \in V_{\pi}$, we define the Petersson inner product
$\left(\phi_1,\phi_2\right)_{\pi}$ on $V_\pi$ by
\[
(\phi_1, \phi_2)_{\pi} = \int_{Z_D\left(\mathbb A\right) G_D(F) 
\backslash G_D\left(\mA\right)} \phi_1(g) \overline{\phi_2(g)} \,dg
\]
where $dg$ denotes the Tamagawa measure. 
Then at each place $v$ of $F$, we take a 
$G_D\left(F_v\right)$-invariant 
hermitian inner product on $V_{\pi_v}$ so that we have a decomposition
$\left(\,\, ,\,\,\right)_\pi=\prod_v\left(\,\,,\,\,\right)_{\pi_v}$.
In the definition of the Bessel period~\eqref{e: def of bessel period},
we take $dr=dt\,du$ where $dt$ and $du$ are the Tamagawa measures on $T_\xi\left(\mathbb A\right)$
and $N_D\left(\mathbb Z\right)$, respectively.
We take and fix the local measures $du_v$ and $dt_v$ so that
$du=\prod_v du_v$ and
\begin{equation}
\label{C_{S_D}}
dt = C_\xi \prod dt_v
\end{equation} 
where $C_\xi$ is a constant called the Haar measure constant
in \cite{II}.
Then  the local Bessel period
$\alpha^{\xi,\Lambda}_v:V_{\pi_v}\times V_{\pi_v}\to\mathbb C$ 
and the local hermitian inner product $\left(\,\,,\,\,\right)_{\pi_v}$
are defined
as in Section~\ref{s:def local bessel}.
%

Suppose that $D$ is not split.
Then by Li~\cite{JSLi},
there exists a pair $\left(\xi^\prime, \Lambda^\prime\right)$
such that $\pi$ has the $\left(\xi^\prime,\Lambda^\prime,
\psi\right)$-Bessel period. Here $\xi^\prime\in D^-\left(F\right)$
such that 
$E^\prime:=F\left(\xi^\prime\right)$ is a quadratic extension of $F$ and 
$\Lambda^\prime$ is a character on $\mathbb A_{E^\prime}^\times\slash
\mathbb A^\times {E^\prime}^\times$.
Then by Proposition~\ref{exist gen prp}, which is a consequence
of the proof of Theorem~\ref{ggp SO} (\ref{theorem1-1-(1)}), there exists an
irreducible cuspidal automorphic representation $\pi^\circ$ of
$G\left(\mathbb A\right)$ which is generic and locally
$G^+$-equivalent to $\pi$.
We take the functorial lift of $\pi^\circ$ to $\mathrm{GL}_4
\left(\mathbb A\right)$ by
Cogdell, Kim, Piatetski-Shapiro and Shahidi~\cite{CKPSS},
which is of the form $\Pi_1 \boxplus \cdots \boxplus \Pi_{\ell_0}$
with $\Pi_i$ an irreducible cuspidal automorphic representation
of $\mathrm{GL}_{m_i}\left(\mA\right)$ for each $i$.
Then we define an integer $\ell\left(\pi\right)$
by $\ell\left(\pi\right)=\ell_0$.
We note that $\pi^\circ$ may not be unique, but 
$\ell\left(\pi\right)$ does not depend 
on the choice of the pair $\left(\xi^\prime,\Lambda^\prime\right)$
by Proposition~\ref{exist gen prp} and Lemma~\ref{compo number}, \ref{comp number not dep on S},
and thus it depends only on $\left(\pi, V_\pi\right)$.
When $D$ is split, then $\pi$ has the functorial lift
to $\mathrm{GL}_4\left(\mathbb A\right)$ by
Arthur~\cite{Ar} (see also Cai-Friedberg-Kaplan~\cite{CFK}) and we define $\ell\left(\pi\right)$
in a similar way.

Our second main result is the refined Gross-Prasad conjecture
formulated by Liu~\cite{Liu2},
i.e.
the Ichino-Ikeda type explicit central value formula,
in the case of $\left(\mathrm{SO}\left(5\right),
\mathrm{SO}\left(2\right)\right)$.
%
%
%
%
%
\begin{theorem}
\label{ref ggp}
Let $\left(\pi, V_\pi\right)$ be an irreducible cuspidal tempered automorphic representation of $G_D(\mA)$ with a trivial central character.

Then 
 %
 %
  for any non-zero decomposable cusp form 
 $\phi=\otimes_v\,\phi_v\in V_\pi$, we have
\begin{multline}
\label{e: main identity}
 \frac{\left|B_{\xi, \Lambda,\psi}\left(\phi\right)\right|^2}{
( \phi,\phi )_\pi}
\\
 =2^{-\ell(\pi)}\,C_{\xi}
 \cdot
 \left(\prod_{j=1}^2\zeta_F\left(2j\right)\right)
 \frac{L\left(\frac{1}{2},\pi \times \mathcal{AI} \left(\Lambda \right) \right)
 }{
 L\left(1,\pi,\mathrm{Ad}\right)L\left(1,\chi_E\right)}
 \cdot\prod_v
 \frac{\alpha_v^\natural\left(\phi_v\right)}{
(\phi_v,\phi_v)_{\pi_v}}.
\end{multline}
Here $\zeta_F(s)$ denotes the complete zeta function of $F$ and
$\alpha_v^\natural\left(\phi_v\right)$ is defined by
\[
\alpha_v^\natural\left(\phi_v\right) =
 \frac{L\left(1,\pi_v,\mathrm{Ad}\right)L\left(1,\chi_{E,v}\right)}{L\left(1/2,\pi_v \times \Pi \left(\Lambda \right)_v \right) \prod_{j=1}^2\zeta_{F_v}\left(2j\right)
 }
\cdot  \alpha_{\Lambda_v, \psi_{\xi, v}}\left(\phi_v,\phi_v\right).
 \]
 We note that 
 $\displaystyle{\frac{\alpha_v^\natural\left(\phi_v\right)}{\left(\phi_v,\phi_v\right)_{\pi_v}}=1}$
for almost all places $v$ of $F$ by \cite{Liu2}.
\end{theorem}
%
\begin{Remark}
\label{rem Arthur + ishimoto}
Under the assumption~\eqref{arthur classification}, 
we have $|\mathcal{S}(\phi_\pi)| =2^{\ell(\pi)}$,
where $\phi_\pi$ denotes the Arthur parameter of $\pi$
and $\mathcal{S}\left(\phi_\pi\right)$  the centralizer of 
$\phi_\pi$ in the complex dual group $\hat{G}$.
Hence \eqref{e: main identity} coincides with the conjectural formula
in Liu~\cite[Conjecture~2.5 (3)]{Liu2}.
Thus when $D$ is split, i.e. $G_D\simeq G$, our theorem proves Liu's conjecture
since the assumption~\eqref{arthur classification} is indeed
fulfilled. 
After submitting this paper, Ishimoto posted a preprint~\cite{Ishimoto} on arXiv, in which he gives
the endoscopic classification of representations of non-quasi split orthogonal groups for generic Arthur parameters. 
Hence, our theorem proves  \cite[Conjecture~2.5 (3)]{Liu2} completely
in the case of $(\mathrm{SO}(5), \mathrm{SO}(2))$.
\end{Remark}
%
\begin{Remark}
Let $\pi_{\rm gen}$ denote the irreducible cuspidal globally generic automorphic representation of $G\left(\mathbb A\right)$
which has the same $L$-parameter as $\pi$.
When $\pi_v$ is unramified at any finite place $v$ of $F$,
Chen and Ichino~\cite{CI} proved
an explicit formula of the ratio  $L\left(1,\pi,\mathrm{Ad}\right)\slash
\left(\Phi_{\rm gen}, \Phi_{\rm gen} \right)$
for a suitably normalized  cusp form $\Phi_{\rm gen}$
in the space of $\pi_{\rm gen}$.
\end{Remark}
%
%
%
\begin{Remark}
In the unitary case,  a remarkable progress has been made
in the Gan-Gross-Prasad conjecture and its refinement for Bessel periods,
by studying the Jacquet-Rallis relative trace formula.
In the striking paper~\cite{BPLZZ} by Beuzart-Plessis, Liu, Zhang and Zhu,
a proof
in the co-dimension one case
for irreducible cuspidal tempered automorphic representations of unitary groups such that  their base change lifts are cuspidal was given
by establishing an ingenious  method to isolate the cuspidal spectrum.
In yet another striking paper by Beuzart-Plessis, Chaudouard and Zydor~\cite{BPCZ}, 
a proof  for all endoscopic cases in the co-dimension one setting was given
by a precise study of the relative trace formula.
Very recently, in a remarkable preprint by Beuzart-Plessis and Chaudouard ~\cite{BPC}, 
the above results are extended to arbitrary co-dimension cases.
 Thus the Gan-Gross-Prasad conjecture and its refinement for Bessel periods on unitary groups are now proved in general.
%

On the contrary, the orthogonal case in general is still open.
We note that, in the $\left(\mathrm{SO}\left(5\right),\mathrm{SO}\left(2\right)\right)$
case, the first author has formulated relative trace formulas 
to approach the formula~\eqref{e: main identity}
and proved the fundamental lemmas
in his joint work with Shalika~\cite{FS}, 
Martin~\cite{FuMa} and Matrin-Shalika~\cite{FuMaS}.
In order to deduce the $L$-value formula
from these relative trace formulas, 
several issues such as smooth transfer of test functions
must be overcome.
In the above mentioned co-dimension one unitary group case,
reductions to Lie algebras played crucial roles to solve similar
issues.
However Bessel periods in our case
involves integration over unipotent subgroups
and it is not clear, at least to the first author, how to make
the reduction to Lie algebras
work.
\end{Remark}
%
%
%
%
\begin{Remark}
In the co-dimension one orthogonal group case,
the refined Gross-Prasad conjecture  has been deduced
from the Waldspurger formula~\cite{Wal} in the 
$\left(\mathrm{SO}\left(3\right), \mathrm{SO}\left(2\right)\right)$ case
and from the Ichino formula~\cite{Ich2} in the $\left(\mathrm{SO}\left(4\right), \mathrm{SO}\left(3\right)\right)$ case,
respectively.
Gan and Ichino~\cite{GI0} studied the $\left(\mathrm{SO}\left(5\right), \mathrm{SO}\left(4\right)\right)$-case 
when the representation of $\mathrm{SO}\left(5\right)$ is a theta lift from $\mathrm{GSO}(4)$
by reduction to the $\left(\mathrm{SO}\left(4\right), \mathrm{SO}\left(3\right)\right)$ case.

Liu~\cite{Liu2} proved
Theorem~\ref{ref ggp} when $D$ is split
and $\pi$ is an endoscopic lift, i.e. a Yoshida lift,
by reducing it to the Waldspurger formula~\cite{Wal}.
The case when $\pi$ is a non-endoscopic
Yoshida lift was proved later by
Corbett~\cite{Co}  in a similar manner.
\end{Remark}
%
%
%
As a corollary of Theorem~\ref{ref ggp},
we prove the $\left(\mathrm{SO}(5), \mathrm{SO}(2)\right)$
case of the Gan-Gross-Prasad conjecture
in the form as stated in \cite[Conjecture~24.1]{GGP}.
%
\begin{corollary}
\label{ggp alpha ver}
Let $(\pi, V_\pi)$ be an irreducible cuspidal tempered automorphic 
representation of $G_D\left(\mA\right)$ with a
trivial central character.
Then the following three conditions are equivalent.
\begin{enumerate}
\item\label{1-ggp alpha ver}
The $\left(\xi,\Lambda, \psi \right)$-Bessel period does not
vanish on $\pi$.
\item\label{2-ggp alpha ver} 
$L \left(\frac{1}{2}, \pi \times \mathcal{AI}\left(\Lambda \right) \right) \ne 0$ 
and  the local Bessel period $\alpha_{\Lambda_v, \psi_{\xi, v}}\not\equiv 0$ on $\pi_v$
at any place $v$ of $F$.
\item
$L \left(\frac{1}{2}, \pi \times \mathcal{AI} \left(\Lambda \right) \right) \ne 0$ 
and  $\mathrm{Hom}_{R_{\xi,v}}\left(\pi_v,\chi^{\xi,\Lambda}_v\right)
\ne\left\{0\right\}$
at any place $v$ of $F$.
\end{enumerate}
\end{corollary}
\begin{Remark}
The equivalence between the conditions
(\ref{1-ggp alpha ver}) and (\ref{2-ggp alpha ver}) is immediate
from Theorem~\ref{ref ggp}.
The equivalence
%
\begin{equation}\label{equiv wald}
\alpha_{\Lambda_v, \psi_{\xi, v}}\not\equiv 0
\Longleftrightarrow
\mathrm{Hom}_{R_{\xi,v}}\left(\pi_v,\chi^{\xi,\Lambda}_v\right)
\ne\left\{0\right\}
\end{equation}
is proved by Waldspurger~\cite{Wal12}
at any non-archimedean place $v$
and by Luo~\cite{Luo} recently
at any archimedean place $v$, respectively.
\end{Remark}
\subsection{Method}\label{ss:method}
In \cite{FM1} and \cite{FM2} we used
the theta correspondence for the dual pair
$\left(\mathrm{SO}\left(2n+1\right),
\mathrm{Mp}_n\right)$.

The main tool in \cite{FM1} was the 
pull-back formula by the first author~\cite{Fu}
for the Whittaker period on $\mathrm{Mp}_n$,
which is expressed by a certain integral
involving the \emph{Special} Bessel period on $\mathrm{SO}\left(2n+1
\right)$.
This forced us the restriction that the character
$\Lambda$ on $\mathrm{SO}\left(2\right)$ is trivial.

In \cite{FM2},
to prove  the refined Gross-Prasad conjecture for
$\left(\mathrm{SO}\left(2n+1\right),\mathrm{SO}\left(2\right)\right)$
when $\Lambda$ is trivial, 
the following additional restrictions were necessary:
\begin{enumerate}
\item\label{c2} The base field $F$ is totally real
and at every archimedean place $v$ of $F$,
the representation $\pi_v$ is a discrete series representation.
\item\label{c3}
The assumption \eqref{arthur classification}.
\end{enumerate}
Additional main tool needed in \cite{FM2} was the Ichino-Ikeda type 
formula for the Whittaker periods on $\mathrm{Mp}_n$
by Lapid and Mao~\cite{LM17},
which imposed on us the  condition~(\ref{c2}).
In fact, their proof was to reduce the global identity
to certain local identities.
They proved the local identities in general at 
non-archimedean places.
On the other hand, at archimedean places, 
their proof was to note the
equivalence between their local identities
and the formal degree conjecture
by Hiraga-Ichino-Ikeda~\cite{HII1, HII2}
and then to prove the latter when $\pi$ is 
a discrete series representation.
Our proof in \cite{FM2} was to reduce to the case when $\pi$ has the special
Bessel period by the assumption~\eqref{arthur classification} and to combine
these two main tools with the Siegel-Weil formula.

It does not seem plausible that  a straightforward generalization of
the method of \cite{FM1} and \cite{FM2} would allow us to remove
these restrictions.
Thus  we need to adopt a new strategy in this paper.

Our main method here is again theta correspondence but we use it differently
and in a more intricate way.
First we consider the quaternionic dual pair $\left(G^+_D,\mathrm{GSU}_{3,D}\right)$
where $\mathrm{GSU}_{3,D}$ denotes the identity component
of the similitude quaternion unitary group $\mathrm{GU}_{3,D}$
defined by \eqref{def of gu_3,D} and
$G^+_D$ defined by \eqref{e: G^+_D}.
Then we recall the accidental isomorphism
\begin{equation}\label{e: accidental 3,D}
\mathrm{PGSU}_{3,D}\simeq \mathrm{PGU}_{4,\varepsilon}
\end{equation}
when $D\simeq D_\varepsilon$ given by \eqref{d: quaternion}
and $\mathrm{GU}_{4,\varepsilon}$
is the 
similitude unitary group defined by \eqref{d: unitary GD}.
Hence we have
\begin{equation}\label{e: accidental U_4}
\mathrm{GU}_{4,\varepsilon}\simeq
\begin{cases}\mathrm{GU}_{2,2},&\text{when $D$ is split,
i.e. $\varepsilon\in\mathrm{N}_{E\slash F}
\left(E^\times\right)$};
\\
\mathrm{GU}_{3,1},&\text{when $D$ is non-split,
i.e. $\varepsilon\notin\mathrm{N}_{E\slash F}
\left(E^\times\right)$ }.
\end{cases}
\end{equation}
Thus  our theta correspondence  for $\left(G^+_D,\mathrm{GSU}_{3,D}\right)$
induces
a correspondence for the pair $\left(\mathbb G_D,\mathrm{PGU}_{4,\varepsilon}\right)$.
Then we note that the pull-back of a certain Bessel period on
$\mathrm{PGU}_{4,\varepsilon}$ is an integral involving the 
$\left(\xi,\Lambda, \psi\right)$-Bessel period on $G_D$.

Theorem~\ref{ggp SO} is reduced essentially to the 
Gan-Gross-Prasad conjecture for the Bessel periods on $\mathrm{GU}_{4,\varepsilon}$,
which we proved in \cite{FM3} using the theta correspondence for the pair 
$\left(\mathrm{GU}_{4,\varepsilon},\mathrm{GU}_{2,2}\right)$.

Similarly Theorem~\ref{ref ggp} is reduced to the
refined Gan-Gross-Prasad conjecture for 
the Bessel periods on $\mathrm{GU}_{4,\varepsilon}$.
For the reader's sake,
here we  present an outline of the proof
when the $\left(\xi,\Lambda,\psi\right)$-Bessel period
does not vanish.
Note that in the following paragraph the notation used  is provisionally 
and the argument is not rigorous since 
our intention here is to 
present a rough sketch of the main idea.

Let $\left(\pi, V_\pi\right)$ be an irreducible cuspidal tempered automorphic representation of $G_D(\mA)$ with a trivial central character.
Suppose that the $\left(\xi,\Lambda,\psi\right)$-Bessel period,
which we denote by $B$,
does not vanish on $\pi$.
Let $\theta\left(\pi\right)$ be the theta lift of $\pi$
to $\mathrm{GSU}_{3,D}$.
When $G_D=G$ and the theta lift of $\pi$ to 
$\mathrm{GSO}_{3,1}$ is non-zero, $\theta\left(\pi\right)$
is not cuspidal but
the explicit formula \eqref{e: main identity} has been already
proved by
Corbett~\cite{Co}.
Thus suppose otherwise. Then $\theta\left(\pi\right)$ is a non-zero
irreducible cuspidal tempered automorphic representation.
The pull-back of a certain Bessel period, which we denote by
$\mathcal B$ on $\mathrm{GSU}_{3,D}$ is written
as an integral involving $B$.
As in our previous paper~\cite{FM2}, the explicit formula
for $B$ is reduced to the one for $\mathcal B$,
which we obtain in the following steps.
\begin{enumerate}
\item
Via the isomorphism~\eqref{e: accidental 3,D},
regard $\theta\left(\pi\right)$ as an automorphic 
representation of $\mathrm{GU}_{4,\varepsilon}$
and then consider its theta lift $\theta_\Lambda\left(\theta\left(
\pi\right)\right)$, which depends on $\Lambda$, 
to $\mathrm{GU}_{2,2}$.
The temperedness of $\pi$ implies that 
$\theta_\Lambda\left(\theta\left(
\pi\right)\right)$ is an irreducible cuspidal automorphic representation
of $\mathrm{GU}_{2,2}$.
Then the pull-back of a certain Whittaker period
$\mathcal W$ on $\mathrm{GU}_{2,2}$
is written as an integral involving the Bessel period $\mathcal B$.
Then in \cite{FM3}, it is shown that the explicit formula
for $\mathcal B$ follows from the one for $\mathcal W$.
Thus we are reduced to show the explicit formula for $\mathcal W$.
\item
Via the isomorphism $\mathrm{PGU}_{2,2}\simeq\mathrm{PGSO}_{4,2}$, regard $\theta_\Lambda\left(\theta\left(
\pi\right)\right)$ as an automorphic representation of $\mathrm{GSO}_{4,2}$.
Let $\pi^\prime$ be the theta lift of $\theta_\Lambda\left(\theta\left(
\pi\right)\right)$ to $G=\mathrm{GSp}_2$.
Then it is shown that $\pi^\prime$ is a globally generic cuspidal automorphic
representation of $G$ and indeed the pull-back of the
Whittaker period $W$ on $G$
is expressed as an integral involving $\mathcal W$.
Hence we are  reduced to the explicit formula for $W$.
\item
Since the theta lift of the globally generic cuspidal automorphic
representation $\pi^\prime$ of $G$ 
to either $\mathrm{GSO}_{2,2}$
or $\mathrm{GSO}_{3,3}$ is non-zero and cuspidal,
we are further reduced to the explicit formulas
for the Whittaker periods on $\mathrm{PGSO}_{2,2}$
and $\mathrm{PGSO}_{3,3}$
by the pull-back computation.
\item
Recall the accidental isomorphisms
$\mathrm{PGSO}_{2,2}\simeq \mathrm{PGL}_2\times
\mathrm{PGL}_2$,
$\mathrm{PGSO}_{3,3}\simeq\mathrm{PGL}_4$.
Since the explicit formula for the Whittaker period
on $\mathrm{PGL}_n$ is already proved by Lapid
and Mao~\cite{LM}, we are done.
\end{enumerate}
%
%
%
%
%
%
\begin{Remark}
Though we only consider the case when $\mathrm{SO}\left(2\right)$
is non-split in this paper,  the split case is proved by a similar argument
as follows.
First we note that $D$ is necessarily split when $\mathrm{SO}\left(2\right)$
is split and hence $G_D\simeq G$.
If the theta lift to $\mathrm{GSO}_{2,2}$ is non-zero,
it is a Yoshida lift and Liu~\cite{Liu2} proved the explicit formula.
Suppose otherwise. Then the theta lift to $\mathrm{GSO}_{3,3}$
is non-zero and cuspidal.
The pull-back of a certain Bessel period on $\mathrm{GSO}_{3,3}$
is an integral involving the split Bessel period on $G$  (see Section~\ref{sp4 so42}).
We recall the accidental isomorphism $\mathrm{PGSO}_{3,3}
\simeq \mathrm{PGL}_4$.
We consider the theta correspondence for the pair
$\left(\mathrm{GL}_4,\mathrm{GL}_4\right)$
instead of $\left(\mathrm{GU}_{4,\varepsilon},\mathrm{GU}_{4,\varepsilon}\right)$
in the non-split case.
Then the  pull-back computation may be interpreted
as expressing the pull-back of the Whittaker period on $\mathrm{GL}_4$
as an integral involving the Bessel period on $\mathrm{GSO}_{3,3}$,
which is given in \cite{FM3}.
Thus as in the non-split case, we are  reduced to the Ichino-Ikeda type
explicit formula for the Whittaker period on $\mathrm{GL}_4$.
\end{Remark}
%
%
%
Here is the statement of the theorem in the split case.
\begin{theorem}
\label{main thm split}
Let $(\pi, V_\pi)$ be an irreducible cuspidal automorphic representation of $G(\mA)$ 
with trivial central character.
Suppose that $D$ is split and  the  Arthur parameter of $\pi$ is generic.

Let $\xi\in D^-\left(F\right)$ such that $F\left(\xi\right)\simeq F\oplus F$ and
fix an $F$-isomorphism $T_\xi\simeq F^\times\times F^\times$.
For a character $\Lambda$ of $\mathbb A^\times\slash F^\times$,
we also denote by $\Lambda$ the character of $T_\xi\left(\mathbb A\right)$
defined by $\Lambda\left(a,b\right):=\Lambda\left(ab^{-1}\right)$.

The following assertions hold.
\begin{enumerate}
\item
The $\left(\xi, \Lambda,\psi\right)$-Bessel period does not vanish on $V_\pi$
if and only if $\pi$ is generic and $L\left(\frac{1}{2},\pi\times\Lambda\right)\ne 0$.
Here we note that 
$L\left(\frac{1}{2},\pi\times\Lambda^{-1}\right)$ is the complex
conjugate of $L\left(\frac{1}{2},\pi\times\Lambda\right)$
since $\pi$ is self-dual.
\item
Further assume that $\pi$ is tempered.
Then for any non-zero decomposable cusp form 
 $\phi=\otimes_v\,\phi_v\in V_\pi$, we have
\begin{multline*}
 \frac{\left|B_{\xi, \Lambda,\psi}\left(\phi\right)\right|^2}{
( \phi,\phi )_\pi}
 =2^{-\ell(\pi)}\,C_{\xi}
 \cdot
 \left(\prod_{j=1}^2\zeta_F\left(2j\right)\right)
 \\
 \times
 \frac{L\left(\frac{1}{2},\pi \times \Lambda \right)L\left(\frac{1}{2},\pi \times \Lambda^{-1} \right)
 }{
 L\left(1,\pi,\mathrm{Ad}\right)\zeta_F(1)}
 \cdot\prod_v
 \frac{\alpha_v^\natural\left(\phi_v\right)}{
(\phi_v,\phi_v)_{\pi_v}}
\end{multline*}
where $\zeta_F\left(1\right)$ stands for $\mathrm{Res}_{s=1}\, \zeta_F(s)$.
\end{enumerate}
\end{theorem}
%
%
%

%
%
%
%
%
%
%
%
%
\subsection{Generalized B\"ocherer conjecture}
Thanks to the meticulous local computation by Dickson, Pitale, Saha and Schmidt~\cite{DPSS},
Theorem~\ref{ref ggp} implies the generalized B\"ocherer conjecture.
For brevity we only state the scalar valued full modular case here
in the introduction.
Indeed  a  more general version shall be proved in \ref{generalized boecherer statement}
   as Theorem~\ref{t: vector valued boecherer}.
%
%
%
%
\begin{theorem}
\label{Boecherer:scalar}
Let $\varPhi$ be a holomorphic Siegel cusp form of degree two and weight $k$
with respect to $\mathrm{Sp}_2\left(\mathbb Z\right)$
which is a Hecke eigenform and $\pi\left(\varPhi\right)$
the associated automorphic representation 
of $\mathbb{G}\left(\mathbb A_\mQ\right)$.
Let 
\begin{equation}\label{e: Fourier}
\varPhi\left(Z\right)=\sum_{T>0}a\left(\varPhi, T\right)
\exp\left[2\pi\sqrt{-1}\operatorname{tr}\left(TZ\right)\right],
\,\, Z\in\mathfrak H_2,
\end{equation}
be the Fourier expansion of $\varPhi$ where $T$ runs over semi-integral positive definite 
two by two symmetric matrices 
and $\mathfrak H_2$ denotes the Siegel upper half space of degree two.

Let $E$ be an imaginary quadratic extension of $\mathbb Q$.
We denote by $-D_E$ its  discriminant, $\mathrm{Cl}_E$ its ideal class group
and $w\left(E\right)$ the number of distinct roots of unity in $E$.
In \eqref{e: Fourier}, when $T^\prime={}^t\gamma T\gamma$
for some $\gamma\in\mathrm{SL}_2\left(\mathbb Z\right)$,
we have 
$a\left(\varPhi, T^\prime\right)=a\left(\varPhi,T\right)$.
By the Gauss composition law, we may naturally identify the 
$\mathrm{SL}_2\left(\mathbb Z\right)$-equivalence classes of binary quadratic forms
of discriminant $-D_E$ with the elements of $\mathrm{Cl}_E$.
Thus the notation  $a(\varPhi, c)$ for $c \in \mathrm{Cl}_E$ makes sense.
For a character $\Lambda$ of $\mathrm{Cl}_E$,
we define $\mathcal{B}_\Lambda\left(\varPhi , E\right)$ by
\[
\label{mathcal B Phi E}
\mathcal{B}_\Lambda\left(\varPhi , E\right): = w\left(E\right)^{-1} \cdot \sum_{c \in \mathrm{Cl}_E} a\left(\varPhi, c\right) \Lambda^{-1} \left(c\right).
\]

Suppose that $\varPhi$ is not a Saito-Kurokawa lift.
Then we have
\begin{equation}
\label{intro B conj}
\frac{|\mathcal{B}_\Lambda(\varPhi , E)|^2}{\langle \varPhi, \varPhi \rangle} =2^{2k-4} \cdot D_E^{k-1} \cdot \frac{L\left(\frac{1}{2},\pi \left( \varPhi\right) 
\times \mathcal{AI} \left(\Lambda \right) \right)
 }{
 L\left(1,\pi \left(\varPhi \right),\mathrm{Ad}\right)}.
\end{equation}
Here 
\[
\langle\varPhi,\varPhi\rangle=\int_{\mathrm{Sp}_2\left(\mathbb Z\right)
\backslash \mathfrak H_2}
\left|\varPhi\left(Z\right)\right|^2\det\left(Y\right)^{k-3}\,dX\, dY
\quad\text{where $Z=X+\sqrt{-1}\,Y$.}
\]
\end{theorem}
\begin{Remark}
In Theorem~\ref{t: vector valued boecherer},
we prove \eqref{intro B conj} allowing $\varPhi$ to have a square-free level
and to be vector-valued.
Moreover, assuming the temperedness of  $\pi\left(\varPhi\right)$,
the weight $2$ case, which is of significant interest because of 
the modularity conjecture for abelian surfaces,  is also included.

The formula~\eqref{intro B conj} and 
its generalization~\eqref{e: vector valued boecherer}
are expected to have a broad spectrum
of interesting applications both arithmetic and analytic.
Some of the examples are
 \cite{Blo}, \cite[Section~3]{DPSS}, \cite{Dummigan},\cite{HY},
 \cite{Saha} and \cite{Waibel}.
\end{Remark}
%
%
%
%
%
\subsection{Organization of the paper}
This paper is organized as follows.
In Section~2, we introduce some more notation
and define local and global Bessel periods.
In Section~3, we carry out the pull-back computation of Bessel periods.
In Section~4, we shall prove Theorem~\ref{ggp SO}
using the results in Section~3.
We also note
 some consequences of our proof of Theorem~\ref{ggp SO} (\ref{theorem1-1-(1)}), 
 which will be used in the 
proof of Theorem~\ref{ref ggp} later.
In Section~5, we  recall the 
Rallis inner product formula for similitude groups.
In Section~6, we will give an explicit formula for
 Bessel periods on $\mathrm{GU}_{4,\varepsilon}$ in certain cases
as explained in our strategy for the proof of Theorem~\ref{ref ggp}.
In Section~7, we complete our proof of Theorem~\ref{ref ggp}.
In Section~8, we 
prove the generalized B\"ocherer conjecture,
including the vector valued case.
In Appendix~\ref{appendix A}, we will give an explicit formula of Whittaker periods for irreducible cuspidal tempered automorphic
representations of $G$.
In Appendix~\ref{s:e comp}, we compute the local Bessel periods explicitly
for representation of $G\left(\mR\right)$
corresponding to vector valued
holomorphic Siegel modular forms. This result is used in Section~8.
In Appendix~\ref{appendix c}, we consider  the meromorphic continuation of the $L$-function for 
$\mathrm{SO}\left(5\right) \times \mathrm{SO}\left(2\right)$.
\subsection*{Acknowledgement}
This paper was partly written while the second author stayed at National University of Singapore. He would like to thank the people at
NUS for their warm hospitality. 
The authors would like to thank the anonymous referee for his/her careful reading of the earlier version of the manuscript
and providing many helpful comments and suggestions.
%
%
%
%
%
%
%
%
%
%
%
%
%
%
%
%
%
%
%
%
\section{Preliminaries}
%
%
%
%
%
%
%
%
%
%
%
%
%
%
\subsection{Groups}
\label{ss:groups}
%
\subsubsection{Quaternion algebras}\label{ss: quaternion}
Let $X(E : F)$ denote the set of $F$-isomorphism classes of central simple algebras over $F$
containing $E$. 
Then we recall that the map $\varepsilon \mapsto D_{\varepsilon}$ gives a 
bijection between $F^\times \slash \mathrm{N}_{E \slash F}(E^\times)$
and $X(E : F)$ (see \cite[Lemma~1.3]{FS}) where
\begin{equation}\label{d: quaternion}
D_\varepsilon := \left\{ \begin{pmatrix}a&\varepsilon b\\ b^\sigma&a^\sigma \end{pmatrix} : a, b \in E \right\}
\quad\text{for $\varepsilon\in F^\times$}.
\end{equation}
Here we regard $E$ as a subalgebra of $D_\varepsilon$ by
\[
E \ni a \mapsto \begin{pmatrix}a&0\\ 0&a^\sigma \end{pmatrix} \in D_\varepsilon.
\]
We also note that  $D_\varepsilon\simeq \mathrm{Mat}_{2 \times 2}\left(F\right)$
when $\varepsilon\in \mathrm{N}_{E\slash F}\left(E^\times\right)$.
The canonical involution $D_\varepsilon\ni x\mapsto \bar{x}\in D_\varepsilon$
is given by
\[
\bar{x}=
\begin{pmatrix}a^\sigma&-\varepsilon b\\ -b^\sigma&a\end{pmatrix}
\quad\text{for $x=\begin{pmatrix}a&\varepsilon b\\ b^\sigma&a^\sigma\end{pmatrix}$}.
\]
We denote  the reduced  trace of $D$ by 
$\mathrm{tr}_D$.
%
%
%
\subsubsection{Orthogonal groups}\label{ss: orthogonal}
For a non-negative integer $n$, a symmetric matrix $S_n\in
\mathrm{Mat}_{(2n+2) \times (2n+2)}\left(F\right)$ 
is defined inductively by
\begin{equation}\label{d: symmetric}
S_0: = \begin{pmatrix}2&0\\ 0&-2d \end{pmatrix} \quad \text{and} \quad S_n: = \begin{pmatrix}0&0&1\\ 0&S_{n-1}&0\\ 1&0& 0\end{pmatrix}\quad
\text{for $n\ge 1$}.
\end{equation}
We recall that $E=F\left(\eta\right)$ where $\eta^2=d$.
Then we write 
the corresponding orthogonal group,
the special orthogonal group
and the similitude orthogonal group by
\begin{equation}\label{d: orthogonal groups}
\mathrm{O}\left(S_n\right)=
\mathrm{O}_{n+2,n},\quad
\mathrm{SO}\left(S_n\right)=
\mathrm{SO}_{n+2,n}
\quad
\text{and}\quad
\mathrm{GO}\left(S_n\right)=\mathrm{GO}_{n+2,n},
\end{equation}
respectively.
Let $\mathrm{GSO}_{n+2,n}$ denote
the identity component of $\mathrm{GO}_{n+2,n}$.
Thus
\begin{equation}\label{d: GSO_n+2,n}
\mathrm{GSO}_{n+2,n}\left(F\right) = \{ g \in \mathrm{GO}_{n+2,n}\left(F\right)
 : \det (g) = \lambda(g)^{n+1} \}
\end{equation}
where
\begin{equation}\label{d: GO_n+2,n}
\mathrm{GO}_{n+2,n}(F) = \left\{ g \in \mathrm{GL}_{2n+2}(F) : 
{}^{t}g\, S_{n} \,g
= \lambda(g)S_n, \,\lambda(g) \in F^\times \right\}.
\end{equation}

For a positive integer $n$, we denote by $J_{2n}$ the $2n\times 2n$ symmetric matrix
with ones on the non-principal diagonal and zeros elsewhere,
i.e.
\begin{equation}\label{d: J_m}
J_2=\begin{pmatrix}0&1\\1&0\end{pmatrix}
\quad\text{and}\quad
J_{2\left(n+1\right)}=
\begin{pmatrix}0&0&1\\0&J_{2n}&0\\1&0&0\end{pmatrix}
\quad\text{for $n\ge 1$}.
\end{equation}
Then the similitude orthogonal group $\mathrm{GO}_{n,n}$ is defined by
\begin{equation}\label{d: GO_n,n}
\mathrm{GO}_{n,n}\left(F\right):=
\left\{g \in\mathrm{GL}_{2n}\left(F\right): {}^{t}g\, J_{2n}\, g = \lambda(g) J_{2n},
\,\lambda\left(g\right)\in F^\times \right\}
\end{equation}
and we denote by $\mathrm{GSO}_{n,n}$
its identity component, which is given by 
\begin{equation}\label{d: GSO_n,n}
\mathrm{GSO}_{n,n}\left(F\right) = \{ g \in \mathrm{GO}_{n,n}\left(F\right) : \det (g) = \lambda(g)^{n} \}.
\end{equation}
%
%
%
\subsubsection{Quaternionic unitary groups}\label{ss: quaternionic unitary groups}
Let $D$ be a quaternion algebra over $F$ containing $E$.
Recall that $G_D$ denotes the similitude quaternionic unitary group of degree $2$ defined
by \eqref{e: G_D}.

We  define a similitude quaternionic unitary group $\mathrm{GU}_{3,D}$ of degree $3$  by
\begin{equation}\label{def of gu_3,D}
\mathrm{GU}_{3,D}(F):= \left\{g \in \mathrm{GL}_3(D) : {}^t \bar{g} \,{\bf J}_\eta\, g = \lambda(g) {\bf J}_\eta, \,
\lambda\left(g\right)\in F^\times\right\}
\end{equation}
where we define a skew-hermitian matrix $ {\bf J}_\eta$ by 
\begin{equation}\label{d: j_eta}
 {\bf J}_\eta := \begin{pmatrix} 0&0&\eta\\0 &\eta&0\\ \eta&0&0\end{pmatrix}.
\end{equation}
Here $\bar{A}=\left(\bar{a}_{ij}\right)$ for $A=\left(a_{ij}\right)\in\mathrm{Mat}_{m \times n}\left(D\right)$.
Let us denote by $\mathrm{GSU}_{3,D}$ the identity component of $\mathrm{GU}_{3,D}$.
Then unlike the orthogonal case, as noted in
 \cite[p.21--22]{MVW},
we have 
\[
\mathrm{GSU}_{3,D}(F) = \mathrm{GU}_{3,D}(F)
\]
and
\[
\text{$\mathrm{GSU}_{3,D}(F_v) = \mathrm{GU}_{3,D}(F_v)$
when $D \otimes_F F_v$ is not split.}
\]
Moreover when $D\otimes_F F_v$ is split at a place $v$ of $F$, we have
\begin{equation}\label{e: GU3D cases}
\mathrm{GU}_{3, D}(F_v)
\simeq
\begin{cases}
\mathrm{GO}_{4,2}(F_v)
&\text{if $E\otimes F_v$ is a quadratic extension of $F_v$};
\\
\mathrm{GO}_{3,3}(F_v)&\text{if $E\otimes F_v\simeq F_v\oplus F_v$.}
\end{cases}
\end{equation}

We also define $\mathrm{GU}_{1,D}$ by
\begin{equation}\label{d: gu_1,d}
\mathrm{GU}_{1,D}(F):= \left\{\alpha\in D^\times :\bar{\alpha} \eta \alpha
 = \lambda\left(\alpha\right) \eta,\,\lambda\left(\alpha\right)
 \in F^\times \right\}
\end{equation}
and denote its identity component by $\mathrm{GSU}_{1,D}$.
Then we note that
\begin{align}\label{d: gsu1,d}
\mathrm{GSU}_{1,D}\left(F\right)&=\left\{\alpha\in D^\times :\bar{\alpha} \eta \alpha
 = \mathrm{n}_D\left(\alpha\right) \eta \right\}
 \\
 \notag
 &=\left\{x\in D^\times\mid x\eta=\eta x\right\}=T_\eta
 \end{align}
 where $T_\eta$ is defined by \eqref{e: T_xi} with $\xi=\eta$
 and $\mathrm{n}_D$ denotes the reduced norm of $D$.
%
%
%
%
\subsubsection{Unitary groups}\label{ss: unitary}
Suppose that  $D=D_\varepsilon$ defined by \eqref{d: quaternion}.
Then we define $\mathrm{GU}_{4,\varepsilon}$ a similitude unitary group of degree $4$ by
\begin{equation}\label{d: unitary GD}
\mathrm{GU}_{4,\varepsilon}\left(F\right):=
\left\{ g \in \mathrm{GL}_4(E) : {}^{t}g^\sigma \mathcal J_\varepsilon g = \lambda(g) \mathcal J_\varepsilon,\,
\lambda\left(g\right)\in F^\times \right\}
\end{equation}
where we define a hermitian matrix $\mathcal J_\varepsilon$ by
\[
\mathcal J_\varepsilon := \begin{pmatrix} 0&0&0&1\\ 0&-1&0&0\\ 0&0&\varepsilon&0\\ 1&0&0&0\end{pmatrix}.
\]
Here $A^\sigma=\left(a_{ij}^\sigma\right)$
for $A=\left(a_{ij}\right)\in\mathrm{Mat}_{m \times n}\left(E\right)$.
Then we have
\begin{equation}\label{e: gu(2,2) or gu(3,1)}
\mathrm{GU}_{4,\varepsilon}\simeq
\begin{cases}\mathrm{GU}_{2,2},&\text{when $D$ is split,
i.e. $\varepsilon\in \mathrm{N}_{E\slash F}\left(E^\times\right)$};
\\
\mathrm{GU}_{3,1},&\text{when $D$ is non-split,
i.e. $\varepsilon\notin \mathrm{N}_{E\slash F}\left(E^\times\right)$}.
\end{cases}
\end{equation}

We also define $\mathrm{GU}_{2,\varepsilon}$ a similitude unitary group
of degree $2$ by
\begin{multline}\label{d: GU2D}
\mathrm{GU}_{2,\varepsilon}(F) := \left\{ g \in \mathrm{GL}_{2}(E) : {}^{t}g^\sigma J_\varepsilon g
=\lambda(g)  J_\varepsilon, \,\lambda(g) \in F^\times \right\}
\\
\text{where}\quad
J_\varepsilon=\begin{pmatrix}-1&0\\0&\varepsilon  \end{pmatrix}.
\end{multline}
%
%
%
%
%
%
%
\subsection{Accidental isomorphisms}
\label{acc isom}
We need to explicate the accidental isomorphisms of our concern,
since we use  them in a crucial way to transfer an automorphic  period on one group to 
the one on the other group. 
The reader may consult, for example,  Satake~\cite{Satake} and Tsukamoto~\cite{Tsukamoto}
about the details of  the material here.
%
%
%
\subsubsection{$\mathrm{PGSU}_{3,D}\simeq \mathrm{PGU}_{4,\varepsilon}$}
Suppose that $D=D_\varepsilon$.
Then we may naturally realize $\mathrm{GSU}_{3,D}(F)$ as a subgroup of $\mathrm{GL}_6(E)$.
We note that
\[
\begin{pmatrix} 1&0&0&0&0&0\\ 0&-\varepsilon&0&0&0&0\\ 0&0&1&0&0&0\\ 0&0&0&-\varepsilon&0&0\\0 &0&0&0&1&0\\ 0&0&0&0&0&-\varepsilon \end{pmatrix}
\,{}^{t}\bar{g}\,
 \begin{pmatrix} 1&0&0&0&0&0\\ 0&-\varepsilon&0&0&0&0\\ 0&0&1&0&0&0\\ 0&0&0&-\varepsilon&0&0\\0 &0&0&0&1&0\\ 0&0&0&0&0&-\varepsilon \end{pmatrix}^{-1}= {}^{t}g^\sigma
\]
and 
\[
\begin{pmatrix}0&1&0&0&0&0\\ -1&0&0&0&0&0\\ 
0&0&0&1&0&0\\ 0&0&-1&0&0&0\\ 0&0&0&0&0&1\\ 0&0&0&0&-1& 0\end{pmatrix}
\,{}^{t}\bar{g}\,
 \begin{pmatrix}0&1&0&0&0&0\\ -1&0&0&0&0&0\\ 
0&0&0&1&0&0\\ 0&0&-1&0&0&0\\ 0&0&0&0&0&1\\ 0&0&0&0&-1& 0\end{pmatrix}^{-1}
= {}^{t}g.
\]
Thus in this realization, we have
\begin{multline}\label{d: gsu_3,D}
\mathrm{GSU}_{3,D}(F)=
\{ g \in \mathrm{GSO}_{3,3}(E) : {}^{t} g^\sigma\, \mathcal J_\varepsilon^\circ \,g 
=\lambda(g)\mathcal J_{\varepsilon}^\circ,\,\lambda\left(g\right)\in F^\times \}
\\
\text{where}\quad
\mathcal J_{\varepsilon}^\circ =-  \begin{pmatrix} 0&0&0&0&1&0\\ 0&0&0&0&0&\varepsilon\\0&0&1&0&0&0
\\0&0&0&\varepsilon&0&0\\1&0&0&0&0&0\\0&\varepsilon&0&0&0&0\end{pmatrix}.
\end{multline}
Here we recall that 
\begin{equation}\label{iso: GSO(3,3)}
\mathrm{GSO}_{3,3}(E) \simeq \mathrm{GL}_4(E) \times \mathrm{GL}_1(E) \slash \{ (z, z^{-2} ) :z \in E^\times \}.
\end{equation}
In fact  the isomorphism \eqref{iso: GSO(3,3)} is realized as follows.
Let us take the standard basis 
\[
b_1 ={}^{t}(1,0,0,0),\quad
b_2 ={}^{t}(0,1,0,0), 
\quad
b_3 ={}^{t}(0,0,1,0), 
\quad
b_4 ={}^{t}(0,0,0,1),
\]
of $E^4$. 
Then we may consider $V:=\wedge^2 E^4$ as an orthogonal space
over $E$ with a quadratic form 
$\left(\,\, ,\,\,\right)_V$ defined by
\[
v_1\wedge v_2=(v_1 ,v_2)_V \cdot  b_1 \wedge  b_2 \wedge  b_3 \wedge  b_4
\]
for $v_1, v_2 \in V$.
As a basis of $V$ over $E$, we take $\{ \varepsilon_i: 1\le i\le 6 \}$ given by 
\[
\varepsilon_1 =b_1 \wedge b_2,  \varepsilon_2 =b_1 \wedge b_3, \varepsilon_3 = b_1 \wedge b_4,
\varepsilon_4 = b_2 \wedge b_3,  \varepsilon_5 = b_4 \wedge b_2,  \varepsilon_6 = b_3 \wedge b_4.
\]
Let the group $\mathrm{GL}_4(E) \times \mathrm{GL}_1(E)$ act on $V$ by
$(g, a)(w_1 \wedge w_2) = a \cdot \left( gw_1 \wedge gw_2 \right)$
where $w_1,w_2\in E^4$.
This action defines a homomorphism
\begin{equation}\label{e: hom from GL4xGL1}
\mathrm{GL}_4(E) \times \mathrm{GL}_1(E)\to
\mathrm{GSO}_{3,3}\left(E\right)
\end{equation}
where we take 
$\left\{\varepsilon_i: 1\le i\le 6\right\}$ as a basis of $V$
and the homomorphism \eqref{e: hom from GL4xGL1} induces
the isomorphism~\eqref{iso: GSO(3,3)}.
By a direct computation we observe that $\left(-\mathcal J_\varepsilon, 1\right)$ is mapped
to $\mathcal J_\varepsilon^\circ$ under \eqref{e: hom from GL4xGL1}
and the restriction of the homomorphism~\eqref{e: hom from GL4xGL1} gives a
homomorphism
\begin{equation}\label{e: hom from GU4}
\mathrm{GU}_{4,\varepsilon}\left(F\right)\to
\mathrm{GSU}_{3,D}\left(F\right).
\end{equation}
Then it is easily seen that the isomorphism
\begin{equation}
\label{acc isom1}
\Phi_D : \mathrm{PGU}_{4,\varepsilon}(F) 
\overset{\sim}{\to} \mathrm{PGSU}_{3,D}(F)
\end{equation}
is induced.
%
%
%
%
%
%
\subsubsection{$\mathrm{PGU}_{2,2}
\simeq\mathrm{PGSO}_{4,2}$}
When $\varepsilon \in \mathrm{N}_{E \slash F}(E^\times)$, 
the quaternion algebra $D=D_\varepsilon$ is split and 
the isomorphism~\eqref{acc isom1} gives an isomorphism
$\mathrm{PGU}_{2,2}\simeq \mathrm{PGSO}_{4,2}$.
We recall the concrete realization of this isomorphism.
First we define $\mathrm{GU}_{2,2}$ by
\begin{multline*}
\mathrm{GU}_{2,2}:=
\left\{
g\in\mathrm{GL}_4\left(E\right):
{}^tg^\sigma\,
J_4\,g
=\lambda\left(g\right)J_4,
\,\lambda\left(g\right)\in F^\times\right\}
\\
\text{where $J_4=\begin{pmatrix}0&0&0&1\\0&0&1&0\\
0&1&0&0\\1&0&0&0\end{pmatrix}$}
\end{multline*}
as \eqref{d: J_m}.
Let 
\begin{equation*}
{\mathcal V}: = \left\{
B\left(\left(x_i\right)_{1\le i\le 6}\right):=
\left(\begin{smallmatrix}
0&\eta x_1& x_3 + \eta x_4 &x_2\\
-\eta x_1 &0 &x_5& -x_3 +\eta x_4\\
-x_3 -\eta x_4 &-x_5& 0 &\eta^{-1} x_6\\
-x_2 &x_3 -\eta x_4 &-\eta^{-1} x_6 &0\\
\end{smallmatrix}\right)
: x_i \in F \, \left(1\le i\le6\right)
\right\}.
\end{equation*}
We define $\Psi : {\mathcal V} \rightarrow F$ by
\[
\Psi \left(B\right): = {\rm Tr} \left(B\, \begin{pmatrix} 0&1_2\\ 1_2&0\end{pmatrix}
\,{}^{t}B^\sigma\,
  \begin{pmatrix}0 &1_2\\ 1_2&0\end{pmatrix}\right).
\]
Then  we have
\[
\Psi\left(B\left(\left(x_i\right)_{1\le i\le 6}\right)\right)=
-4\left\{x_1x_6+x_2x_5-\left(x_3^2-dx_4^2\right)\right\}.
\]

Let $\mathrm{GSU}_{2,2}$ denote the identity component
of $\mathrm{GU}_{2,2}$, i.e.
\[
\mathrm{GSU}_{2,2} = \{ g \in \mathrm{GU}_{2,2}: \det(g) = \lambda(g)^2 \}.
\]
We let $\mathrm{GSU}_{2,2}$ act on $\mathcal V$ by
\[
\mathrm{GSU}_{2,2}\times \mathcal V\ni
\left(g,B\right)\mapsto
\left(wg w \right)B\left(w\,{}^tg w \right)\in\mathcal V
\quad\text{where 
$w=\begin{pmatrix} 1&0&0&0\\ 0&1&0&0 \\ 0&0&0&1\\ 0&0&1& 0\end{pmatrix}$.}
\]
Then this action induces a homomorphism 
$\phi : \mathrm{GSU}_{2,2} \rightarrow \mathrm{GO}({\mathcal V})$.
We note that
\[
\lambda(\phi(g)) = \det\left(g\right)\quad
\text{for $g\in \mathrm{GSU}_{2,2}$}
\]
and this implies that the image of $\phi$ is contained
in $\mathrm{GSO}\left(\mathcal V\right)$.
As a basis of ${\mathcal V}$, we may take
\begin{align*}
f_1&=\begin{pmatrix}
0&\eta&0&0\\
-\eta&0&0&0\\
0&0&0&0\\
0&0&0&0\\
\end{pmatrix}, \quad
&f_2&=\begin{pmatrix}
0&0&0&1\\
0&0&0&0\\
0&0&0&0\\
-1&0&0&0\\
\end{pmatrix}, \quad
&f_3&=\begin{pmatrix}
0&0&1&0\\
0&0&0&-1\\
-1&0&0&0\\
0&1&0&0\\
\end{pmatrix}, 
\\
f_4&=
\begin{pmatrix}
0&0&\eta &0\\
0&0&0&\eta\\
-\eta &0&0&0\\
0&-\eta &0&0\\
\end{pmatrix},  \quad
&f_5&=\begin{pmatrix}
0&0&0&0\\
0&0&1&0\\
0&-1&0&0\\
0&0&0&0\\
\end{pmatrix}, \quad
&f_6&=\begin{pmatrix}
0&0&0&0\\
0&0&0&0\\
0&0&0&\eta^{-1}\\
0&0&-\eta^{-1}&0\\
\end{pmatrix}.
\end{align*}
With respect to this basis, we may regard
$\phi$ as a homomorphism from $\mathrm{GSU}_{2,2}$
to $\mathrm{GO}_{4,2}$,
where the group $\mathrm{GO}_{4,2}$ is  given by \eqref{d: GO_n+2,n} for $n=2$.
Let us consider $\mathrm{GSU}_{2,2}\rtimes E^\times$
where the action of $\alpha\in E^\times$ on $g\in\mathrm{GSU}_{2,2}$
is given by
\[
\alpha \cdot g = \begin{pmatrix} \alpha&0&0&0\\0 &1&0&0\\0 &0&1&0\\ 0&0&0&\left(\alpha^{\sigma}
\right)^{-1}\end{pmatrix}
\,g \,\begin{pmatrix} \alpha&0&0&0\\0 &1&0&0\\0 &0&1&0\\ 0&0&0&\left(\alpha^{\sigma}
\right)^{-1}\end{pmatrix}^{-1}.
\]
Then as in \cite[p.32--34]{Mo},
$\phi$ may be extended to $\mathrm{GSU}_{2,2}\rtimes E^\times$
and we have a homomorphism
$\mathrm{GSU}_{2,2}\rtimes E^\times\to
\mathrm{PGSO}_{4,2}$
which induces the isomorphism
\begin{equation}\label{acc isom2}
\Phi : \mathrm{PGU}_{2,2} \overset{\sim}{\to} \mathrm{PGSO}_{4,2}.
\end{equation}
%
%
%
%
%
%
%
\subsection{Bessel periods}
Let us introduce Bessel periods on various groups.
%
%
%
%
%
%
%
\subsubsection{Bessel periods on $G=\mathrm{GSp}_2$}
\label{s:def bessel G}
Though we already introduced Bessel periods on $G_D$ in general
as \eqref{e: def of bessel period},
we would like to describe them concretely in the case of $G$ here for our 
explicit pull-back computations in the next section.

Let $P$ be the Siegel parabolic subgroup of $G$ with the Levi decomposition $P = MN$
where 
\[
\label{d:N}
M(F) = \left\{\begin{pmatrix} g&0\\ 0&\lambda \cdot {}^tg^{-1}\end{pmatrix} : 
\begin{aligned}
&g \in \mathrm{GL}_2(F), 
\\
&\lambda \in F^\times 
\end{aligned}
\right\},\,\,
N(F) =  \left\{ \begin{pmatrix} 1&X\\ 0&1\end{pmatrix} : X \in \mathrm{Sym}_2(F) \right\}.
\]
Here $\mathrm{Sym}_n(F)$ denotes the set of $n$ by $n$ symmetric matrices with entries
in $F$ for a positive integer $n$.
For $S \in \mathrm{Sym}_2(F)$, let us define a character $\psi_{S}$ of $N(\mA)$ by 
\[
\psi_{S}  \begin{pmatrix} 1&X\\ 0&1\end{pmatrix}  = \psi
\left[\mathrm{tr}(SX)\right].
\]
For $S\in\mathrm{Sym}_2\left(F\right)$ such that  $\det S\ne 0$,
let
\[
\label{T_S}
T_S := \left\{ g \in\mathrm{GL}_2 : {}^{t}gSg = \det(g)S \right\}.
\]
We identify $T_S$ with the subgroup of $G$ given by
\[
 \left\{ \begin{pmatrix}g&0\\ 0&\det (g) \cdot {}^{t}g^{-1} \end{pmatrix} : g \in T_S \right\}.
\]
%
\begin{Definition}
Let us take $S\in\mathrm{Sym}_2\left(F\right)$
such that $T_S\left(F\right)$ is isomorphic to $E^\times$.
Let $\pi$ be an irreducible cuspidal automorphic representation of $G\left(\mathbb A\right)$
whose central character is trivial
and $V_\pi$ its space of automorphic forms.
Fix an $F$-isomorphism $T_S\left(F\right)\simeq E^\times$.
Let $\Lambda$ be a character of $\mathbb A_E^\times\slash E^\times$
such that $\Lambda\mid_{\mathbb A^\times}$ is trivial.
We
regard $\Lambda$ 
as a character of $T_S\left(\mathbb A\right) \slash \mathbb A^\times\,
T_S\left(F\right)$.

Then for $\varphi\in V_\pi$, we define $ B_{S,\Lambda,\psi}\left(\varphi\right)$,
the $\left(S,\Lambda,\psi\right)$-Bessel period of $\varphi$ by
\begin{equation}
\label{Beesel def gsp}
B_{S,  \Lambda,\psi}(\varphi) = \int_{\mA^\times \,T_{S}(F) \backslash T_{S}(\mA)} \int_{N(F) \backslash N(\mA)}
\varphi(uh) \Lambda^{-1}(h) \psi_{S}^{-1}(u) \, du\, dh.
\end{equation}
We say that $\pi$ has the $\left(S,\Lambda,\psi\right)$-Bessel period
when $B_{S,  \Lambda,\psi}\not\equiv 0$ on $V_\pi$.
Then we also say that $\pi$ has the $\left(E,\Lambda\right)$-Bessel period as in Definition~\ref{def of E,Lambda-Bessel period}.
\end{Definition}
%
%
%
%
%
%
%
%
%
%
%
%
%
%
\subsubsection{Bessel periods on $\mathrm{GSU}_{3,D}$}
Let us introduce Bessel periods on the group $\mathrm{GSU}_{3,D}$
defined in   \ref{ss: quaternionic unitary groups}.
Let $P_{3,D}$ be a maximal parabolic subgroup of $\mathrm{GSU}_{3,D}$
with the Levi decomposition $P_{3,D}=M_{3,D}N_{3,D}$ where
\[
\label{d:N3,D}
M_{3,D} = \left\{ \begin{pmatrix}g&0&0\\ 0&h&0\\ 0&0&g \end{pmatrix} : 
\begin{aligned}
&g\in D^\times,
\\
&h \in T_\eta,  
\\
& \mathrm{n}_D\left(g\right)=\mathrm{n}_D\left(h\right)
\end{aligned}
\right\},
\,\,
N_{3,D} = \left\{ \begin{pmatrix} 1&A^\prime&B\\ 0&1&A\\0 &0&1 \end{pmatrix} \in \mathrm{GSU}_{3,D} \right\}.
\]
As for $T_\eta$, we recall \eqref{d: gsu1,d} and 
$T_\eta \simeq E^\times$.
For $X \in D^\times$, we define a character $\psi_{X, D}$ of $N_{3, D}(\mA)$ by 
\[
\psi_{X, D}  \begin{pmatrix} 1&A^\prime&B\\ 0&1&A\\0 &0&1 \end{pmatrix} = \psi \left[\mathrm{tr}_D(XA) \right].
\]
Then the identity component of the stabilizer of $\psi_{X, D}$ in $M_{3,D}$ is 
\[
\label{M_X D}
M_{X, D} =  \left\{ \begin{pmatrix} h^{X}&0&0\\ 0&h&0\\0 &0&h^{X}\end{pmatrix} : h \in T_\eta \right\}
\quad
\text{where}\quad h^X = XhX^{-1}.
\]
We identify $M_X$ with $T_\eta$ by
\begin{equation}\label{e: identify M_X with T_eta}
M_{X, D}\ni \begin{pmatrix} h^{X}&0&0\\ 0&h&0\\ 0&0&h^{X}\end{pmatrix}
\mapsto h\in T_\eta
\end{equation}
and we fix an $F$-isomorphism $T_\eta\simeq E^\times$.
%
\begin{Definition}\label{d: bessel for GSU_3,D}
Let $\sigma_D$ be an irreducible cuspidal automorphic representation of $\mathrm{GSU}_{3,D}\left(\mA \right)$ 
and $V_{\sigma_D}$ its space of automorphic forms.
Let $\chi$ be a character of $\mathbb A_E^\times\slash E^\times$
and we regard $\chi$ as a character of $M_{X,D}\left(\mathbb A\right) \slash M_{X,D}\left(F\right) $.
Suppose that $\chi|_{\mA^\times} = \omega_{\sigma_D}$, the central character of $\sigma_D$.

Then for $\varphi\in V_{\sigma_D}$, we define $\mathcal{B}^D_{X, \chi, \psi} (\varphi)$,
the $\left(X,\chi, \psi \right)$-Bessel period of $\varphi$ by
\begin{multline}\label{Besse def gsud}
\mathcal{B}^D_{X, \chi, \psi} (\varphi) =\int_{\mA^\times M_{X,D}(F) \backslash M_{X,D}(\mA)} \int_{N_{3,D}(F) \backslash N_{3,D}(\mA)} \varphi(uh)
\\
\times \chi (h)^{-1} \psi_{X, D}(u)^{-1} \, du \, dh.
\end{multline}
\end{Definition}
%
%
%
%
%
%
\subsubsection{Bessel periods on $\mathrm{GU}_{4,\varepsilon}$}
In light of the accidental isomorphism \eqref{acc isom1},
Bessel periods on the group $\mathrm{GU}_{4,\varepsilon}$
is defined as follows.

Let $P_{4,\varepsilon}$ be a maximal parabolic subgroup of $\mathrm{GU}_{4,\varepsilon}$
with the Levi decomposition $M_{4,\varepsilon}N_{4,\varepsilon}$ 
where
\begin{align*}
M_{4,\varepsilon}(F) &= \left\{ \begin{pmatrix} a&0&0\\ 0&g&0\\ 0&0&\lambda(g) (a^\sigma)^{-1} \end{pmatrix} : a \in E^\times, g \in \mathrm{GU}_{2,\varepsilon}\left(F\right) \right\}
,
\\
N_{4,\varepsilon}(F) &= \left\{  \begin{pmatrix} 1&A&B\\ 0&1_2&A^\prime\\ 0&0&1\end{pmatrix} \in \mathrm{GU}_{4,\varepsilon}\left(F\right) \right\}.
\end{align*}
Let us take an anisotropic vector $e\in E^4$ of the form ${}^{t}(0, \ast, \ast, 0)$.
Then we define a character $\chi_e$ of $N_{4,\varepsilon}(\mA)$ by
\[
\chi_e \left(u\right)
= \psi( (ue, b_1)_\varepsilon)
\quad\text{where
$\left(x,y\right)_\varepsilon={}^tx^\sigma J_{\varepsilon}y$}.
\]
Here we recall that $J_\varepsilon$ is
as given in \eqref{d: GU2D} 
and $b_1={}^t\left(1,0,0,0\right)$.
Let $D_e$ denote the subgroup of $M_{4,\varepsilon}$
given by
\[
D_{e}  := \left\{ \begin{pmatrix}1&0&0\\ 0&h&0\\ 0&0&1\end{pmatrix} :  h \in \mathrm{U}_{2,\varepsilon}, \, h e = e \right\}.
\]
Then the group $D_e\left(\mA\right)$ stabilizes the character $\chi_e$ by conjugation.
We note that 
\[
D_e(F) \simeq \mathrm{U}_1(F):= \{ a \in E^\times : \bar{a}a =1 \}.
\]
Hence for a character $\Lambda$ of $\mathbb A_E^\times$ which is trivial on $\mA^\times$,
we may regard $\Lambda$ as a character of $D_e(\mA)$ by $d \mapsto \Lambda(\det d)$.
Then we define a character $\chi_{e, \Lambda}$ of $R_e(\mA)$ where $R_e:=D_eN_{4,\varepsilon}$ by
\begin{equation}\label{d of chi_e,Lambda}
\chi_{e, \Lambda}(ts) := \Lambda(t) \chi_e(s) \quad \text{for}
\quad t \in D_e(\mA), \,
s \in N_{4,\varepsilon}(\mA).
\end{equation}
%
\begin{Definition}\label{d: Bessel on G_4,varepsilon}
For a cusp form $\varphi$ on $\mathrm{GU}_{4,\varepsilon}(\mA_F)$
with a trivial central character, we define $B_{e, \Lambda,\psi}(\varphi)$,
the $(e, \Lambda,\psi)$-Bessel period of $\varphi$, by
\begin{equation}\label{d bessel period on GU_4,varepsilon}
B_{e, \Lambda,\psi}(\varphi) = \int_{D_e(F) \backslash D_e(\mA_F)} \int_{N_{4,\varepsilon}(F) \backslash N_{4,\varepsilon}(\mA_F)}
\chi_{e, \Lambda}(ts)^{-1}\, \varphi(ts) \, ds \, dt.
\end{equation}
\end{Definition}
%
%
%
%
%
%
%
\subsubsection{Bessel periods on $\mathrm{GSO}_{4,2}$ and $\mathrm{GSO}_{3,3}$}
By combining
 the accidental isomorphisms \eqref{acc isom1} and \eqref{acc isom2}
in the split case, we shall define Bessel periods on $\mathrm{GSO}_{4,2}$ and $\mathrm{GSO}_{3,3}$
as the following.

Let $P_{4,2}$ denote a maximal parabolic subgroup of $\mathrm{GSO}_{4,2}$ with the Levi decomposition 
$P_{4,2}= M_{4,2}N_{4,2}$ where
 \[
 \label{d:N4,2}
 M_{4,2} = \left\{ \begin{pmatrix} g&0&0\\0 &h&0\\0 &0&
 g^\ast\cdot \det h\end{pmatrix} : 
 \begin{aligned}
 &g \in \mathrm{GL}_2, 
 \\
 &h \in \mathrm{GSO}_{2 ,0}
 \end{aligned}
 \right\}, \,\,
 N_{4,2}= \left\{ \begin{pmatrix}1_2&A^\prime&B\\ 0&1_2&A\\ 0&0&1_2 \end{pmatrix} \in \mathrm{GSO}_{4,2}\right\}.
 \]
 Here 
 \[
 g^\ast = \begin{pmatrix}0&1\\1&0\end{pmatrix} {}^{t}g^{-1} \begin{pmatrix}0&1\\1&0\end{pmatrix}\quad\text{
 for}\quad
  g\in\mathrm{GL}_2. 
 \]
 Then for $X\in\mathrm{Mat}_{2\times 2}\left(F\right)$,  we define
a character $\psi_X$ of $N_{4,2}\left(\mA\right)$ by
\[
\psi_{X} \begin{pmatrix}1_2&A^\prime&B\\ 0&1_2&A\\0 &0&1_2 \end{pmatrix} 
 = \psi \left[\mathrm{tr}(XA) \right].
 \]
 Suppose that $\det X \ne 0$ and let
 \[
 \label{M_X}
M_{X} :=  \left\{ \begin{pmatrix}  (\det h) \cdot (h^X)^\ast&0&0\\ 0&h&0\\ 0&0&h^X \end{pmatrix} : h \in \mathrm{GSO}_{2,0} \right\}
 \]
 where $h^X = XhX^{-1}$ . 
 Then $M_X\left(\mA\right)$ stabilizes the character $\psi_X$ and
 $M_X$ is isomorphic to $\mathrm{GSO}_{2,0}$.
 We fix an isomorphism $\mathrm{GSO}_{2,0}(F)\simeq E^\times$ and
 we regard a character of $\mathbb A_E^\times$ as a character of $M_X\left(\mA\right)$.
 %
 \begin{Definition}\label{def of Bessel on GS)_4,2}
  Let $\sigma$ be an irreducible cuspidal automorphic representation of 
  $\mathrm{GSO}_{4,2}(\mA)$ with its space of automorphic forms $V_\sigma$
  and the central character $\omega_\sigma$.
For a character $\chi$ of $\mA_E^\times$ such that $\chi |_{\mA^\times} = \omega_{\sigma}$, we define $\mathcal B_{X,\chi, \psi}\left(\varphi\right)$, 
the $(X, \chi, \psi)$-Bessel period of $\varphi \in V_\sigma$ by
\begin{equation}\label{e: Bessel GSO(4,2)}
\mathcal{B}_{X, \chi, \psi}(\varphi) = \int_{N_{4,2}(F) \backslash N_{4,2}(\mA)} \int_{M_{X}(F) \mA^\times \backslash M_{X}(\mA)}
\varphi(u h) \chi(h)^{-1} \psi_{X}(u)^{-1} \, du \, dh.
\end{equation}
\end{Definition}
When $d \in (F^\times)^2$, we know that $\mathrm{GSO}(S_2) \simeq \mathrm{GSO}_{3,3}$.
Hence, as above, for a cusp form $\varphi$ on $\mathrm{GSO}_{3,3}$ with central character $\omega$  and characters 
$\Lambda_1, \Lambda_2$ of $\mA^\times \slash F^\times$ such that $\Lambda_1 \Lambda_2=\omega$, 
we define $(X, \Lambda_1, \Lambda_2, \psi)$-Bessel period by
\[
\mathcal{B}_{X, \Lambda, \psi}(\varphi) = \int_{N_{4,2}(F) \backslash N_{4,2}(\mA)} \int_{M_{X}(F) \mA^\times \backslash M_{X}(\mA)}
\varphi(u h) \chi_{\Lambda_1, \Lambda_2}(h)^{-1} \psi_{X}(u)^{-1} \, du \, dh.
\]
Here, since $M_{4,2} \simeq \mathrm{GL}_2 \times \mathrm{GSO}_{1,1}$ and 
$\mathrm{GSO}_{1,1}(F) = \left\{ \left( \begin{smallmatrix} a&\\ &b\end{smallmatrix} \right) : a, b \in F^\times \right\}$, 
we define a character $\chi_{\Lambda_1, \Lambda_2}$ of $\mathrm{GSO}_{1,1}(\mA)$ by 
\[
\chi_{\Lambda_1, \Lambda_2} \begin{pmatrix} a&\\ &b\end{pmatrix} =\Lambda_1(a) \Lambda_2(b).
\]
When $\omega$ is trivial, we have $\Lambda_2 = \Lambda_1^{-1}$. In this case, we simply call
$(X, \Lambda_1, \Lambda_1^{-1}, \psi)$-Bessel period as $(X, \Lambda_1, \psi)$-Bessel period
and simply write $\chi_{\Lambda_1, \Lambda_1^{-1}} = \Lambda_1$.
%
%
%
%
%
%
%
%
%
\subsection{Local Bessel periods}
\label{s:def local bessel}
Let us introduce local counterparts to the global Bessel periods.
Let $k$ be a local field of characteristic zero and $D$ a quaternion algebra over $k$. 

Since the local Bessel periods are deduced from the global ones in a uniform way,
by abuse of notation,
let a quintuple $\left(H, T, N,\chi,\psi_N\right)$ stand for one of
\begin{align*}
&\text{$\left(G_D, T_\xi, N_D, \Lambda,\psi_\xi\right)$  in \eqref{e: def of bessel period},}
\\
&\text{$\left(\mathrm{GSp}_2, T_S, N, \Lambda,\psi_S\right)$ in \eqref{Beesel def gsp},
or,}
\\
&\text{
$\left(\mathrm{GSU}_{3,D}, M_X, N_{4,2},\chi,\psi_{X}\right)$ in \eqref{Besse def gsud}.}
\end{align*}
Let $(\pi, V_\pi)$ be an irreducible tempered representation of $H=H\left(k\right)$ with trivial central 
character 
and $[\,,\,]$ a $H$-invariant hermitian pairing  on $V_\pi$, the space of $\pi$.
Let us denote by $V_\pi^\infty$ the space of smooth vectors in $V_\pi$. When $k$ is non-archimedean, 
clearly $V_\pi^\infty = V_\pi$.
Let $\chi$ be a character of $T=T\left(k\right)$ which is trivial on $Z_H=Z_H\left(k\right)$,
where $Z_H$ denotes the center of $H$.

Suppose that $k$ is non-archimedean.
Then for $\phi,\phi^\prime\in V_\pi$, we define the local Bessel period
$\alpha_{\chi, \psi_N}^H(\phi, \phi^\prime) = \alpha_{\chi, \psi_N}(\phi, \phi^\prime) = \alpha\left(\phi,\phi^\prime\right)$ by
\begin{equation}\label{e: local integral 1}
\alpha\left(\phi,\phi^\prime\right)
:=
\int_{T \slash Z_H }\int_{N}^{\mathrm{st}}
[\pi \left(ut \right)\phi,\phi^\prime]\,
\chi\left(t\right)^{-1} \psi_N(u)^{-1}\, du\,dt.
\end{equation}
Here the inner integral of \eqref{e: local integral 1} is the stable integral
in the sense of Lapid and Mao~\cite[Definition~2.1, Remark~2.2]{LM}.
Indeed it is shown that for any $t\in T$ the inner integral stabilizes at 
a certain compact open subgroup of $N=N\left(k\right)$ and the outer integral
converges by Liu~\cite[Proposition~3.1, Theorem~2.1]{Liu2}.
We note that it is also shown in Waldspurger~\cite[Section~5.1, Lemme]{Wa2} that \eqref{e: local integral 1}
is well-defined. We often simply write $\alpha(\phi) = \alpha(\phi, \phi)$.

Now suppose that $k$ is archimedean.
Then the local Bessel period is defined 
as a regularized integral whose regularization is achieved by the Fourier transform
as in Liu~\cite[3.4]{Liu2}.
Let us briefly recall  the definition.
We define a subgroup $N_{-\infty}$ of $N=N\left(k\right)$ by:
\begin{align*}
N_{-\infty}:&=
\left\{\begin{pmatrix}1&u\\0&1\end{pmatrix}\in N_D:
\mathrm{tr}_D\left(\xi u\right)=0\right\}
\quad\text{in the $G_D$-case};
\\
N_{-\infty}:&=\left\{\begin{pmatrix}1&Y\\0&1\end{pmatrix}\in N:
\mathrm{tr}\left(SY\right)=0\right\}
\quad\text{in the $\mathrm{GSp}_2$-case};
\\
N_{-\infty}:&=\left\{\begin{pmatrix}1&A^\prime&B\\0&1&A\\0&0&1\end{pmatrix}\in N_{3,D}:
\mathrm{tr}_D\left(XA\right)=0\right\}
\quad\text{in the $\mathrm{GSU}_{3,D}$-case},
\end{align*}
respectively.
Then it is shown in Liu~\cite[Corollary~3.13]{Liu2} that for $u \in N$,
\[
\alpha_{\phi, \phi^\prime}(u):=
\int_{T \slash Z_G} \int_{N_{-\infty}} [ \pi\left(us t \right)\phi,\phi^\prime] \,\chi(t)^{-1} \, ds \, dt
\]
converges absolutely for $\varphi,\varphi^\prime\in V_\pi^\infty$ 
and it gives a tempered distribution on $N \slash N_{-\infty}$.

For an abelian Lie group $\mathcal{N}$, we denote by $\mathcal{D}(\mathcal{N})$ (resp. $\mathcal{S}(\mathcal{N})$)
the space of tempered distributions (resp. Schwartz functions) on $\mathcal{N}$.
Then we recall that the Fourier transform 
$\hat{} : \mathcal{D}(\mathcal{N})  \rightarrow \mathcal{D}(\mathcal{N}) $
is defined by the formula
\[
\left(\hat{\mathfrak{a}}, \phi\right) = \left(\mathfrak{a}, \hat{\phi}\right) 
\quad\text{for $\mathfrak a\in \mathcal D\left(\mathcal N\right)$
and $\phi\in \mathcal S\left(\mathcal N\right)$},
\]
where $\left(\, ,\right)$ denotes the natural pairing 
$\mathcal{D}(\mathcal{N}) \times \mathcal{S}(\mathcal{N}) \rightarrow \mC$
and $\hat{\phi}$ is the Fourier transform of $\phi\in \mathcal S\left(\mathcal N\right)$.

Then by Liu~\cite[Proposition~3.14]{Liu2},
the Fourier transform $\widehat{\alpha_{\phi, \phi^\prime}}$
is smooth on the regular locus 
$(\widehat{N \slash N_{-\infty}})^{\rm reg}$ of the 
Pontryagin dual $\widehat{N \slash N_{-\infty}}$
and we define the local Bessel period $\alpha\left(\phi,\phi^\prime\right)$ by
\begin{equation}\label{local archimedean Bessel}
\alpha_{\chi, \psi_N}^H(\phi, \phi^\prime) = \alpha_{\chi, \psi_N}(\phi, \phi^\prime) =
\alpha\left(\phi,\phi^\prime\right) := \widehat{\alpha_{\phi, \phi^\prime}} \left( \psi_N \right).
\end{equation}
As in the non-archimedean case, we often simply write $\alpha(\phi) = \alpha(\phi, \phi)$.
%
%
%
%
%
%
%
%
%
%
%
%
%
%
%
%
%
%
%
%
%
\section{Pull-back of Bessel periods}
In this section, we establish the  pull-back formulas of the global 
Bessel periods with respect to the dual pairs, 
$\left(\mathrm{GSp}_2,\mathrm{GSO}_{4,2}\right)$, $\left(\mathrm{GSp}_2,\mathrm{GSO}_{3,3}\right)$
and $\left(G_D,\mathrm{GSU}_{3,D}\right)$.
We recall that the first two cases may be regarded as the special case 
when $D$ is split of the last one,
by the accidental isomorphisms explained in \ref{acc isom}.
%
%
%
%
%
%
%
%
%
\subsection{$\left(\mathrm{GSp}_2,\mathrm{GSO}_{4,2}\right)$ and $\left(\mathrm{GSp}_2,\mathrm{GSO}_{3,3}\right)$ case}
%
%
%
%
%
%
%
%
%
%
%
%
%
%
%
%
%
%
%
%
%
%
%
%
%
%
%
\subsubsection{Symplectic-orthogonal theta correspondence with similitudes}
\label{def theta}
Let $X$ (resp. $Y$) be a finite dimensional vector space over $F$
equipped with a non-degenerate alternating (resp. symmetric) bilinear
form. Assume that $\dim_F Y$ is even.
We denote their similitude groups by $\mathrm{GSp}(X)$ and $\mathrm{GO}(Y)$,
and,
 their isometry groups by $\mathrm{Sp}(X)$ and $\mathrm{O}(Y)$,
 respectively.
We denote the identity component of $\mathrm{GO}(Y)$ and $\mathrm{O}(Y)$
by $\mathrm{GSO}(Y)$ and $\mathrm{SO}(Y)$, respectively.
We let $\mathrm{GSp}(X)$ (resp. $\mathrm{GO}(Y)$) act on $X$ from right (resp. left).
The space $Z = X \otimes Y$ has a natural non-degenerate alternating form $\langle \,, \, \rangle$,
and we have an embedding $\mathrm{Sp}(X) \times \mathrm{O}(Y) \rightarrow \mathrm{Sp}(Z)$
defined by 
\begin{equation}
\label{SP times O to SP}
(x \otimes y)(g, h) =  x g \otimes h^{-1}y , \quad \text{for } x \in X, y \in Y, h \in \mathrm{O}(Y), g \in \mathrm{Sp}(X).
\end{equation}
Fix a polarization $Z = Z_{+} \oplus Z_{-}$. 
Let us denote by $(\omega_{\psi}, \mathcal{S}(Z_{+}(\mA)))$ the Schr\"{o}dinger model of 
the Weil representation of $\widetilde{\mathrm{Sp}}(Z)$
corresponding to this polarization with the Schwartz-Bruhat space $\mathcal{S}(Z_{+})$ on $Z_{+}$. 
We write a typical element of $\mathrm{Sp}(Z)$ by
\[
\begin{pmatrix}
A&B\\
C&D\\
\end{pmatrix}
\quad\text{where}
\quad
\begin{cases}
A \in {\rm Hom}(Z_{+} , Z_{+}),\,\,
B \in {\rm Hom}(Z_{+} , Z_{-}),
\\
C \in {\rm Hom}(Z_{-} , Z_{+}),\,\,
D \in {\rm Hom}(Z_{-} , Z_{-}).
\end{cases}
\]
Then the action of $\omega_{\psi}$ on $\phi\in \mathcal{S}(Z_+)$ is given by the
following formulas:
\begin{equation}
\label{weil act 1}
\omega_{\psi}
\left(
\begin{pmatrix}
A&B\\
0&^{t}A^{-1}\\
\end{pmatrix},\varepsilon
\right)\phi(z_{+})
= \varepsilon 
\frac{\gamma _{\psi }(1)}{\gamma _{\psi }({\rm det}A)} |{\rm det}(A)|^{\frac{1}{2}} \psi 
\left(\frac{1}{2} \langle z_{+}A ,z_{+}B \rangle\right)\phi(z_{+}A)
\end{equation}
\begin{equation}
\label{weil act 2}
\omega_{\psi}
\left(
\begin{pmatrix}
0&I\\
-I&0\\
\end{pmatrix},\varepsilon
\right)\phi(z_{+})
=\varepsilon (\gamma _{\psi }(1))^{-\dim Z_{+}} \int _{Z_{+}}
\psi 
\left(
\langle z^{\prime} , z
\begin{pmatrix}
0&I\\
-I&0\\
\end{pmatrix}
\rangle \right)
\phi(z^{\prime}) \, d z^{\prime},
\end{equation}
where $\gamma_{\psi}(t)$ is a certain eighth root of unity called 
the Weil factor.
Moreover, since the embedding given by \eqref{SP times O to SP} splits in the metaplectic group $\mathrm{Mp}(Z)$,
we obtain the Weil representation of $\mathrm{Sp}(X, \mA) \times \mathrm{O}(Y, \mA)$
by restriction.
We also denote this representation by $\omega_\psi$.

We have a natural homomorphism
\[
 i : \mathrm{GSp}(X) \times  \mathrm{GO}(Y)  \rightarrow \mathrm{GSp}(Z)
\]
given by the action \eqref{SP times O to SP}.
Then we note that $\lambda(i(g,h)) = \lambda(g)\lambda(h)^{-1}$. Let 
\[
R:=\{(g,h) \in \mathrm{GSp}(X) \times  \mathrm{GO}(Y) \, | \, \lambda(g) = \lambda(h) \}
\supset
\mathrm{Sp}(X) \times  \mathrm{O}(Y).
\]
We may define an extension of the Weil representation 
of $\mathrm{Sp}(X, \mA) \times  \mathrm{O}(Y, \mA)$ 
to $R(\mA)$ as follows.
Let $X= X_+ \oplus X_-$ be a polarization of $X$
and use the polarization $Z_{\pm} = X_{\pm} \otimes Y$
of $Z$ to realize the Weil representation $\omega_\psi$.
Then we note that
\[
\omega_{\psi}(1, h)\phi(z) = \phi\left( i \left(h \right)^{-1} z \right) \quad \text{for $h\in\mathrm{O}\left(
\mA\right)$ and $\phi\in \mathcal{S}(Z_+(\mA))$}.
\]
Thus we define an action $L$
of $\mathrm{GO}\left(Y,\mA\right)$ on $\mathcal{S}(Z_+(\mA))$
by
\[
L\left(h\right)\phi\left(z\right)=|\lambda(h)|^{-\frac{1}{8} \dim~X \cdot \dim~Y }
\phi\left(i \left( h \right)^{-1}z\right).
\]
Then we may extend the Weil representation  $\omega_\psi$ 
of $\mathrm{Sp}(X, \mA) \times  \mathrm{O}(Y, \mA)$
to $R(\mA)$ by 
\[
\omega_{\psi}(g, h) \phi = \omega_\psi(g_1, 1) L\left(h\right)\phi
\quad\text{for $\phi\in\mathcal{S}(Z_+(\mA))$ and $\left(g,h\right)\in R\left(\mA\right)$},
\]
where
\[
g_1 = g \begin{pmatrix}\lambda(g)^{-1}&0\\ 0&1 \end{pmatrix} \in \mathrm{Sp}(X,\mA).
\]

In general, for any polarization $Z= Z_+^\prime \oplus Z_-^\prime$, there exists an
$\mathrm{Sp}(X, \mA) \times \mathrm{O}(Y)(\mA)$-isomorphism 
$p : \mathcal{S}(Z_+(\mA)) \to \mathcal{S}(Z_+^\prime(\mA))$ given by an 
integral transform (see Ichino-Prasanna~\cite[Lemma~3.3]{IP}).
Let us denote the realization of the Weil representation
of $\mathrm{Sp}(X, \mA) \times \mathrm{O}(Y)(\mA)$
on $\mathcal{S}(Z_+^\prime(\mA))$ by $\omega_\psi^\prime$.
Then we may extend $\omega_\psi^\prime$
to $R\left(\mA\right)$ by
\[
\omega_\psi^\prime\left(g,h\right)=p\circ \omega_\psi\left(g,h\right)\circ
p^{-1}\quad\text{for $(g, h) \in R(\mA)$}.
\]

For $\phi \in \mathcal{S}(Z_{+}(\mA))$, we define the theta kernel $\theta^\phi$ by 
\[
\theta_{\psi}^{\phi}(g , h) = \theta^{\phi}(g , h) := \sum _{z_{+} \in Z_{+}(F) } 
\omega_{\psi} \left(g,h\right)
 \phi(z_{+}) \quad 
\text{for $(g, h) \in R(\mA)$}.
\]
Let
\begin{equation}
\label{gl def gap+}
\mathrm{GSp}(X, \mA)^{+}= \left\{ g \in \mathrm{GSp}(X, \mA) \mid \, \lambda(g) = \lambda(h) \text{ for some } 
h \in \mathrm{GO}(Y, \mA) \right\}
\end{equation}
and $\mathrm{GSp}(X, F)^{+} = \mathrm{GSp}(X, \mA)^{+} \cap \mathrm{GSp}(X, F)$.

As in \cite[Section~5.1]{HK}, for a cusp form $f$ on $\mathrm{GSp}(X, \mA)^{+}$, 
we define its theta lift to $\mathrm{GO}(Y, \mA)$ by 
\[ 
\Theta_\psi^{X, Y} (f, \phi)(h) = \Theta(f, \phi)(h) := \int _{\mathrm{Sp}(X, F) \backslash \mathrm{Sp}(X, \mA)} \theta ^{\phi}(g_1g , h)f(g_1 g) \, dg_1
\]
for $h \in \mathrm{GO}(Y, \mA)$,
where $g \in \mathrm{GSp}(X, \mA)^+$ is chosen so that $\lambda(g) = \lambda(h)$.
It defines an automorphic form on $\mathrm{GO}(Y, \mA)$. 
For a cuspidal automorphic representation $(\pi_+, V_{\pi_+})$ of $\mathrm{GSp}(X, \mA)^+$, 
we denote by $\Theta_\psi(\pi_+)$ the theta lift of $\pi_+$ to $\mathrm{GO}(Y, \mA)$. Namely
\[
\Theta_\psi^{X, Y}(\pi_+) = \Theta_\psi(\pi_+): = \left\{ \Theta (f, \phi) : f \in V_{\pi_+}, \,
\phi \in \mathcal{S}(Z_+(\mA)) \right\}.
\]
Furthermore, for an irreducible cuspidal automorphic representation $(\pi, V_\pi)$ of $\mathrm{GSp}(X, \mA)$, we define 
\[
\Theta_\psi(\pi): = \Theta_\psi(\pi|_{\mathrm{GSp}(X, \mA)^+})
\]
where $\pi|_{\mathrm{GSp}(X, \mA)^+}$ denotes the automorphic representation of $\mathrm{GSp}(X, \mA)^+$
with its space of automorphic forms
$\left\{  \varphi |_{\mathrm{GSp}(X, \mA)^+} : \varphi \in V_\pi \right\}$.

As for the opposite direction,
 for a cusp form $f^\prime$ on $\mathrm{GO}(Y, \mA)$, we define its theta lift
$\Theta (f^\prime, \phi)$
to $\mathrm{GSp}(X, \mA)^+$ by 
\[ 
\Theta (f^\prime, \phi)(g): = \int _{\mathrm{O}(Y, F) \backslash \mathrm{O}(Y, \mA)} \theta ^{\phi}(g , h_1h)f(h_1 h) \, dh_1
\quad\text{for $g\in\mathrm{GSp}\left(X,\mA\right)^+$},
\]
where $h \in \mathrm{GO}(Y,\mA)$ is chosen so that $\lambda(g) = \lambda(h)$.
For an irreducible cuspidal automorphic representation $(\sigma, V_\sigma)$ of $\mathrm{GO}(Y, \mA)$, 
we define the theta lift $\Theta_\psi(\sigma)$ of $\sigma$ to 
$\mathrm{GSp}(X, \mA)^+$ by
\[
\Theta_\psi(\sigma):=
\left\{\Theta (f^\prime, \phi): f^\prime\in V_\sigma,\, \phi \in \mathcal{S}(Z_+(\mA)) \right\}.
\]
Moreover we extend $\theta (f^\prime, \phi)$ to an automorphic form on $\mathrm{GSp}(X, \mA)$ by the natural embedding 
\[
\mathrm{GSp}(X, F)^+ \backslash \mathrm{GSp}(X, \mA)^+ \rightarrow \mathrm{GSp}(X, F) \backslash \mathrm{GSp}(X, \mA)
\]
and extension by zero. Then we define the theta lift $\Theta_\psi(\sigma)$
of $\sigma$ to $\mathrm{GSp}(X, \mA)$ 
as the $\mathrm{GSp}\left(X,\mA\right)$ representation generated by such
$\theta\left(f^\prime,\phi\right)$ for $f^\prime\in V_\sigma$ 
and $\phi\in \mathcal S\left(Z_+\left(\mA\right)\right)$.

For some $X$ and $Y$, theta correspondence for the dual pair $(\mathrm{GSp}(X)^+, \mathrm{GO}(Y))$ 
gives  theta correspondence between $\mathrm{GSp}(X)^+$ and $\mathrm{GSO}(Y)$ by the restriction
of representations of $ \mathrm{GO}(Y)$ to $\mathrm{GSO}(Y)$.
Indeed, when $\dim X=4$ and $\dim Y = 6$, we may consider theta correspondence for the pair 
$(\mathrm{GSp}(X)^+, \mathrm{GSO}(Y))$. 
In Gan-Takeda~\cite{GT10, GT0}, they study the case when $\mathrm{GSO}(Y) \simeq \mathrm{GSO_{3,3}}$ or  $\mathrm{GSO}_{5,1}$,
and, in \cite{Mo}, the case when $\mathrm{GSO}(Y) \simeq \mathrm{GSO_{4,2}}$ is studied.
In these cases, for a cusp form $f$ on $\mathrm{GSp}(X, \mA)^{+}$, we denote by $\theta (f, \phi)$ the restriction of $\Theta (f, \phi)$ 
to $\mathrm{GSO}(Y, \mA)$.
Moreover, for a cuspidal automorphic representation $(\pi_+, V_{\pi_+})$ of $\mathrm{GSp}(X, \mA)^+$, we define the theta lift $\theta_\psi\left(\pi_+\right)$ of $\pi_+$ 
to $\mathrm{GSO}(Y, \mA)$ by
\[
\theta_\psi^{X, Y}(\pi_+) = \theta_\psi(\pi_+) := \left\{ \theta (f, \phi) : f \in V_{\pi_+}, \phi \in \mathcal{S}(Z_+(\mA)) \right\}.
\]
Similarly, for a cusp form $f^\prime$ on $\mathrm{GSO}(Y, \mA)$, we define its theta lift
$\theta (f^\prime, \phi)$
to $\mathrm{GSp}(X, \mA)^+$ by 
\[ 
\theta (f^\prime, \phi)(g): = \int _{\mathrm{SO}(Y, F) \backslash \mathrm{SO}(Y, \mA)} \theta ^{\phi}(g , h_1h)f(h_1 h) \, dh_1
\quad\text{for $g\in\mathrm{GSp}\left(X,\mA\right)^+$},
\]
where $h \in \mathrm{GSO}(Y,\mA)$ is chosen so that $\lambda(g) = \lambda(h)$. We extend it to an automorphic from on 
$\mathrm{GSp}(X, \mA)$ as above.
For a cuspidal automorphic representation $(\sigma, V_\sigma)$ of $\mathrm{GSO}(Y, \mA)$, 
we define the theta lift $\theta_\psi(\sigma)$ of $\sigma$ to 
$\mathrm{GSp}(X, \mA)^+$ by
\[
\theta_\psi(\sigma):=
\left\{\theta (f^\prime, \phi): f^\prime\in V_\sigma,\, \phi \in \mathcal{S}(Z_+(\mA)) \right\}.
\]
\color{black}
\begin{Remark}
\label{theta irr rem}
Suppose that $\Theta_\psi(\pi_+)$ (resp. $\theta_\psi(\sigma)$)
is non-zero and cuspidal
where 
$(\pi_+, V_{\pi_+})$ (resp. $(\sigma, V_\sigma)$) is an irreducible cuspidal automorphic 
representation of $\mathrm{GSp}(X, \mA)^+$ (resp. $\mathrm{GO}(Y, \mA)$).
Then Gan~\cite[Proposition~2.12]{Gan} has shown that 
the Howe duality, which was proved by Howe~\cite{Ho1} at archimedean places, by Waldspurger~\cite{Wa}
at odd finite places and finally by Gan and Takeda~\cite{GT}
at all finite places, implies that $\Theta_\psi(\pi_+)$ (resp.  $\theta_\psi(\sigma)$)
is  irreducible and cuspidal.
Moreover in the case of our concern, namely when $\dim_FX=4$ and $\dim_FY=6$,
the irreducibility of $\Theta_\psi(\pi_+)$
implies that of $\theta_\psi(\pi_+)$
by the conservation relation due to Sun and Zhu~\cite{SZ}.
\end{Remark}
%
%
%
%
%
%
%
%

\subsubsection{
Pull-back of the global
Bessel periods for the dual pairs 
$\left(\mathrm{GSp}_2,\mathrm{GSO}_{4,2}\right)$ and $\left(\mathrm{GSp}_2,\mathrm{GSO}_{3,3}\right)$ 
}
\label{sp4 so42}
Our goal here is to prove the pull-back formula~\eqref{f: first pull-back formula}.

First we introduce the set-up.
Let  $X$ be the space of $4$ dimensional
row vectors over $F$ equipped with the symplectic form
\[
\langle w_1,w_2\rangle=w_1\begin{pmatrix}0&1_2\\-1_2&0\end{pmatrix}
\,{}^tw_2.
\]
Let us take the standard basis of $X$ and name the basis vectors as
\begin{equation}\label{sb for X}
x_1 = (1, 0, 0, 0), \quad x_2 = (0, 1, 0, 0),
\quad x_{-1} = (0, 0, 1, 0), \quad x_{-2} = (0, 0, 0, 1).
\end{equation}
Then the matrix representation of $\mathrm{GSp}\left(X\right)$ with
respect to the standard  basis is $G=\mathrm{GSp}_2$ defined by
\eqref{gsp}.
We let $G$ act on $X$ from the right.

Let $Y$ be the space of $6$ dimensional column vectors over $F$
equipped with the non-degenerate symmetric bilinear form
\[
\left(v_1,v_2\right)={}^tv_1 S_2 v_2
\]
where the symmetric matrix $S_2$ is given by \eqref{d: symmetric}.
Let us take the standard basis of $Y$ and name the basis vectors as
\begin{align*}
y_{-2} &= {}^{t}(1, 0, 0, 0, 0, 0), \quad y_{-1} = {}^{t}(0, 1, 0, 0, 0, 0),\\
e_1&= {}^{t}(0, 0, 1, 0, 0, 0), \quad e_2 = {}^{t}(0, 0, 0, 1, 0, 0),\\
y_{1} &= {}^{t}(0, 0, 0, 0, 1, 0), \quad y_{2} = {}^{t}(0, 0, 0, 0, 0, 1).
\end{align*}
We note that $( y_i , y_j )=\delta_{ij}, ( e_1 , e_1)=2$ and $( e_2 , e_2) =-2d$. 
Since $d \in F^\times \setminus (F^\times)^2$, with respect to the standard basis, the matrix representations
of $\mathrm{GO}\left(Y\right)$ and $\mathrm{GSO}\left(Y\right)$
are $\mathrm{GO}_{4,2}$ defined by \eqref{d: GO_n+2,n}
and $\mathrm{GSO}_{4,2}$ defined by \eqref{d: GSO_n+2,n},
respectively.
In this section, we also study the theta correspondence for the dual pair $(\mathrm{GSp}(X), \mathrm{GSO}_{3,3})$, for which,
we may use the above matrix representation with $d \in (F^\times)^2$.
Hence, in the remaining of this section, we study theta correspondence for $(\mathrm{GSp}(X), \mathrm{GSO}(Y))$ for an arbitrary $d \in F^\times$.

We shall denote $\mathrm{GSp}(X, \mA)^+$ as $G(\mA)^+$
and also $\mathrm{GSp}(X, F)^+ $ as $G(F)^+$.
We note that when $d \in (F^\times)^2$, $\mathrm{GSp}(X)^+ = \mathrm{GSp}(X)$.

Let $Z=X\otimes Y$ and we take a polarization 
$Z=Z_+\oplus Z_-$ as follows.
First we take $X=X_+\oplus X_-$ where
\[
X_{+} = F \cdot x_1 + F \cdot x_2 \quad \text{and} \quad X_{-}= F \cdot x_{-1} + F \cdot x_{-2}
\] 
as the polarization of $X$.
Then we decompose $Y$ as $Y=Y_+\oplus Y_0\oplus Y_-$
where
\[
Y_{+} =F \cdot y_1+ F \cdot y_2, \quad
Y_0 = F \cdot e_1 + F \cdot e_2\quad
 \text{and $ Y_{-} =F \cdot y_{-1} + F \cdot y_{-2}$}.
\]
Then let
\[
Z_{\pm} = \left(X \otimes Y_{\pm} \right)\oplus
\left(X_{\pm} \otimes Y_0\right)
\]
where the double sign corresponds.
To simplify the notation,
we sometimes write $z_+\in Z_+$ as
$z_+=\left(a_1,a_2;b_1,b_2\right)$ when
\[
z_{+} = a_1 \otimes y_1 + a_2 \otimes y_2 + b_1 \otimes e_1 + b_2 \otimes e_2\in Z_+,
\quad\text{where $ a_i \in X, \, b_i \in X_{+}\, \left(i=1,2\right)$}.
\]


Let us compute the pull-back of $(X, \chi, \psi)$-Bessel periods on $\mathrm{GSO}(Y)$ defined by \eqref{e: Bessel GSO(4,2)}
with respect to the theta lift
from $G$.
%
\begin{proposition}
\label{pullback Bessel gsp}
Let $\left(\pi ,V_\pi \right)$ be an irreducible cuspidal
automorphic representation of $G\left(\mA\right)$
whose central character is $\omega_\pi$ and 
$\chi$ a character of 
$\mathbb A_E^\times$ such that $\chi\mid_{\mathbb A^\times}=\omega_\pi^{-1}$. 
Let $X\in\mathrm{Mat}_{2\times 2}\left(F\right)$
such that $\det X\ne 0$.

Then for 
$f \in V_\pi$ and $\phi \in \mathcal{S}(Z_+(\mA))$, we have 
\begin{equation}\label{f: first pull-back formula}
\mathcal{B}_{X, \chi, \psi}(\theta(f: \phi)) = \int_{N(\mA) \backslash G^1(\mA)} B_{S_X, \chi^{-1},\psi}(\pi(g)f)\left( \omega_\psi(g, 1) \phi\right)(v_X) \, dg
\end{equation}
where $B_{S_X, \chi^{-1},\psi}$ is the 
$\left(S_X,\chi^{-1},\psi\right)$-Bessel period on $G$ 
defined by \eqref{Beesel def gsp}.

Here, for $X= \begin{pmatrix}x_{11}&x_{12}\\x_{21}&x_{22} \end{pmatrix}$, we define
a vector $v_X\in Z_+$ by
 \begin{equation}\label{e: def of v_X}
v_X := \left(x_{-2}, x_{-1};\frac{x_{21}}{2}x_1+\frac{x_{11}}{2} x_2,  -\frac{x_{22}}{2d}x_1-\frac{x_{12}}{2d} x_2\right)
\end{equation}
 and a $2$ by $2$ symmetric matrix $S_X$ by
\begin{equation}\label{e: def of S_X}
S_X: = \frac{1}{4d}{}^{t}(J_2 \,{}^{t}X J_2)S_0 ( J_2\,{}^{t}X J_2).
\end{equation}
We regard $\chi$ as a character of $\mathrm{GSO}(S_X)(\mA)$ by 
\begin{equation}\label{chi as a character}
\mathrm{GSO}(S_X) \ni k \mapsto \chi((J_2{}^{t}XJ_2) k (J_2{}^{t}XJ_2)^{-1})\in\mathbb C^\times.
\end{equation}

In particular, 
the $(S_X, \chi^{-1},\psi)$-Bessel period does not vanish on $V_\pi$
if and only if the $\left(X,\chi, \psi\right)$-Bessel period does not vanish
on $\theta_\psi\left(\pi\right)$.
\end{proposition}
%
%
%
\begin{proof}
We compute the $\left(X,\chi, \psi\right)$-Bessel period defined by \eqref{e: Bessel GSO(4,2)}
in stages.
We consider subgroups of $N_{4,2}$ given by:
\begin{align}
N_0(F) &= \left\{ u_0(x) :=
\begin{pmatrix}
1&-{}^{t}X_0 S_1& 0\\
0&1_4&X_0\\
0&0&1
\end{pmatrix}
\mid X_0 = \begin{pmatrix} x\\0\\0\\0\end{pmatrix}
\right\};\label{d: N_0}
\\
N_1(F) &= \left\{ u_1(s_1, t_1) :=
\begin{pmatrix}
1&-{}^{t}X_1 S_1&-\frac{1}{2}{}^{t}X_1 S_1 X_1\\
0&1_4&X_1\\
0&0&1
\end{pmatrix}
\mid X_1 = \begin{pmatrix} 0\\s_1\\t_1\\0 \end{pmatrix}
\right\};\label{d: N_1}
\\
N_2(F) &= \left\{ u_2(s_2, t_2) :=
\begin{pmatrix}
1&0&0&0&0\\
0&1&-{}^{t}X_2 S_0&-\frac{1}{2}{}^{t}X_2 S_0 X_2&0\\
0&0&1_2&X_2&0\\
0&0&0&1&0\\
0&0&0&0&1
\end{pmatrix}
\mid X_2 = \begin{pmatrix} s_2\\t_2\end{pmatrix}
\right\}\label{d: N_2}
\end{align}
where $S_0$ and $S_1$ are given by \eqref{d: symmetric}.
Then we have
\[
N_0 \lhd N_0 N_1 \lhd N_0 N_1 N_2 =N_{4.2}.
\]
Thus we may write 
\begin{multline}
\label{bessel u0 u1 u2}
\mathcal{B}_{X, \chi, \psi}(\theta(f: \phi))
=
\int_{\mA^\times M_{X}(F) \backslash M_X(\mA)}
 \int_{(F  \backslash \mA_{F})^2} 
\int_{(F  \backslash \mA_{F})^2} \int_{F  \backslash \mA_{F}}
\\
 \theta(f, \phi)(u_0(x) u_1(s_1, t_1) u_2(s_2, t_2) h) 
 \\
\times \psi(x_{21}s_1+x_{22}t_1+x_{11}s_2+x_{12}t_2)^{-1}\chi (h)^{-1} \, dx \,ds_1 \,dt_1 \,ds_2 \,dt_2 \,dh.
\end{multline}
For $h \in \mathrm{GSO}(Y, \mA)$, let us define 
\[
W_0(\theta(f: \phi))(h) 
:= \int_{F \backslash \mA_F} \theta(f, \phi)(u_0(x)h) \, dx.
\]
From the definition of the theta lift, we have
\begin{multline}
\label{W_0 1}
W_0(\theta(f, \phi))(h) \\
=  \int_{F \backslash \mA_F} \int_{G^1(F) \backslash G^1(\mA_F)} \sum_{a_i \in X, b_i \in X_{+}} \left(\omega_\psi
(g_1 \lambda_s(\nu(h)), u_0(x) h) \phi\right)(a_1, a_2; b_1, b_2)
\\
\times
f(g_1 \lambda_s(\lambda(h)))\, dg_1 \, dx.
\end{multline}
Here, for $a \in \mA^\times$, we write
\[
\lambda_s(a) = \begin{pmatrix} 1_2&0\\0 &a \cdot 1_2 \end{pmatrix}.
\]
Since $Z_{-}\,(1, u_0(x))= Z_{-}$ and we have
\[
z_{+}\,(1, u_0(x)) = z_{+} + (x \cdot a_1 \otimes y_{-2} -x \cdot a_2 \otimes y_{-1}),
\]
we observe that
\begin{align}
\label{global comp1}
\left(\omega_\psi(1, u_0(x)) \phi\right)(z_+) 
&=\psi \left(\frac{1}{2} \langle z_+, x \cdot a_1 \otimes y_{-2} -x \cdot a_2 \otimes y_{-1} \rangle \right)  \phi(z_+)
\\
\notag
&= \psi \left( -x \langle a_1, a_2 \rangle \right) \phi(z_+).
\end{align}
Thus in the  summation of the right-hand side of \eqref{W_0 1},  
only $a_i$ such that  $\langle a_1, a_2 \rangle = 0$ contributes to the integral 
$W_0(\theta(f, \phi))$, and we obtain
\begin{multline*}
W_0(\theta(f, \phi))(h) 
= \int_{G^1(F) \backslash G^1(\mA_F)}
\\
 \sum_{\substack{a_i \in X, \langle a_1, a_2 \rangle =0,\\ b_i \in X_{+}}} \left(\omega_\psi(g_1 \lambda_s(\lambda(h)), h) \phi\right)(a_1, a_2; b_1, b_2)
f(g_1 \lambda_s(\lambda(h)))\, dg_1.
\end{multline*}
Since the space spanned by $a_1$ and $a_2$ is isotropic,
there exists $\gamma \in G^1(F)$ such that $a_1 \gamma^{-1}, a_2 \gamma^{-1} \in X_{-}$.
Let us define an equivalence relation $\sim$ on $(X_{-})^{2}$ by
\[
(a_1 , a_2) \sim (a_1^\prime , a_2^\prime)
\underset{\text{def.}}{\Longleftrightarrow}
\text{there exists $\gamma\in G^1\left(F\right)$
such that $a_i^\prime = a_i \gamma$ for $i=1,2$}.
\]
Let us denote by $\mathcal{X}_{-}$ the set of equivalence classes
$\left(X_{-}\right)^2 \slash \sim$
and by $\overline{(a_1, a_2)}$ the equivalence class 
containing $\left(a_1,a_2\right)\in\left(X_{-}\right)^2$.
Then we may write $W_0 (\theta (f, \phi)) (h)$ as
\begin{multline*}
\int_{G^1(F) \backslash G^1(\mA_F)} 
\\
\sum _{ \overline{(a_1 , a_2)} \in \mathcal{X}_{-}}\, \sum_{\gamma \in V(a_1 , a_2) \backslash G^1(F)}\, \sum_{b_i \in X_{+}}
\left(\omega_\psi (g_1 \lambda_s(\lambda(h)) , h) \phi\right)(a_1 \gamma, a_2 \gamma; b_1, b_2) 
\\
\times f(g_1  \lambda_s(\lambda(h))) \, dg_1.
\end{multline*}
Here 
\[
V(a_1, a_2) = \{ g \in G^1(F) \mid \text{$a_ig =a_i$ for $i=1,2$}\}.  
\]
\begin{lemma}
For any $g \in G(\mA)^+$ and $h \in \mathrm{GSO}(Y, \mA)$
such that $\lambda(g) = \lambda(h)$,
\[
\sum_{b_i \in X_{+}}
\left(\omega_\psi (g, h) \phi\right)(a_1 \gamma, a_2 \gamma, b_1, b_2)
=\sum_{b_i \in X_{+}}
\left(\omega_\psi (\gamma g, h) \phi\right)(a_1 , a_2 , b_1, b_2).
\]
\end{lemma}
\begin{proof}
This is proved by an argument
 similar to the one for  \cite[Lemma~2]{Fu}.
\end{proof}
%
Further, by an argument similar to the one for
 $W_0 (\theta (f, \phi)) (h)$, we shall prove the following lemma.
%
\begin{lemma}
\label{pull comp 2}
For any $g \in G(\mA)^+$ and $h \in \mathrm{GSO}(Y, \mA)$
such that $\lambda(g) = \lambda (h)$,
\begin{align*}
&\int_{(F \backslash \mA_F)^{2}} \psi^{-1}(x_{21}s_1+x_{22}t_1) \left(\omega_\psi (g , u_{1}(s_1, t_1) h) \phi\right)(a_1 , a_2 , b_1, b_2)
 \, ds_1 \, dt_1\\
=& \left\{
\begin{array}{ll}
\left(\omega_\psi (g , h)\phi\right)(a_1 , a_2 , b_1, b_2) & \text{if $ \langle a_2 , b_1 \rangle =-\frac{x_{21}}{2}$ and $\langle a_2, b_2 \rangle= \frac{x_{22}}{2d}$};\\
& \\
0 & \text{otherwise}\\
\end{array}
\right.
\end{align*}
and
%
%
%
\begin{align*}
&\int_{(F \backslash \mA_F)^{2}} \psi^{-1}(x_{11} s_2+x_{12}t_2) 
\left(\omega_\psi(g , u_{2}(s_2, t_2) h) \phi\right)(a_1 , a_2 , b_1, b_2)
 \, ds_2\, dt_2\\
=& \left\{
\begin{array}{ll}
\left(\omega_\psi (g , h)\phi\right)(a_1 , a_2 , b_1, b_2) & \text{if $ \langle a_1 , b_1 \rangle =-\frac{x_{11}}{2}$ and $\langle a_1, b_2 \rangle =\frac{x_{12}}{2d}$};\\
& \\
0 & \text{otherwise}.\\
\end{array}
\right.
\end{align*}
\end{lemma}
%
%
%
%
\begin{proof}
Since $Z_-\, (1, u_1(s_1, t_1)) = Z_-$ and we have 
\begin{multline*}
z_+\,(1, u_1(s_1, t_1)) = z_{+} +2s_1 (b_1 \otimes y_{-2}) - 2dt_1(b_2 \otimes y_{-2}) 
\\
+
(-s_1^2+2dt_1^2)a_2 \otimes y_{-2}-s_1a_2 \otimes e_1-t_1 a_2 \otimes e_2,
\end{multline*}
we obtain
\begin{align*}
&\left(\omega_\psi(1, u_1(s_1, t_1)) \phi\right)(z_+) =
\psi \left(\frac{1}{2} \left( 2s_1\langle a_2, b_1 \rangle -2dt_1 \langle a_2, b_2 \rangle   \right) \right) 
\\
&\qquad\quad\times
\psi \left(\frac{1}{2} \left( (-s_1^2+2dt_1^2) \langle a_2, a_2 \rangle -2s_1 \langle b_1, a_2 \rangle +2dt_1 \langle b_2, a_2 \rangle \right) \right) 
\phi(z_+)
\\
=&\psi \left(2s_1\langle a_2, b_1 \rangle -2dt_1  \langle a_2, b_2  \rangle \right)\phi(z_+).
\end{align*}
Then the first assertion readily follows.

Similarly, since $Z_- \,(1, u_2(s_2, t_2)) = Z_-$ and we have 
\[
z_+\,(1, u_2(s_2, t_2)) = z_+ + a_1 \otimes ((s_2^2-dt_2^2)y_{-1}-s_2 e_1-t_2e_2 ) + 2s_2 b_1 \otimes y_{-1} -2dt_2 b_2 \otimes  y_{-1},
\]
we obtain
\begin{align*}
&\omega(1, u_2(s_2, t_2)) \phi(z_+) 
 = \psi \left(\frac{1}{2} \left( -2s_2 \langle b_1, a_1 \rangle +2dt_2 \langle b_2, a_1 \rangle   \right) \right)
\\
&\qquad\qquad\qquad\times
\psi \left(\frac{1}{2} \left( 2s_2 \langle a_1, b_1 \rangle
-2dt_2 \langle a_1, b_2 \rangle   \right) \right)\phi(z_+)
\\
=
&\psi \left(2s_2 \langle a_1, b_1 \rangle -2dt_2 \langle a_1, b_2 \rangle \right) \phi(z_+)
\end{align*}
and the second assertion follows.
\end{proof}
%
%
%
Lemma~\ref{pull comp 2} implies that
%
\begin{multline*}
\mathcal{B}_{X, \chi, \psi}(\theta(f: \phi)) 
= \int_{\mA^\times M_{X}(F) \backslash M_X(\mA)}
\int_{G^1(F) \backslash G^1(\mA_F)}  \chi (h)^{-1}
\\
\times
\sum _{ \overline{(a_1 , a_2)} \in \mathcal{X}_{-}}\, \sum_{\gamma \in V(a_1 , a_2) \backslash G^1(F)}\, \sum_{\substack{
b_i \in X_{+}, \langle a_i, b_1 \rangle = \frac{x_{i1}}{2},\\
 \langle a_i, b_2 \rangle =- \frac{x_{i2}}{2d} }}
 \\
\left(\omega_\psi (\gamma g_1 \lambda_s(\lambda(h)) , h) \phi\right)(a_1, a_2, b_1, b_2) \,
 f(g_1  \lambda_s(\lambda(h))) \, dg_1\,dh.
\end{multline*}
We note that $a_1$ and $a_2$ are linearly independent from the conditions on $a_i$ and $\det (X) \ne 0$.
Since $a_i \in X_-$ and $\dim X_- =2$,  we may take 
$\left(a_1,a_2\right)=\left(x_{-2},x_{-1}\right)$ as 
a representative.
Then we should have 
\[
b_1 =\frac{x_{21}}{2}x_1+\frac{x_{11}}{2} x_2, \quad  b_2= -\frac{x_{22}}{2d}x_1-\frac{x_{12}}{2d} x_2.
\]
Hence we get
\begin{align}
\label{MX to TS}
& \mathcal{B}_{X, \chi, \psi}(\theta(f: \phi)) 
=
 \int_{\mA^\times M_{X}(F) \backslash M_X(\mA)}
\int_{G^1(F) \backslash G^1(\mA_F)} 
\chi\left(h\right)^{-1}
\\
\notag
&\times
 \sum_{\gamma \in N(F) \backslash G^1(F)}
\left(\omega_\psi (\gamma g_1 \lambda_s(\lambda(h)) , h) \phi\right)(v_X) 
\, f(g_1  \lambda_s(\lambda(h))) \, dg_1\,dh
\\
=&  \notag
\int_{N(\mA) \backslash G^1(\mA_F)}  \int_{\mA^\times M_{X}(F) \backslash M_X(\mA)} \int_{N(F) \backslash N(\mA)}
\\
\notag  &\qquad\qquad\chi(h)^{-1}
\omega (v g_1 \lambda_s(\lambda(h)) , h) \phi(v_X) f(v g_1  \lambda_s(\lambda(h))) \, dv\,dg_1\,dh
\end{align}
where we put $v_X = (x_{-2}, x_{-1}; \frac{x_{21}}{2}x_1+\frac{x_{11}}{2} x_2,  -\frac{x_{22}}{2d}x_1-\frac{x_{12}}{2d} x_2)$.

For $u=\begin{pmatrix}1_2&A\\0&1_2\end{pmatrix}$ where
$A= \begin{pmatrix}a&b\\ b&c \end{pmatrix}
 \in \mathrm{Sym}^2$, we have 
\begin{align*}
&\left( x_{-2} \otimes y_1 + x_{-1} \otimes y_2+ \left(\frac{x_{21}}{2}x_1+\frac{x_{11}}{2} x_2 \right) \otimes e_1 +
\left( -\frac{x_{22}}{2d}x_1-\frac{x_{12}}{2d} x_2 \right)
\otimes e_2  \right) \left(u,1\right)
\\
=&
x_{-2} \otimes y_1 + x_{-1} \otimes y_2+
 \left(\frac{x_{21}}{2}(x_1 +a x_{-1}+b x_{-2} )+\frac{x_{11}}{2} (x_2+bx_{-1}  +c x_{-2}) \right) \otimes e_1 
 \\
 &\qquad\qquad+
\left( -\frac{x_{22}}{2d}(x_1 +a x_{-1}+b x_{-2} )-\frac{x_{12}}{2d} (x_2+bx_{-1}  +c x_{-2}) \right)
\otimes e_2. 
\end{align*}
Hence, when we put 
\begin{align*}
S_X&=\frac{1}{4d}{}^{t}(J_2{}^{t}XJ_2)S_0 (J_2{}^{t}XJ_2)
\\
& =\frac{1}{2d} \begin{pmatrix}x_{22}^2-dx_{21}^2&x_{22}x_{12}-dx_{21}x_{11} \\x_{22}x_{12}-dx_{21}x_{11} &x_{12}^2-dx_{11}^2 \end{pmatrix} 
\in \mathrm{Sym}^2(F),
\end{align*}
for $u\in N\left(\mA\right)$, 
we have 
\[
\left(\omega_\psi(ug\lambda_s(\lambda(h)), h) \phi\right)
(v_X) = \psi_{S_X} \left( u \right)^{-1} \omega_\psi(g\lambda_s(\lambda(h)), h) \phi(v_X).
\]
Therefore, we get
\begin{align*}
&\int_{N(\mA) \backslash G^1(\mA_F)}  \int_{\mA^\times M_{X}(F) \backslash M_X(\mA)} \int_{N(F) \backslash N(\mA)}  \chi(h)^{-1} 
\\
&\times
\left(\omega_\psi (g_1 \lambda_s(\lambda(h)) , h) \phi\right)(v_X) 
\,
f(u g_1  \lambda_s(\lambda(h)))  \psi_{S_X}\left(u\right)^{-1} \, du \, dh\, dg_1
\\
=&
\int_{N(\mA) \backslash G^1(\mA_F)}  \int_{\mA^\times M_{X}(F) \backslash M_X(\mA)} \int_{N(F) \backslash N(\mA)}  \chi (h)^{-1}
\\
&\times
\omega_\psi (\lambda_s(\lambda(h)) g_1, h) \phi(v_X) 
\,
|\lambda(h)|^3 f(u  \lambda_s(\lambda(h)) g_1) 
\psi_{S_X}\left(u\right)^{-1}
  \, du \, dh\, dg_1.
\end{align*}
By a direct computation, we see that
\[
\left(\omega_\psi (\lambda_s(\lambda(h)) g_1, h) \phi\right)(v_X) =|\lambda(h)|^{-3} 
\left(\omega_\psi (h_0 \lambda_s(\lambda(h)) g_1, 1) \phi\right)(v_X)
\]
when we write 
\[
h = \begin{pmatrix}(\det h) (h^X)^\ast &0&0\\ 0&h&0\\ 0&0&h^X\end{pmatrix},
\,
h_0 = \begin{pmatrix} ({}^{t}XJ_2)^{-1} {}^{t}h ({}^{t}XJ_2)&0\\ 0&(J_2X) h^{-1} (J_2X)^{-1}\end{pmatrix}.
\]
For $g \in \mathrm{GSO}(S_0)$, we have ${}^{t}g = wgw$ and we may write
\[
h_0 = \begin{pmatrix} (J_2{}^{t}XJ_2)^{-1} {}^{t}h (J_2{}^{t}XJ_2)&0\\ 0&{}^{t}\left(  (J_2{}^{t}XJ_2)^{-1} {}^{t}h (J_2{}^{t}XJ_2)\right)^{-1} \end{pmatrix}.
\]
Since we have
\[
\mathrm{GSO}(S_X) = (J_2{}^{t}XJ_2)^{-1} \mathrm{GSO}(S_0) (J_2{}^{t}XJ_2),
\]
we get
\begin{align}
\label{pull-back gsp complete}
&\int_{N(\mA) \backslash G^1(\mA_F)}  \int_{\mA^\times T_{S_X}(F) \backslash T_{S_X}(\mA)} \int_{N(F) \backslash N(\mA)}  \chi(h) 
\\
\notag
&
\qquad\qquad\times
\left(\omega_\psi (g_1, 1) \phi\right)(v_X) f(u h g_1) 
\psi_{S_X}\left(u\right)^{-1}  \, du \, dh\, dg_1
\\
=& \notag
\int_{N(\mA) \backslash G^1(\mA_F)} 
B_{S_X, \chi^{-1}}(\pi(g_1)f)   \left(\omega_\psi (g_1, 1) \phi\right)(v_X)  dg_1
\end{align}
where we regard $\chi$ as a character of $\mathrm{GSO}(S_X)(\mA)$ by \eqref{chi as a character}.

Finally the last statement concerning the equivalence 
of the non-vanishing conditions on
the $\left(S_X,\chi^{-1},\psi\right)$-Bessel period
and the $\left(X,\chi\right)$-Bessel period
follows from the pull-back formula~\eqref{f: first pull-back formula}
by an argument similar to the one in 
the proof of Proposition~ 2 in \cite{FM1}.
\end{proof}
%
%
%
%
%
%
%
%
%
%
\subsection{$\left(G_D, \mathrm{GSU}_{3,D}\right)$ case}
\subsubsection{Theta correspondence for quaternionic dual pair
with similitudes}
\label{def theta D}
Let $D$ be a quaternion division algebra over $F$.
Let $X_D$ (resp. $Y_D$) be a right (resp. left) $D$-vector space of finite rank
equipped with a non-degenerate hermitian bilinear form $(\,,\,)_{X_D}$ (resp. non-degenerate skew-hermitian bilinear form $\langle \, , \, \rangle_{Y_D}$).
Hence
 $(\,,\,)_{X_D}$ and  $\langle \, ,  \,\rangle_{Y_D}$ are $D$-valued
 $F$-bilinear form on $X_D$ and $Y_D$ satisfying:
\begin{gather*}
\overline{(x ,x^\prime)_{X_D}} = (x^\prime,x)_{X_D}, 
\quad (xa,x^\prime b)_{X_D} = \bar{a}(x,x^\prime)_{X_D} b,
\\
\overline{\langle y,y^\prime \rangle_{Y_D}} = -\langle y^\prime,y\rangle_{Y_D},  \quad
\langle ay,y^\prime b  \rangle_{Y_D} = a \langle y,y^\prime \rangle_{Y_D} \bar{b},
\end{gather*}
for $x, x^\prime \in X_D$, $y, y^\prime \in Y_D$ and $a,b\in D$.
We denote the isometry group of $X_D$ and $Y_D$ by $\mathrm{U}(X_D)$ and $\mathrm{U}(Y_D)$, respectively.
Then the space $Z_D = X_D \otimes_D Y_D$ is regarded as a symplectic space over $F$ with the non-degenerate alternating form $\langle \,, \, \rangle$
defined by 
\begin{equation}\label{d: quaternion dual pair}
\langle x_1 \otimes y_1, x_2 \otimes y_2 \rangle =   \mathrm{tr}_D ((x_1, x_2)_{X_D} \overline{\langle y_1, y_2 \rangle_{Y_D}}) \in F
\end{equation}
and we have a homomorphism $\mathrm{U}(X_D) \times \mathrm{U}(Y_D) \rightarrow \mathrm{Sp}(Z_D)$
defined by 
\begin{equation}
\label{DSP times DO to DSP}
(x \otimes y)(g, h) =  x g \otimes h^{-1}y 
 \quad \text{for $x \in X, y \in Y$, $h \in \mathrm{U}(Y_D)$
 and $g \in \mathrm{U}(X_D)$.}
\end{equation}
As in the case when $D \simeq \mathrm{Mat}_{2 \times 2}$, this mapping
 splits in the metaplectic group $\mathrm{Mp}(Z_D)$.
Hence we have the Weil representation $\omega_\psi$ of $\mathrm{U}(X_D, \mA) \times \mathrm{U}(Y_D, \mA)$ by restriction.

From now on, we suppose that the rank of $X_D$ is $2k$ and $X_D$ is maximally split, in the sense that its maximal isotropic subspace has rank $k$.

Let us denote by $\mathrm{GU}(X_D)$ (resp. $\mathrm{GU}(Y_D)$) the similitude unitary group of $X_D$ (resp. $Y_D$) with the similitude character $\lambda_D$
(resp. $\nu_D$).
Also we write the identity component of $\mathrm{GU}(Y_D)$ by $\mathrm{GSU}(Y_D)$.
Then the action \eqref{DSP times DO to DSP}
extends to a homomorphism
\[
i_D : \mathrm{GU}(X_D) \times \mathrm{GU}(Y_D) \rightarrow \mathrm{GSp}(Z_D)
\]
with the property
 $\lambda(i_D(g, h)) = \lambda_D(g) \nu_D(h)^{-1}$.
Let 
\[
R_D:=\{(g,h) \in \mathrm{GU}(X_D) \times  \mathrm{GU}(Y_D) \, | \, \lambda_D(g) = \nu_D(h) \}
\supset \mathrm{U}(X_D) \times  \mathrm{U}(Y_D).
\]
Since $X_D$ is maximally split, we have a Witt decomposition $X_D = X_D^+ \oplus X_D^-$ with 
maximal isotropic subspaces $X_{D}^{\pm}$. Then as in Section~\ref{SP times O to SP},
we may realize the Weil representation $\omega_\psi$ of $\mathrm{U}(X_D) \times  \mathrm{U}(Y_D)$ on 
$\mathcal{S}((X_D^+ \otimes Y_D)(\mA))$. 
In this realization,
for $h \in \mathrm{U}(Y_D)$ and
$\phi \in \mathcal{S}((X_D^+ \otimes Y_D)(\mA))$,
we have
\[
\omega_{\psi}(1, h)\phi(z) = \phi(i_D\left(h\right)^{-1} z).
\]
Hence, as in Section~\ref{SP times O to SP}, we may extend $\omega_\psi$ to $R_D(\mA)$ by 
\[
\omega_{\psi}(g, h) \phi\left(z\right) = |\lambda(h)|^{ -2 \mathrm{rank}~X_D \cdot \mathrm{rank}~Y_D} \omega_\psi(g_1, 1) \phi
\left(i_D\left(h\right)^{-1}z\right)
\]
for $\left(g,h\right)\in R_D\left(\mA\right)$,
where
\[
g_1 = g \begin{pmatrix}\lambda_D(g)^{-1}&0\\ 0&1 \end{pmatrix} \in \mathrm{U}(X_D).
\]

Then as  in Section~\ref{SP times O to SP},
we may extend the Weil representation  $\omega_\psi$
of $\mathrm{U}(X_D) \times  \mathrm{U}(Y_D)$ 
on $\mathcal S\left(
Z_+\left(\mathbb A_F\right)\right)$,
where $Z_D = Z_D^+ \oplus Z_D^-$ is an arbitrary polarization,
 to $R_D\left(\mA\right)$, 
by using the $\mathrm{U}(X_D) \times  \mathrm{U}(Y_D)$-isomorphism
$p:\mathcal{S}((X_D^+ \otimes Y_D)(\mA))\to\mathcal S\left(
Z_+\left(\mathbb A_F\right)\right)$.
Thus for $\phi\in\mathcal S\left(
Z_+\left(\mathbb A_F\right)\right)$, the theta kernel $\theta^\phi_\psi=\theta^\phi$
on $R_D(\mA)$ is defined by
\[
\theta_\psi^{\phi}(g , h) =\theta ^{\phi}(g , h) = \sum _{z_{+} \in Z_D^{+}(F) } \,
\omega_{\psi} (g , h) \phi(z_{+}) \quad 
\text{for $(g, h) \in R_D(\mA)$}.
\]
Let us define 
\[
\mathrm{GU}(X_D, \mA)^+ = \left\{ h \in \mathrm{GU}(X_D, \mA) : \lambda_D(h) \in \nu_D \left(\mathrm{GU}(Y_D, \mA) \right)\right\}.
\]
and
\[
\mathrm{GU}(X_D, F)^+ = \mathrm{GU}(X_D, \mA)^+ \cap \mathrm{GU}(X_D, F).
\]
We note that $\nu_D \left(\mathrm{GU}(Y_D, F_v) \right)$ contains $N_D(D(F_v)^\times)$ for any place $v$.
Thus, if $v$ is non-archimedean or complex, we have $\mathrm{GU}(X_D, F_v)^+ = \mathrm{GU}(X_D, F_v)$, and 
if $v$ is real, $|\mathrm{GU}(X_D, F_v) \slash \mathrm{GU}(X_D, F_v)^+| \leq 2$.

For a cusp form $f$ on $\mathrm{GU}(X_D, \mA)^+$, 
as in \ref{def theta}, we define the theta lift of $f$ to $\mathrm{GU}(Y_D, \mA)$ by 
\[ 
\Theta (f, \phi)(h) := \int _{\mathrm{U}(X_D, F) \backslash \mathrm{U}(X_D, \mA)} \theta ^{\phi}(g_1g , h)f(g_1 g) \, dg_1
\]
where $g \in \mathrm{GU}(X_D, \mA)^+$ is chosen so that $\lambda_D(g) = \nu_D(h)$.
It defines an automorphic form on $\mathrm{GU}(Y_D, \mA)$. 
When we regard $\Theta (f, \phi)(h)$ as an automorphic form on $\mathrm{GSU}(Y_D, \mA)$ by the restriction,
we denote it as $\theta (f, \phi)(h)$.
For an irreducible cuspidal automorphic representation $(\pi_+, V_{\pi_+})$ of $\mathrm{GU}(X_D, \mA)^+$, 
we denote by $\Theta_\psi(\pi_+)$ (resp. $\theta_\psi(\pi_+)$) the theta lift of $\pi_+$ to $\mathrm{GU}(Y_D, \mA)$
(resp. $\mathrm{GSU}(Y_D, \mA)$), namely
\begin{align*}
\Theta_\psi(\pi) &:= \left\{ \Theta (f, \phi) : f \in V_{\pi_+}, \phi \in \mathcal{S}(Z_D^+(\mA)) \right\},
\\
\theta_\psi(\pi) &:= \left\{ \theta (f, \phi) : f \in V_{\pi_+}, \phi \in \mathcal{S}(Z_D^+(\mA)) \right\},
\end{align*}
respectively.
Moreover, for an irreducible cuspidal automorphic representation $(\pi, V_\pi)$ of $\mathrm{GU}(X_D, \mA)$,
we define the theta lift $\Theta_\psi(\pi)$ (resp. $\theta_\psi(\pi)$) of $\pi$ to $\mathrm{GU}(Y_D, \mA)$ (resp. $\mathrm{GSU}(Y_D, \mA)$)
by $\Theta_\psi(\pi) := \Theta_\psi(\pi|_{\mathrm{GU}(X_D, \mA)^+})$ (resp.  $\theta_\psi(\pi) := \theta_\psi(\pi|_{\mathrm{GU}(X_D, \mA)^+})$).

As for the opposite direction, as in \ref{def theta},
for a cusp form  $f^\prime$ on $\mathrm{GSU}(Y_D, \mA)$, we define the theta lift of $f^\prime$ to $\mathrm{GU}(X_D, \mA)^+$ by 
\[ 
\theta (f^\prime, \phi)(g) := \int _{\mathrm{SU}(Y_D, F) \backslash \mathrm{SU}(Y_D, \mA)} \theta ^{\phi}(g , h_1h)f(h_1 h) \, dh_1
\]
where $h \in \mathrm{GSU}(Y_D, \mA)$ is chosen so that $\lambda_D(g) = \nu_D(h)$.
For an irreducible cuspidal automorphic representation $(\sigma, V_\sigma)$ of $\mathrm{GSU}(Y_D, \mA)$, we denote by
$\theta_\psi(\sigma)$ the theta lift of $\sigma$ to $\mathrm{GU}(X_D, \mA)^+$.
Moreover, we extend $\theta (f^\prime, \phi)$ to an automorphic form on $\mathrm{GU}(X_D, \mA)$ by the natural embedding 
\[
\mathrm{GU}(X_D, F)^+ \backslash \mathrm{GU}(X_D, \mA)^+ \rightarrow \mathrm{GU}(X_D, F) \backslash \mathrm{GU}(X_D, \mA)
\]
and extension by zero. Then we define the theta lift 
$\Theta_\psi\left(\sigma\right)$ of $\sigma$ to $\mathrm{GU}(X_D, \mA)$ 
as the $\mathrm{GU}(X_D, \mA)$ representation
generated by such $\theta\left(f^\prime,\phi\right)$
for $f^\prime\in V_\sigma$ and
$\phi\in\mathcal S\left(Z_+\left(\mA\right)\right)$.

\begin{Remark}
\label{theta irr rem D}
Suppose that $(\pi_+, V_{\pi_+})$ (resp. $(\sigma, V_\sigma)$) is
 an irreducible cuspidal automorphic representation of $\mathrm{GU}(X_D, \mA)^+$ 
(resp. $\mathrm{GSU}(Y_D, \mA)$).
Suppose moreover that the theta lift $\Theta_\psi(\pi_+)$ (resp. $\theta_\psi(\sigma)$) is non-zero and cuspidal.
Then by Gan~\cite[Proposition~2.12]{Gan},  $\Theta_\psi(\pi_+)$ (resp. $\theta_\psi(\sigma)$)
is an
irreducible cuspidal automorphic representation because of the Howe duality 
for quaternionic dual pairs proved 
by Gan and Sun~\cite{GSun} and Gan and Takeda~\cite{GT}.
We shall study the case $\dim_D~X_D=2$ and $\dim_D~Y_D=3$. In this case, by the conservation relation proved by Sun and Zhu~\cite{SZ},
the irreducibility of  $\Theta_\psi(\pi_+)$ implies
that of $\theta_\psi(\pi_+)$.
\end{Remark}

\subsubsection{Pull-back of the global Bessel periods for the
dual pair $\left(\mathrm{G}_D,\mathrm{GSU}_{3,D}\right)$}
\label{ss: second pull-back}

The set-up is as follows.

Let $X_D$ be the space of $2$ dimensional row vectors over $D$
equipped with the hermitian form
\[
\left(x,x^\prime\right)_{X_D}=\bar{x}\begin{pmatrix}0&1\\1&0\end{pmatrix}
\,{}^tx^\prime.
\]
Let us take the standard basis of $X_D$ and name the basis vectors
as
\[
x_+=\left(1,0\right),\quad x_-=\left(0,1\right).
\]
Then $G_D$ defined by \eqref{e: G_D} is the matrix representation of the
similitude unitary group $\mathrm{GU}\left(X_D\right)$ for $X_D$
with respect to the standard basis.

Let $Y_D$ be the space of $3$ dimensional column vectors over
$D$ equipped with the skew-hermitian form
\[
\langle y, y^\prime \rangle_{Y_D} = {}^{t} y\, \begin{pmatrix}0&0&\eta\\ 0&\eta&0 \\ \eta&0&0 \end{pmatrix}\, \overline{y^\prime}.
\]
Let us take the standard basis of $Y_D$ and name the basis vectors
as
\[
y_-={}^t\left(1,0,0\right),\quad
e={}^t\left(0,1,0\right),\quad
y_-={}^t\left(0,0,1\right).
\]
Then $\mathrm{GSU}_{3,D}$ defined in \ref{ss: quaternionic unitary groups} is the matrix representation of
the group $\mathrm{GSU}\left(Y_D\right)$ for $Y_D$
with respect to the standard basis.

We take  a polarization $Z_D=Z_{D,+}\oplus Z_{D,-}$
of $Z_D=X_D\otimes_D Y_D$ defined
as follows. Let
\[
X_{D,\pm}=x_{\pm}\cdot D
\]
where the double sign corresponds.
We decompose $Y_D$ as $Y_D=Y_{D,+}\oplus Y_{D,0}\oplus Y_{D,-}$
where
\[
Y_{D,+}=D\cdot y_+,\quad
Y_{D,0}=D\cdot y_0,
\quad
Y_{D,-}=D\cdot y_-.
\]
Then let
\begin{equation}\label{zD def}
Z_{D,\pm}=\left(X_D\otimes  Y_{D,\pm}\right)
\oplus \left(X_{D,\pm}\otimes Y_{D,0}\right)
\end{equation}
where the double sign corresponds.
To simplify the notation, 
we write $z_+ \in Z_{D,+}\left(\mA\right)$ as
$z_+=\left(a,b\right)$ when 
\[
z_+=a\otimes y_+ + b\otimes e
\quad\text{where $a\in X_D\left(\mA\right)$ and $b\in X_{D,+}
\left(\mA\right)$}
\]
and $\phi\left(z_+\right)$ as $\phi\left(a,b\right)$
for $\phi\in \mathcal S\left(Z_{D,+}\left(\mA\right)\right)$.
%
%
%
%
%
%

Let us compute the pull-back of the $\left(X,\chi, \psi\right)$-Bessel periods
on $\mathrm{GSU}_{3,D}$ defined by
\eqref{Besse def gsud} with respect to the theta lift
from $G_D$.
%
%
\begin{proposition}
\label{pullback Bessel gsp innerD}
Let $\left(\pi_D, V_{\pi_D}\right)$ be an irreducible cuspidal automorphic
representation of $G_D\left(\mA\right)$ whose central character is 
$\omega_\pi$ and $\chi$ a character of $\mathbb A_E^\times$ such that
$\chi\mid_{\mathbb A^\times}=\omega_\pi^{-1}$.
Let $X\in D^\times$.

Then for
 $f \in V_{\pi_D}$ and $\phi \in \mathcal{S}(Z_{D, +}(\mA))$, we have 
\begin{equation}\label{2: 2nd pull-back formula}
\mathcal{B}_{X, \chi, \psi}^D(\theta(f: \phi)) = \int_{N_{D}(\mA) \backslash G_D^1(\mA)} B_{\xi_X, \chi^{-1},\psi}(\pi(g)f)\left( \omega(g, 1) \phi\right)(v_{D, X}) \, dg
\end{equation}
where 
\begin{equation}\label{def of S_X}
\xi_X:=X\eta \bar{X}\in D^-\left(F\right),\quad
v_{D, X} := (x_-, -\eta^{-1}X x_+)\in Z_{D,+},
\end{equation}
and 
$B_{\xi_X, \chi^{-1},\psi}$ denotes the $\left(\xi_X, \chi^{-1},\psi\right)$-Bessel
period on $G_D$ defined by \eqref{e: def of bessel period}.

In particular, the $\left(\xi_X, \chi^{-1},\psi\right)$-Bessel period 
does not vanish on $V_{\pi_D}$ if and only if the 
$\left(X,\chi, \psi\right)$-Bessel period does not vanish on 
$\theta_\psi\left(\pi_D\right)$.
\end{proposition}
%
%
\begin{proof}
The proof of this proposition is 
similar to the one for 
Proposition~\ref{pullback Bessel gsp}.

Let $N_{0,D}$ be a subgroup of $N_{3,D}$ given by
\[
N_{0, D}(F) = \left\{ u_D(x):=\begin{pmatrix} 1&0&\eta x\\ 0&1&0\\ 0&0&1 \end{pmatrix} : x \in F\right\}.
\]
Then we note that $N_{0, D}$ is a normal subgroup of $N_{3,D}$ and $\psi_{X, D}$ is trivial on  $N_{0, D}(\mA)$. 
Since 
\[
Z_{D,-}(\mA)(1, u_D(x)) = Z_{D,-}(\mA)\,\,\text{and}\,\,
z_+(1, u_D(x)) = z_+ + a \otimes (-\eta x) y_-\,\,\text{for $x \in \mA$}, 
\]
we have
\[
\left(\omega(1, u_D(x)) \phi\right)(z_+) = \psi \left(- \frac{1}{2} \,\mathrm{tr}_D \left( \langle a, a \rangle  \eta^2 x \right) \right) \phi(z_+).
\]
Thus by an argument similar to the one in the proof of Proposition~\ref{pullback Bessel gsp}, one may show that 
\begin{multline}\label{e: 2nd bessel step 1}
\int_{N_{3,D}(F) \backslash N_{3,D}(\mA)}  \theta(f; \phi)(hu)  \psi^{-1}_{X, D}(u) \, du
\\
=
\int_{N_{3,D}(F) \backslash N_{3,D}(\mA)}
\int_{G_D^1(F) \backslash G_D^1(\mA)}
\,\, \sum _{ \overline{a} \in \mathcal{X}_{D,-}}\, \sum_{\gamma \in V_D(a) \backslash G_D^1(F)}\,
\sum_{b \in X_{D,+}}
\\
\left(\omega (\gamma g_1 \lambda_s^D(\nu(h)) , uh) \phi\right)(a, b) 
f(g_1  \lambda_s(\nu(h))) \, dg_1 \, du.
\end{multline}
Here 
$\mathcal X_{D,-}$ is the set of equivalence classes 
$X_{D,-}\slash\sim$ where $a\sim a^\prime$ if and only if
there exists a $\gamma\in G_D^1\left(F\right)$ such that
$a^\prime=a\gamma$,
$\bar{a}$ denotes the equivalence class
of $\mathcal X_{D,-}$ containing $a\in X_{D,-}$,
and, $V\left(a\right)=\left\{\gamma\in G_D^1\left(F\right)\mid a\gamma=a
\right\}$.
Then we may rewrite \eqref{e: 2nd bessel step 1} as
\begin{multline}\label{e: 2nd bessel step 2}
\int_{N_{3,D}(F) \backslash N_{3,D}(\mA)}  \theta(f; \phi)(hu)  \psi^{-1}_{X, D}(u) \, du
=
\int_{N_{3,D}(F) \backslash N_{3,D}(\mA)}
\int_{G_D^1(F) \backslash G_D^1(\mA)}
\\
 \sum_{N_D\left(F\right) \backslash G_D^1(F)}\,
\sum_{b \in X_{D,+}}
\left(\omega (\gamma g_1 \lambda_s^D(\nu(h)) , uh) \phi\right)(x_{-}, b) 
f(g_1  \lambda_s(\nu(h))) \, dg_1 \, du.
\end{multline}
Since,
for $u =  \begin{pmatrix} 1&-\eta^{-1} \bar{A} \eta&B\\ 0&1&A\\ 0&0&1 \end{pmatrix} \in N_{3,D}(\mA)$, we have 
$Z_{D,-}(\mA)(1, u) = Z_{D,-}(\mA)$ and 
\begin{align*}
z_+ (1, u) &=z_++ x_- \otimes (B^\prime y_- - A e +y_+) +b \otimes (\eta^{-1} \bar{A} \eta y_- + e) 
\\
&= z_+ +  x_- \otimes (B^\prime y_- - A e) +b \otimes (\eta^{-1} \bar{A} \eta y_-) ,
\end{align*}
 we obtain
\[
\left(\omega(1, u) \phi\right)(z_+) = \psi \left( \mathrm{tr}_D \left( \langle b, x_- \rangle \overline{(e, -Ae)} \right) \right) \phi\left(z_+\right)
= \psi \left(\mathrm{tr}_D \left(  \eta \langle b, x_- \rangle A  \right) \right) 
\phi\left(z_+\right).
\]
Hence in \eqref{e: 2nd bessel step 2}, only $b\in X_{D,-}$ satisfying
$\eta \langle b, x_- \rangle= X$, i.e.
$b=x_+(-\overline{X}\eta^{-1} )$ contributes.
Thus our integral is equal to 
\begin{align*}
&\int_{N_{D}(F) \backslash G_D^1(\mA)}  
\left(\omega (g_1 \lambda_s^D(\nu(h)) , uh) \phi\right)(v_{D, X}) f(g_1  \lambda_s^D(\nu(h))) \, dg_1 \, du
\\
=&
\int_{N_{D}(\mA) \backslash G_D^1(\mA)}  \int_{N_{D}(F) \backslash N_{D}(\mA)}
\omega (u g_1 \lambda_s^D(\nu(h)) , uh) \phi(v_{D, X})
\\
&\qquad\qquad\times f( u g_1  \lambda_s^D(\nu(h))) \, dg_1 \, du
\end{align*}
where $v_{D, X} =( x_-,  x_+(-\overline{X}\eta^{-1} ))$. 
Further for $u=\begin{pmatrix}1&a\\0&1\end{pmatrix}
\in N_D\left(\mA\right)$, we have 
\[
\left(
\omega \left(  ug, h \right) \phi
\right)(v_{D, X}) 
= \psi_{\xi_{X}}\left( u\right)^{-1}\left(\omega\left(g,h\right)\phi\right)
\left(v_{D,X}\right)
\]
where we put $\xi_{X} = X \eta \overline{X}$.
Thus our integral becomes
\begin{multline*}
\int_{N_{D}(\mA) \backslash G_D^1(\mA)}  \int_{N_{D}(F) \backslash N_{D}(\mA)}  \psi_{\xi_{X}} (u)^{-1}
\omega (g_1 \lambda_s^D(\nu(h)) , h) \phi(v_{D, X})
\\
\times
 f( u g_1  \lambda_s^D(\nu(h))) \, du \, dg_1 .
\end{multline*}
As for the integration over $\mathbb A^\times M_{X, D}\left(F\right)
\backslash M_{X, D}\left(\mathbb A\right)$ in 
\eqref{Besse def gsud}, by a direct computation, we see that
\[
\omega (\lambda_s^D(\nu(h)) g_1, h) \phi(v_{D, X}) =|\nu(h)|^{-3} \omega (h_0 \lambda_s(\nu(h)) g_1, 1) \phi(v_{D, X})
\]
where
\[
h = \begin{pmatrix} n_D(h) \cdot (h^X)^\ast &0&0\\ 0&h&0\\ 0&0&h^X\end{pmatrix}
\quad
\text{and}
\quad
h_0 = \begin{pmatrix} \overline{h^X}&0\\ 0&(h^X)^{-1}\end{pmatrix}.
\]
Therefore, as in the previous case, we obtain
\[
\mathcal{B}_{X, \chi, \psi}^D(\theta(f: \phi))=
\int_{N_{D}(\mA) \backslash G_D^1(\mA)} 
B_{\xi_X, \chi^{-1},\psi}(\pi(g_1)f)  \left( \omega (g_1, 1) \phi\right)(v_{D, X})  dg_1.
\]

The equivalence of the non-vanishing conditions follows
from the pull-back formula~\eqref{2: 2nd pull-back formula}
as Proposition~\ref{pullback Bessel gsp}.
\end{proof}
%
%
%
%
%
%
%
%
%
%
%
%
\subsection{Theta correspondence for similitude unitary groups}
In our proof of Theorem~\ref{ggp SO} and \ref{ref ggp}, we shall use theta correspondence for similitude unitary groups
besides theta correspondences for dual pairs $(\mathrm{GSp}_2, \mathrm{GSO}_{4,2})$ and $(G_D, \mathrm{GSU}_{3,D})$.
Let us recall the definition of the theta lifts in this case.

Let $(X, (\,,\,)_X)$ be an $m$-dimensional hermitian space over $E$,
and let  $(Y, (\,, \,)_Y)$ be an $n$-dimensional skew-hermitian space over $E$.
Then we may define the quadratic space 
\[
(W_{X, Y}, (\,, \,)_{X, Y}) := \left(\mathrm{Res}_{E \slash F} X \otimes Y, \mathrm{Tr}_{E \slash F}\left((\,, \,)_X \otimes \overline{(\,,\,)_Y} \right) \right).
\]
This is a $2mn$-dimensional symplectic space over $F$.
Then we denote its isometry group by $\mathrm{Sp}\left( W_{X, Y}\right)$.
For each place $v$ of $F$, we denote the metaplectic extension of $\mathrm{Sp}\left(W_{X, Y}\right)(F_v)$
by $\mathrm{Mp}\left( W_{X, Y}\right)(F_v)$. Also, $\mathrm{Mp}\left( W_{X, Y}\right)(\mA)$ denotes the 
metaplectic extension of $\mathrm{Sp}\left(W_{X, Y}\right)(\mA)$.

Let $\chi_X$ and $\chi_Y$  be characters of $\mA_E^\times \slash E^\times$
such that $\chi_{X}|_{\mA^\times} = \chi_E^m$ and $\chi_{Y}|_{\mA^\times} = \chi_E^n$.
For each place $v$ of $F$, let 
\[
\iota_{\chi_v} : \mathrm{U}(X)(F_v) \times \mathrm{U}(Y)(F_v) \rightarrow \mathrm{Mp}(W_{X, Y})(F_v)
\]
be the local splitting given by Kudla~\cite{Ku} depending on the choice of  a pair of characters $\chi_v =(\chi_{X,v}, \chi_{Y,v})$.
Using this local splitting, we get a splitting 
\[
\iota_{\chi} : \mathrm{U}(X)(\mA) \times \mathrm{U}(Y)(\mA) \rightarrow \mathrm{Mp}(W_{X, Y})(\mA),
\]
depending on $\chi = (\chi_X, \chi_Y)$. Then by the pull-back, we obtain the Weil representation $\omega_{\psi, \chi}$ of $\mathrm{U}(X)(\mA) \times \mathrm{U}(Y)(\mA)$.
When we fix a polarization $W_{X, Y} = W_{X, Y}^+ \oplus  W_{X, Y}^-$, we may realize  $\omega_{\psi, \chi}$
so that its space of smooth vectors is given by $\mathcal{S}(W_{X, Y}^+(\mA))$, the space of Schwartz-Bruhat functions on $W_{X, Y}^+(\mA)$.
We define 
\[
R := \left\{ (g, h) \in \mathrm{GU}(X) \times \mathrm{GU}(Y)  : \lambda(g) = \lambda(h) \right\}
\supset \mathrm{U}(X)\times \mathrm{U}(Y).
\]
Suppose that $\dim Y$ is even and $Y$ is maximally split, in the sense that $Y$ has a maximal isotropic subspace of dimension $\frac{1}{2} \dim Y$.
In this case, as in Section~\ref{def theta} and \ref{def theta D}, we may extend $\omega_{\psi,\chi}$ to $R(\mA)$.
On the other hand, in this case, we have an explicit local splitting of $R(F_v) \rightarrow \mathrm{Sp}(W_{X, Y})(F_v)$ by Zhang~\cite{CZ} 
and we may extend $\omega_{\psi,\chi}$ to $R(\mA)$ using this splitting.
These two extensions of $\omega_{\psi,\chi}$ to $R\left(\mA\right)$ coincide.

Then for $\phi \in \mathcal{S}(W_{X, Y}^+(\mA))$, we define the theta function $\theta_{\psi, \chi}^\phi$ on 
$R(\mA)$ by
\begin{equation}
\label{theta fct def}
\theta_{\psi, \chi}^\phi(g, h) = \sum_{w \in W_{X, Y}^+(F)} \omega_{\psi, \chi}(g, h)\phi(w).
\end{equation}
Let us define
\begin{align*}
\mathrm{GU}(X)(\mA)^{+}: &= \left\{g \in \mathrm{GU}(X)(\mA) : \lambda(g) \in \lambda(\mathrm{GU}(Y)(\mA)) \right\},
\\
\mathrm{GU}(X)(F)^{+} :&= \mathrm{GU}(X)(\mA)^{+} \cap \mathrm{GU}(X)(F).
\end{align*}
We define
$\mathrm{GU}(Y)(\mA)^{+}$ and $\mathrm{GU}(Y)(F)^{+} $ in a similar manner.
Let $(\sigma, V_\sigma)$ be an irreducible cuspidal automorphic representation of $\mathrm{GU}(X)(\mA)^{+}$.
Then for $\varphi \in V_\sigma$ and $\phi \in \mathcal{S}(W_{X, Y}^+(\mA))$, we define the theta lift of $\varphi$ by 
\[
\theta_{\psi, \chi}^\phi(\varphi)(h) = \int_{\mathrm{U}(X)(F) \backslash \mathrm{U}(X)(\mA)} \varphi(g_1g) \theta_{\psi, \chi}^\phi(g_1g, h) \, dg_1
\]
where $g_1 \in \mathrm{GU}(X)(\mA)^+$ is chosen so that $\lambda(g) = \lambda(h)$.
Further, we define the theta lift of $\sigma$ by
\[
\Theta_{\psi, \chi}^{X, Y}(\sigma) = \langle \theta^\phi_{\psi, \chi}(\varphi) ; \varphi \in \sigma, \phi \in \mathcal{S}(W_{X, Y}^+(\mA)) \rangle.
\]
When the space we consider is clear, we simply write $\Theta^{X, Y}_{\psi, \chi}(\sigma) = \Theta_{\psi, \chi}(\sigma)$.
Similarly, for an irreducible cuspidal automorphic representation $\tau$ of $\mathrm{U}(Y)(\mA)$,
we define $\Theta_{\psi, \chi}^{Y, X}(\tau)$ and we simply write it by $\Theta_{\psi, \chi}(\tau)$.
%
%
%
%
%
%
%
%
%
%
%
%
\section{Proof of the Gross-Prasad conjecture for $\left(\mathrm{SO}\left(5\right),\mathrm{SO}\left(2\right)\right)$}
In this section we prove  Theorem~\ref{ggp SO},
i.e. the Gross-Prasad conjecture for $\left(\mathrm{SO}\left(5\right),\mathrm{SO}\left(2\right)\right)$,
based on the pull-back formulas obtained
in the previous section.
\subsection{Proof of the statement (1)  in Theorem~\ref{ggp SO}}
Let $(\pi, V_\pi)$ be as in Theorem~\ref{ggp SO} (1).
By the uniqueness of  the Bessel model due to 
Gan, Gross and Prasad~\cite[Corollary~15.3]{GGP} at  
finite places
and to Jiang, Sun and Zhu~\cite[Theorem~A]{JSZ} at archimedean places, 
there exists uniquely an irreducible constituent 
$\pi_+^B$ of $\pi\mid_{G_D(\mA)^+}$
that has the $(\xi, \Lambda, \psi)$-Bessel period.

When $D$ is split and $\pi_+^B$ is a theta lift from an irreducible cuspidal automorphic representation of $\mathrm{GSO}_{3,1}\left(\mA\right)$, 
our assertion has been  proved by Corbett~\cite{Co}. 
Hence in the remainder of this subsection, we assume
 that:
 \begin{multline}\label{assumption not type I-A}
\text{\emph{ 
when $D$ is split, 
$\pi$ is not a theta lift from $\mathrm{GSO}_{3,1}$}}
\\
\text{\emph{of an irreducible cuspidal automorphic representation}}.
\end{multline}

Let us proceed under the assumption~\eqref{assumption not type I-A}.
By Proposition~\ref{pullback Bessel gsp} and \ref{pullback Bessel gsp innerD},
the theta lift $\theta_\psi(\pi_+^B)$ of $\pi_+^B$ to $\mathrm{GSU}_{3,D}\left(\mA\right)$ has the $(X_\xi, \Lambda^{-1}, \psi)$-Bessel period
and, in particular, $\theta_\psi(\pi_+^B) \ne 0$  where we take $X_\xi \in D^-(F)$ so that $\xi_{X_\xi} = \xi$. 
For example, when we take $\xi =\eta$, we may take $X_\xi =1$. 
%
\begin{lemma}
\label{nonzero lemma (1)}
$\theta_\psi(\pi_+^B)$ is an irreducible cuspidal automorphic representation of $\mathrm{GSU}_{3,D}(\mA)$.
\end{lemma}
%
%
%
\begin{proof}
First we note that the irreducibility follows from 
the cuspidality by Remark~\ref{theta irr rem} and \ref{theta irr rem D}.

Let us show the cuspidality.
Suppose on the contrary that $\theta_\psi(\pi_+^B)$ is not cuspidal. 

When $D$ is not split, the Rallis tower property implies that
the the theta lift $\theta_{D, \psi}(\pi_+^B)$ of $\pi_+^B$ to $\mathrm{GSU}_{1,D}( \mA)$ 
is non-zero and cuspidal.
Let $w$ be a finite place of $F$ such that  $D(F_w)$ is split and
 $\pi_{+,w}^B$ is a generic representation of $G(F_w)^+$. 
 Since $\pi_{+,w}^B$ is generic, the theta lift of $\pi_{+, w}^B$  to $\mathrm{GSO}_2(F_w)$ vanishes by the same argument
 as  the one for \cite[Proposition~2.4]{GRS97}.
 We note that $\mathrm{GSU}_{1,D}( F_w) \simeq \mathrm{GSO}_2(F_w)$
 and hence the theta lift of $\pi_+^B$ to $\mathrm{GSU}_{1,D}( \mA)$
 must vanish. This is a contradiction.

Suppose that $D$ is split. Then the theta lift of $\pi_+^B$ to $\mathrm{GSO}_{3,1}$ is non-zero by the Rallis tower property.
Moreover, it is not cuspidal by our assumption on $\pi$.
Thus the theta lift of $\pi_+^B$ to $\mathrm{GSO}_{2,0}$ is non-zero, 
again by the Rallis tower property.
Then we reach a contradiction by 
the same argument as in the non-split case.
\end{proof}
%
%
%
%
We may regard $\theta_\psi(\pi_+^B)$ as an irreducible cuspidal automorphic representation of $\mathrm{PGU}_{2,2}$
or $\mathrm{PGU}_{3,1}$ according to whether $D$ is split or not,
under the isomorphism $\Phi$ in \eqref{acc isom2} or $\Phi_D$
in \eqref{acc isom1}.
Recall  our assumption that $\theta_{\psi_{w}}(\pi_{+,w}^B)$ is generic
at a finite place $w$.
Then the non-vanishing of $(X_\xi, \Lambda^{-1}, \psi)$-Bessel period on $\theta_\psi(\pi_+^B)$ implies
the non-vanishing of the central value of the
standard $L$-function for $\theta_\psi(\pi_+^B)$ of $\mathrm{PGU}_4$ twisted by $\Lambda^{-1}$, namely
\[
L^{S} \left(\frac{1}{2}, \theta_\psi(\pi_+^B) \times \Lambda^{-1} \right) \ne 0
\]
for any finite set $S$ of places of $F$ containing all archimedean places
because of the unitary group case of the 
Gan-Gross-Prasad conjecture for $\theta_\psi(\pi_+^B)$ proved by Proposition~A.2 and Remark~A.1 in \cite{FM3}.
Moreover, from the explicit computation of local theta correspondence in \cite{GT0} and \cite{Mo},
we see that 
\[
L(s, \pi_v \times \mathcal{AI} \left(\Lambda \right)_v) = L \left(s, \theta_\psi(\pi_+^B)_v \times \Lambda_v^{-1} \right)
\]
at a finite place $v$ where all data are unramified.
Thus when we take $S_0$, a finite set of places of $F$ containing
all archimedean places, so that all data are unramified at $v\notin S_0$,
we have
\[
L^{S} \left(\frac{1}{2}, \pi \times \mathcal{AI} \left(\Lambda \right) \right) = L^{S} \left(\frac{1}{2}, \theta_\psi(\pi_+^B) \times \Lambda \right) \ne 0
\]
for any finite set $S$ of places of $F$ with $S\supset S_0$.
%

Let us show an existence of $\pi^\circ$.
We denote $\theta_\psi(\pi_+^B)$ by $\sigma$.
Then the theta lift $\Sigma : = \Theta_{\psi, (\Lambda^{-1}, \Lambda^{-1})}(\sigma)$ of $\sigma$ to $\mathrm{GU}_{2,2}$ which we may regard as an automorphic representation
of $\mathrm{GSO}_{4,2}$ by the accidental isomorphism~\eqref{acc isom2}, is 
an irreducible cuspidal globally generic automorphic representation with trivial central character
by the proof of \cite[Proposition~A.2]{FM3} since $\theta_\psi(\pi_+^B)$ has 
the $(X_\xi, \Lambda^{-1}, \psi)$-Bessel period. 

Here we recall that, by the conservation relation due to Sun and Zhu~\cite[Theorem~1.10, Theorem~7.6]{SZ},
for any irreducible admissible representations 
$\tau$ of $\mathrm{GO}_{4,2}(k)$ (resp. $\mathrm{GO}_{3,3}(k)$) over a local field $k$ of characteristic zero, theta lifts of either $\tau$ or $\tau \otimes \det$ to 
$\mathrm{GSp}_3(k)^+$ (resp. $\mathrm{GSp}_3(k)$) is non-zero. 
Thus   we may extend $\Sigma$ to an automorphic representation
of $\mathrm{GO}_{4,2}(\mA)$ as in 
Harris--Soudry--Taylor~\cite[Proposition~2]{HST} so that 
 its local theta lift to $\mathrm{GSp}_3(F_v)^+$ is non-zero 
 at every place $v$.
 
 On the other hand, since $\Sigma$ is nearly equivalent to $\sigma$,  we have 
\begin{equation}\label{e: pole at 1}
L^S(s, \Sigma, \mathrm{std}) = L^S(s, \pi, \mathrm{std} \otimes \chi_E) \zeta_F^S(s)
\end{equation}
for a sufficiently large finite set $S$ of places of $F$
containing all archimedean places by the explicit computation of local theta correspondences in  \cite{GT0} and \cite{Mo}.
Here
\[
\label{s=1 non-zero}
L^S(1, \pi, \mathrm{std} \otimes \chi_E) \ne 0
\]
by Yamana~\cite[proof of Theorem~10.2, Theorem~10.3]{Yam},
since the theta lift $\theta_\psi(\pi_+^B)$ of $\pi_+^B$ to $\mathrm{GSU}_{3,D}( \mA)$ is non-zero and  cuspidal.
Hence the left hand side of \eqref{e: pole at 1} has a pole
at $s=1$. In particular, it is non-zero and the theta lift of $\Sigma$ to $\mathrm{GSp}_3(\mA)^+$ is non-zero by 
Takeda~\cite[Theorem~1.1 (1)]{Tak}. Further, again by Takeda~\cite[Theorem~1.1 (1)]{Tak}, this theta lift actually descends to 
$\mathrm{GSp}_2(\mA)^+ =G(\mA)^+$. Namely, the theta lift $\pi^\prime_+ := \theta_{\psi^{-1}}(\Sigma)$ of 
$\Sigma$ to $G(\mA)^+$ is non-zero since $L^S(s, \Sigma, \mathrm{std})$ 
actually has  a pole at $s=1$.

Suppose that $\pi^\prime_+$ is not cuspidal.
Then by the Rallis tower property, the theta lift of $\Sigma$ to
$\mathrm{GL}_2\left(\mA\right)^+$ is non-zero and
cuspidal.
Meanwhile the local theta lift of $\Sigma_v$ to $\mathrm{GL}_2\left(
F_v\right)^+$ vanishes by a computation similar to the one
for \cite[Proposition~3.3]{GRS97}
since $\Sigma_v$ is generic.
This is a contradiction
and hence $\pi^\prime_+$ is cuspidal.

Since $\Sigma$ is generic, so is $\pi^\prime_+$
by \cite[Proposition~3.3]{Mo}.
Let us take an extension $\pi^\circ$ of $\pi_+^\prime$ to
$G\left(\mathbb A\right)$.
Since $\left| G(F_v) \slash G(F_v)^+\right|=2$, 
we have $\pi^\prime_{v} \simeq \pi_v$ or $\pi^\prime_{v} \simeq \pi_v \otimes \chi_{E_v}$
at almost all places $v$ such that $\pi_{+,v}^\prime \simeq \pi_{+, v}^B$.
Hence
$\pi$ is locally $G^+$-nearly equivalent to $\pi^\circ$.
\qed
\\
%
%
%
%
%
%
\subsection{Some consequences of the proof of Theorem~\ref{ggp SO} (1)}
As preliminaries for our further considerations, we would like to discuss some 
consequence of the proof of Theorem~\ref{ggp SO} (1)
and related results.

First we note the following result concerning the functorial transfer.
%
\begin{proposition}
\label{exist gen prp}
Let $(\pi, V_\pi)$ be an irreducible cuspidal automorphic representation of $G_D(\mA)$ with a trivial central character.
Assume that there exists a finite place $w$
at which $\pi_w$ is generic and tempered.

Then there exists a globally generic irreducible cuspidal automorphic 
representation $\pi^\circ$ of $G\left(\mA\right)$
and an \'{e}tale quadratic extension $E^\circ$ of $F$
such that $\pi^\circ$ is $G^{+, E^\circ}$-nearly equivalent
to $\pi$.
In particular we have a weak functorial lift of $\pi$ to $\mathrm{GL}_4(\mA_{E^\circ})$ with respect to $\mathrm{BC} \circ \mathrm{spin}$.

Moreover, $\pi$ is tempered if and only if $\pi^\circ$ is  tempered.
\end{proposition}
%
%
%
\begin{Remark}
\label{exist gen rem}
When $D$ is split, our assumption implies that $\pi$ has a generic Arthur parameter.
Though our assertion thus follows from the global descent method
by Ginzburg, Rallis and Soudry~\cite{GRS} and
Arthur~\cite{Ar},
we shall present another  proof which does not refer
to these papers.
\end{Remark}
%
%
%
\begin{proof}
Suppose that $D$ is split.
When $\pi$ participates in the theta correspondence with $\mathrm{GSO}_{3,1}$,
our assertion follows from \cite{Ro}. 
Thus we now assume that the theta lift of $\pi$ to  $\mathrm{GSO}_{3,1}$ is zero.
By \cite{JSLi}, $\pi$ has $(S_\circ, \Lambda_\circ, \psi)$-Bessel period for some $S_\circ$ and $\Lambda_\circ$.
When $\mathrm{GSO}(S_\circ)$ is not split,
the existence of a globally generic irreducible cuspidal automorphic representation follows from Theorem~\ref{ggp SO} (1).
Suppose that $\mathrm{GSO}(S_\circ)$ is split.
Then by Proposition~\ref{pullback Bessel gsp},  the theta lift of $\pi$
 to $\mathrm{GSO}_{3,3}$ is non-zero.
 Since $\pi_w$ is generic, the local theta lift of $\pi_w$ to
 $\mathrm{GSO}_{1,1}$ is zero as in the proof of Theorem~\ref{ggp SO} (1)
 and hence
 the theta lift of $\pi$ to $\mathrm{GSO}_{1,1}$ is zero.
 Hence by the Rallis tower property, 
 either the theta lift of $\pi$  to $\mathrm{GSO}_{2,2}$ or the one
 to $\mathrm{GSO}_{3,3}$  is 
 non-zero and cuspidal.
 Then $\pi$ itself is globally generic by Proposition~\ref{GRS identity}
 in the former case. In the latter case, 
 the global genericity of 
 $\pi$ readily follows from the proof of Soudry~\cite[Proposition~1.1]{So} (see also Theorem in p.264 of \cite{So}).

  In any case when $D$ is split, we have a globally generic
   irreducible cuspidal automorphic representation $\pi^\circ$
   of $G\left(\mA\right)$
   which is nearly equivalent to $\pi$.
 Thus when we take the strong lift of $\pi^\circ$ to $\mathrm{GL}_4\left(\mA\right)$ by \cite{CKPSS}, it is a weak lift of $\pi$ to 
 $\mathrm{GL}_4\left(\mA\right)$.
 %
 %
 %
 %
 
 Suppose that $D$ is not split.
 Then by Li~\cite{JSLi},
 there exist an $\eta_\circ\in D^-(F)$ where $E_\circ:=F\left(\eta_\circ\right)$
 is a quadratic extension of $F$,  and a character
 $\Lambda_\circ$ of $\mathbb A_{E_\circ}^\times \slash E_\circ^\times \mA^\times$
 such that $\pi$ has the $\left(\eta_\circ, \Lambda_\circ\right)$-Bessel
 period.
 Then there exists a desired automorphic representation
 $\pi^\circ$ of $G\left(\mA\right)$ by Theorem~\ref{ggp SO} (1).
 
 Let us discuss the temperedness.
 Let $\sigma, \Sigma$ and $\pi_+^\prime$ denote the same 
 as in the proof of Theorem~\ref{ggp SO} (1).
 Suppose that $\pi$ is tempered.
 Then the temperedness of $\sigma$ 
 follows 
 from a similar argument as in Atobe-Gan~\cite[Proposition~5.5]{AG} 
 (see also  \cite[Proposition~C.1]{GI1}) at finite places,
 from Paul~\cite[Theorem~15, Theorem~30]{Pau}, \cite[Theorem~15, Theorem~18, Corollary~24]{Pau3} 
 and Li-Paul-Tan-Zhu~\cite[Theorem~4.20, Theorem~5.1]{LPTZ} at real places and from
 Adams-Barbasch~\cite[Theorem~2.7]{AB} at complex places.
Then the temperedness of $\sigma$ implies that
of $\Sigma$ by Atobe-Gan~\cite[Proposition~5.5]{AG} at finite places, 
by Paul~\cite[Theorem~3.4]{Pau2} at non-split real places, 
by M\oe glin~\cite[Proposition~III.9]{Moe} at split real places
and  by Adams-Barbasch~\cite[Theorem~2.6]{AB} at complex places.
As we obtained the temperedness of $\sigma$ from that of $\pi$,
the temperedness of $\Sigma$ implies that of $\pi^\prime_+$
and hence $\pi^\circ$ is tempered.
 The opposite direction, i.e., the temperedness of $\pi^\circ$
 implies that of $\pi$, follows by the same argument.
 \end{proof}
\begin{lemma}
\label{compo number}
Let $\pi$ be as in Theorem~\ref{ggp SO} (1). 
Suppose that $\sigma=\theta_\psi\left(\pi_+^B\right)$ is an
irreducible cuspidal autormophic representation of $\mathrm{GSU}_{3,D}
\left(\mA\right)$.
Here $\pi_+^B$ denotes the unique irreducible constituent of $\pi|_{G_D(\mA)^+}$ such that $\pi_+^B$
has the $(E, \Lambda)$-Bessel period.
We regard $\sigma$ as an automorphic representation of 
$\mathrm{GU}_{4,\varepsilon}\left(\mA\right)$ via 
\eqref{acc isom1} or \eqref{acc isom2} and 
let $\Pi_\sigma$ denote the base change lift of $\sigma |_{\mathrm{U}_{4,\varepsilon}\left(\mA\right)}$ to 
$\mathrm{GL}_4\left(\mA_E\right)$.
Let $\pi^\circ$ be a globally generic irreducible cuspidal automorphic representation of
$G\left(\mA\right)$ whose existence is proved in
 Theorem~\ref{ggp SO} (1).
 We denote the functorial lift of $\pi^\circ$ to 
 $\mathrm{GL}_4\left(\mA\right)$ by $\Pi_{\pi^\circ}$.
 
 Suppose that
 \begin{equation}\label{e: decomposition 1}
 \Pi_{\pi^\circ} = \Pi_{1} \boxplus \cdots \boxplus \Pi_\ell
 \end{equation}
where $\Pi_i$ are irreducible cuspidal automorphic representations of $\mathrm{GL}_{n_i}(\mA)$ and 
\begin{equation}\label{e: decomposition 2}
\Pi_{\sigma} =\Pi_1^\prime  \boxplus \cdots \boxplus \Pi_k^\prime
\end{equation}
where $\Pi_j^\prime$ are irreducible cuspidal automorphic representations of  $\mathrm{GL}_{m_j}(\mA_E)$. 

Then we have $\Pi_\sigma=\mathrm{BC}\left(\Pi_{\pi^\circ}\right)$,
$\Pi_{\pi^\circ}\not\simeq\Pi_{\pi^\circ}\otimes\chi_E$
and $\mathrm{BC}\left(\Pi_i\right)$ is cuspidal for each $i$.
In particular, we have $\ell=k$.
Here $\mathrm{BC}$ denotes the base change from $F$ to $E$.
\end{lemma}
%
%
%
%
\begin{proof}
By the explicit computation of local theta correspondences in 
\cite{GT0} and \cite{Mo}, we see that 
$\left(\Pi_{\sigma}\right)_v\simeq \mathrm{BC}(\Pi_{\pi^\circ} )_v$
at almost all finite places $v$ of $E$. Thus, $\Pi_{\sigma}= \mathrm{BC}(\Pi_{\pi^\circ} )$ by the strong multiplicity one theorem.
Also, by \cite{CKPSS}, we know that $\ell = 1$ or $2$.

Suppose that $\ell=1$. 
We note that the cuspidality of $\mathrm{BC}(\Pi_{\pi^\circ})$
is equivalent to $\Pi_{\pi^\circ} \otimes \chi_E \not \simeq \Pi_{\pi^\circ}$.
Suppose otherwise, i.e.
$\Pi_{\pi^\circ} \simeq \Pi_{\pi^\circ} \otimes \chi_E$.
Then $\Pi_{\pi^\circ} = \mathcal{AI}(\tau)$ for some irreducible cuspidal automorphic 
representation $\tau$ of $\mathrm{GL}_2(\mA_E)$. 
Since $\Pi_{\pi^\circ}$ is a lift from $\mathrm{PGSp}_2$,
the central character of $\tau$ needs to be trivial
and hence $\tau\simeq \tau^\vee$.
On the other hand, we have 
\[
\Pi_{\sigma} = \mathrm{BC}\left( \mathcal{AI}(\tau) \right)= \tau \boxplus \tau^\sigma.
\]
Since this is a base change lift of $\sigma |_{\mathrm{U}_{4,\varepsilon}\left(\mA\right)}$, we  have $\tau = \left( \tau^\sigma \right)^\vee$
and $\tau \not \simeq \tau^\sigma$ by \cite{AC} (see also \cite[Proposition~3.1]{PR}).
In particular, $\tau\not\simeq \tau^\vee$ and we have a contradiction.
Thus $\mathrm{BC}\left(\Pi_{\pi^\circ}\right)$ is cuspidal
and $k=1$.

Suppose that $\ell=2$. 
First we show that $\Pi_{\pi^\circ} \not \simeq \Pi_{\pi^\circ}  \otimes \chi_E$.
Suppose otherwise, i.e. $\Pi_{\pi^\circ} \simeq \Pi_{\pi^\circ}  \otimes \chi_E$.
Then either
$\Pi_i \simeq \Pi_i \otimes \chi_E$ for $i=1,2$,
or, $\Pi_2 \simeq \Pi_1 \otimes \chi_E$.
In the former case, we have $\Pi_i=\mathcal{AI}\left(\chi_i\right)$
with a character $\chi_i$ of $\mA_E^\times \slash E^\times$
for $i=1,2$.
Then we have $\Pi_{\pi^\circ}=\mathcal{AI}\left(\chi_1\right)\boxplus
\mathcal{AI}\left(\chi_2\right)$ and
$\Pi_\sigma=\chi_1 \boxplus \chi_1^\sigma  \boxplus \chi_2  \boxplus \chi_2^\sigma$.
Since $\Pi_{\pi^\circ}$ is a lift from $\mathrm{PGSp}_2$,
the central character of $\mathcal{AI}\left(\chi_i\right)$
is trivial and hence $\chi_i\mid_{\mA^\times}=\chi_E$.
On the other hand, since $\Pi_\sigma$ is a base change lift
of $\sigma |_{\mathrm{U}_{4,\varepsilon}\left(\mA\right)}$, we see that $\chi_i\mid_{\mA^\times}$ is trivial.
This is a contradiction.
In the latter case, we have
$\mathrm{BC}\left(\Pi_2\right)=\mathrm{BC}\left(\Pi_1 \otimes \chi_E\right)
=\mathrm{BC}\left(\Pi_1\right)$
and hence $\Pi_\sigma=\mathrm{BC}\left(\Pi_1\right)\boxplus
\mathrm{BC}\left(\Pi_1\right)$.
This implies that $\Pi_\sigma$ is not in the image of the base change lift
from the unitary group
and again we have a contradiction.
Thus we have $\Pi_{\pi^\circ} \not \simeq \Pi_{\pi^\circ}  \otimes \chi_E$.
Then $\Pi_i \not\simeq \Pi_i \otimes \chi_E$ at least one of $i=1,2$.
Suppose that this is so only for one of the two, say $i=2$.
Then $\Pi_1=\mathcal{AI}\left(\chi\right)$ for some character
$\chi$ of $\mA_E^\times\slash E^\times$ and
$\mathrm{BC}\left(\Pi_2\right)$ is cuspidal.
We have
$\Pi_{\pi^\circ}=\mathcal{AI}\left(\chi\right)\boxplus \Pi_2$ and
$\Pi_\sigma=\chi\boxplus\chi^\sigma\boxplus\mathrm{BC}\left(\Pi_2\right)$.
Then $\chi\mid_{\mA^\times}$ is trivial from the former equality
and $\chi\mid_{\mA^\times}=\chi_E$ from the latter equality
as above. Hence we have a contradiction.
Thus $\mathrm{BC}\left(\Pi_i\right)$ for $i=1,2$ are both cuspidal,
$\Pi_\sigma=\mathrm{BC}\left(\Pi_1\right)\boxplus\mathrm{BC}\left(\Pi_2\right)$
and $k=2$.
\end{proof}
The following lemma gives the uniqueness of the constant $\ell(\pi)$ defined before Theorem~\ref{ref ggp}.
%
%
%
%
%
%
%
\begin{lemma}
\label{comp number not dep on S}
Let $\pi$ be as in Theorem~\ref{ggp SO} (1). 
For $i=1,2$, let $E_i$ be a quadratic extension of $F$
and $\pi^\circ_i$ an irreducible cuspidal automorphic representation
of $G\left(\mA\right)$ which is $G^{+,E_i}$-locally near equivalent
to $\pi$.
Let $\Pi_{\pi^\circ_i}$ be the functorial lift of $\pi_i^\circ$ to 
$\mathrm{GL}_4\left(\mA\right)$ and consider the decomposition
\[
\Pi_{\pi^\circ_i}=\Pi_{i,1}\boxplus\cdots\boxplus\Pi_{i,\ell_i}
\quad\text{for $i=1,2$}
\]
as \eqref{e: decomposition 1}.
Then we have $\ell_1=\ell_2$.
\end{lemma}
%
%
%
%
\begin{proof}
Since the case when $E_1=E_2$ is trivial, suppose that $E_1\ne E_2$.
Let $K=E_1E_2$.
From the definition of the base change, we have 
\[
\mathrm{BC}_{K\slash E_1} \left(  \mathrm{BC}_{E_1
 \slash F} (\Pi_{\pi^\circ_1}) \right) =
\mathrm{BC}_{K \slash E_1} \left(  \mathrm{BC}_{E_1 \slash F}
 (\Pi_{\pi^{\circ}_2}) \right).
\]
Hence
\[
 \mathrm{BC}_{E_1 \slash F} (\Pi_{\pi^\circ_1}) 
 = \mathrm{BC}_{E_1 \slash F} (\Pi_{\pi^{\circ}_2}) 
 \quad
 \text{or}
 \quad
  \mathrm{BC}_{E_1 \slash F} (\Pi_{\pi^\circ_1}) 
 = \mathrm{BC}_{E_1 \slash F} (\Pi_{\pi^{\circ}_2})\otimes
 \chi_{K\slash E_1}
\]
where $\chi_{K\slash E_1}$ denotes
 the character  of $\mA_E^\times$ corresponding to $K \slash E_1$.
In the former case, we have 
\[
\Pi_{\pi^\circ_1} = \Pi_{\pi^{\circ}_2}  \quad
 \text{or}
 \quad \Pi_{\pi^\circ_1} = \Pi_{\pi^{\circ}_2} \otimes \chi_{E_1}
\]
and our claim follows. In the latter case, since $\chi_{K\slash E_1} = \chi_{E_2} \circ N_{E_1 \slash F}$,
we have
\[
\Pi_{\pi^\circ_1} = \Pi_{\pi^{\circ}_2} \otimes \chi_{E_2}
 \quad
 \text{or}
 \quad
 \Pi_{\pi^\circ_1} = \Pi_{\pi^{\circ}_2} \otimes \chi_{E_2} \chi_{E_1}
\]
and our claim follows.
\end{proof}
%
%
%
%
%
%
%
\begin{Definition}
\label{type}
Let $\pi$ be as in Theorem~\ref{ggp SO} (1).
Then we say that $\pi$ is of Type I  if $\pi$ and $\pi \otimes \chi_E$ are nearly equivalent.
Moreover, we say that $\pi$ is of type I-A  if  $\pi$ participates in the theta correspondence with $\mathrm{GSO}(S_1) =\mathrm{GSO}_{3,1}$
and that $\pi$ is of type I-B  if  $\pi$ participates in the theta correspondence with $\mathrm{GSO}(X_\circ)$
for some four dimensional anisotropic
orthogonal space $X_\circ$ over $F$ with discriminant algebra $E$.
\end{Definition}
\begin{Remark}
From the proof of Theorem~\ref{ggp SO} (1), if $\pi$ is not of type I-A, then the theta lift of $\pi$
to $\mathrm{GSU}_{3,D}$ is cuspidal.
Further, we note that $D$ is necessarily split when $\pi$ is
of type I-A or I-B, by definition.
\end{Remark}

%
%
%
%
%
%
In order to study an explicit formula using theta lifts from $G_D(\mA)$, the following lemma 
will be important later.
%
%
%
%
\begin{lemma}
\label{irr lemma}
Let $\pi$ be as in Theorem~\ref{ggp SO} (1). Then
$\pi$ is either type I-A or I-B if and only if 
$\pi$ is nearly equivalent to $\pi \otimes \chi_E$.
In particular, when $\pi$ is neither of type I-A nor I-B, $\pi|_{\mathcal{G}_D}$ is irreducible where
\begin{equation}\label{d: mathcal G_D}
\mathcal G_D=Z_{G_D}\left(\mA\right)G_D\left(\mA\right)^+ G_D\left(F\right).
\end{equation}
\end{lemma}
%
%
%
\begin{proof}
Suppose that $\pi$ is nearly equivalent to $\pi \otimes \chi_E$.
Then at almost all places $v$ of $F$, $\mathrm{Ind}_{G_D\left(F_v\right)^+}^{G_D\left(F_v\right)}\left(\pi_{+,v}\right)$ is irreducible
where $\pi_{+,v}$ is an irreducible constituent of $\pi_v\mid_{G_D\left(
F_v\right)^+}$.
This implies that $\pi$ and $\pi^\circ$ are nearly equivalent
and hence $\pi^\circ$ is nearly equivalent to $\pi^\circ \otimes \chi_E$.
Thus $\Pi_{\pi^\circ}$ is nearly equivalent to $\Pi_{\pi^\circ} \otimes \chi_E$
and hence
$\Pi_{\pi^\circ}  = \Pi_{\pi^\circ} \otimes \chi_E$
by the strong multiplicity one theorem.
When $\pi$ is neither of type I-A nor I-B, this does not happen
by Lemma~\ref{nonzero lemma (1)} and Lemma~\ref{compo number}.

Suppose that $\pi$ is either of type I-A or I-B.
Then $D$ is split and the functorial lift $\Pi_\pi$
of $\pi$ to $\mathrm{GL}_4\left(\mA\right)$ is of the form
$\mathcal{AI}\left(\tau\right)$ for an irreducible automorphic representation
$\tau$ of $\mathrm{GL}_2\left(\mA_E\right)$ by Roberts~\cite{Ro}.
Then we have $\Pi_\pi=\Pi_\pi\otimes\chi_E$.
Hence $\pi$ is nearly equivalent to $\pi\otimes\chi_E$.

When $\pi$ is not nearly equivalent to $\pi\otimes\chi_E$,
$\pi\mid_{\mathcal G_D}$ is irreducible
since $\mathcal G_D$ is of index $2$ in $G_D\left(\mA\right)$.
\end{proof}
\begin{Remark}
This lemma give a classification of $\pi$ such that the twist $\pi \otimes \chi_E$ of $\pi$ by $\chi_E$
has the same Arthur parameter as $\pi$. 
A classification of
$\pi$ such that $\pi$ and $\pi \otimes \chi_E$ are isomorphic when $G_D \simeq G$
is given in Chan~\cite{PSC}.
\end{Remark}
%
%
%
%
%
%
%
%
\subsection{Proof of the statement (2) in Theorem~\ref{ggp SO}}
\label{s:Proof of the statement (2)}
%
Suppose that $\pi$ has a generic Arthur parameter.

When there exists a pair $\left(D^\prime,\pi^\prime\right)$
as described in Theorem~\ref{ggp SO} (2),
$\pi$ and $\pi^\prime$ share the same generic Arthur parameter
since they are nearly equivalent to each other.
Hence by Theorem~\ref{ggp SO} (1),
we have 
\[
L^S \left(\frac{1}{2}, \pi \times \mathcal{AI} \left(\Lambda \right) \right) = L^S \left(\frac{1}{2}, \pi^\prime \times \mathcal{AI} \left(\Lambda \right) \right) \ne 0
\]
when $S$ is a sufficiently large finite set of places of $F$.
Then by Remark~\ref{L-fct def rem},  we have
\[
L\left(\frac{1}{2},\pi\times\mathcal{AI}\left(\Lambda \right)\right)\ne 0,
\]
i.e. \eqref{e:L non-zero} holds.

%

Conversely suppose that $L \left(\frac{1}{2}, \pi \times \mathcal{AI} \left(\Lambda \right) \right) \ne 0$. 
There exists an irreducible cuspidal globally generic automorphic representation $\pi^\circ$ of $G(\mA)$ which is nearly equivalent to $\pi$
since $\pi$ has a generic Arthur parameter.
Let $U$ be a maximal unipotent subgroup of $\mathrm{GSO}_{4,2}$ and $\psi_U$ be a non-degenerate character of 
$U(\mA)$ defined below by \eqref{def:U} and \eqref{def:psi_U}, which are
the  same as \cite[(2.4)]{Mo} and \cite[(3.1)]{Mo}, respectively.
Let $U_G$ be the  maximal unipotent subgroup of $\mathrm{GSp}_{2}$ 
defined by \eqref{d: U_G} and $\psi_{U_G}$  the non-degenerate character of 
$U_G(\mA)$ defined by \eqref{def:psi_UG} in \ref{s: pf wh gso}.
Note that in \cite{Mo}, $U_G$ is denoted by $N$ and $\psi_{U_G}$ is denoted by $\psi_{N}$ in \cite[p.34]{Mo} and \cite[(3.2)]{Mo}, respectively.
Then we note that the restriction of $\pi^\circ$ to $G(\mA)^+$ contains 
a unique $\psi_{U_G}$-generic irreducible constituent and we denote it by $\pi_+^\circ$.
Let us consider the theta lift $\Sigma:=\theta_\psi(\pi_+^\circ)$ of $\pi_+^\circ$ to $\mathrm{GSO}_{4,2}(\mA)$.
Then by \cite[Proposition~3.3]{Mo}, we know that 
$\Sigma$ is $\psi_U$-globally generic and hence non-zero.
We divide into two cases according to the cuspidality of $\Sigma$.
%

Suppose that $\Sigma$ is not cuspidal.
Then by Rallis tower property, $\pi^\circ_+$ participates in the theta correspondence with $\mathrm{GSO}_{3,1}$.
As in the proof of Lemma~\ref{nonzero lemma (1)},
the theta lift of $\pi^\circ_+$ to $\mathrm{GSO}_2$ is zero
since $\pi^\circ_+$ is generic.
Hence the theta lift $\tau:=\theta_\psi^{X, S_1}(\pi^\circ_+)$ of $\pi^\circ_+$ to $\mathrm{GSO}_{3,1}$ is cuspidal
and non-zero.
By Remark~\ref{theta irr rem}, $\tau$ is also irreducible.

Recall that 
\[
\mathrm{GSO}_{3,1}(F) \simeq \mathrm{GL}_2(E) \times F^\times \slash \{(z, \mathrm{N}_{E \slash F}(z)) : z \in E^\times \},
\,\,
\mathrm{PGSO}_{3,1}(F) \simeq \mathrm{PGL}_2(E).
\]
Then we may regard $\tau$ as an irreducible cuspidal automorphic 
representation of $\mathrm{GL}_2(\mA_E)$ with a
trivial central character
since the central character of $\pi^\circ_+$ is trivial.
%

Let $\Pi$ denote the strong functorial lift of $\pi^\circ$ to $\mathrm{GL}_4(\mA)$ by \cite{CKPSS}. 
Then at almost all finite places $v$ of $F$, 
we have $\Pi_v \simeq \mathcal{AI}(\tau)_v$, and thus by the strong 
multiplicity one theorem,
$\Pi = \mathcal{AI}(\tau)$ holds. 
Since $\pi$ is nearly equivalent to $\pi^\circ$, Remark~\ref{L-fct def rem} and our assumption imply that for a sufficiently large finite set $S$ of places of $F$, we have
\begin{align*}
&L^S \left(\frac{1}{2}, \tau \times \Lambda \right) L^S \left(\frac{1}{2}, \tau \times \Lambda^{-1} \right) 
= L^S \left(\frac{1}{2}, \pi^\circ \times \mathcal{AI} \left(\Lambda \right) \right)
\\
 = &L^S \left(\frac{1}{2}, \pi \times \mathcal{AI} \left(\Lambda \right) \right) \ne 0.
\end{align*}
Then by Waldspurger~\cite{Wal}, $\tau$ has the
split torus model with respect to the character $(\Lambda, \Lambda^{-1})$.
Hence, the equation in Corbett~\cite[p.78]{Co} implies
 that $\pi^\circ$ has the $(E, \Lambda)$-Bessel period.
Hence we may take $D^\prime = \mathrm{Mat}_{2 \times 2}$ and $\pi^\prime = \pi^\circ$.
Thus the case when $\Sigma$ is not cuspidal is settled.

Suppose that  $\Sigma$ is cuspidal.
We may regard $\Sigma$ as an irreducible cuspidal globally generic automorphic representation of $\mathrm{GU}(2,2)$ with trivial central character because of the accidental isomorphism \eqref{acc isom2}.
As in the proof of Theorem~\ref{ggp SO} (1), our assumption implies that
$L \left( \frac{1}{2}, \Sigma \times \Lambda \right) \ne 0$.
Then by \cite[Proposition~A.2]{FM3}, there exists an irreducible cuspidal automorphic representation $\Sigma^\prime$ of $\mathrm{GU}(V)$
such that $\Sigma^\prime$ is locally $\mathrm{U}(V)$-nearly equivalent to $\Sigma$ and $\Sigma^\prime$ has the 
$(e, \Lambda, \psi)$-Bessel period where $V$
is a $4$-dimensional hermitian space over $E$ whose Witt index is at least $1$.
Then we note that 
$\mathrm{PGU}(V) \simeq \mathrm{PGSO}_{4,2}$
or $\mathrm{P}\mathrm{GU}_{3,D^\prime}$
for some quaternion division algebra $D^\prime$ over $F$.

In the first case, we consider the theta lift $\pi^\prime_+:=\theta_{\psi^{-1}}(\Sigma^\prime)$ of $\Sigma^\prime$ to 
$G(\mA)^+$.
Then by the same argument as the one in the proof of Theorem~\ref{ggp SO} (1), we see that $\pi_+^\prime \ne 0$ by Takeda~\cite[Theorem~1.1 (1)]{Tak}
and that it is an irreducible cuspidal automorphic representation
of $G\left(\mA\right)^+$.
Since $\Sigma^\prime$ has the 
$(e, \Lambda, \psi)$-Bessel period, $\pi^\prime_+$ has the $(E, \Lambda)$-Bessel period by Proposition~\ref{pullback Bessel gsp}.
From the definition, $\pi_+^\prime$ is nearly equivalent to $\pi_+^\circ$.
Let us take an irreducible cuspidal automorphic representation $(\pi^\prime, V_{\pi^\prime})$ of $G(\mA)$
such that $\pi^\prime\mid_{G\left(\mA\right)^+} \supset \pi_+^\prime$.
Then $\pi^\prime$ is locally $G^+$-nearly equivalent, and thus 
either $\pi^\prime$ or $\pi^\prime \otimes \chi_E$ is nearly equivalent to $\pi$
by Remark~\ref{rem loc ne}. Since both $\pi^\prime$ and $\pi^\prime \otimes \chi_E$ have the $(E, \Lambda)$-Bessel period,
our claim follows.

In the second case,  we consider the theta lift of $\Sigma^\prime$ to $G_{D^\prime}(\mA)$.
Then by an argument similar to the one in the  first case, we may
 show that the theta lift of $\Sigma^\prime$ to $G_{D^\prime}(\mA)$ contains an irreducible constituent which
 is cuspidal, locally $G^+$-nearly equivalent to $\pi$
and has the $(E, \Lambda)$-Bessel period.  Here we use \cite[Lemma~10.2]{Yam} and its proof in the case of ($\rm{I}_1$)
with $n=3, m=2$, noting Remark~\ref{yam typo}.
This completes our proof of the 
existence of a pair $\left(D^\prime,\pi^\prime\right)$.
%
%
%

Let us show the uniqueness of a pair $\left(D^\prime,\pi^\prime\right)$ 
under the assumption that $\pi$ is tempered.
Suppose that for $i=1,2$ there exists a pair$\left(D_i,\pi_i\right)$ 
where $D_i$ is a quaternion algebra over $F$
and $\pi_i$ is an irreducible cuspidal automorphic representation
of $G_{D_i}\left(\mA\right)$ which is nearly equivalent to $\pi$
such that $\pi_i$ has the $(E, \Lambda)$-Bessel period.
%

Suppose that $\pi_i$ is nearly equivalent to $\pi_i \otimes \chi_E$
for $i=1,2$.
Then by Proposition~\ref{irr lemma}, $\pi_1$, $\pi_2$ are of type I-A or I-B
and in particular $D_1 \simeq D_2 \simeq \mathrm{Mat}_{2 \times 2}$.
Hence for $i=1,2$, there exist a four dimensional orthogonal space $X_i$ over $F$ with discriminant algebra $E$
and an irreducible cuspidal automorphic representation $\sigma_i$ of $\mathrm{GSO}(X_i, \mA)$
such that $\pi_i= \theta_{\psi}(\sigma_i)$.
Since $\mathrm{PGSO}(X_i, F) \simeq (D_i^\prime)^\times(E)\slash E^\times$
for some quaternion algebra $D_i^\prime$ over $F$, 
we may regard $\sigma_i$
as an automorphic representation of $(D_i^\prime)^\times(\mA_E)$
with the trivial central character.
Since $\pi_i$ has the $(E, \Lambda)$-Bessel period,
$\sigma_i$ has the split torus period with respect to a character
$\left(\Lambda,\Lambda^{-1}\right)$  by \cite[p.78]{Co}.
Hence $D_i^\prime(E) \simeq \mathrm{Mat}_{2 \times 2}(E)$
by \cite{Wal}.
Since $\sigma_1$ is nearly equivalent to $\sigma_2$,
we have $\sigma_1=\sigma_2$ by the strong multiplicity one.
Thus $\pi_1\simeq \pi_2$.
%
%
%

Suppose that $\pi_i$ is neither type I-A nor I-B for $i=1,2$.
For each $i$, 
let us take a unique irreducible constituent $\pi_{i, +}^B$ of $\pi_{i}|_{G_{D_i}(\mA)^+}$
that has the $(\xi_i, \Lambda, \psi)$-Bessel period. 
Note that $\pi_{1, +}^B$ and $\pi_{2, +}^B$ are nearly equivalent
to each other.

Now
let $\sigma_i$ denote the theta lift $\theta_{\psi}(\pi_{i, +}^B)$ 
of $\pi_{i, +}^B$ to $\mathrm{GSU}_{3,D_i}$.
Then we regard $\sigma_i$ as an automorphic representation
of $\mathrm{GU}_{4,\varepsilon}$ via \eqref{acc isom1},
\eqref{acc isom2}
and let $\Sigma_i := \Theta_{\psi, (\Lambda^{-1}, \Lambda^{-1})}$ denote the theta lift of $\sigma_i$ to $\mathrm{GU}_{2,2}$.
In turn, we regard $\Sigma_i$ as an automorphic representation
of $\mathrm{GSO}_{4,2}$ via \eqref{acc isom2} and we denote by 
$\pi_{i,+}^\prime$ its theta lift to $G\left(\mA\right)^+$.
Then from the proof of Theorem~\ref{ggp SO} (1),
$\sigma_i$, $\Sigma_i$ and $\pi_{i,+}^\prime$
are irreducible and cuspidal.
Moreover $\pi_{1,+}^\prime$ and $\pi_{2,+}^\prime$
are both globally generic and nearly equivalent to each other.
Furthermore, since $\pi_i$ is tempered, 
$\sigma_i=\theta_{\psi}(\pi_{i, +}^B)$ is tempered
at finite places by an argument similar to the one
in Atobe-Gan~\cite[Proposition~5.5]{AG} (see also  \cite[Proposition~C.1]{GI1})
and similarly at real and complex places by
Paul~\cite[Theorem~15, Theorem~30]{Pau} and 
Li-Paul-Tan-Zhu~\cite[Theorem~4.20, Theorem~5.1]{LPTZ},
and, by Adams-Barbasch~\cite[Theorem~2.7]{AB}, respectively.
Similarly $\Sigma_i$ and $\pi_{i,+}^\prime$ are
also tempered.

By Proposition~\ref{pullback Bessel gsp} and 
Proposition~\ref{pullback Bessel gsp innerD}, we know that 
$\sigma_i$   has the $(X_{\xi_i}, \Lambda, \psi)$-Bessel period.
Let $\mathrm{GU}_i$ denote the similitude unitary group which 
modulo center is isomorphic to $\mathrm{PGSU}_{3, D_i}$
by \eqref{acc isom1}.
Then $\sigma_i\mid_{\mathrm{U}_i}$ has a unique irreducible constituent $\nu_i$
which has  the $(X_{\xi_i}, \Lambda, \psi)$-Bessel period.
Then by Beuzart-Plessis~\cite{BP1,BP2}  (also by Xue~\cite{Xuea} at the real place),
we see that $\mathrm{U}_1 \simeq \mathrm{U}_2$
since $\nu_1$ and $\nu_2$ are equivalent to each other.
This implies  that $D_1 \simeq D_2$
and hence $G_{D_1} \simeq G_{D_2}$.
Let us denote $D^\prime\simeq D_i$ for $i=1,2$.

We take an
irreducible cuspidal automorphic representation $\pi_{i}^\prime$ of $G(\mA)$
such that $\pi_i^\prime|_{G(\mA)^+}$ contains  $\pi_{i,+}^\prime$.
Then by Remark~\ref{rem loc ne}, we may suppose that $\pi_1^\prime$ is nearly equivalent to $\pi_2^\prime$
or $\pi_2^\prime \otimes \chi_E$.
Thus replacing $\pi_2^\prime$ by $\pi_2^\prime\otimes\chi_E$ if necessary,
we may assume that $\pi_1^\prime$ and $\pi_2^\prime$
are nearly equivalent to each other.
Then since $\pi_1^\prime$ and $\pi_2^\prime$
are generic and they have the same $L$-parameter because of the temperedness of $\pi_i^\prime$,
we have $\pi_1^\prime \simeq \pi_2^\prime$ by the uniqueness of the
generic member in the $L$-packet by Atobe~\cite{At} or Varma~\cite{Var} at finite places and by Vogan~\cite{Vo} at archimedean places.
Hence in particular, $\pi_{1,+}^\prime\simeq\pi_{2,+}^\prime$.

From the definition of $\pi_{+,i}^\prime$, we get $\pi_{1,+}^B \simeq \pi_{2,+}^B$.
Then, we see that $\pi_1 \simeq \pi_2 \otimes \omega$ for
 some character $\omega$ of $G_{D^\prime}(\mA)$
such that $\omega_v$ is trivial or $\chi_{E,v}$ at each place $v$ of $F$.
Since $\pi_1$ and $\pi_2$ have the same $L$-parameter, $\pi_{1,v}$ and $\pi_{1,v}\otimes \omega_v$ are in the same  $L$-packet 
for every place $v$ of $F$.

Let us take a place $v$ of $F$, and write the $L$-parameter of $\pi_{1, v}$ as $\phi_v : WD_{F_v} \rightarrow G^1(\mC)$. 
If $\phi_v$ is an irreducible four dimensional representation, the $L$-packet of $\phi_v$ is singleton,
and thus $\pi_{1, v} \simeq \pi_{2,v}$. 
So let us suppose that $\phi_{v} = \phi_1 \oplus \phi_2$
with two dimensional irreducible representations $\phi_i$.
Further, we may suppose that $\omega_v = \chi_{E,v}$
since there is nothing to prove when $\omega_v$
is trivial. This implies that 
$\phi_v \otimes \chi_{E ,v}\simeq \phi_v$. Then, by \cite[Proposition~3.1]{PR}, we have
$\phi_i=\pi(\chi_i)$ for some character $\chi_i$ of $E_v^\times$
for $i=1,2$.
Moreover, any member of the 
$L$-packet of $\pi_1$ is given by the theta lift 
from an irreducible representation $\mathrm{JL}(\pi(\chi_1)) \boxtimes \pi(\chi_2)$
of $D^\prime(F_v)^\times \times \mathrm{GL}_2(F_v)$
where $\mathrm{JL}$ denotes the
Jacquet-Langlands transfer.
Since the theta lift preserves the character twist,
we see that 
\[
\theta(\mathrm{JL}(\pi(\chi_i)) \boxtimes \pi(\chi_j)) \otimes \chi_{E,v} \simeq 
\theta(\mathrm{JL}(\pi(\chi_i)) \boxtimes \pi(\chi_j)) 
\]
by $\pi(\chi_i) \otimes \chi_{E, v} \simeq \pi(\chi_i)$. This shows that in this case, any element in the $L$-packet
is invariant under the twist by $\chi_{E,v}$. 
Thus $\pi_{1, v} \otimes \chi_{E, v} \simeq \pi_{2,v}$
and hence $\pi_{1,v} \simeq \pi_{2,v}$.
\qed
%
%
%
\begin{Remark}
As we remarked in the end of Section~\ref{s:Gross-Prasad conjecture}, the uniqueness of $(D^\prime, \pi^\prime)$
follows from the
local Gan-Gross-Prasad conjecture for $(\mathrm{SO}(5), \mathrm{SO}(2))$, which is proved by Luo~\cite{Luo} at archimedean places 
and by
Prasad--Takloo-Bighash~\cite{PT} (see also Waldspurger~\cite{Wal12} in general case) at finite places. Our proof gives another proof of the uniqueness.
\end{Remark}
\begin{Remark}
\label{yam typo}
There is a typo in the statement of \cite[Lemma~10.2]{Yam}.
The first condition stated there should be the holomorphy at $s=-s_m+\frac{1}{2}$.
\end{Remark}

%
%
%
%
%
%
%
%
%
%
%
%
%
%
%
%
%
%
%
%
\section{Rallis inner product formula for similitude groups}
In this section, as a preliminary for the 
 proof of Theorem~\ref{ref ggp}, we recall Rallis inner product formulas
for similitude dual pairs.
%
%
%
%
%
%
%
%
%
%
%
%
%
%
%
%
%
%
%
%
\subsection{For the theta lift from $G$ to $\mathrm{GSO}_{4,2}$}
\label{s:RI H gso}
In this section, we shall recall the Rallis inner product formula for the theta lift from $G$ to $\mathrm{GSO}_{4,2}$.
It is derived from the isometry case in a manner similar to
the one in Gan-Ichino~\cite[Section~6]{GI0}, where the case of the 
theta lift from $\mathrm{GL}_2$ to $\mathrm{GSO}_{3,1}$ is treated.

Let $(\pi, V_\pi)$ be an irreducible cuspidal automorphic representation of $G(\mA)$ with a trivial central character.
Let us define a subgroup $\mathcal G$ of $G\left(\mA\right)$ by
\begin{equation}
\label{cal G def}
\mathcal{G} := Z_G(\mA)G(\mA)^+ G(F)
\end{equation}
and in this section we assume that:
\begin{equation}\label{assumption 1}
\text{\emph{the restriction of $\pi$ to $\mathcal{G}$ is irreducible, i.e. $\pi \otimes \chi_E \not \simeq \pi$}}
\end{equation}
for our later use.
%
%
%
%

Let us recall the notation in \ref{sp4 so42}.
Thus $X$ denotes the four dimensional symplectic space
on which $G$ acts on the right and $Y$ denotes the six dimensional
orthogonal space on which $\mathrm{GSO}_{4,2}$ 
acts on the left.
Then $Z=X\otimes Y$ is a symplectic space over $F$.
Here we take $X_\pm\otimes Y$ as the polarization
and we realize the Weil representation $\omega_\psi$
of $\mathrm{Mp}\left(Z\right)\left(\mA\right)$
on $V_\omega:=\mathcal{S}((X_+ \otimes Y)(\mA))$.

Put $X^\Box = X \oplus (-X)$.
Then $X^\Box$ is naturally  a symplectic space. Let
$\widetilde{G}:=\mathrm{GSp}\left(X^\Box\right)$ and
we denote by  $\mathbf{G}$ a subgroup of $G\times G$ given by
\[
\mathbf{G}: = \left\{ (g_1, g_2) \in G \times G
 : \lambda(g_1) = \lambda(g_2) \right\},
\]
which has a natural embedding $\iota:\mathbf {G}\to\widetilde{G}$.
We define the canonical pairing $\mathcal{B}_\omega : V_\omega \otimes V_\omega \rightarrow \mC$ by 
\[
\mathcal{B}_\omega(\varphi_1, \varphi_2) := \int_{(X_+ \otimes Y)(\mA)} \varphi_1(x) \,\overline{\varphi_2(x)} \, dx
\quad\text{for $\varphi_1, \varphi_2 \in V_\omega$}
\]
where $dx$ denotes the Tamagawa measure on $(X_+ \otimes Y)(\mA)$.

Let $\widetilde{Z} = X^\Box \otimes Y$ and 
we take a polarization $\widetilde{Z}=\widetilde{Z}_+\oplus \widetilde{Z}_-$ with
\[
\widetilde{Z}_\pm:=\left(X_\pm\oplus \left(-X_\pm\right)\right)\otimes Y
\]
where the double sign corresponds.
Let us denote by $\widetilde{\omega}_\psi$ the Weil representation 
of $\mathrm{Mp}(\widetilde{Z}(\mA))$ on $\mathcal{S}( \widetilde{Z}^+(\mA))$.
On the other hand, let
\[
X^\nabla:=\left\{(x,-x) : x \in X \right\} 
\quad\text{and}\quad \widetilde{X}^\nabla := X^\nabla \otimes Y .
\]
Then we have a natural isomorphism
\[
V_\omega \otimes V_\omega\simeq
\mathcal{S}( \widetilde{X}^\nabla(\mA)) 
\]
by which we regard $\mathcal{S}( \widetilde{X}^\nabla(\mA))$
as a representation of $\mathrm{Mp}(Z)(\mA) \times \mathrm{Mp}(Z)(\mA)$.
Meanwhile
we may realize $\widetilde{\omega}_\psi$ on $\mathcal{S}(\widetilde{X}^\nabla(\mA) )$
and indeed
we have an isomorphism
\[
\delta : \mathcal{S}(\widetilde{Z}_+(\mA)) \rightarrow  \mathcal{S}( \widetilde{X}^\nabla (\mA))
\]
as representations of $\mathrm{Mp}(\widetilde{Z})(\mA)$ such that 
\[
\delta(\varphi_1 \otimes \overline{\varphi}_2)(0) = \mathcal{B}_\omega(\varphi_1, \varphi_2)
\quad\text{for $\varphi_1, \varphi_2 \in V_\omega$.}
\]

Let us define Petersson inner products on $G(\mA)$ and $G(\mA)^+$ as follows.
For $f_1, f_2 \in V_\pi$, we define the Petersson inner product 
$\left(\, ,\,\right)_\pi$ on  $G(\mA)$ by
\[
(f_1, f_2)_\pi := \int_{\mA^\times G(F) \backslash G(\mA)} f_1(g) \,\overline{f_2(g)} \,dg
\]
where $dg$ denotes the Tamagawa measure. 
Then regarding $f_1, f_2$ as automorphic forms on $G(\mA)^+$,
we define
\[
(f_1, f_2)_{\pi}^+ := \int_{\mA^\times G(F)^+ \backslash G(\mA)^+} f_1(h) \,\overline{f_2(h)} \,dh
\]
where the measure $dh$ is normalized so that 
\[
\mathrm{vol}(\mA^\times G(F)^+ \backslash G(\mA)^+)=1.
\]
Then from our assumption \eqref{assumption 1}
on $\pi$, as in \cite[Lemma~6.3]{GI0}, we see that 
\[
\left(f_1, f_2\right)_{\pi}^+= \frac{1}{2}\,\left(f_1, f_2\right)_{\pi}
\]
since $\mathrm{Vol}(\mA^\times G(F) \backslash G(\mA)) =2$.
For each place $v$ of $F$, we take a hermitian $G(F_v)$-invariant local pairing $\left(\, ,\,\right)_{\pi_v}$ of $\pi_v$
so that
\begin{equation}
\label{e:inner decomp H}
\left(f_1,f_2\right)_\pi = \prod_v
 \left(f_{1,v},f_{2,v}\right)_{\pi_v}
 \quad\text{for $f_i=\otimes_v \,f_{i,v}\in V_\pi\, \left(i=1,2\right)$}.
\end{equation}
We also choose a local Haar measure $d g_v$ on $G(F_v)$ for each place $v$ of $F$ so that
$\mathrm{Vol}(K_{G, v}, d_{g_v})=1$ at almost all $v$, where $K_{G, v}$ is a maximal compact subgroup of $G(F_v)$.
We define positive constants $C_{G}$ by
\[
dg = C_{G} \cdot \prod_v dg_v
\]

Local doubling zeta integrals are defined as follows.
Let $I(s)$ denote the degenerate principal series representation of $\widetilde{G}(\mA)$ defined by 
\[
I(s): = \mathrm{Ind}_{\widetilde{P}(\mA)}^{\widetilde{G}(\mA)}
\left(\chi_E \,\delta_{\widetilde{P}}^{\,s\slash 9}\right)
\]
where $\widetilde{P}$ denotes the Siegel parabolic subgroup of $\widetilde{G}$.
Then for each place $v$, we define a local zeta integral  by
\[
Z_v(s, \Phi_v, f_{1, v}, f_{2,v}) := \int_{G^1(F_v)} \Phi_v(\iota(g_v, 1), s) \left( \pi_v(g_v)f_{1,v}, f_{2,v} \right)_{\pi_v} \, dg_v
\]
for $\Phi_v\in I\left(s\right)$, $f_{1,v},f_{2,v}\in V_{\pi_v}$,
where $G^1=\left\{g\in G:\lambda\left(g\right)=1\right\}$.
The integral  converges absolutely at $s=\frac{1}{2}$ when $\Phi_v \in I_v(s)$ is a holomorphic section by \cite[Proposition~6.4]{PSR} (see also \cite[Lemma~6.5]{GI0}).
Moreover, when we define a map
$
\mathcal{S}( \widetilde{X}^\nabla(\mA))
\ni\varphi\mapsto \left[\varphi\right]\in  I \left(\frac{1}{2} \right)
$
by
\[
[\varphi] \left(g, \frac{1}{2} \right) := |\nu(g)|^{-4}  \left(\widetilde{\omega}_\psi(\begin{pmatrix}1_4&\\&\lambda\left(g\right)^{-1}1_4\end{pmatrix}g))\varphi\right)(0),
\]
we may naturally extend $\left[\varphi\right]$
 to a holomorphic section in $I\left(s\right)$. 

By an argument similar to the one in the proof of 
\cite[Proposition~6.10]{GI0}, 
we may derive the following Rallis inner product formula in 
the similitude groups case from  the one \cite[Theorem~8.1]{GQT}
in the isometry groups case.
%
\begin{proposition}
\label{RIPF H}
Keep the above notation. 

Then for decomposable vectors $f  = \otimes f_v \in V_\pi$ and $\phi = \otimes \phi_v\in V_\omega$, we have
\begin{multline*}
\frac{\langle \Theta(f; \phi),  \Theta(f; \phi) \rangle}{\left( f, f \right)_\pi} =C_G \cdot \frac{1}{2} \cdot
\frac{L\left(1, \pi, \mathrm{std} \otimes \chi_E \right)}{L(3, \chi_E) L(2, \mathbf{1}) L(4, \mathbf{1}) }
\\
\times
\prod_{v} 
\,
Z_v^\sharp \left(\frac{1}{2}, [\delta(\phi_v \otimes \phi_v)], f_v, f_v \right).
\end{multline*}
Here we recall that $\Theta_\psi(f; \phi)$ is the
theta lift of $f$ to $\mathrm{GO}_{4,2}$, 
$\langle\,\, , \,\,\rangle$  denotes the Petersson inner product
with respect to the Tamagawa measure and we define
\begin{multline*}
Z_v^\sharp \left(\frac{1}{2}, [\delta(\phi_v \otimes \phi_v)], f_v, f_v \right): = 
\frac{1}{\left(f_v, f_v \right)_{\pi_v}}
\frac{L(3, \chi_{E_v \slash F_v}) L(2, \mathbf{1}_v) 
L(4, \mathbf{1}_v) }{L\left(1, \pi_v, \mathrm{std} \otimes 
\chi_{E_v \slash F_v} \right)}
\\
\times
Z_v \left(\frac{1}{2}, [\delta(\phi_v \otimes \phi_v)], f_v, f_v \right),
\end{multline*}
which is equal to $1$ at almost all places $v$ of $F$ by \cite{PSR}.
\end{proposition}
%
%
%
%
%
Recall that $ \theta(f; \phi)$ denotes
the restriction of $ \Theta_\psi(f; \phi)$ to $\mathrm{GSO}_{4,2}(\mA)$, namely the theta lift of $f$ to $\mathrm{GSO}_{4,2}$.
Then as in \cite[Lemma~2.1]{GI0}, we see that
\[
2\langle \Theta(f; \phi),  \Theta(f; \phi) \rangle
= \langle \theta(f; \phi), \theta(f; \phi) \rangle
\]
where the right hand side denotes the Petersson inner product
on $\mathrm{GSO}_{4,2}$ with respect to the Tamagawa measure.
Hence, Proposition~\ref{RIPF H} yields
\begin{multline}
\label{RIPF H GSO}
\frac{\langle  \theta(f; \phi), \theta(f; \phi) \rangle}{\left(f, f \right)_{\pi}}=C_G \cdot  \frac{L\left(1, \pi, \mathrm{std} \otimes \chi_E \right)}{L(3, \chi_E) L(2, \mathbf{1}) L(4, \mathbf{1}) }
\\
\times
\prod_{v}\, Z_v^\sharp \left(\frac{1}{2}, [\delta(\phi_v \otimes \phi_v)], f_v, f_v \right).
\end{multline}
%
%
%
%
%
%
%
%
%
%
%
%
%
%
%
%
\subsection{Theta lift from $G_D$ to $\mathrm{GSU}_{3,D}$}
\label{s:RI HD gsu}
In this subsection, we shall consider the Rallis inner product formula for the theta lift from $G_D$ to $\mathrm{GSU}_{3,D}$ as in the previous section.
We recall  that  the formula in the case of isometry groups is proved by Yamana~\cite[Lemma~10.1]{Yam} 
where our case corresponds to ($\rm{I}_3$)
with $m=3, n=2$.

Let $(\pi, V_\pi)$ be an irreducible cuspidal automorphic representation of $G_D(\mA)$ with a trivial central character.
Recall that $\mathcal G_D$ denotes the subgroup of $G_D\left(\mA\right)$
given by \eqref{d: mathcal G_D}.
In this section, assume that:
\begin{equation}\label{assumption 1_D}
\text{\emph{the restriction of $\pi$ to $\mathcal{G}_D$ is irreducible}}
\end{equation}
for our later use.

Let us recall the notation in \ref{ss: second pull-back}.
Thus $X_D$ denotes the hermitian space of degree two over $D$
on which $G_D$ acts on the right and $Y_D$ 
denotes the skew-hermitian space of degree three
over $D$ on which $\mathrm{GSU}_{3,D}$ acts on the left.
Then $Z_D= X_D \otimes_D Y_D$ is a symplectic space over $F$
by \eqref{d: quaternion dual pair}.
Here we take $X_{D,\pm}\otimes_D Y_D$ as the polarization
and we realize the Weil representation $\omega_\psi$ of 
$\mathrm{Mp}\left(Z_D\right)\left(\mA\right)$ on 
$V_{\omega,D}:=\mathcal S\left(\left(X_{D,+}\otimes_D Y_D\right)\left(\mA\right)\right)$.

Put $X_D^\Box = X_D \oplus \overline{X_D}$. 
Then $X_D^\Box$ is naturally a hermitian space over $D$.
Let $\widetilde{G}_D:=\mathrm{GU}\left(X_D^\Box\right)$ and we denote by $\mathbf{G}_D$ 
a subgroup of $G_D\times G_D$ given by
\[
\mathbf{G}_D := \left\{ (g_1, g_2) \in G_D \times G_D : \lambda(g_1) = \lambda(g_2) \right\}
\]
which has a natural embedding $\iota:\mathbf{G}_D\to\widetilde{G}_D$.
We define the canonical pairing $\mathcal{B}_\omega : V_{\omega,D} \otimes V_{\omega,D}\rightarrow \mC$ by 
\[
\mathcal{B}_\omega(\varphi_1, \varphi_2) := \int_{(X_{D,+} \otimes Y_D)(\mA)} \varphi_1(x) \,\overline{\varphi_2(x)} \, dx
\quad\text{for $\varphi_1, \varphi_2 \in V_{\omega,D}$}
\]
where $dx$ denotes the Tamagawa measure on $(X_{D,+} \otimes Y_D)(\mA)$.

Let $\widetilde{Z}_D = X_D^\Box \otimes Y_D$ and 
we take a polarization $\widetilde{Z}_D=\widetilde{Z}_{D,+}\oplus \widetilde{Z}_{D,-}$ with
\[
\widetilde{Z}_{D,\pm}=
\left(
X_{D,\pm}\oplus \overline{-X_{D,\pm}}\,\right)\otimes Y_D
\]
where the double sign corresponds.
Let us denote by $\widetilde{\omega}_\psi$ the Weil representation 
of $\mathrm{Mp}(\widetilde{Z}_D)(\mA)$ on $\mathcal{S}( \widetilde{Z}_{D,+}(\mA))$.
On the other hand, let
\[
X_D^\nabla:=\left\{(x,\overline{x}) : x \in X_D \right\} 
\quad\text{and}\quad \widetilde{X}_D^\nabla := X_D^\nabla \otimes Y_D .
\]
Then we have a natural isomorphism
\[
V_{\omega,D} \otimes V_{\omega,D}\simeq
\mathcal{S}( \widetilde{X}_D^\nabla(\mA)) 
\]
by which we regard $\mathcal{S}( \widetilde{X}_D^\nabla(\mA))$
as a representation of $\mathrm{Mp}(Z_D)(\mA) \times \mathrm{Mp}(Z_D)(\mA)$.
Meanwhile
we may realize $\widetilde{\omega}_\psi$ on $\mathcal{S}(\widetilde{X}_D^\nabla(\mA) )$
and indeed
we have an isomorphism
\[
\delta : \mathcal{S}(\widetilde{Z}_{D,+}(\mA)) \rightarrow  \mathcal{S}( \widetilde{X}_D^\nabla (\mA))
\]
as representations of $\mathrm{Mp}(\widetilde{Z}_D)(\mA)$ such that 
\[
\delta(\varphi_1 \otimes \overline{\varphi}_2)(0) = \mathcal{B}_\omega(\varphi_1, \varphi_2)
\quad\text{for $\varphi_1, \varphi_2 \in V_{\omega,D}$.}
\]

Let us define Petersson inner products on $G_D(\mA)$ and $G_D(\mA)^+$ as follows.
For $f_1, f_2 \in V_{\pi_D}$, we define the Petersson inner product 
$\left(\, ,\,\right)_{\pi_D}$ on  $G_D(\mA)$ by
\[
(f_1, f_2)_{\pi_D}: = \int_{\mA^\times G_D(F) \backslash G_D(\mA)} f_1(g) \,\overline{f_2(g)} \,dg
\]
where $dg$ denotes the Tamagawa measure. 
Then regarding $f_1, f_2$ as automorphic forms on $G_D(\mA)^+$,
we define
\[
(f_1, f_2)_{\pi_D}^+ := \int_{\mA^\times G_D(F)^+ \backslash G_D(\mA)^+} f_1(h) \,\overline{f_2(h)} \,dh
\]
where the measure $dh$ is normalized so that 
\[
\mathrm{vol}(\mA^\times G_D(F)^+ \backslash G_D(\mA)^+)=1.
\]
Then from our assumption \eqref{assumption 1_D}
on $\pi_D$, as in \cite[Lemma~6.3]{GI0}, we see that 
\[
\left(f_1, f_2\right)_{\pi_D}^+= \frac{1}{2}\,\left(f_1, f_2\right)_{\pi_D}
\]
since $\mathrm{Vol}(\mA^\times G_D(F) \backslash G_D(\mA)) =2$.
For each place $v$ of $F$, we take a hermitian $G_D(F_v)$-invariant local pairing $\left(\, ,\,\right)_{\pi_{D,v}}$ of $\pi_{D,v}$
so that
\begin{equation}
\label{e:inner decomp H_D}
\left(f_1,f_2\right)_{\pi_D} = \prod_v
 \left(f_{1,v},f_{2,v}\right)_{\pi_{D,v}}
 \quad\text{for $f_i=\otimes_v \,f_{i,v}\in V_{\pi_D}\, \left(i=1,2\right)$}.
\end{equation}
As in the previous section, we choose local Haar measures $dg_v$  on $G_D(F_v)$ at each place $v$ of $F$ and we have 
\[
dg = C_{G_D} \cdot \prod_v dg_v
\]
for some positive constant $C_{G_D}$.

Local doubling zeta integrals are defined as follows.
Let $I_D(s)$ denote the degenerate principal series representation of $\widetilde{G}_D(\mA)$ defined by 
\[
I_D(s) := \mathrm{Ind}_{\widetilde{P}_D(\mA)}^{\widetilde{G}_D(\mA)}\left(\chi_E \,\delta_{\widetilde{P}_D}^{\,s \slash 9}\right)
\]
where $\widetilde{P}_D$ denotes the Siegel parabolic subgroup of $\widetilde{G}_D$.
Then for each place $v$, we define a local zeta integral  
for $\Phi_v\in I_{D,v}\left(s\right)$, $f_{1,v},f_{2,v}\in V_{\pi_{D,v}}$
by
\[
Z_v(s, \Phi_v, f_{1, v}, f_{2,v}) := \int_{G^1_{D}(F_v)} \Phi_v(\iota(g_v, 1), s) \left( \pi_{D,v}(g_v)f_{1,v}, f_{2,v} \right)_{\pi_v} \, dg_v
\]
where $G_D^1=\left\{g\in G_D:\lambda\left(g\right)=1\right\}$.
The integral  converges absolutely at $s=\frac{1}{2}$ when 
$\Phi_v \in I_{D,v}(s)$ is a holomorphic section by \cite[Proposition~6.4]{PSR} (see also \cite[Lemma~6.5]{GI0}).
Moreover, when we define a map
$
\mathcal{S}( \widetilde{X}_D^\nabla(\mA))
\ni\varphi\mapsto \left[\varphi\right]\in  I_D \left(\frac{1}{2} \right)
$
by
\[
[\varphi] \left(g, \frac{1}{2} \right): = |\lambda(g)|^{-4}  \left(\widetilde{\omega}_\psi(\begin{pmatrix}1_4&\\&\lambda\left(g\right)^{-1}1_4\end{pmatrix}g))\varphi\right)(0),
\]
we may naturally extend $\left[\varphi\right]$
 to a holomorphic section in $I_D\left(s\right)$. 

By an argument similar to the one in the proof of 
\cite[Proposition~6.10]{GI0}, 
we may derive the following Rallis inner product formula in 
the similitude groups case from  the one \cite[Theorem~2]{Yam2}
in the isometry groups case.
%
\begin{proposition}
\label{RIPF HD}
Keep the above notation. 

Then for decomposable vectors $f  = \otimes f_v \in V_{\pi_D}$ and $\phi = \otimes \phi_v\in V_{\omega,D}$, we have
\[
\frac{\langle \theta(f; \phi),  \theta(f; \phi) \rangle}{\left( f_{\pi_D}, f_{\pi_D} \right)} =
\frac{L\left(1, \pi, \mathrm{std} \otimes \chi_E \right)}{L(3, \chi_E) L(2, \mathbf{1}) L(4, \mathbf{1}) }\,
\prod_{v} \,Z_v^\sharp \left(\frac{1}{2}, [\delta(\phi_v \otimes \phi_v)], f_v, f_v \right).
\]
Here recall that $\theta_\psi(f; \phi)$ is the
theta lift of $f$ to $\mathrm{GSU}_{3,D}$,
$\langle\,\, ,\,\,\rangle$ denotes the Petersson inner product
with respect to the Tamagawa measure and we define
\begin{multline*}
Z_v^\sharp \left(\frac{1}{2}, [\delta(\phi_v \otimes \phi_v)], f_v, f_v \right) := \frac{1}{\left( f_v, f_v \right)_{\pi_{D, v}}}
\frac{L(3, \chi_{E_v \slash F_v}) L(2, \mathbf{1}_v) 
L(4, \mathbf{1}_v) }{L\left(1, \pi_v, \mathrm{std} \otimes 
\chi_{E_v \slash F_v} \right)}
\\
\times
Z_v \left(\frac{1}{2}, [\delta(\phi_v \otimes \phi_v)], f_v, f_v \right),
\end{multline*}
which is equal to $1$ at almost all places $v$ of $F$ by \cite{PSR}.
\end{proposition}
%
%
%
%
%
%
%
%
%
%
%
%
%
%
%
%
%
%
%
%
%
\section{Explicit formula for Bessel periods on $\mathrm{GU}(4)$}
Let $\mathrm{GU}\left(4\right)$ stand for
 one of $\mathrm{GU}_{2,2}$ or $\mathrm{GU}_{3,1}$.
In \cite{FM3}, the explicit formula for the Bessel periods on
$\mathrm{GU}\left(4\right)$ is proved under the assumption that 
 the  explicit formula for the Whittaker periods on
 $\mathrm{GU}_{2,2}$ holds.
In this section
we shall show that this assumption is indeed satisfied 
in the cases we need,
 from the explicit formula for the Whittaker periods on $G=\mathrm{GSp}_2$,
 which in turn will be proved in Appendix~\ref{appendix A}.
Thus the explicit formula for the Bessel periods on 
$\mathrm{GU}(4)$ holds by \cite{FM3},
in the cases which we need for the proof of
 Theorem~\ref{ref ggp}.
 %
 %
 %
\subsection{Explicit formulas}
\label{s:Explicit formulas whittaker}
Let $(\pi, V_{\pi})$ be an irreducible cuspidal tempered globally generic automorphic representation of $G(\mA)$ 
such that $\pi |_{\mathcal{G}}$ is irreducible.
We recall that the subgroup $\mathcal G$ 
of $G\left(\mA\right)$ is defined by \eqref{cal G def}.
Let $\pi^\circ$ denote the unique generic irreducible constituent of $\pi |_{G(\mA)^+}$.
Let $\left(\Sigma, V_\Sigma\right)$ denote the theta lift of 
$\pi^\circ$ to $\mathrm{GSO}_{4,2}(\mA)$.
Then as in \cite[Proposition~3.3]{Mo}, we know that $\Sigma$ is 
an irreducible globally generic  cuspidal  tempered automorphic representation. 
Here we prove the explicit formula for the Whittaker periods for $\Sigma$ 
assuming the explicit formula for the Whittaker periods for
 $\pi$.

Let us recall some notation.
Let $X, Y, Y_0$ and $Z$ be as in Section~\ref{sp4 so42}
and 
we use a polarization $Z = Z_+ \oplus Z_-$ with
\[
Z_{\pm} = (X \otimes Y_{\pm}) \oplus (X_{\pm} \otimes Y_0)
\]
where the double sign corresponds.
We write $z_+=(a_1, a_2; b_1, b_2)$ when
\[
z_+=a_1 \otimes y_1 + a_2 \otimes y_2 + b_1 \otimes e_1+b_2\otimes e_2\in Z_+
\quad\text{with $a_i \in X$, $b_i \in X_+$}
.
\]
Recall that the unipotent subgroups
$N_0$, $N_1$ and $N_2$ of $\mathrm{GSO}_{4,2}$
are defined by \eqref{d: N_0}, \eqref{d: N_1}
and \eqref{d: N_2}, respectively.
Let us define an unipotent subgroup $\tilde{U}$
of $\mathrm{GSO}_{4,2}$ by
\begin{equation}\label{e: def of u_tilde}
\tilde{U}:  = \left\{ \tilde{u}(b): =
\begin{pmatrix}
1&-{}^{t}\widetilde{X} S_1& 0\\
0&1_4&\widetilde{X}\\
0&0&1
\end{pmatrix}
: \widetilde{X} = \begin{pmatrix} 0\\0\\0\\-b\end{pmatrix}
\right\}
\end{equation}
where $S_1$ is given by \eqref{d: symmetric}.
Let 
\begin{equation}
\label{def:U}
U:=N_{4,2}\,\tilde{U}.
\end{equation}
Then $U$ is a maximal unipotent subgroup of $\mathrm{GSO}_{4,2}$
and we have
\[
N_0 \triangleleft N_0 N_1 \triangleleft N_0 N_1 N_2=
N_{4,2}  \triangleleft N_{4,2}\, \tilde{U} = U.
\]
Then we define a non-degenerate character $\psi_U$ of $U(\mA)$ by 
\begin{equation}
\label{def:psi_U}
\psi_U(u_0(x)u_1(s_1, t_1)u_2(s_2, t_2)\tilde{u}(b)) := \psi(2dt_2+b).
\end{equation}
By \cite[Proposition~3.3]{Mo}, $\Sigma$ is $\psi_U$-generic. Namely
\[
\label{W U}
W^{\psi_U}(\varphi) := \int_{U(F) \backslash U(\mA)} \varphi(u)\, \psi_U^{-1}(u) \,du
 \quad\text{for $ \varphi \in V_\Sigma$},
\]
is not identically zero on $V_\Sigma$.
Now we regard $\Sigma$ as an automorphic representation
of $\mathrm{GU}_{2,2}$ by the accidental isomorphism
\eqref{acc isom2} and
let $\Pi_{\Sigma} =\Pi_{1}^\prime \boxplus \cdots \boxplus \Pi_\ell^\prime$ denote the base change lift of 
$\Sigma\mid_{\mathrm{U}_{2,2}}$ to $\mathrm{GL}_4(\mA_E)$ where 
$\Pi_i^\prime$ is an irreducible cuspidal
automorphic representations  of $\mathrm{GL}_{m_i}(\mA_E)$.
Here the existence of $\Pi_\Sigma$  follows from \cite{KMSW}.

Recall that in Section~\ref{s:RI H gso}, 
the Petersson inner products on $G\left(\mA\right)$
and $\mathrm{GSO}_{4,2}(\mA)$ using the Tamagawa measures,
denoted respectively as $\left(\,\, ,\,\,\right)$ and $\left<\,\, ,\,\,\right>$,
are introduced.
Moreover at each place $v$ of $F$, we choose and fix an  $G(F_v)$-invariant hermitian
inner product $\left(\,\, ,\,\,\right)_v$ on $V_{\pi^\circ_v}$ so that 
the decomposition formula
\eqref{e:inner decomp H} holds.
Similarly at each place $v$, we choose and fix a $\mathrm{GSO}_{4,2}(F_v)$-invariant hermitian
inner product $\langle\,\,,\,\,\rangle_v$ on $V_{\Sigma_v}$
so that the decomposition formula
\begin{equation}\label{e:inner decomp GSO_4,2}
\langle \phi_1,\phi_2\rangle=
\prod_v\,\langle \phi_{1,v},\phi_{2,v}\rangle_v
\quad
\text{for $\phi_i=\otimes_v\,\phi_{i,v}\in V_{\Sigma}\quad(i=1,2)$}
\end{equation}
holds.

Then as in Section~\ref{s:def local bessel}, 
at each place $v$ of $F$,
we may define a local  period $\mathcal{W}_v(\varphi_v)$
for $\varphi_v \in V_{\Sigma_v}$ by the stable integral
\begin{equation}
\label{e: local integral whittaker}
\mathcal{W}_v(\varphi_v):=
\int_{U(F_v)}^{\mathrm{st}}
\frac{\langle \Sigma_v\left(n_v\right)\varphi_v,
\varphi_v \rangle_v}{{\langle \varphi_v, \varphi_v \rangle_v}}
\cdot \psi_U^{-1}\left(n_v\right)\, dn_v
\end{equation}
when $v$ is finite.
When $v$ is archimedean, we use
the Fourier transform to define $\mathcal{W}_v(\varphi_v)$.
 See \cite[Proposition~3.5, Proposition~3.15]{Liu2} for 
 the details.

We shall prove the following theorem,
namely 
the explicit formula for
the Whittaker periods on $V_\Sigma$,
in \ref{s: pf wh gso}. 
\begin{theorem}
\label{gso whittaker}
For a non-zero decomposable vector $\varphi = \otimes \varphi_v \in V_\Sigma$, we have
\begin{equation}
\label{e:gso whittaker}
\frac{|W^{\psi_U}(\varphi)|^2}{ \langle \varphi, \varphi \rangle} = 
\frac{1}{2^\ell}\cdot \frac{\prod_{j=1}^4 L \left(j, \chi_{E}^j \right)}{L \left(1, \Pi_{\Sigma}, \mathrm{As}^+ \right)}\cdot \prod_v 
\mathcal{W}_v^\circ(\varphi_v)
\end{equation}
where 
\[
\mathcal{W}_v^\circ(\varphi_v):=
\frac{L\left(1,\Pi_{\Sigma_v}, \mathrm{As}^+\right)}{
\prod_{j=1}^4 L \left(j, \chi_{E_v}^j \right)}
\cdot \mathcal{W}_v(\varphi_v).
\]
Here we note that $\mathcal{W}_v^\circ(\varphi_v) = 1$ at almost all places $v$ by Lapid and Mao~\cite{LM}.
\end{theorem}
%
%
%
%
%
Before proceeding to the proof of Theorem~\ref{gso whittaker},
by assuming it, we prove the following theorem, namely
the explicit formula for the Bessel periods on $\mathrm{GU}\left(4\right)$.
%
%
%
\begin{theorem}
\label{gsu Bessel}
Let $(\pi, V_\pi)$ be an irreducible cuspidal tempered automorphic representation of $G_D(\mA)$ with trivial central character. 
Suppose that
$\pi$ has the $\left(\xi,\Lambda,\psi\right)$-Bessel period and that $\pi$ is neither of type I-A nor type I-B.
Let $\pi_+^B$ denote the unique irreducible constituent of $\pi|_{G_D(\mA)^+}$
which has  the $\left(\xi,\Lambda,\psi\right)$-Bessel period.
We denote by $\left(\sigma, V_\sigma\right)$
the theta lift of $\pi_+^B$ to $\mathrm{GSU}_{3,D}$,
which is an irreducible cuspidal automorphic representation by Lemma~\ref{nonzero lemma (1)} and Lemma~\ref{irr lemma}. 

Then for a non-zero decomposable vector $\varphi = \otimes \varphi_v \in V_\sigma$, we have
\[
\frac{|\mathcal B_{X, \psi. \Lambda}(\varphi)|^2}{(\varphi, \varphi)} =\frac{1}{2^\ell} \left( \prod_{j=1}^4 L(1, \chi_E^j) \right) \frac{  L \left(\frac{1}{2}, \sigma \times \Lambda^{-1} \right)}{L(1, \pi, \mathrm{std} \otimes \chi_E) L(1, \chi_E)} \prod_v \alpha_{\Lambda_v, \psi_{X, v}}^{\natural}(\varphi_v) 
\]
where 
\[
\alpha_{\Lambda_v, \psi_{X, v}}^{\natural}(\varphi_v) =
\left( \prod_{j=1}^4 L(1, \chi_{E, v}^j) \right)^{-1} \frac{L(1, \pi_v, \mathrm{std} \otimes \chi_{E,v}) L(1, \chi_{E_v})}{ L \left(\frac{1}{2}, \sigma_v \times \Lambda_v^{-1} \right)}
\cdot \frac{\alpha_{\Lambda_v, \psi_{X, v}}(\varphi_v)}{(\varphi_v, \varphi_v)_v}
\]
and $X\in D^\times$ is taken so that $\xi=S_X$
in \eqref{def of S_X}.
\end{theorem}
%
%
%
\begin{proof}
Let us regard $\sigma$ as an automorphic representation of $\mathrm{GU}(4)$ with trivial central character via the accidental isomorphisms $\Phi$
\eqref{acc isom2} or $\Phi_D$ \eqref{acc isom1},
depending whether $D$ is split or not.
Let $\theta(\sigma) = \Theta_{\psi, (\Lambda^{-1},\Lambda^{-1})}(\sigma)$ denote the theta lift of $\sigma$ to $\mathrm{GU}_{2,2}$ with respect to $\psi$ and 
$(\Lambda^{-1}, \Lambda^{-1})$.
By \cite[Proposition~3.1]{FM3}, $\theta(\sigma)$ is globally generic
and, in particular, non-zero.
By the same argument as in
the proof of \cite[Theorem~1]{FM3}, we see that $\theta(\sigma)$ is 
cuspidal and hence irreducible by Remark~\ref{theta irr rem}
and \ref{theta irr rem D}.
Moreover by the unramified computations in \cite{Ku2} and \cite[(3.6)]{Mo}, we see that 
$L^S(s, \Sigma, \wedge_t^2)$ has a pole at $s=1$
 when $S$ is 
 a sufficiently large finite set  of places of $F$ containing all archimedean places,
where 
$L^S(s, \Sigma, \wedge_t^2)$ denotes the twisted exterior square $L$-function of $\Sigma$ (see \cite[Section~2.1.1]{FM0} for the definition).
Since $\theta(\sigma)$ is generic, \cite[Theorem~4.1]{FM0}  implies
 that it has the unitary Shalika period defined in \cite[(2.5)]{FM0}.
Then, by \cite[Theorem~B]{Mo}, the theta lift of $\theta(\sigma)$ to $G(\mA)^+$,
which we denote by  $(\pi_{+}^\prime, V_{\pi_{+}^\prime})$,
is an irreducible cuspidal globally generic automorphic representation
of $G\left(\mA\right)^+$.
We note that $\pi_+^B$ is nearly equivalent to $\pi_+^\prime$.

Let us take an irreducible cuspidal automorphic representation $(\pi^\prime, V_{\pi^\prime})$ of $G(\mA)$
such that $V_{\pi^\prime}|_{G(\mA)^+} \supset V_{\pi_+^\prime}$. 
Then $\pi^\prime$ is globally generic. 
Moreover
$\pi^\prime \otimes \chi_E$ is not nearly equivalent to $\pi^\prime$
by our assumption on $\pi$.
Hence $\pi^\prime|_{\mathcal{G}}$ is irreducible.
Thus we may apply Theorem~\ref{gso whittaker},
taking  $\pi^\circ = \pi^\prime$ and $\Sigma = \theta(\sigma)$,
and we obtain the explicit formula 
for  the Whittaker periods on $\theta(\sigma)$.
Then by \cite[Theorem~A.1]{FM3}, the required explicit formula for the
Bessel periods follows.
\end{proof}
%
%
%
%
%
%
%
%
\subsection{Proof of Theorem~\ref{gso whittaker}}
\label{s: pf wh gso}
We reduce Theorem~\ref{gso whittaker} to a certain local identity
in \ref{Reduction to a local identity}
and then prove the local identity in \ref{ss: local pull-back computation}.

As we stated in the beginning of this section, what we do essentially is to 
deduce the explicit formula \eqref{e:gso whittaker} 
for the Whittaker periods on $\mathrm{GSO}_{4,2}$ from
\eqref{e:gsp whittaker} below, the one for the Whittaker periods on $G$.
%
\subsubsection{Explicit formula for the Whittaker periods on $G=\mathrm{GSp}_2$}
Let $U_G$ denote the maximal unipotent subgroup of $G$. 
Namely
\begin{equation}\label{d: U_G}
U_G := \left\{ m\left(n\right)\begin{pmatrix}1_2&X\\ 0&1_2 \end{pmatrix} : X \in \mathrm{Sym_2}, n \in N_2 \right\}
\end{equation}
where $m\left(h\right)=\begin{pmatrix} h&0 \\0 &{}^{t}h^{-1}\end{pmatrix} $ for $h\in\mathrm{GL}_2$
and $N_2$ denotes the group of upper unipotent matrices in $\mathrm{GL}_2$.
Then we define a non-degenerate character $\psi_{U_G}$ of $U_G(\mA)$ by 
\begin{equation}
\label{def:psi_UG}
\psi_{U_G}(u) := \psi(u_{1\,2}+d \,u_{2\,4})
\quad\text{for $u=\left(u_{i\,j}\right)\in U_G\left(\mA\right)$}.
\end{equation}
Then for an automorphic form $\phi$ on $G(\mA)$, we define 
the Whittaker period $W_{\psi_{U_G}}(\phi)$ 
of $\phi$ by 
\[
\label{W U_G}
W_{\psi_{U_G}}(\phi) = \int_{U_G(F) \backslash U_G(\mA)} \phi(n) \,\psi_{U_G}^{-1}(n) \,dn.
\]
The following theorem shall be  proved in Appendix~\ref{appendix A}.
%
%
%
\begin{theorem}
\label{gsp whittaker}
Suppose that $\left(\pi,V_{\pi}\right)$ is an irreducible cuspidal
tempered globally generic automorphic representation of $G\left(\mA\right)$.
Let $\Pi_{\pi} = \Pi_1 \boxplus \cdots \boxplus \Pi_k$ denote the functorial lift of $\pi$ to $\mathrm{GL}_4(\mA)$. 

Then for any non-zero decomposable vector $\varphi = \otimes \varphi_v \in V_{\pi}$, we have
\begin{equation}
\label{e:gsp whittaker}
\frac{|W_{\psi_{U_G}}(\varphi)|^2}{\left( \varphi, \varphi \right)} =\frac{1}{2^k}\cdot \frac{\prod_{j=1}^2\xi_F(2j)}{L\left(1, \Pi_{\pi}, \mathrm{Sym}^2\right)} \cdot \prod_v \mathcal{W}_{G,v}^\circ(\varphi_v).
\end{equation}
Here $\mathcal{W}^\circ_{G, v}(\varphi_v)$ is defined by
\[
\label{mathcal W_G}
\mathcal{W}^\circ_{G, v}(\varphi_v) =\frac{L\left(1, \Pi_{\pi, v}, \mathrm{Sym}^2\right)}{\prod_{j=1}^2 \zeta_{F_v}(2j)} \mathcal{W}_{G, v}(\varphi_v)
\]
and $\mathcal{W}_{G, v}(\varphi_v)$ is defined by
\[
\mathcal{W}_{G, v}(\varphi_v) =  \int_{U_G(F_v)}^{st} 
\frac{\left( \pi_v^\circ(n)\varphi_v, \varphi_v \right)}{\left( \varphi_v, \varphi_v \right)} \,\psi_{U_G}^{-1}(n) \, dn
\]
when $v$ is finite
and by the Fourier transform when $v$ is archimedean.
\end{theorem}
%
%
%
\subsubsection{Reduction to a local identity}
\label{Reduction to a local identity}
Let us go back to the situation
stated in the beginning of \ref{s:Explicit formulas whittaker}.

First we note that the unramified computation in \cite{Ku2}
implies the following lemma.
\begin{lemma}
\label{pi sigma identity}
There exists a finite set $S_0$ of places of $F$ containing all archimedean
places such that for a place $v\notin S_0$, 
we have 
\[
L \left(1, \Pi_{\Sigma_v}, \mathrm{As}^+ \right) = L\left(1, \pi_v, \mathrm{std} \otimes \chi_E\right) L\left(1, \Pi_{\pi,v}, \mathrm{Sym}^2\right)L\left(1, \chi_{E_v}\right).
\]
\end{lemma}
Let us recall the following pull-back formula
 for the Whittaker period on $\Sigma = \theta_\psi\left(\pi^\circ\right)$.
\begin{proposition}\cite[p.~40]{Mo}
\label{mo 40}
Let $f \in V_{\pi^\circ}$ and $\phi \in \mathcal{S}(Z_+(\mA))$.
Then 
\begin{multline}\label{e: global w pull-back}
W^{\psi_U}(\theta(\phi; f)) = \int_{N(\mA) \backslash G^1(\mA)} 
(\omega_{\psi}(g_1, 1)\phi)( (x_{-2}, x_{-1}, 0, x_2)) 
\\
\times W_{\psi_{U_G}}(\pi^\circ(g_1) f) \, dg_1.
\end{multline}
\end{proposition}
%
%
%
Suppose that $f=\otimes f_v$ and $\phi=\otimes\phi_v$.
Then by an argument similar to the one in 
obtaining \cite[(2.27)]{FM2},
when $W_{\psi_{U_G}}(f) \ne 0$, we have
\[
W^{\psi_U}(\theta(\phi; f)) = C_{G^1} \cdot W_{\psi_{U_G}}(f) \cdot \prod_v \mathcal{L}_v^\circ(\phi_v, f_v) 
\]
where 
\begin{multline*}
\mathcal{L}_v^\circ(\phi_v, f_v)  := \int_{N(F_v) \backslash G^1(F_v)} 
(\omega_{\psi_v}(g_1, 1)\phi_v)((x_{-2}, x_{-1}, 0, x_2))
\\
\times \mathcal{W}_{G, v}^\circ(\pi_v^\circ(g_1)f_v) \, dg_{1,v}
\end{multline*}
when $\phi = \otimes_v \phi_v$ and $f = \otimes_v f_v$.
We also define
\[
\label{mathcal L}
\mathcal{L}_v(\phi_v, f_v): =  \frac{\prod_{j=1}^2
\xi_f\left(2j\right)}{L\left(1,\Pi_{\pi, v},
\mathrm{Sym}^2\right)} 
\,
\frac{\mathcal{L}_v^\circ(\phi_v, f_v)}{ \mathcal{W}_{G, v}(f_v)}.
\]
Here the measures are taken as the following.
Let $dg_v$ be the measure on $G^1\left(F_v\right)$
defined by the gauge form and $dn_v$ the measure on 
$N\left(F_v\right)$ defined in the manner 
stated in
\ref{ss: measures}.
Then we take the measure $dg_{1,v}$ on $N(F_v) \backslash G^1(F_v)$ so that $dg_v=dn_v \, dg_{1,v}$.

Let $\Theta\left(\pi_v^\circ,\psi_v\right):=\operatorname{Hom}_{G(F_v)^+}
\left(\Omega_{\psi_v},\bar{\pi}_v^\circ\right)$
where $\Omega_{\psi_v}$ is the extended local Weil representation
of $G(F_v)^+ \times \mathrm{GSO}_{4,2}\left(F_v\right)$
realized on $\mathcal S\left(Z_+\left(F_v\right)\right)$,
the space of Schwartz-Bruhat functions on $Z_+\left(F_v\right)$.
We recall that the action of $G(F_v)^+ \times \mathrm{GSO}_{4,2}\left(F_v\right)$
on $\mathcal S\left(Z_+\left(F_v\right)\right)$ via $\Omega_{\psi_v}$
is defined as in the global case (e.g. see \cite[2.2]{Mo}).
We also recall that for $\Sigma=\theta_\psi\left(\pi^\circ\right)$,
we have $\Sigma=\otimes_v\,\Sigma_v$
where $\Sigma_v=\theta_{\psi_v}\left(\pi_v^\circ\right)$ is the local theta lift of $\pi_v^\circ$.

Let 
\[
\theta_v:\mathcal S\left(Z_+\left(F_v\right)\right)\otimes V_{\pi_v^\circ}\to
V_{\Sigma_v}
\]
be a $G(F_v)^+ \times \mathrm{GSO}_{4,2}\left(F_v\right)$-equivariant
linear map, which is unique up to a scalar multiplication.
Since the global mapping
\[
\mathcal S\left(Z_+\left(\mA\right)\right)\otimes V_{\pi^\circ}
\ni\left(\phi^\prime,f^\prime\right)\mapsto
\theta_\psi\left(\phi^\prime; f^\prime\right)\in V_\Sigma
\]
is $G(F_v)^+ \times \mathrm{GSO}_{4,2}\left(F_v\right)$-equivariant  at any place $v$,
by the uniqueness of $\theta_v$, we may adjust $\left\{\theta_v\right\}_v$ so that 
\[
\theta_\psi\left(\phi^\prime; f^\prime\right)=\otimes_v\,
\theta_v\left(\phi_v^\prime\otimes f_v^\prime\right)
\quad\text{for $f^\prime=\otimes_v\,f_v^\prime\in V_{\pi^\circ}$,
$\phi^\prime=\otimes_v\, \phi_v^\prime
\in\mathcal S\left(Z_+\left(\mA\right)\right)$.}
\]
%
%
%
Then as in \cite[Section~2.4]{FM2}, combining Theorem~\ref{gsp whittaker}, 
the Rallis inner product formula \eqref{RIPF H GSO}, Lemma~\ref{pi sigma identity},
Lemma~\ref{compo number} and Proposition~\ref{mo 40},
we see that a proof of Theorem~\ref{gso whittaker} is reduced to 
a proof of the  following local identity~\eqref{e: 1st local id}.
%
%
%
\begin{proposition}
\label{1st local id}
Let $v$ be an arbitrary place of $F$ . 
For a given $f_v \in V_{\pi^\circ_v}^\infty$ satisfying  $\mathcal{W}_{G_,v}(f_v) \ne 0$,
there exists $\phi_v \in \mathcal{S}(Z_+(F_v))$ such that the local integral $\mathcal{L}_v(\phi_v, f_v)$ 
converges absolutely, $\mathcal{L}_v(\phi_v, f_v) \ne 0$ and the equality
\begin{equation}\label{e: 1st local id}
\frac{Z_v(\phi_v, f_v, \pi_v)  \cdot \mathcal{W}_v(\theta(\phi_v \otimes f_v))}{|\mathcal{L}_v(\phi_v, f_v)|^2} =  \mathcal{W}_{G_, v}(f_v)
\end{equation}
holds with respect to the specified local measures.
\end{proposition}
\noindent
%
%

Let us define a hermitian inner product $\mathcal{B}_{\omega_v}$
on 
$\mathcal{S}(Z_+(F_v))$ by 
\[
\mathcal{B}_{\omega_v}(\phi, \phi^\prime) = \int_{Z_+(F_v)} \phi(x) 
\,\overline{\phi^\prime}(x) \,dx
\quad \text{for} \quad \phi, \phi^\prime \in \mathcal{S}(Z_+(F_v)).
\]
Here on $Z_+(F_v) \simeq (F_v)^{12}$,
we take the product measure of the one on $F_v$.
Then we consider the integral
\begin{align}
\label{theta local PI}
&Z^{\flat}(f, f^\prime; \phi, \phi^\prime) = \int_{G^1(F_v)} \langle \pi_v^\circ(g) f, f^\prime \rangle_v \,
 \mathcal{B}_{\omega_v}(\omega_{\psi}(g)\phi, \phi^\prime) \, dg
 \\
 = &\int_{G^1(F_v)} \int_{Z_+(F_v)} \langle \pi_v^\circ(g) f, f^\prime \rangle_v \,
\left(\omega_{\psi_v}(g,1)\phi \right)(z) \overline{\phi^\prime(z)} \, dx \,dg
\quad\text{for $f, f^\prime \in V_{\pi_v^\circ}$}.
\notag
\end{align}
The integral \eqref{theta local PI} converges absolutely by Yamana~\cite[Lemma~7.2]{Yam}.
As in Gan and Ichino~\cite[16.5]{GI1}, we may define a
$\mathrm{GSO}_{4,2}(F_v)$-invariant hermitian inner product $\mathcal{B}_{\Sigma_v} : V_{\Sigma_v} \times V_{\Sigma_v} \rightarrow \mC$ by 
\[
\mathcal{B}_\Sigma(\theta(\phi \otimes f), \theta(\phi^\prime \otimes f^\prime))
:= Z^{\flat}(f, f^\prime; \phi, \phi^\prime).
\]
Here
we note that for $h \in \mathrm{SO}_{4,2}(F_v)$, we have 
\[
\mathcal{B}_\Sigma(\Sigma(h)\theta(\phi \otimes f), \theta(\phi^\prime \otimes f^\prime))
=
\mathcal{B}_\Sigma(\theta(\omega_{\psi}(1,h)\phi \otimes f), \theta(\phi^\prime \otimes f^\prime)).
\]
As in the definition of $W_v$, we define
\[
\label{mathcal W psi U}
\mathcal{W}^{\psi_U} (\tilde{\phi}_1, \tilde{\phi}_2):= \int_{U(F_v)}^{st} \mathcal{B}_\Sigma(\Sigma(n)\tilde{\phi}_1, \tilde{\phi}_2) \psi_{U}(n)^{-1} \, dn
\quad\text{for 
$\,\widetilde{\phi}_i \in \Sigma_v\,\,\left(i=1,2\right)$}.
\]
Then by an argument similar 
 to the one in 
 \cite[3.2--3.3]{FM2}, 
 indeed by word for word,
 Proposition~\ref{1st local id} is reduced to the following another local identity,
which is regarded as a local pull-back computation of the
Whittaker periods with respect to  the theta lift.
\begin{proposition}\label{main identity whittaker}
For any $f, f^\prime\in V_{\pi_v^\circ}$ and any $\phi,\phi^\prime\in C_c^\infty\left(Z_+(F_v)\right)$,
we have
\begin{multline}\label{e: weil hermitian13}
\mathcal{W}_{\psi_U}\left(
\theta\left(\phi\otimes f
\right),
\theta\left(\phi^\prime\otimes f^\prime\right)
\right)
=
\int_{N(F_v)\backslash G^1(F_v)}
\int_{N(F_v)\backslash G^1(F_v)}
\\
\mathcal{W}_{G,v}\left(\pi_v^\circ\left(g\right)f,\pi_v^\circ\left(g^\prime\right)f^\prime\right)
\left(\omega_{\psi_v}\left(g,1\right)\phi\right)\left(x_0\right)\,
\overline{\left(\omega_{\psi_v}\left(g^\prime,1\right)\phi^\prime\right)\left(x_0\right)}\,
dg\, dg^\prime.
\end{multline}
\end{proposition}
\begin{Remark}
Since $\{ g \cdot x_0 : g \in G^1(F_v) \}$  is locally closed in $Z_+(F_v)$, 
the mappings
\[
N(F_v)\backslash G^1(F_v) \ni g \mapsto \phi(g^{-1} \cdot x_0)\in
\mathbb C,  
\quad 
N(F_v)\backslash G^1(F_v) \ni g^\prime \mapsto \phi^\prime(g^{-1} \cdot x_0)\in
\mathbb C
\]
are compactly supported, and thus
the right-hand side of \eqref{e: weil hermitian13} converges absolutely for $\phi, \phi^\prime \in C_c^\infty\left(Z_+(F_v)\right)$.
\end{Remark}
\subsubsection{Local pull-back computation}
\label{ss: local pull-back computation}
Here we shall prove Proposition~\ref{main identity whittaker}
and thus complete our proof of Theorem~\ref{gso whittaker}.

Since we work over a fixed place $v$ of $F$,
we shall suppress $v$  from the notation in this subsection,
e.g. $F$ means $F_v$.
Further, 
for any algebraic group $K$ over $F$,
we denote its group of $F$-rational points $K\left(F\right)$
by $K$ for simplicity.

\subsubsection*{The case when $F$ is non-archimedean}
Suppose that $F$ is non-archimedean.
From the definition, the local Whittaker period is equal to
\[
 \int_{U}^{st} \int_{G^1} \int_{Z_+} (\omega_{\psi}(g, n) \phi)(x) \overline{ \phi^\prime(x)} \langle \pi^\circ(g)f, f^\prime \rangle \psi_{U}(n)^{-1} \, dx \, dg \, dn.
\]
Recall that we have defined subgroups $N_0, N_1, N_2$ and $\widetilde{U}$ of $U$ in 
\eqref{d: N_0}, \eqref{d: N_1},
\eqref{d: N_2} and \eqref{e: def of u_tilde}, respectively.
Then because of the absolute convergence of the integral \eqref{theta local PI}, the above local integral can be written as 
%
\begin{multline}
\label{e:1}
\int_{\widetilde{U}}^{st} \int_{N}^{st} \int_{N_1} \int_{N_0} \int_{Z_+}  \int_{G^1}
(\omega_{\psi}(g, u_0 u_1 u_2 \tilde{u}) \phi)(x) \overline{ \phi^\prime(x)} 
\\
\times
\langle \pi^\circ(g)f, f^\prime \rangle \psi_{U}(u_2 \tilde{u})^{-1} \, dx \, dg \, du_0\, du_1 \, du_2\, d\tilde{u}.
\end{multline}

Let us define $Z_{+, \circ} := \left\{ (a_1, a_2 ; 0,0) \in Z_+ : \text{$a_1$ and $a_2$ are linearly independent}\right\}$.
Then since
$Z_{+, \circ} \oplus (X_+ \otimes Y_0)$ is open and dense in $Z_+$, we have
\[
\int_{Z_+} \Phi(z) \, dz = \int_{Z_{+, \circ}} \int_{X_+ \otimes Y_0} \Phi(z_1 +z_2) \, dz_2 \, dz_1
\]
for any $\Phi \in L^1(Z_+)$. 
We consider a map $p : Z_{+, \circ} \rightarrow F$ defined by $p((a_1, a_2; 0, 0)) = \langle a_1, a_2\rangle$.
This is clearly surjective.
For each $t \in F$, we fix $x_t \in Z_{+, \circ}$ such that $p(x_t) = t$.
Then by Witt's theorem, the fiber $p^{-1}(x_t)$ of $x_t :=(a_1^t, a_2^t; 0, 0)$ is given by 
\[
p^{-1}(x_t) = \left\{ \gamma \cdot x_t :=(\gamma a_1^t, \gamma a_2^t: 0, 0) : \gamma \in G^1 \right\}.
\]
We may identify this space with $G^1 \slash R_t$ as a $G^1$-homogeneous space.
Here $R_t$ denotes the stabilizer of $x_t$ in $G^1$.
From this observation, the following lemma readily follows (cf. \cite[Lemma~3]{FM2}).
\begin{lemma}
\label{meas decomp 1}
For each $x_t \in Z_{+, \circ}$, there exists a Haar measure $dr_t$ on $R_t$ such that 
\[
\int_{Z_+}\Phi(z) \, dz = \int_{F} \int_{R_t \backslash G^1} \int_{X_+ \otimes Y_0} \Phi(g^{-1} \cdot x_t+z) \, dz \, dg_t \, dt.
\]
Here $dg_t$ denotes the quotient measure $dr_t \backslash dg$ on $R_t \backslash G^1$.
\end{lemma}
Further, we note that the following lemma, which is proved by an argument similar to the one for 
\cite[Lemma~3.20]{Liu2}.
(cf. \cite[Lemma~3]{FM2}).
\begin{lemma}
\label{lem L1}
For $\phi_1,\phi_2\in C_c^\infty\left(Z_+\right)$
and $f_1,f_2\in V_{\pi^\circ}$, let 
\[
\mathcal G_{\phi_1,\phi_2,f_1,f_2}
\left(t\right)=\int_{G^1}
\int_{R_t \backslash G^1}
\phi_1\left((g g^\prime)^{-1} \cdot x_t\right)\,
\phi_2\left(g^{-1} \cdot x_t\right)\,
\langle\pi^\circ\left(g^\prime\right)f_1,f_2\rangle\,
dg\, dg^\prime
\]
for $t \in F$.
Then the integral is absolutely convergent and is 
locally constant.
\end{lemma}
\begin{Remark}
\label{archi lem L1}
When $F$ is archimedean, by an 
argument similar to the one for \cite[Proposition 3.22]{Liu2},
we see that this integral is absolutely convergent  and 
is a continuous function on $F$ not only for $C_c^\infty\left(Z_+\right)$ but also for $\mathcal{S}(Z_+)$.
\end{Remark}
By Lemma~\ref{meas decomp 1}, the integral
\eqref{e:1} can be written as 
\begin{multline*}
 \int_{N_0} \int_{F} \int_{R_t \backslash G^1} \int_{X_+ \otimes Y_0} \int_{G^1} 
 (\omega_{\psi}(g, u_0 h) \phi)(\gamma^{-1} \cdot x_t+z) \overline{ \phi^\prime(\gamma^{-1} \cdot x_t+z)} 
 \\
 \times\langle \pi^\circ(g)f, f^\prime \rangle 
  \, dg \, dz \, d\gamma_t \, dt   \, du_0.
\end{multline*}
Moreover, by the computation in \cite[Section~3.1]{Mo}, we have 
\[
 (\omega_{\psi}(g, u_0(x) h) \phi)(\gamma^{-1} \cdot x_t+z) = \psi(-xt)  \phi(\gamma^{-1} \cdot x_t+z).
\]
Then because of Lemma~\ref{lem L1}, we may apply the Fourier inversion with respect to $x$ and $t$, and thus the above integral is equal to
\begin{multline}
\label{FI 1}
\int_{R_0 \backslash G^1} \int_{X_+ \otimes Y_0} \int_{G^1} (\omega_{\psi}(g, h) \phi)(\gamma^{-1} \cdot x_0+z) \overline{ \phi^\prime(\gamma^{-1} \cdot x_0+z)} 
\\
\times
\langle \pi^\circ(g)f, f^\prime \rangle
  \, dg \, dz \, d\gamma_0 \, dt  \, du_0
\\
=
\int_{R_0 \backslash G^1} \int_{X_+ \otimes Y_0} \int_{G^1} (\omega_{\psi}(\gamma g, h) \phi)(x_0+z) 
(\overline{\omega_{\psi}(\gamma, 1) \phi^\prime)(x_0+z)}
\\
\times \langle \pi^\circ(g)f, f^\prime \rangle
  \, dg \, dz \, d\gamma_0. 
\end{multline}
The support of $\phi^\prime(\gamma^{-1} \cdot x_0+z)$ as a function of $X_+ \otimes Y_0$ is compact
since $\phi^\prime \in C_c^\infty(Z_+)$.
Therefore this integral converges absolutely and is equal to
\begin{multline*}
\int_{X_+ \otimes Y_0} \int_{R_0 \backslash G^1}  \int_{G^1} (\omega_{\psi}(\gamma g, h)\phi)(x_0+z) 
(\overline{\omega_{\psi}(\gamma, 1)\phi^\prime)(x_0+z)} 
\\
\times
\langle \pi^\circ(g)f, f^\prime \rangle
  \, dg  \, d\gamma_0 \, dz.
\end{multline*}
Now, let us take $(x_{-2}, x_{-1}: 0, 0)$ as $x_0$. Then we have
\[
R_0 = N.
\]
Let us define a map $q : X_+ \otimes Y_0 \rightarrow \mathrm{Mat}_{2 \times 2}$ by 
\[
q(b_1 \otimes e_1+b_2 \otimes e_2) = \begin{pmatrix} \langle x_{-2}, b_1 \rangle & \langle x_{-2}, b_2 \rangle \\ \langle x_{-1}, b_1 \rangle & \langle x_{-1}, b_2 \rangle  \end{pmatrix}
\]
with $b_i \in X_+$. Clearly this map is bijective. Hence, there exists a measure $dT$ on $\mathrm{Mat}_{2 \times 2}$ such that we have 
\[
\int_{X_+ \otimes Y_0} \Phi(x_{-2}, x_{-1}: z) \, dz = \int_{\mathrm{Mat}_{2 \times 2}} \Phi(x_{-2}, x_{-1}: x_T) \, dT
\]
with $x_T = q^{-1}(T)$.
Here
we note that the measure $dz$ on $X_+ \otimes Y_0$
is taken to be the Tamagawa measure and hence
we have the Fourier inversion
\[
\int_{\mathrm{Mat}_{2 \times 2}} \int_{\mathrm{Mat}_{2 \times 2}} \Phi(T) \psi \left(\mathrm{tr} \left(T S_0 T^\prime \right) \right) \, dT \, dT^\prime = \Phi(0)
\]
with the above Haar measures $dT, dT^\prime$ on $\mathrm{Mat}_{2 \times 2}$ if the integral converges.
Thus we have
\begin{multline*}
\int_{N_2}^{st} \int_{N_1}
\int_{X_+ \otimes Y_0} \int_{N \backslash G^1}  \int_{G^1} (\omega_{\psi}(\gamma g, u_1 u_2 h)\phi)(x_0+z) 
(\overline{\omega_{\psi}(\gamma, 1)\phi^\prime)(x_0+z)} 
\\
\times
\langle \pi^\circ(g)f, f^\prime \rangle
  \, dg \, d\gamma_0  \, dz  \, du_1 \, du_2
  \\
=
\int_{N}^{st} \int_{N_1}
\int_{\mathrm{Mat}_{2 \times 2}} \int_{N \backslash G^1}  \int_{G^1} (\omega_{\psi}(\gamma g, u_1 u_2 h)\phi)(x_0+x_T) 
(\overline{\omega_{\psi}(\gamma, 1)\phi^\prime)(x_0+x_T)} 
\\
\times \langle \pi^\circ(g)f, f^\prime \rangle
  \, dg \, d\gamma_0  \, dT  \, du_1 \, du_2.
\end{multline*}
Moreover, similarly to the global computation in \cite[Section~3.1]{Mo}, we may write this integral as
\begin{multline}
\label{e:pre 1}
\int_{N_2}^{st} \int_{N_1}
\int_{\mathrm{Mat}_{2 \times 2}} \int_{N \backslash G^1}  \int_{G^1} 
\psi \left(\mathrm{tr} \left( \begin{pmatrix}s_1&s_2\\ t_1 &t_2 \end{pmatrix}S_0 (x_T-x_{T_0}) \right) \right)
\\ \times
(\omega_{\psi}(\gamma g, h)\phi)(x_0+x_T) 
(\overline{\omega_{\psi}(\gamma, 1)\phi^\prime)(x_0+x_T)} \langle \pi^\circ(g)f, f^\prime \rangle
  \, dg \, d\gamma_0 \, dT  \, du_1 \, du_2
\end{multline}
where we write $u_1 = u_1(s_1, t_1)$ and $u_2 = u_2(s_2, t_2)$, and we put
$
T_0 = \begin{pmatrix}0&0\\0&1 \end{pmatrix}$.
By an  argument similar to  the proof to show \eqref{FI 1}, we may apply the Fourier inversion to this integral, 
and we see that this is equal to
\[
\int_{N_H \backslash G^1}  \int_{G^1} 
(\omega_{\psi}(\gamma g, h) \phi)(x_0+x_{T_0}) 
(\overline{\omega_{\psi}(\gamma, 1) \phi^\prime)(x_0+x_{T_0})} \langle \pi^\circ(g)f, f^\prime \rangle
  \, dg \, d\gamma_0.
\]
Now we note that from the argument to obtain \eqref{FI 1}, this integral converges absolutely.
Then by telescoping the $G^1$-integration, we obtain
\begin{multline*}
\int_{N \backslash G^1}  \int_{N \backslash G^1}  \int_{N}
(\omega_{\psi}( r g, h) \phi)(x_0+x_{T_0}) 
(\overline{\omega_{\psi}(\gamma, 1)\phi^\prime)(x_0+x_{T_0})} 
\\
\times
\langle \pi^\circ(r g)f, \pi^\circ(\gamma)f^\prime \rangle
\, dr  \, dg \, d\gamma_0.
\end{multline*}
Put $z_0 = x_0+x_{T_0} = (x_{-2}, x_{-1}, 0, x_2)$. 
Recall that from the computation in \cite[Section~3.1]{Mo}, we have
\begin{equation}
\label{da22}
\omega_{\psi}(v(A)g, \tilde{u}(b) h) \phi(z_0) = \psi(-d a_{22})\omega_{\psi}(g, \tilde{u}(b) h) \phi(z_0)
\end{equation}
when we write $A =\begin{pmatrix}a_{11}&a_{12}\\a_{21}&a_{22} \end{pmatrix} $,
and we have
\begin{equation}
\label{z0 w(b)}
z_0(1, \tilde{u}(b)) = z_0(w(b), 1).
\end{equation}
Therefore, $\mathcal{W}_{\psi_U}\left(
\theta\left(\phi\otimes f
\right),
\theta\left(\phi^\prime\otimes f^\prime\right)
\right)
$ is equal to
\begin{multline*}
\int_{F}^{st}
\int_{N \backslash G^1}  \int_{N \backslash G^1} \int_{N} \psi(-b)
(\omega_{\psi}(r g, \tilde{u}(b)) \phi)(z_0) 
(\overline{\omega_{\psi}(\gamma, 1) \phi^\prime)(z_0)}
\\
\times
 \langle \pi^\circ(rg)f, \pi^\circ(\gamma)f^\prime \rangle
 \, dr \, dg  \, d\gamma_0 \, db
\\
=
    \int_{F}^{st}
\int_{N \backslash G^1}  \int_{N \backslash G^1} \int_{\mathrm{Sym}^2}  \psi(-b-da_{22})
(\omega_{\psi}(w(b)g, 1)\phi)(z_0) 
\\
\times
(\overline{\omega_{\psi}(\gamma, 1)\phi^\prime)(z_0)} \langle \pi^\circ(v(A)g)f, \pi^\circ(\gamma)f^\prime \rangle
\, dA  \, dg  \, d\gamma_0 \, db.
\end{multline*}
By an  argument similar to the one in \cite{FM2}
showing that
 \cite[(3.30)]{FM2} is equal to $\alpha(\pi(g) \phi, \pi(h)\phi^\prime)$
 there,
indeed, by word for word, we see that this integral is equal to
\begin{multline*}
\int_{N \backslash G^1}  \int_{N \backslash G^1} \int_{U_G}^{st}  \psi_{U_{G}}^{-1}(n)
(\omega_{\psi}(g, 1)\phi)(z_0) 
(\overline{\omega_{\psi}(\gamma, 1)\phi^\prime)(z_0)}
\\
\times \langle \pi^\circ(ng)f, \pi^\circ(\gamma)f^\prime \rangle
\, dn  \, dg \, d\gamma_0. 
\end{multline*}
Thus Proposition~\ref{main identity whittaker} in the non-archimedean case is proved.
%
%
%
%
\subsubsection*{The case when $F$ is archimedean}
Suppose that $F$ is archimedean. Recall that 
\[
\mathcal{W}^{\psi_U}(\tilde{\phi}_1, \tilde{\phi}_2)=
\widehat{\mathcal{W}_{\tilde{\phi}_1, \tilde{\phi}_2}}\left(
\psi_U\right)
\quad \text{for $\tilde{\phi}_i \in \Sigma^\infty\,\,\left(i=1,2\right)$},
\]
where we set
\[
\mathcal{W}_{\tilde{\phi}_1, \tilde{\phi}_2}(n) = \int_{U_{-\infty}} \mathcal{B}_{\Sigma}(\Sigma(n u)\tilde{\phi}_1, \tilde{\phi}_2) \psi_{U}^{-1}(n u) \, du
\quad\text{for $n\in U$},
\]
which converges absolutely and gives a  tempered distribution on 
$U \slash U_{-\infty}$ by \cite[Corollary~3.13]{Liu2}. 
Let us define $U^\prime = N_0 N_1 N_2$. Then $U^\prime_{-\infty} = U_{-\infty}$.
Moreover, for any $\tilde{u} \in \widetilde{U}$ and $u^\prime \in U^\prime$, we have $\tilde{u} u^\prime \tilde{u}^{-1} (u^\prime)^{-1} \in U^\prime_{-\infty}$
and we obtain $\mathcal{W}_{\tilde{\phi}_1, \tilde{\phi}_2}(\tilde{u} u^\prime) = \mathcal{W}_{\tilde{\phi}_1, \tilde{\phi}_2}(u^\prime \tilde{u})$.
Hence, we may regard it as a tempered distribution on $\tilde{U} \times \left(U^\prime \slash U_{-\infty}^\prime \right)$.
Then for a tempered distribution $I$ on $\tilde{U} \times \left(U^\prime \slash U_{-\infty}^\prime \right)$, 
we define partial Fourier transforms $\widehat{I^j}$ of $I$ 
for $j=1,2$ by 
\[
\langle I, \widehat{f_1} \otimes f_2 \rangle = \langle  \widehat{I^1}, f_1 \otimes f_2 \rangle 
\quad
\text{and}
\quad
\langle I, f_1 \otimes \widehat{f_2} \rangle = \langle  \widehat{I^2}, f_1 \otimes f_2 \rangle 
\]
where $f_1 \in \mathcal{S} \left(  \tilde{U} \right)$
and $f_2 \in \mathcal{S} \left( U^\prime \slash U_{-\infty}^\prime \right)$,
respectively.
Then we have 
\[
\widehat{ \widehat{I}^2}^1(\psi_U) = \widehat{ \widehat{I}^1}^2(\psi_U)
 = \widehat{I}(\psi_U).
\]
From the definition of $\mathcal{B}_\Sigma$, we have
\begin{multline*}
\mathcal{W}_{\theta(\phi \otimes f), \theta(\phi^\prime \otimes f^\prime)}(n)
=  \int_{U_{-\infty}} \int_{G^1} \int_{Z_+}
(\omega_{\psi}(g, nu) \phi)(x) \overline{ \phi^\prime(x)} 
\\
\times
\langle \pi^\circ(g)f, f^\prime \rangle
 \psi_{U}^{-1}(n u) \, dx \, dg \, du
 \\
 =
  \int_{U_{-\infty} \slash N_0} \int_{N_0} \int_{G^1} \int_{Z_+}
(\omega_{\psi}(g, nu_0u) \phi)(x) \overline{ \phi^\prime(x)} 
\\
\times
\langle \pi^\circ(g)f, f^\prime \rangle
 \psi_{U}^{-1}(n u) \, dx \, dg \, du_0 \, du,
\end{multline*}
for $\phi, \phi^\prime \in \mathcal{S}(Z_+)$ and $f, f^\prime \in V_{\pi^\circ}^\infty$.
Clearly, Lemma~\ref{meas decomp 1} holds in the archimedean case also. Then as in \eqref{FI 1}, because of Remark~\ref{archi lem L1}
and the Fourier inversion, the above integral is equal to
\begin{multline*}
 \int_{U_{-\infty} \slash N_0} \int_{N \backslash G^1} \int_{X_+ \otimes Y_0} \int_{G^1} (\omega_{\psi}(\gamma g, nu) \phi)(x_0+z) 
(\overline{\omega_{\psi}(\gamma, 1) \phi^\prime)(x_0+z)} 
\\
\times
\langle \pi^\circ(g)f, f^\prime \rangle
  \, dg \, dz \, d\gamma_0 \, du.
\end{multline*}
As \eqref{FI 1}, this integral converges absolutely.
Let us denote this integral by $J_{\phi, \phi^\prime, f, f^\prime}(n)$. Then  from the definition,
\[
\widehat{J_{\phi, \phi^\prime, f, f^\prime}} = \widehat{\mathcal{W}_{\theta(\phi \otimes f), \theta(\phi^\prime \otimes f^\prime)}}.
\]
Again,
 from the definition, for $\varphi \in \mathcal{S}(U^\prime \slash U^\prime_{-\infty})$, we have
\begin{multline*}
(\widehat{J_{\phi, \phi^\prime, f, f^\prime}}^2, \psi_U \cdot \varphi) =(J_{\phi, \phi^\prime, f, f^\prime}, \widehat{\psi_U \cdot \varphi})
= \int_{U^\prime \slash U^\prime_{-\infty}} \int_{U_{-\infty}^\prime \slash N_0} \int_{N \backslash G^1} \int_{X_+ \otimes Y_0} \int_{G^1} 
\\
\times
(\omega_{\psi}(\gamma g, nu) \phi)(x_0+z) 
(\overline{\omega_{\psi}(\gamma, 1) \phi^\prime)(x_0+z)}
\\
\times
 \langle \pi^\circ(g)f, f^\prime \rangle \widehat{\varphi}(n) \psi_U^{-1}(n)
  \, dg \, dz \, d\gamma_0 \, du \, dn.
\end{multline*}
By  a computation similar to the one to obtain \eqref{e:pre 1}, this integral is equal to
\begin{multline*}
\int_{N_1} \int_{N_2} \int_{N \backslash G^1} \int_{X_+ \otimes Y_0} \int_{G^1} 
(\omega_{\psi}(\gamma g, u_1 u_2u) \phi)(x_0+z) 
(\overline{\omega_{\psi}(\gamma, 1) \phi^\prime)(x_0+z)}
\\
\times \langle \pi^\circ(g)f, f^\prime \rangle 
\widehat{\varphi}(u_1 u_2) \psi_U^{-1}(u_1 u_2)
  \, dg \, dz \, d\gamma_0 \, du \, du_1 \, du_2
  \\
  =
\int_{N_1} \int_{N_2}
\int_{\mathrm{Mat}_{2 \times 2}} \int_{N \backslash G^1}  \int_{G^1} 
\psi \left(\mathrm{tr} \left( \begin{pmatrix}s_1&s_2\\ t_1 &t_2 \end{pmatrix}S_0 (x_T-x_{T_0}) \right) \right)
\\
\times
(\omega_{\psi}(\gamma g, h)\phi)(x_0+x_T) 
(\overline{\omega_{\psi}(\gamma, 1)\phi^\prime)(x_0+x_T)}
\\
\times
 \langle \pi^\circ(g)f, f^\prime \rangle \widehat{\varphi}(u_1 u_2)
  \, dg \, d\gamma_0 \, dT  \, du_1 \, du_2.
\end{multline*}
As above, we may apply the Fourier inversion, and thus this is equal to
\begin{multline*}
\widehat{\varphi}(1) \cdot
 \int_{N \backslash G^1} \int_{G^1} 
(\omega_{\psi}(\gamma g,1) \phi)(x_0+x_{T_0}) 
(\overline{\omega_{\psi}(\gamma, 1) \phi^\prime)(x_0+x_{T_0})} 
\\
\times
\langle \pi^\circ(g)f, f^\prime \rangle 
  \, dg  \, d\gamma_0.
\end{multline*}
Hence, 
\begin{multline*}
\widehat{J_{\phi, \phi^\prime, f, f^\prime}}^2(\psi_U)
= \int_{N \backslash G^1} \int_{G^1} 
(\omega_{\psi}(\gamma g,1) \phi)(x_0+x_{T_0}) 
(\overline{\omega_{\psi}(\gamma, 1) \phi^\prime)(x_0+x_{T_0})} 
\\
\times
\langle \pi^\circ(g)f, f^\prime \rangle 
  \, dg  \, d\gamma_0.
\end{multline*}
Here, we note that by Remark~\ref{archi lem L1}, this integral converges absolutely.
Then this identity shows that we have
\begin{multline}
\label{J12}
\widehat{\widehat{J_{\phi, \phi^\prime, f, f^\prime}}^2}^1(\varphi) =
\int_{\tilde{U}} \int_{N \backslash G^1} \int_{G^1} 
(\omega_{\psi}(\gamma g,b) \phi)(x_0+x_{T_0}) 
(\overline{\omega_{\psi}(\gamma, 1) \phi^\prime)(x_0+x_{T_0})} 
\\
\times
\langle \pi^\circ(g)f, f^\prime \rangle  \varphi(b)
  \, dg  \, d\gamma_0 \, db
\end{multline}
for $\varphi \in \mathcal{S}(\tilde{U})$.
As in the non-archimedean case, by \eqref{da22} and \eqref{z0 w(b)}, we may easily show that this is equal to 
\begin{multline*}
\int_{N \backslash 
G^1}  \int_{N \backslash G^1} \int_{N} \int_{F} \psi_{U_{G}}^{-1}(v(x) n)
(\omega_{\psi}(g, 1)\phi)(z_0) 
(\overline{\omega_{\psi}(\gamma, 1)\phi^\prime)(z_0)}
\\
\times
 \langle \pi^\circ(v(x)ng)f, \pi^\circ(\gamma)f^\prime \rangle
\varphi(\tilde{u}(x))
 \,dx \, dn  \, dg \, d\gamma_0 
\end{multline*}
since the integral in \eqref{J12} converges absolutely. Thus 
Proposition~\ref{main identity whittaker}
is proved
in the archimedean case also.
%
%
%
%
%
%
%
%
%
%
%
%
\section{Proof of Theorem~\ref{ref ggp}}
In this section, we complete our proof of Theorem~\ref{ref ggp}.
Let $(\pi, V_\pi)$ be an irreducible cuspidal tempered automorphic representation of $G_D(\mA)$ with  a  trivial central character.
Throughout this section, we suppose that $\pi$ is neither of type I-A nor type I-B.
When $\pi$ is one of these types, our theorem is already
proved in \cite[Theorem~7.5]{Co}.

The case when $B_{\xi, \Lambda,\psi} \not \equiv 0$ on $V_\pi$
is treated in \ref{s: pf main thm 1}
and the case when $B_{\xi, 
\Lambda, \psi} \equiv 0$ on $V_\pi$
is treated in \ref{s: pf main thm 3}, respectively.
%
%
%
%
%
%
%
%
%
%
%
%
\subsection{Proof of Theorem~\ref{ref ggp}
when  $B_{\xi, \Lambda,\psi} \not\equiv 0$}
\label{s: pf main thm 1}
\subsubsection{Reduction to a local identity}
Suppose that $B_{\xi, \Lambda,\psi} \not \equiv 0$ on $V_\pi$.
Let $(\sigma, V_\sigma)$ denote the 
theta lift of $\pi$ to $\mathrm{GSU}_{3,D}(\mA)$, which is an irreducible cuspidal automorphic representation.
As in the proof of Theorem~\ref{gso whittaker}, our theorem may be reduced to a certain local identity. 
Let us set some notation to explain our local identity.

As in Section~\ref{s:RI H gso} and Section~\ref{s:RI HD gsu}, we fix the
Petersson inner product $(\,,\,)$ on $V_{\pi}$
and the local hermitian pairing $(\,,\,)_v$ on $\pi_v$.
As in \eqref{zD def}, we define the maximal isotropic subspaces $Z_{D, \pm}$.
Let 
\[
\theta_{D, v}:\mathcal S\left(Z_{D, +}\left(F_v\right)\right)\otimes V_{\pi_v}\to
V_{\sigma_v}
\]
be the $G_D(F_v)^+ \times \mathrm{GSU}_{3, D}\left(F_v\right)$-equivariant
linear map, which is unique up to multiplication by a scalar.
As in Section~\ref{s:Explicit formulas whittaker}, 
let us adjust $\left\{\theta_{D, v}\right\}_v$ so that 
\[
\theta_{D, \psi}\left(\phi^\prime; f^\prime\right)=\otimes_v\,
\theta_{D, v}\left(\phi_v^\prime\otimes f_v^\prime\right)
\]
for $f^\prime=\otimes_v\,f_v^\prime\in V_{\pi}$
and $\phi^\prime=\otimes_v\, \phi_v^\prime
\in\mathcal S\left(Z_{D, +}\left(\mA\right)\right)$.
Let us choose $X \in D^\times(F)$ so that $S_{X} =\xi$. Then by Proposition~\ref{pullback Bessel gsp innerD}, we have 
\begin{equation}
\label{pullback main decomp}
\mathcal{B}_{X, \Lambda^{-1}}(\theta(f: \phi)) =B_{\xi, \Lambda}(f) \cdot \prod_v \mathcal{K}_v(f_v; \phi_v)
\end{equation}
where $f = \otimes f_v \in V_{\pi_D}$ and $\phi = \otimes \phi_v \in \mathcal{S}(Z_{D, +}(\mA))$,
and we define 
\[
\mathcal{K}_v(f_v ;\phi_v) = \int_{N_D(F_v) \backslash G_D^1(F_v)} \alpha_{\Lambda_v, \psi_{\xi, v}}(\pi_v(g)f_v) \phi_v(g^{-1} \cdot v_{D, X}) \, dg.
\]
Here, we take the measure $dh_{v}$ on $G_D^1(F_v)$ defined by the
gauge form, the measure $dn_v$ on $N_{G_D}(F_v)$ defined in \ref{ss: measures}
under the identification $D(F_v) \simeq F_v^4$
and the measure $dg_{1,v}$ on $N_{G_D}(F_v) \backslash G_D^1(F_v)$ such that $dh_v=dn_v \, dg_{1,v}$.
%
%
%
Then by combining the explicit formula of the Bessel periods on $\sigma$ given in Theorem~\ref{gsu Bessel}, 
the Rallis inner product formulas \eqref{RIPF H GSO} and Proposition~\ref{RIPF HD},
 Lemma~\ref{pi sigma identity}
and Lemma~\ref{compo number}, and the above pull-back formula \eqref{pullback main decomp},
we see that Theorem~\ref{ref ggp} is reduced to the following local identity.
\begin{proposition}
\label{main prp bessel}
Let $v$ be an arbitrary place of $F$.
For a given $f_v \in V_{\pi_v}$ satisfying $\alpha_{\xi, \Lambda, v}\left(f_v\right)\ne 0$, 
there exists $\phi_v \in \mathcal{S}(Z_{D, +}\left(F_v\right))$
such that the local integral $\mathcal{K}_v\left(f_v;\phi_v\right)$
converges absolutely, $\mathcal{K}_v\left(f_v;\phi_v\right) \ne 0$
 and the  equality
 \[
\frac{Z_v(\phi_v, f_v, \pi_v) \,  \alpha_{\Lambda_v^{-1}, \psi_{X, v}}\left(\theta(\phi_v \otimes f_v\right)}{|\mathcal{K}_v(f_v;\phi_v)|^2}
= \frac{\alpha_{\Lambda_v, \psi_{\xi, v}}(f_v)}{(f_v, f_v)_v}
\]
holds.
\end{proposition}
\begin{Remark}
In Corollary~\ref{cor nonzero bessel}, the existence of $f_v$ with
$\alpha_{\Lambda_v, \psi_{\xi, v}}\left(f_v\right)\ne 0$ is shown.
\end{Remark}
Let us define hermitian inner product on $\mathcal{S}(Z_{D, +}(F_v))$ by 
\[
\mathcal{B}_{\omega_v, D}(\phi, \phi^\prime) = \int_{Z_{D, +}(F_v)} \phi(x) \overline{\phi^\prime}(x) \,dx
\quad \text{for} \quad \phi, \phi^\prime \in \mathcal{S}(Z_{D, +}(F_v)).
\]
Then we consider the integral
\[
Z^{\bullet}(f, f^\prime; \phi, \phi^\prime) = \int_{G^1(F_v)} \langle \pi_v(g) f, f^\prime \rangle_v \,
 \mathcal{B}_{\omega_v}(\omega_{\psi}(g)\phi, \phi^\prime) \, dg
\]
for $f, f^\prime \in \pi_v$ and $\phi, \phi^\prime \in \mathcal{S}(Z_{D, +}(F_v))$.
As in Section~\ref{s: pf wh gso}, this converges absolutely and gives
a $\mathrm{GSU}_{3,D}(F_v)$-invariant hermitian inner product 
\[
\mathcal{B}_{\sigma_v} : V_{\sigma_v} \times V_{\sigma_v} \rightarrow \mC
\]
 by 
\[
\mathcal{B}_{\sigma_v}(\theta(\phi \otimes f), \theta(\phi^\prime \otimes f^\prime))
:= Z^{\bullet}(f, f^\prime; \phi, \phi^\prime).
\]
By the Rallis inner product formula \eqref{RIPF H GSO} and Proposition~\ref{RIPF HD}, 
at any place $v$, there exist $f_v, f_v^\prime, \phi, \phi^\prime$ such that $Z^{\bullet}(f, f^\prime; \phi, \phi^\prime) \ne 0$
since $\theta_{\psi, D}(\pi) \ne 0$. Thus, $\mathcal{B}_{\sigma_v} \not\equiv 0$.

For $\widetilde{\phi}_i \in \sigma_v$, we define
\[
\mathcal{A}(\tilde{\phi}_1, \tilde{\phi}_2):= \int_{N_{3, D}(F_v)}^{st} \int_{M_{X}(F_v)} \mathcal{B}_{\sigma_v}(\sigma_v(nt)\tilde{\phi}_1, \tilde{\phi}_2)
\Lambda_{D,v}(t) \psi_{X, D, v}(n)^{-1} \, dt \, dn.
\]
Here, at an archimedean place $v$, a stable integration means 
the Fourier transform as in the definition of $\alpha_{\chi, \psi_N}$.
Then by an
argument similar to the one in  \cite[3.2--3.3]{FM2}, we may reduce Proposition~\ref{main prp bessel} to the following identity.
\begin{proposition}
For any $f, f^\prime\in V_{\pi_v}$ and any $\phi,\phi^\prime\in C_c^\infty\left(Z_{D, +}(F_v)\right)$,
we have
\begin{multline}
\label{e:main local pullback bessel}
\mathcal{A}\left(
\theta\left(\phi\otimes f
\right),
\theta\left(\phi^\prime\otimes f^\prime\right)
\right)
=
\\
\int_{N_D(F_v)\backslash G_D^1(F_v)}
\int_{N_D(F_v)\backslash G_D^1(F_v)}
\alpha_{\Lambda_v, \psi_{\xi, v}}\left(\pi_v\left(h\right)f,\pi_v\left(h^\prime\right)f^\prime\right)
\\
\times
\left(\omega_{\psi_v}\left(h,1\right)\phi\right)\left(x_0\right)\,
\overline{\left(\omega_{\psi_v}\left(h^\prime,1\right)\phi^\prime\right)\left(x_0\right)}\,
dh\, dh^\prime.
\end{multline}
\end{proposition}
Before proceeding to a proof of this proposition, we give some corollaries of this identity.
\begin{corollary}
\label{cor nonzero bessel}
For an arbitrary place $v$ of $F$, 
we have 
$\alpha_{\Lambda_v, \psi_{\xi, v}} \not \equiv 0$ on $\pi_v$.
\end{corollary}
\begin{proof}
Sine $\mathcal{B}_{\sigma_v} \not\equiv 0$, \eqref{e:main local pullback bessel} implies that $\alpha_{\Lambda_v, \psi_{\xi, v}} \not \equiv 0$ on $\pi_v$
 if and only if $\alpha_{\Lambda_v^{-1}, \psi_{X, v}} \not \equiv 0$ on $\sigma_v$.
Moreover, by \cite[Corollary~5.1]{FM2},  $\alpha_{\Lambda_v^{-1}, \psi_{X, v}} \not \equiv 0$ on $\sigma_v$ since the theta lift of $\sigma_v$
to $\mathrm{GU}_{2,2}(F_v)$ is generic.
Thus our claim follows.
\end{proof}
As another corollary, a non-vanishing of local theta lifts follows from a non-vanishing of local periods.
\begin{corollary}
\label{supple lem  RIPF}
Let $k$ be a local field of characteristic zero and $\mathcal{D}$ be a quaternion algebra over $k$.
Let $\tau$ be an irreducible admissible tempered representation of $G_{\mathcal{D}}$ with a trivial central character.
Let $S_{\mathcal{D}} \in \mathcal{D}^1$ and $\chi$ be a character of $T_{\mathcal{D}, S_{\mathcal{D}}}$.
Suppose that $\alpha_{\chi, \psi_{S_{\mathcal{D}}}} \not \equiv 0$ on $\tau$. Then $\mathcal{A} \not \equiv 0$ 
on $\theta_{\psi, \mathcal{D}}(\tau) \times \theta_{\psi, \mathcal{D}}(\tau)$. 
In particular  $\theta_{\psi, \mathcal{D}}(\tau) \ne 0$ and 
$Z^\bullet\left(\phi, \phi^\prime, f, f^\prime\right) \ne 0$ for some $f, f^\prime \in \tau$ and $\phi, \phi^\prime \in \mathcal{S}(Z_{\mathcal{D}, +})$.
\end{corollary}
\begin{Remark}
By \cite[Lemma~8.6, Remark~8.4 (1)]{Yam}, we know that the  existence of such $f, f^\prime, \phi, \phi^\prime$ is 
equivalent to the non-vanishing of the  theta lift of $\tau$ to $\mathrm{GSU}_{3,\mathcal{D}}$ when $k \ne \mR$.
Though the equivalence is not clear when $k=\mR$,
we shall use Corollary~\ref{supple lem  RIPF} to show 
that the local non-vanishing of the theta lifts 
implies the global non-vanishing of the theta lifts
in \ref{s: pf main thm 3}.
\end{Remark}
\begin{proof}
By our assumption,
 the right-hand side of \eqref{e:main local pullback bessel} is not zero for some 
$f, f^\prime, \phi, \phi^\prime$ when $F_v \ne \mR$.
Hence, the left-hand side is not zero, and in particular  
$Z^\bullet\left(\phi, \phi^\prime, f, f^\prime\right) \ne 0$.
\end{proof}
%
%
%
%
%
%
%
%
%
%
%
%
%
%
%
%
\subsubsection{Local pull-back computation}
Here we shall prove the identity
\eqref{e:main local pullback bessel} and  thus
we complete our proof of Theorem~\ref{ref ggp} 
when
$B_{\xi, \psi, \Lambda} \not \equiv 0$.
Here we give a proof of \eqref{e:main local pullback bessel}
only in  the non-archimedean case
since the
archimedean case is similarly proved 
as in the proof of Proposition~\ref{main identity whittaker}.
Our proof is a local analogue of the proof of Proposition~\ref{pullback Bessel gsp}
and  Proposition~\ref{pullback Bessel gsp innerD}.
Moreover we will consider only the case
when $D$ is split since the proof is similar 
and indeed is easier in the non-split case
as in the global computation.
Since the argument in this subsection
is purely local, 
in order to simplify the notation, 
we omit subscripts $v$ and 
we simply write $K(F)$ by $K$ for any algebraic group
$K$ defined over $F=F_v$.

From the definition, we may write the left-hand side of \eqref{e:main local pullback bessel} as
 \[
 \int_{N_{4,2}}^{st} \int_{M_{X}} \int_{G^1} \int_{Z_+}  \langle \pi(g)f, f^\prime \rangle (\omega_{\psi}(g, nt) \phi)(x) \overline{\phi^\prime(x)}
\Lambda(t) \psi_{X}(n)^{-1} \, dx \, dg \, dt \, dn
 \]
 where $X$ is chosen so that $S_X =S$.
 Further as in \eqref{bessel u0 u1 u2}, this is equal to
 \begin{multline*}
\int_{F}^{st} \int_{F^2}^{st} \int_{F^2}^{st} \int_{M_{X}} \int_{G^1} \int_{Z_+}   
(\omega_{\psi}(g, u_0(s) u_1(s_1, t_1) u_2(s_2, t_2) t) \phi)(x) \overline{\phi^\prime(x)}
\\
\times
 \langle \pi(g)f, f^\prime \rangle
\Lambda(t) \psi(x_{21}s_1+x_{22}t_1+x_{11}s_2+x_{12}t_2)^{-1} \, dx \, dg \, dt \, ds_2 \,dt_2 \,ds_1 \, dt_1 \,ds
\end{multline*}
when we write $X = \begin{pmatrix} x_{11}&x_{12}\\ x_{21}&x_{22}\end{pmatrix}$.
For each $r \in F$, we may take $A_r =(a_1^r, a_2^r, 0, 0) \in Z_+$ such that $a_1^r,a_2^r$ are linearly independent and $\langle a_1^r, a_2^r \rangle =r$.
Let us denote by $Q_r$ the stabilizer of $x_r$ in $G^1$.
Then as in the proof of Proposition~\ref{main identity whittaker}, for each $r \in F$, there is 
a Haar measure $dq_r$ of $Q_r$
such that
\[
\int_{Z_+} \Phi(x) \, dx = \int_{F} \int_{Q_r \backslash G^1} \int_{X_+^2} \Phi(h^{-1} \cdot A_r+b) \, db \,dh_r  \,dr
\] 
with $dh_r=dq_r \backslash dh$,
provided that the both sides converge. Then applying the Fourier inversion, because of \eqref{global comp1}, our integral becomes
 \begin{multline*}
 \int_{F^2}^{st} \int_{F^2}^{st} \int_{M_{X} }\int_{G^1}\int_{Q_0 \backslash G^1} \int_{X_+^2}   
 \langle \pi(g)f, f^\prime \rangle
\Lambda(t) \psi(x_{21}s_1+x_{22}t_1+x_{11}s_2+x_{12}t_2)^{-1}
\\
\times
(\omega_{\psi}(hg, u_1(s_1, t_1) u_2(s_2, t_2) t) \phi)(A_0+b) \overline{(\omega_{\psi}(h,1)\phi^\prime)(A_0+b)}
\\ db  \, dh \, dx \, dg \, dt \, ds_2 \,dt_2 \,ds_1 \, dt_1 
\end{multline*}
with $A_0=(x_{-2}, x_{-1}, 0, 0)$.
This is verified by an argument similar to the one 
for  \cite[Lemma~3.20]{Liu2}.
We note that $Q_0 = N$
from the definition. Moreover, as in \cite[Lemma~3.19]{Liu2},
the inner integral $\int_{M_{X}} \int_{G^1}\int_{Q_0 \backslash G^1} \int_{X_+^2} $ converges absolutely, and thus this is equal to
 \begin{multline*}
 \int_{F^2}^{st} \int_{F^2}^{st}   \int_{Q_0 \backslash G^1} \int_{G^1}  \int_{M_{X}}  \int_{X_+^2}  \langle \pi(g)f, f^\prime \rangle 
 \Lambda(t) \psi(x_{21}s_1+x_{22}t_1+x_{11}s_2+x_{12}t_2)^{-1}
 \\
\times
(\omega_{\psi}(hg, u_1(s_1, t_1) u_2(s_2, t_2) t) \phi)(A_0+b) \overline{(\omega_{\psi}(h,1)\phi^\prime)(A_0+b)}
\\
 \\ db  \, dh \, dx \, dg \, dt \, ds_2 \,dt_2 \,ds_1 \, dt_1 .
\end{multline*}
From the proof of Lemma~\ref{pull comp 2}, this integral is equal to
 \begin{multline}
 \label{pre F4 X+}
 \int_{F^2}^{st} \int_{F^2}^{st}  \int_{Q_0 \backslash G^1} \int_{G^1}  \int_{M_{X}}   \int_{X_+^2}  \langle \pi(g)f, f^\prime \rangle 
(\omega_{\psi}(hg, t) \phi)(A_0+b)
\\
\times \overline{(\omega_{\psi}(h,1)\phi^\prime)(A_0+b)}
\,
\Lambda(t)\, \psi \left( \mathrm{tr} \begin{pmatrix}s_2&t_2\\ s_1&t_1 \end{pmatrix} \left( S_0 \begin{pmatrix} \langle x_{-2}, b_1 \rangle & \langle x_{-2}, b_2 \rangle\\ 
\langle x_{-1}, b_1 \rangle& \langle x_{-1}, b_2 \rangle \end{pmatrix} -X\right) \right)
\\ db  \, dh \, dx \, dg \, dt \, ds_2 \,dt_2 \,ds_1 \, dt_1 .
\end{multline}
Now we claim that we may define the stable integral
 \begin{multline*}
 \int_{F^2}^{st} \int_{F^2}^{st} \int_{X_+^2}  \langle \pi(g)f, f^\prime \rangle 
(\omega_{\psi}(hg, t) \phi)(A_0+b) \overline{(\omega_{\psi}(h,1)\phi^\prime)(A_0+b)}
\\
\times
\Lambda(t) \psi \left( \mathrm{tr} \begin{pmatrix}s_2&t_2\\ s_1&t_1 \end{pmatrix} \left( S_0 \begin{pmatrix} \langle x_{-2}, b_1 \rangle & \langle x_{-2}, b_2 \rangle\\ 
\langle x_{-1}, b_1 \rangle& \langle x_{-1}, b_2 \rangle \end{pmatrix} -X\right) \right)
\, db \, ds_2 \,dt_2 \,ds_1 \, dt_1
\end{multline*}
and we may
 choose a sufficiently large compact open subgroup $F_i$ of $F$ ($1 \leq i \leq 4$)
so that it depends only on $\psi$ and $\int_{F^2}^{st} \int_{F^2}^{st}\dots = \int_{F_1} \int_{F_2}\int_{F_3}\int_{F_4}\cdots$.
This claim easily follows from the following lemma 
in the one dimensional case.
\begin{lemma}\label{l: integral formula}
Let $f$ be a locally constant function on $F$ which is in $L^1(F)$.
Then there exists a compact open subgroup $F_0$ of $F$ such that for any compact open subgroups
$F^\prime$ and $F^{\prime \prime}$
of $F$ containing $F_0$, we have
\begin{equation}\label{e: integral formula}
\int_{F^\prime} \int_{F} f(x) \psi(xy) \, dx \, dy = \int_{F^{\prime \prime}} \int_{F} f(x) \psi(xy) \, dx \, dy.
\end{equation}
\end{lemma}
\begin{proof}
Suppose that $\psi$ is trivial on $F_0 : =\varpi^{m}\mathcal{O}_F$ and not trivial on $\varpi^{m-1}\mathcal{O}_F$.
Put $F^\prime = \varpi^{m^\prime} \mathcal{O}_F$ with $m^{\prime} \leq m$. Then we may write the left-hand side of
\eqref{e: integral formula} as 
\begin{equation}\label{e: integral formula2}
  \int_{F^\prime} \int_{F \setminus \mathcal{O}} f(x) \psi(xy)\, dx \, dy
+\int_{F^\prime}\int_{\mathcal{O}}  f(x) \psi(xy) \, dx \, dy .
\end{equation}
The first integral of 
\eqref{e: integral formula2} 
converges absolutely.
Hence by interchanging the order of integration, it is equal
to
\[
\int_{F \setminus \mathcal{O}} \int_{F^\prime}  f(x) \psi(xy)\, dy \, dx
=\int_{F \setminus \mathcal{O}}f\left(x\right)
 \left(\int_{F^\prime}\psi\left(xy\right)\,dy\right)\, dx=0
\]
since
$y \mapsto \psi(xy)$ is a non-trivial character of $F^\prime$
for each $x \in F\setminus \mathcal{O}$.
As for the second integral of \eqref{e: integral formula2},
we have
\begin{multline*}
\int_{F^\prime}\int_{\mathcal{O}}  f(x) \psi(xy) \, dx \, dy
\\
=\int_{\varpi^{m} \mathcal{O}} \int_{\mathcal{O}}
  f(x) \psi(xy) \, dx \, dy
+\int_{\varpi^{m^\prime} \mathcal{O} \setminus \varpi^{m} \mathcal{O}}f\left(x\right)
\left(\int_{\mathcal{O}}   \psi(xy) \, dy\right) \, dx
\end{multline*}
where the inner integral of the second integral vanishes
as above.
Thus the left hand side of \eqref{e: integral formula}
is equal to
\[
 \int_{\varpi^{m} \mathcal{O}} \int_{\mathcal{O}} f(x) \psi(xy) \, dx \, dy.
\]
Similarly the right-hand side of \eqref{e: integral formula}
becomes as above, and our claim follows.
\end{proof}
By Lemma~\ref{l: integral formula}, we see that \eqref{pre F4 X+} is equal to
\begin{multline*}
 \int_{N \backslash G^1} \int_{G^1}  \int_{M_{X}}   \int_{F^2}^{st} \int_{F^2}^{st}  \int_{X_+^2}  \langle \pi(g)f, f^\prime \rangle 
(\omega_{\psi}(hg, t) \phi)(A_0+b)
\\
\times \overline{(\omega_{\psi}(h,1)\phi^\prime)(A_0+b)}
\,
\Lambda(t)\, \psi \left( \mathrm{tr} \begin{pmatrix}s_2&t_2\\ s_1&t_1 \end{pmatrix} \left( S_0 \begin{pmatrix} \langle x_{-2}, b_1 \rangle & \langle x_{-2}, b_2 \rangle\\ 
\langle x_{-1}, b_1 \rangle& \langle x_{-1}, b_2 \rangle \end{pmatrix} -X\right) \right)
\\ db  \, dh \, dx \, dg \, dt \, ds_2 \,dt_2 \,ds_1 \, dt_1 .
\end{multline*}
Then applying the Fourier inversion,  we get
\begin{multline}
\label{eq 1}
\int_{N \backslash G^1} \int_{G^1}  \int_{M_{X}}  \langle \pi(g)f, f^\prime \rangle 
\\
\times
(\omega_{\psi}(hg, t) \phi)(A_0+B_0) \overline{(\omega_{\psi}(h,1)\phi^\prime)(A_0+B_0)}
\,
\Lambda(t) \, db  \, dh \, dx \, dg \, dt 
\end{multline}
where $B_0 = (0, 0,\frac{x_{21}}{2}x_1+\frac{x_{11}}{2} x_2, -\frac{x_{22}}{2d}x_1-\frac{x_{12}}{2d} x_2)$
and $x_0=A_0+B_0$.
By \cite[Proposition~3.1]{Liu2}, for a sufficiently large compact open subgroup $N_0$ of $N$,
we have 
\[
\int_{M_X} \int^{st}_{N} f(nt) \chi(nt) \,dn \,dt =
\int_{N_0} \int_{M_X}  f(nt) \chi(nt) \,dn \,dt 
\]
and thus we may define 
\[
\int_{N}^{st} \int_{M_X}  f(nt) \chi(nt) \,dn \,dt .
\]
Further, we note a simple fact that we have 
\[
\int_{G} g(h) \,dh = \int_{N \backslash G} \int_{N}^{st} g(nh) \, dn dh
\]
when both sides are defined. Thus \eqref{eq 1} is equal to
\begin{multline*}
\int_{N \backslash G^1} \int_{N \backslash G^1}   \int_{M_{X}} \int_{N}^{st}  \langle \pi(g)f, f^\prime \rangle 
\\
\times
(\omega_{\psi}(hg, t) \phi)(A_0+B_0) \overline{(\omega_{\psi}(h,1)\phi^\prime)(A_0+B_0)}
\Lambda(t) \, db  \, dh \, dx \, dg \, dt .
\end{multline*}
Then the same computation 
as the one to get \eqref{pull-back gsp complete} from \eqref{MX to TS} may be applied to the above integral, and thus 
we see that our integral is equal to 
\begin{multline*}
\int_{N\backslash G^1}
\int_{N\backslash G^1}
\alpha_{\Lambda, \psi_{S}}\left(\pi_v\left(h\right)f,\pi_v\left(h^\prime\right)f^\prime\right)
\\
\times
\left(\omega_{\psi}\left(h,1\right)\phi\right)\left(x_0\right)\,
\overline{\left(\omega_{\psi}\left(h^\prime,1\right)\phi^\prime\right)\left(x_0\right)}\,
dh\, dh^\prime.
\end{multline*}
Hence the identity \eqref{e:main local pullback bessel} holds
when $B_{\xi, \Lambda,\psi} \not\equiv 0$.
%
%
%
%
%
%
%
%
%
%
%
%
%
%
%
%
\subsection{Proof of Theorem~\ref{ref ggp}
when  $B_{\xi, \Lambda,\psi} \equiv 0$}
\label{s: pf main thm 3}
First we note the following proposition concerning
the  non-vanishing of the $L$-values.
%
\begin{proposition}
Let $\pi$ be an irreducible cuspidal tempered automorphic representation of $G_D(\mA)$ with trivial central character.
If $G_D \simeq G$ and $\pi$ is a theta lift from $\mathrm{GSO}_{3,1}$, then $L(s, \pi, \mathrm{std} \otimes \chi_E)$ has a simple pole at $s=1$.
Otherwise $L(s, \pi, \mathrm{std} \otimes \chi_E)$ is  holomorphic and non-zero at $s=1$.
\end{proposition}
\begin{proof}
Suppose that $G_D \simeq G$, i.e. $D$ is split.
Then  there exists an irreducible cuspidal globally generic automorphic representation $\pi_0$ of $G(\mA)$
such that $\pi$ and $\pi_0$ are nearly equivalent.
Then our claim follows from \cite[Lemma~10.2]{Yam}  and \cite[Theorem~5.1]{Sha0}.

Suppose that $D$ is not split.
Let us take a quadratic extension $E_0$ of $F$ such that $\pi$ has $(E_0, \Lambda_0)$-Bessel period 
for some character $\Lambda_0$ of $\mA_{E_0}^\times \slash E_0^\times$.
Then by Theorem~\ref{ggp SO} (1), we see that there exists
 an irreducible cuspidal tempered automorphic representation $\pi_0$
of $G(\mA)$ such that for a sufficiently large finite set $S$ of places of $F$
containing all archimedean places, $\pi_v, \pi_{0,v}$ are unramified and 
$\mathrm{BC}_{E_0 \slash F}(\pi_v) \simeq \mathrm{BC}_{E_0 \slash F}(\pi_{0,v})$ for $v\not\in S$.
This implies that 
\begin{multline*}
L^S(s, \pi_0, \mathrm{std} \otimes \chi_{E_0} \chi_E)
L^S(s, \pi_0, \mathrm{std} \otimes \chi_E)
\\
=
L^S(s, \pi, \mathrm{std} \otimes  \chi_{E_0} \chi_E)
L^S(s, \pi, \mathrm{std} \otimes \chi_E).
\end{multline*}
From the case when  $G_D \simeq G$, the left-hand side of this identity is not zero at $s=1$,
and thus so is the right-hand side, which possibly has a pole at $s=1$.

Suppose that $L^S(s, \pi, \mathrm{std} \otimes \chi_{E_0 \slash F} \chi_E)$ has a pole at $s=1$.
We may take a quadratic extension $E_1 \subset D$ of $F$ such that $\chi_{E_1} =  \chi_{E_0} \chi_E$.
Then by Yamana~\cite[Lemma~10.2]{Yam}, $\pi$ is a theta lift from $\mathrm{GSU}_{1,D}$,
which is a similitude quaternion unitary group of degree one defined by an element in $E_1$ as in \eqref{d: gu_1,d}.
In this case, $\pi$ is not tempered, and thus it contradicts to our assumption on $\pi$.
Thus, $L^S(s, \pi, \mathrm{std} \otimes \chi_{E_0 \slash F} \chi_E)$ is holomorphic at $s=1$.
Further, by an argument  similar to the one for $L^S(s, \pi, \mathrm{std} \otimes \chi_{E_0 \slash F} \chi_E)$,
we see that  $L^S(s, \pi, \mathrm{std} \otimes \chi_E)$ is holomorphic.
Therefore, it is holomorphic and non-zero at $s=1$.
\end{proof}
%
Suppose that $B_{\xi, \Lambda,\psi} \equiv 0$ on $V_\pi$.
We shall show that the right-hand side of \eqref{e: main identity} is zero.
If $L\left(\frac{1}{2}, \pi \times \mathcal{AI} \left(\Lambda \right)\right) = 0$, then there is nothing to prove.
Hence, we may suppose that $L\left(\frac{1}{2}, \pi \times \mathcal{AI} \left(\Lambda \right)\right) \ne 0$.
Then we shall show that for some place $v$ of $F$, we have
$\alpha_{\Lambda_v, \psi_{\xi, v}}\equiv 0$ on $\pi_v$.

Assume contrary, i.e. 
 $\alpha_{\Lambda_v, \psi_{\xi, v}} \not \equiv 0$ on $\pi_v$ for any $v$.
Let us denote by $\pi_+^{B, \rm{loc}}$ the unique irreducible constituent of $\pi|_{G_D(\mA)^+}$
such that $\alpha_{\Lambda_v, \psi_{\xi, v}} \not \equiv 0$ on $\pi_{+,v}^{B, \mathrm{loc}}$ for any $v$.
From our assumption $\alpha_{\Lambda_v, \psi_{\xi, v}} \not \equiv 0$ on $\pi_v$ and Corollary~\ref{supple lem  RIPF},
we see that $\alpha_{\Lambda_v^{-1}, \psi_{X, v}} \not \equiv 0$ on the theta lift $\theta_{\psi_v, D}(\pi_v)$ of $\pi_v$
to $\mathrm{GSU}_{3,D}(F_v)$ and $Z_v(\phi_v, f_v, \pi) \ne 0$ for some $f_v \in \pi_v$ and $\phi_v \in \mathcal{S}(Z_{D, +}(F_v))$.
Since $\pi^\prime$ is nearly equivalent to $\pi$, we have $L(1, \pi, \mathrm{std} \otimes \chi_E) \ne 0$.
Therefore, the theta lift $\theta_{\psi, D}(\pi_+^{B, \mathrm{loc}})$ of $\pi_+^{B, \mathrm{loc}}$ to 
$\mathrm{GSU}_{3,D}(\mA)$ is non-zero by Yamana~\cite[Theorem~10.3]{Yam},
which states that the non-vanishing of local theta lifts at all places together with  the non-vanishing of the $L$-value
implies the non-vanishing of the global theta lift.
We note that actually in \cite[Theorem~10.3]{Yam},
there is an assumption that $D$ is not split at real places,
which was necessary to ensure that 
 the non-vanishing of the local theta lift implies $Z_v(\phi_v, f_v, \pi) \ne 0$ 
 for some $f_v \in \pi_v$ and $\phi_v \in \mathcal{S}(Z_{D, +}(F_v))$.
 Since the non-vanishing of $Z_v(\phi_v, f_v, \pi)$ for some $f_v$ and $\phi_v$
 is shown in our case by the argument above, 
  we may apply \cite[Theorem~10.3]{Yam} regardless of the assumption.

Recall that from the proof of Theorem~\ref{ggp SO} (1), $\theta_{\psi, D}(\pi_+^{B, \mathrm{loc}})$ is tempered.
Let us regard $\theta_{\psi, D}(\pi_+^{B, \mathrm{loc}})$ as automorphic representations of $\mathrm{GU}_{4, \varepsilon}$.
By the uniqueness of the Bessel model for $\mathrm{GU}_{4, \varepsilon}$ proved in \cite[Proposition~A.1]{FM3}, there uniquely exists
an irreducible constituent $\tau$ of  $\theta_{\psi, D}(\pi_+^{B, \mathrm{loc}})|_{\mathrm{U}(4)}$
such that  $\tau$ has  the local $(X, \Lambda_v^{-1},\psi_v)$-Bessel model at any place $v$.

On the other hand, we note $L\left(1 \slash 2, \tau \times \Lambda^{-1} \right) \ne 0$ since  $L\left(\frac{1}{2}, \pi \times \mathcal{AI} \left(\Lambda \right)\right) \ne 0$.
Then by \cite[Theorem~1.2]{FM3}, there exists an irreducible cuspidal automorphic representation $\tau^\prime$
of $\mathrm{U}(V_0)$ with four dimensional hermitian space $V_0$ over $E$ such that $\tau^\prime$ has 
$(X, \Lambda_v,\psi_v)$-Bessel period.
Then we know that $\tau$ and $\tau^\prime$ have the same $L$-parameter, in particular, $\tau_v \simeq \tau_v^\prime$
when $v$ is split.
At a non-split place $v$, by the uniqueness of an element of the tempered $L$-packet which has the same Bessel period due to Beuzart-Plessis~\cite{BP1,BP2},
we see that $\mathrm{U}(V_0) \simeq \mathrm{U}(J_D)$ and $\tau \simeq \tau^\prime$.
Moreover, by Mok~\cite{Mok}, we have $\tau = \tau^\prime$.
Therefore, $\tau = \tau^\prime$ has $(X, \Lambda^{-1},\psi)$-Bessel period, and 
this implies that $\theta_{\psi, D}(\pi_+^{B, \mathrm{loc}})$ also has $(X, \Lambda^{-1},\psi)$-Bessel period.
Then Proposition~\ref{pullback Bessel gsp}
and \ref{pullback Bessel gsp innerD} show that $\pi$ has 
$(E,\Lambda)$-Bessel period, and this is a contradiction.
Thus,  \eqref{e: main identity} holds when $B_{\xi, \Lambda,\psi} \equiv 0$ on $V_\pi$.
%
%
%
%
%
%
%
%
%
%
%
%
%
%
%
%
\section{Generalized B\"{o}cherer conjecture}
\label{GBC}
In this section we  prove the generalized B\"{o}cherer conjecture.
In fact, we shall prove Theorem~\ref{t: vector valued boecherer}
below,
which is more general than Theorem~\ref{Boecherer:scalar}
stated in the introduction.
%
%
%
\subsection{Temperedness condition}
In order to apply Theorem~\ref{ref ggp} to holomorphic Siegel
cusp  forms of degree two,
we need to verify the  temperedness 
for corresponding automorphic representations.
%
\begin{proposition}
\label{temp prp}
Suppose that $F$ is totally real.
Let $\tau$ be an irreducible cuspidal automorphic representation of $G_D(\mA)$ with a trivial central character
such that $\tau_v$ is a discrete series representation for every real place $v$ of $F$.
Suppose moreover  that $\tau$ is not CAP. Then $\tau$ is tempered.
\end{proposition}
\begin{Remark}
When $D$ is split, i.e. 
$G_D \simeq G$,
Weissauer~\cite{We} proved that $\tau_v$ is tempered at a place $v$
when $\tau_v$ is unramified.
Moreover, when $\tau_v$ is a holomorphic discrete series 
representation at each archimedean place $v$, 
Jorza~\cite{Jo} showed the temperedness at finite places 
not dividing $2$.
\end{Remark}
%
%
%
\begin{proof}
First suppose that $G_D \simeq G$.
Let $\Pi$ denote the functorial lift of $\tau$ to $\mathrm{GL}_4(\mA)$ established by Arthur~\cite{Ar} (see also Cai-Friedberg-Kaplan~\cite{CFK}).

When $\Pi$ is not cuspidal, since $\tau$ is not CAP, 
$\Pi$ is of the form
 $\Pi=\Pi_1 \boxplus \Pi_2$
with irreducible cuspidal automorphic representations $\Pi_i$ of $\mathrm{GL}_2(\mA)$.
Since $\tau_v$ is a discrete series representation for any real place $v$,
$\Pi_{i,v}$ is also a discrete series representation.
Then $\Pi_i$ is tempered by \cite{Bl}
and thus the Langlands parameter of $\Pi_v$ is tempered at all places $v$ of $F$.
Hence $\tau$ is tempered.
%
%
%

Suppose that  $\Pi$ is cuspidal.
Then by Raghuram-Sarnobat~\cite[Theorem~5.6]{RS}, $\Pi_v$ is tempered and cohomological at any real place $v$.
Let us take an imaginary quadratic extension $E$ of $F$ such that the base change lift $\mathrm{BC}(\Pi)$
of $\Pi$ to $\mathrm{GL}_4(\mA_E)$ is cuspidal. 
Note that $\mathrm{BC}(\Pi)$ is cohomological and 
that $\mathrm{BC}(\Pi)^\vee  \simeq \mathrm{BC}(\Pi^\vee)\simeq \mathrm{BC}(\Pi) \simeq \mathrm{BC}(\Pi)^\sigma$.
Then Caraiani~\cite[Theorem~1.2]{Car} shows that $\mathrm{BC}(\Pi)$ is tempered at all finite places.
This implies that $\Pi_v$ is also tempered for any finite place $v$.
Thus $\tau$ is tempered.
%
%
%

Now let us consider the case when  $D$ is not split. 
Since $\tau$ is not CAP, 
by Proposition~\ref{exist gen prp}, 
there exists an irreducible cuspidal automorphic representation
$\tau^\prime$
of $G\left(\mA\right)$ and a quadratic extension $E_0$ of $F$
such that $\tau^\prime$ is $G^{+, E_0}$-locally equivalent
to $\tau$.
Moreover $\tau$ is tempered if and only if $\tau^\prime$
is tempered.
By \cite{LPTZ, Moe, Pau, Pau2}, $\tau^\prime_v$
is a discrete series representation at 
any real place $v$.
Then the temperedness of $\tau^\prime$ follows from the split
case. Hence $\tau$ is also tempered.
\end{proof}
%
%
%
As an application of Proposition~\ref{temp prp},
the following corollary holds.
\begin{corollary}
\label{GRC}
Suppose that $F$ is totally real.
Let $\tau$ be an irreducible cuspidal globally generic automorphic representation of $G(\mA)$
such that $\tau_v$ is a discrete series representation at any real place $v$.
Then $\tau$ is tempered and hence the
 explicit formula \eqref{e:gsp whittaker} for the
 Whittaker periods holds for any non-zero decomposable vector in  $V_\tau$.
\end{corollary}
\begin{proof}
Recall that the functorial lift $\Pi$ of $\tau$ to $\mathrm{GL}_4\left(\mA\right)$ is cuspidal
or an isobaric sum of irreducible cuspidal automorphic representations of $\mathrm{GL}_2$ by \cite{CKPSS}.
In particular $\tau$  is not CAP by Arthur~\cite{Ar}. Then by Proposition~\ref{temp prp}, $\tau$ is tempered and our claim follows from Theorem~\ref{gsp whittaker}.
\end{proof}
%
%
%
%
%
%
%
%
%
%
%
%
%
%
%
%
\subsection{Vector valued Siegel cusp forms and Bessel periods}
Let $\mathfrak H_2$ be the Siegel upper half space of degree two, i.e.
the set of two by two symmetric complex matrices whose imaginary parts are 
positive definite.
Then the group $G\left(\mathbb R\right)^+=\left\{
g\in G\left(\mathbb R\right):\nu\left(g\right)>0\right\}$
acts on $\mathfrak H_2$ by
\[
g\langle Z\rangle=\left(AZ+B\right)\left(CZ+D\right)^{-1}
\quad
\text{for 
$g=\begin{pmatrix}A&B\\C&D\end{pmatrix}\in
G\left(\mathbb R\right)^+$ and $Z\in\mathfrak H_2$}
\]
and the factor of automorphy $J\left(g,Z\right)$ is defined by
\[
J\left(g,Z\right)=CZ+D.
\]

For an integer $N\ge 1$, let
\[
\Gamma_0\left(N\right)=
\left\{
\gamma\in G^1\left(\mathbb Z\right)
:
\gamma=\begin{pmatrix}A&B\\C&D\end{pmatrix},\,
C
\equiv
0
\pmod{N\mathbb{Z}}
\right\}.
\]
%
\subsubsection{Vector valued Siegel cusp forms}
Let $\left(\varrho, V_\varrho\right)$ be 
an algebraic representation of $\gl_2\left(\mathbb C\right)$.
Then a holomorphic mapping
$\varPhi:\mathfrak H_2\to V_\varrho$ is 
a \emph{Siegel cusp form of weight $\varrho$ with respect to 
$\Gamma_0\left(N\right)$}
when $\varPhi$ 
vanishes at the cusps 
and satisfies
\begin{equation}\label{e: transformation}
\varPhi\left(\gamma\langle Z\rangle\right)=
\varrho\left(J\left(\gamma, Z\right)\right)\varPhi\left(Z\right)
\quad
\text{for $\gamma\in \Gamma_0\left(N\right)$ and $Z\in\mathfrak H_2$}.
\end{equation}
We denote by $S_\varrho\left(\Gamma_0\left(N\right)\right)$
the complex vector space of Siegel cusp forms of weight $\varrho$
with respect to $\Gamma_0\left(N\right)$.
Then $\varPhi\in S_\varrho\left(\Gamma_0\left(N\right)\right)$
has a Fourier expansion
\[
\varPhi\left(Z\right)=\sum_{T>0}
a\left(T,\varPhi\right)\,\exp\left[2\pi\sqrt{-1}\,
\mathrm{tr}\left(TZ\right)\right]
\quad\text{where
$Z\in\mathfrak H_2$ and $a\left(T,\Phi\right)\in V_\varrho$}.
\]
Here $T$ runs over positive definite two by two symmetric matrices
which are semi-integral, i.e. $T$ is of the form
$T=\begin{pmatrix}a&b\slash 2\\ b\slash 2&c\end{pmatrix}$,
$a,b,c\in\mathbb Z$.
We note that \eqref{e: transformation} implies
\begin{equation}\label{e: transformation2}
a\left(\varepsilon \,T\,{}^t\varepsilon,\varPhi\right)=
\varrho\left(\varepsilon\right) a\left(T,\varPhi\right)
\quad\text{for $\varepsilon\in\gl_2\left(\mathbb Z\right)$}.
\end{equation}

From now on till the end of this paper, 
we assume $\varrho$ to be \emph{irreducible}.
It is well known that the irreducible algebraic representations
of $\gl_2\left(\mathbb C\right)$ are parametrized by
\begin{equation}\label{e: L}
\mathbb L=\left\{\left(n_1,n_2\right)\in\mathbb Z^2: n_1\ge n_2\right\}.
\end{equation}
Namely the parametrization is given by assigning 
\[
\varrho_\kappa:=\mathrm{Sym}^{n_1-n_2}\otimes {\det}^{n_2}
\quad\text{to $\kappa=\left(n_1,n_2\right)\in\mathbb L$}.
\]

Suppose that $\varrho=\varrho_\kappa$ with 
$\kappa=\left(n+k,k\right)\in\mathbb L$.
Then we realize $\varrho$ concretely by 
taking its space of representation
$V_{\varrho}$ to be $\mathbb C\left[X,Y\right]_{n}$, the space of degree $n$ homogeneous
polynomials of $X$ and $Y$, where the action of $\gl_2\left(\mathbb C\right)$
is given by
\[
\varrho\left(g\right)P\left(X,Y\right)=\left(\det g\right)^k\cdot
P\left(\left(X,Y\right)g\right)
\quad
\text{for
$g\in\gl_2\left(\mathbb C\right)$ and
$P\in\mathbb C\left[X,Y\right]_n$}.
\]
%
Let us define
a bilinear form 
\[
\mathbb C\left[X,Y\right]_n\times
\mathbb C\left[X,Y\right]_n\ni\left(P,Q\right)\mapsto
\left(P,Q\right)_n\in \mathbb C
\]
by
\begin{equation}\label{e: bilinear form}
\left(X^iY^{n-i}, X^j Y^{n-j}\right)_n=
\begin{cases}
\displaystyle{\left(-1\right)^i \begin{pmatrix}n\\ i\end{pmatrix}}
&\text{if $i+j=n$};
\\
0&\text{otherwise.}
\end{cases}
\end{equation}
Then 
we have
\begin{equation}\label{e: equivariance}
\left(\varrho\left(g\right)P,\varrho\left(g\right)Q\right)_n=
\left(\det g\right)^{n+2k}\left(P,Q\right)_n
\quad\text{for $g\in\gl_2\left(\mathbb C\right)$.}
\end{equation}
We define a positive definite hermitian inner product
$\langle\,,\,\rangle_\varrho$
on $V_\varrho$ 
by
\begin{equation}\label{e: def of hermitian}
\langle P,Q\rangle_\varrho:=
\left(P,\varrho\left(w_0\right)\overline{Q}\,\right)_n
\quad\text{where $w_0=\begin{pmatrix}0&1\\-1&0\end{pmatrix}$.}
\end{equation}
Here $\overline{Q}$ denotes the polynomial obtained from $Q$ by taking
the complex conjugates of its coefficients.
Then  \eqref{e: equivariance} implies that
we have
\begin{equation}\label{e: invariant inner product}
\langle\varrho\left(g\right)v,w\rangle_{\varrho}=
\langle v,\varrho\left({}^t\bar{g}\right)w\rangle_{\varrho}
\quad
\text{for 
$g\in\gl_2\left(\mathbb C\right)$ and $v,w\in V_\varrho$}.
\end{equation}
In particular the hermitian inner product $\langle\, ,\,\rangle_\varrho$
is $\mathrm{U}_2\left(\mathbb R\right)$-invariant.
Then for $\varPhi,\varPhi^\prime\in S_{\varrho}\left(\Gamma_0\left(N\right)\right)$, 
we define the Petersson inner product $\langle\varPhi,\varPhi^\prime\rangle_{\varrho}$ by
\begin{equation}\label{e: Petersson}
\langle\varPhi,\varPhi^\prime\rangle_{\varrho}=
\frac{1}{\left[\mathrm{Sp}_2\left(\mathbb Z\right):
\Gamma_0\left(N\right)\right]}
\int_{\Gamma_0\left(N\right)\backslash \mathfrak H_2}
\langle\varPhi\left(Z\right),\varPhi^\prime\left(Z\right)\rangle_{\varrho}\,
\left(\det Y\right)^{k-3}\, dX\, dY
\end{equation}
where $X=\mathrm{Re}\left(Z\right)$ and $Y=\mathrm{Im}\left(Z\right)$.
The space $S_{\varrho}\left(\Gamma_0\left(N\right)\right)$
has a natural orthogonal  decomposition with respect to the Petersson inner 
product 
\[
S_{\varrho}\left(\Gamma_0\left(N\right)\right)=
S_{\varrho}\left(\Gamma_0\left(N\right)\right)^{\mathrm{old}}\oplus
S_{\varrho}\left(\Gamma_0\left(N\right)\right)^{\mathrm{new}}
\]
into the oldspace and the newspace
in the sense of Schmidt~\cite[3.3]{Sch1}.
We note that when $n$ is odd, we have
$S_\varrho\left(\Gamma_0\left(N\right)\right)=\left\{0\right\}$
for $\varrho$ with $\kappa=\left(n+k,k\right)$
by \eqref{e: transformation}
since $-1_4\in\Gamma_0\left(N\right)$.
%
%
%
\subsubsection{Adelization}\label{sss: adelization}
Given $\varPhi\in S_\varrho\left(\Gamma_0\left(N\right)\right)$,
its adelization $\varphi_\varPhi:G\left(\mathbb A\right)\to V_\varrho$
 is defined as follows
(cf. \cite[3.1]{Sa}, \cite[3.2]{Sch1}).
For each prime number $p$, let us define a compact open subgroup
$P_{1,p}\left(N\right)$ of $G\left(\mathbb Q_p\right)$ by
\[
P_{1,p}\left(N\right):=
\left\{ g\in G\left(\mathbb Z_p\right):
g=\begin{pmatrix}A&B\\C&D\end{pmatrix},\,
C\equiv 0\pmod{N\,\mathbb Z_p}
\right\}.
\]
Then we  define a mapping $\varphi_{\varPhi}:G\left(\mathbb A\right)
\to V_{\varrho}$ by 
\begin{equation}\label{e: vector valued}
\varphi_{\varPhi}\left(g\right)=
\nu\left(g_\infty\right)^{k+r}
\varrho\left(J\left(g_\infty, \sqrt{-1}\, 1_2\right)\right)^{-1}
\varPhi\left(g_\infty\langle\sqrt{-1}\, 1_2\rangle\right)
\end{equation}
when
\[
g=\gamma\, g_\infty\,k_0\quad
\text{with $\gamma\in G\left(\mathbb Q\right)$,
$g_\infty\in G\left(\mathbb R\right)^+$
and $k_0\in\prod_{p<\infty}P_{1,p}\left(N\right)$}.
\]

Let $L$ be any non-zero linear form on $V_\varrho$.
Then $L\left(\varphi_\varPhi\right):G\left(\mA\right)\to\mathbb C$
defined by $L\left(\varphi_\varPhi\right)\left(g\right)=L\left(\varphi_\varPhi\left(g\right)\right)$
is a scalar valued automorphic form on $G\left(\mA\right)$.
Let 
$V\left(\varPhi\right)$ denote the the space generated by  right 
$G\left(\mathbb A\right)$-translates of $L\left(\varphi_\varPhi\right)$.
Then $V\left(\varPhi\right)$ does not depend on the choice
of $L$ and we denote by $\pi\left(\varPhi\right)$
the right regular representation of $G\left(\mA\right)$ on 
$V\left(\varPhi\right)$.
Note that the central character of $\pi\left(\varPhi\right)$ is trivial.

We recall that for  scalar valued automorphic forms 
$\phi$, $\phi^\prime$ on
$G\left(\mathbb A\right)$ with a trivial central character,
their  Petersson inner product $\langle \phi,\phi^\prime\rangle$
is defined by
\[
\langle \phi,\phi^\prime\rangle
=\int_{Z_G\left(\mathbb A\right) G\left(\mathbb Q\right)
\backslash  G\left(\mathbb A\right)}
\phi\left(g\right)\overline{\phi^\prime\left(g\right)}\,dg
\]
where $Z_G$ denotes the center of $G$ 
and $dg$ is the Tamagawa measure.
%
%
%
\begin{lemma}\label{l: norm constant}
Let $L$ be a non-zero linear form on $V_\varrho$.
Take $v^\prime\in V_\varrho$ such that $L\left(v\right)=\langle v,v^\prime\rangle_\varrho$ for any $v\in V_\varrho$.

Then
we have
\[
\langle L\left(\varphi_\varPhi\right),L\left(\varphi_\varPhi\right)\rangle
=C\left(v^\prime\right)\cdot
\langle\varPhi,\varPhi\rangle_\varrho
\quad\text{for any $\varPhi\in S_\varrho\left(\Gamma_0\left(N\right)\right)$}
\]
where 
\begin{equation}\label{e: norm constant}
C\left(v^\prime\right)=
\frac{\mathrm{Vol}\left(Z_G\left(\mathbb A\right) G\left(\mathbb Q\right)
\backslash G\left(\mathbb A\right)\right)}{
\mathrm{Vol}\left(\mathrm{Sp}_2\left(\mathbb Z\right)\backslash
\mathfrak H_2\right)}\cdot
\frac{\langle v^\prime,v^\prime\rangle_\varrho}{
\dim V_\varrho}.
\end{equation}
\end{lemma}
%
\begin{proof}
Let $K_\infty=\mathrm{U}_2\left(\mathbb R\right)$.
We identify $K_\infty$
as a subgroup of $\mathrm{Sp}_2\left(\mathbb R\right)$
via 
\[
K_\infty\ni A+\sqrt{-1}\, B\mapsto
\begin{pmatrix}A&-B\\B&A\end{pmatrix}
\in \mathrm{Sp}_2\left(\mathbb R\right).
\]
Let $dk$ be the Haar measure on $K_\infty$
such that $\mathrm{Vol}\left(K_\infty, dk\right)=1$.
Then by the Schur orthogonality relations, we have
\[
\int_{K_\infty}L\left(\varrho\left(k\right)^{-1}v\right)
\cdot\overline{L\left(\varrho\left(k\right)^{-1}w\right)}\,dk
=
\frac{\langle v,w\rangle_\varrho\cdot\langle v^\prime,v^\prime\rangle_\varrho}{
\dim V_\varrho}.
\]
On the other hand, it is easily seen that
for $\varPhi\in S_\varrho\left(\Gamma_0\left(N\right)\right)$, we have
\[
\frac{\langle\varPhi,\varPhi\rangle_\varrho}{
\mathrm{Vol}\left(\mathrm{Sp}_2\left(\mathbb Z\right)\backslash
\mathfrak H_2\right)}
=\frac{\langle\varphi_\varPhi,\varphi_\varPhi\rangle_\varrho}{
\mathrm{Vol}\left(Z_G\left(\mathbb A\right) G\left(\mathbb Q\right)
\backslash G\left(\mathbb A\right)\right)
}
\]
where
\[
\langle\varphi_\varPhi,\varphi_\varPhi\rangle_\varrho:=
\int_{Z_G\left(\mathbb A\right) G\left(\mathbb Q\right)
\backslash  G\left(\mathbb A\right)}
\langle\varphi_\varPhi\left(g\right),\varphi_\varPhi\left(g\right)\rangle_\varrho
dg.
\]
Hence
\begin{align*}
\langle\varPhi,\varPhi\rangle_\varrho=
&C\left(v^\prime\right)^{-1}
\int_{Z_G\left(\mathbb A\right)  G\left(\mathbb Q\right)
\backslash  G\left(\mathbb A\right)}
\int_{K_\infty}
\left| L\left(\varrho\left(k\right)^{-1}\varphi_\varPhi\left(g\right)\right)
\right|^2
\,
dk\,dg
\\
=&
C\left(v^\prime\right)^{-1}\int_{K_\infty}
\int_{Z_G\left(\mathbb A\right)  
G\left(\mathbb Q\right)
\backslash  G\left(\mathbb A\right)}
\left| L\left(\varphi_\varPhi\left(gk\right)\right)
\right|^2\,dg\,dk
\\
=&C\left(v^\prime\right)^{-1}\cdot\langle L\left(\varphi_\varPhi\right),L\left(\varphi_\varPhi\right)
\rangle_\varrho.
\end{align*}
\end{proof}
%
\subsubsection{Bessel periods of vector valued Siegel cusp forms}
Let $E$ be an imaginary quadratic field of $\mQ$ and $-D_E$ its discriminant.
We put
\begin{equation}\label{e: matrix S}
S_E:=\begin{cases}
\,\begin{pmatrix}1&0\\0&D_E\slash 4\end{pmatrix}
&\text{when $D_E\equiv 0\pmod{4}$};
\\
\begin{pmatrix}1&1\slash 2\\ 1\slash 2&\left(1+D_E\right)\slash 4\end{pmatrix}
&\text{when $D_E\equiv -1\pmod{4}$}.
\end{cases}
\end{equation}
Given $S=S_E$ as above, we define 
$T_S$, $N$ and
$\psi_{S}$ as in \ref{s:def bessel G}.
Then $T_S\left(\mathbb Q\right)\simeq E^\times$.
%
%

Let $\Lambda$ be a character of $T_S\left(\mathbb A\right)$
which is trivial on $\mathbb A^\times T_S\left(\mathbb Q\right)$.
Let $\psi$ be the unique character of $\mathbb A\slash \mathbb Q$
such that $\psi_\infty\left(x\right)=e^{-2\pi\sqrt{-1}\, x}$ and
the conductor of $\psi_\ell$ is $\mathbb Z_\ell$ for any prime number $\ell$.
%
Then for a scalar valued
 automorphic form $\phi$ on $G\left(\mathbb A\right)$
 with a trivial central character, we define its
 $(S, \Lambda, \psi)$-Bessel period $B_{S,\Lambda, \psi}\left(\phi\right)$ by \eqref{Beesel def gsp}
 with the Haar measures $du$ on $N\left(\mathbb A\right)$
 and $dt=dt_\infty\, dt_f$ on 
 $T_S\left(\mathbb A\right)=T_S\left(\mathbb R\right)
 \times T_S\left(\mathbb A_f\right)$
 are taken 
 so that
 $\mathrm{Vol}\left(N\left(\mathbb Q\right)\backslash N\left(\mathbb A\right),
 du\right)=1$ and
 \[
\mathrm{Vol}\left(\mathbb R^\times
 \backslash T_S\left(\mathbb R\right), dt_\infty\right)=
 \mathrm{Vol}\left(T_S\left(\mathbb{Z}_p \right), dt_f\right)=1.
 \]
 Then we note that 
 \[
 \mathrm{Vol}(\mathbb A^\times T_S\left(\mathbb Q\right)
 \backslash T_S\left(\mathbb A\right),dt) = \frac{2h_E}{w(E)} = D_E^{1 \slash 2} \cdot L(1, \chi_E).
 \]
 %
 
 For a $V_\varrho$-valued
 automorphic form $\varphi$ with a trivial central character,
 it is clear that for a linear form $L:V_\varrho\to\mathbb C$ we have
 \begin{equation}\label{e: vector into scalar}
 B_{S,\Lambda, \psi}\left(L\left(\varphi\right)\right)
 =L\left[ \int_{\mathbb A^\times T_S\left(\mathbb Q\right)
 \backslash T_S\left(\mathbb A\right)}
 \int_{N\left(\mathbb Q\right)\backslash N\left(\mathbb A\right)}
\Lambda \left(t\right)^{-1}\psi_S \left(u\right)^{-1} \varphi\left(tu\right)\, dt\, du
 \right].
  \end{equation}
%
Recall that we may identify
the ideal class group $\mathrm{Cl}_E$ of $E$
with
the quotient group
\[
T_S\left(\mathbb A\right)
\slash
T_S\left(\mathbb Q\right)T_S\left(\mathbb R\right) T_S\left(\hat{\mathbb Z}\right) .
\]
Let $\left\{t_c:
c\in\mathrm{Cl}_E\right\}$ be a set of representatives of $\mathrm{Cl}_E$
such that $t_c\in \prod_{p<\infty}T\left(\mathbb Q_p\right)$.
We 
write
$t_c$ as $t_c=\gamma_c\,m_c\,\kappa_c$ with
$\gamma_c\in \gl_2\left(\mathbb Q\right)$,
$m_c\in\left\{g\in\gl_2\left(\mathbb R\right):\det g>0\right\}$,
$\kappa_c\in\prod_{p<\infty}\gl_2\left(\mathbb Z_p\right)$.
Let $S_c=\left(\det \gamma_c\right)^{-1}\cdot {}^t\gamma_c S\gamma_c$.
Then the set $\left\{S_c: c\in\mathrm{Cl}_E\right\}$ is
a set of representatives for
 the $\mathrm{SL}_2\left(\mathbb Z\right)$-equivalence
classes of primitive semi-integral positive definite two by two symmetric 
matrices of discriminant $D_E$.
%

Thus when $\varphi=\varphi_\varPhi$
for  $\varPhi\in S_\varrho\left(\Gamma_0\left(N\right)\right)$
and $\Lambda$ is a character of $\mathrm{Cl}_E$, 
we may write  \eqref{e: vector into scalar} as
\begin{equation}\label{e: scalar vs vector}
B_{S,\Lambda, \psi}\left(L\left(\varphi_\varPhi\right)\right)=
2\cdot e^{-2\pi\mathrm{tr}\left(S\right)}\cdot
L\left(B_\Lambda\left(\varPhi;E\right)\right)
\end{equation}
where
\begin{equation}\label{e: sbp for vector valued}
B_\Lambda \left(\varPhi; E\right):=
w\left(E\right)^{-1}\cdot \pi_\varrho\left(\sum_{c\in\mathrm{Cl}_E}
\Lambda(c)^{-1} \cdot a\left(S_c,\varPhi\right)\right)
\end{equation}
is the vector valued $(S, \Lambda, \psi)$-Bessel period where
\begin{equation}\label{e: representation part}
\pi_\varrho=\int_{T_S^1\left(\mathbb R\right)}
\varrho\left(t\right)\, dt
\quad\text{with
$T_S^1=\mathrm{SL}_2\cap T_S$,
$\mathrm{Vol}\left(T_S^1\left(\mathbb R\right), dt\right)=1$}
\end{equation}
(e.g.  Dickson et al.~\cite[Proposition 3.5]{DPSS} and
Sugano~\cite[(1-26)]{Su}).
%
\begin{Remark}[An erratum to \cite{FM1}]
The definition of $B\left(\varPhi;E\right)$ in the vector valued case
in \cite[Theorem~5]{FM1} should be replaced by
 \eqref{e: sbp for vector valued}.
The statement and the proof of \cite[Theorem~5]{FM1}  remain valid.
\end{Remark}

Suppose that $\varrho=\varrho_\kappa$ where $\kappa=\left(2r+k,k\right)
\in\mathbb L$.
We define $Q_{S,\varrho}\in \mathbb C\left[X,Y\right]_{2r}$
by
\begin{equation}\label{e: def of Q}
Q_{S,\varrho}\left(X,Y\right):=
\left(\left(X,Y\right) S\begin{pmatrix}X\\ Y\end{pmatrix}\right)^r
\cdot \left(\det S\right)^{-\frac{2r+k}{2}}
\quad\text{where $S=S_E$ in \eqref{e: matrix S}}.
\end{equation}
Then for $\varPhi\in S_\varrho\left(\Gamma_0\left(N\right)\right)$,
the scalar valued $(S, \Lambda, \psi)$-Bessel period
${\mathcal B}_\Lambda\left(\Phi; E\right)$ of $\varPhi$
is defined by
\begin{equation}\label{e: scalar bs1}
{\mathcal B}_\Lambda\left(\Phi; E\right):=
\left(B_\Lambda\left(\varPhi;E\right), Q_{S,\varrho}\right)_{2r}.
\end{equation}
%
%
%
%
\subsection{Explicit $L$-value formula in the vector valued case}
\label{generalized boecherer statement}
%
Let us state our explicit formula for holomorphic Siegel modular forms.
In what follows, 
whenever
we refer to a type of 
an admissible representation of $G$ 
over a non-archimedean local field,
we use the standard classification
due to Roberts and Schmidt~\cite{RS}.
%

Let $N$ be a squarefree integer.
We say that a non-zero $\varPhi\in S_\varrho\left(\Gamma_0\left(
N\right)\right)$ is a \emph{newform} if
\begin{enumerate}
\item $\varPhi\in S_\varrho\left(\Gamma_0\left(
N\right)\right)^{\text{new}}$.
\item
$\varPhi$ is an eigenform for the local Hecke algebras for
all primes $p$ not dividing $N$ and an eigenfunction of 
the local $U\left(p\right)$ operator (see Saha and Schmidt~\cite[2.3]{SS}) for all primes dividing $N$.
\item
The representation $\pi\left(\varPhi\right)$ of $G\left(\mA\right)$
is irreducible.
\end{enumerate}
%
%
%
%

Then the following theorem is derived from
Theorem~\ref{ref ggp} exactly 
as Dickson, Pitale, Saha and Schmidt~\cite[Theorem~1.13]{DPSS}
except that we need to compute local Bessel periods
at the real place
adapting to the  vector valued case.
We perform the computation of them in Appendix~\ref{s:e comp}.
%
\begin{theorem}\label{t: vector valued boecherer}
Let $N\ge 1$ be an odd squarefree integer.
Let $\varrho=\varrho_\kappa$ where $\kappa=\left(2r+k,k\right)$
with $k\ge 2$.
Let $\varPhi$ be a non-CAP newform in $S_{\varrho}\left(\Gamma_0\left(N\right)\right)$.
Suppose that $\displaystyle{\left(\frac{D_E}{p}\right)=-1}$ for all primes $p$ dividing $N$.
When $k=2$, suppose moreover that
$\pi\left(\varPhi\right)$ is tempered.

Then we have
\begin{equation}\label{e: vector valued boecherer}
\frac{\left| {\mathcal B}_\Lambda\left(\varPhi;E\right)\right|^2}{
\langle\varPhi,\varPhi\rangle_{\varrho}}
=\frac{2^{4k+6r-c}}{D_E}\cdot
\frac{L\left(1\slash 2,\pi\left(\varPhi\right) \times \mathcal{AI} \left(\Lambda \right)\right)}{
L\left(1,\pi\left(\varPhi\right),\mathrm{Ad}\right)}
\cdot \prod_{p | N}J_p
\end{equation}
where $c=5$ if $\varPhi$ is a Yoshida lift in the sense of
Saha~\cite[Section~4]{Sa}
and $c=4$ otherwise.
The quantities $J_p$ for $p$ dividing $N$ are given by
\[
J_p=\left(1+p^{-2}\right)\left(1+p^{-1}\right)
\times
\begin{cases}
1&\text{if $\pi\left(\varPhi\right)_p$ is of type
$\mathrm{IIIa}$};
\\
2&\text{if $\pi\left(\varPhi\right)_p$
is of type $\mathrm{VIb}$};
\\
0&\text{otherwise.}
\end{cases}
\]
\end{theorem}
%
%
%
%
\begin{Remark}\label{temperedness condition}
When $k\ge 3$, $\pi\left(\varPhi\right)$ is tempered
by Proposition~\ref{temp prp}.
\end{Remark}
%
\begin{Remark}\label{e: scalar valued case}
Since $\mathcal B\left(\Phi;E\right)=2^{k}D_E^{-\frac{k}{2}}
\cdot B\left(\varPhi;E\right)$
when $r=0$, \eqref{intro B conj} follows from \eqref{e: vector valued boecherer}
by putting $N=1$ and $r=0$.
\end{Remark}
%
\begin{Remark}\label{e: Yoshida space}
In the statement of the theorem, 
we used the notion of Yoshida lifts in the sense of Saha~\cite{Sa}.
Though
 it is necessary to extend the arguments concerning
Yoshida lifts
in \cite[Section~4]{Sa} in the scalar valued case to the vector valued case to be rigorous, 
we omit it here 
since it is straightforward.
We also mention that the arguments in \cite[4.4]{Sa} now work
unconditionally since the classification theory in Arthur~\cite{Ar}
is  complete
for $\mathbb G=\mathrm{PGSp}_2\simeq\mathrm{SO}\left(3,2\right)$.
\end{Remark}
%
\begin{Remark}\label{r: finite part}
Recall that the $L$-functions in \eqref{e: vector valued boecherer}
are complete $L$-functions.
We may rewrite the explicit formula 
in terms of the finite parts of the $L$-functions by observing that
the relevant archimedean $L$-factors are given by
\[
L\left(1\slash 2,\pi\left(\varPhi\right)_\infty \times \mathcal{AI} \left(\Lambda \right)_\infty \right)
=2^4\left(2\pi\right)^{-2\left(k+r\right)}
\Gamma\left(k+r-1\right)^2\Gamma\left(r+1\right)^2
\]
and
\begin{multline*}
L\left(1,\pi\left(\varPhi\right)_\infty,\mathrm{Ad}\right)
=2^6\left(2\pi\right)^{-\left(4k+6r+1\right)}
\\
\times
\Gamma\left(k+2r\right)\Gamma\left(k-1\right)\Gamma\left(2r+2\right)
\Gamma\left(2k+2r-2\right)
\end{multline*}
respectively.
\end{Remark}
%
%
%
\begin{Remark}
Let us consider the case when $D$ is a quaternion algebra over $\mQ$ which
is split at the real place, i.e. 
$D(\mR) \simeq \mathrm{Mat}_{2 \times 2}(\mR)$.
Assuming  that 
the endoscopic classification holds for $\mathbb G_D=G_D\slash
Z_D$,
we may apply Theorem~\ref{ref ggp} to holomorphic modular forms on  $\mathbb G_D\left(\mA\right)$.
In this case, Hsieh-Yamana~\cite{HY} compute local Bessel periods and show an explicit formula for Bessel periods such as 
\eqref{e: vector valued boecherer}
for scalar valued holomorphic modular forms, 
including the case when $G_D =G$ and $N$ is an even 
squarefree integer.
Meanwhile we shall maintain $N$ to be odd
 in Theorem~\ref{t: vector valued boecherer},
 since 
our computation of  the local Bessel period at the real place
in the vector valued case
in Appendix~\ref{s:e comp} is performed under the assumption
that $N$ is odd.

As we noted in Remark~\ref{rem Arthur + ishimoto}, after the submission
of this paper, Ishimoto~\cite{Ishimoto}
showed the
endoscopic classification of $\mathrm{SO}(4,1)$ for generic Arthur parameters.
Therefore, we may apply our theorem to the case of $\Bbb G_D \simeq \mathrm{SO}(4,1)$.
\end{Remark}
%
\begin{Remark}
A global explicit formula such as \eqref{e: vector valued boecherer}
is obtained in a certain non-squarefree level case
by Pitale, Saha and Schmidt~\cite[Theorem~4.8]{PSS2}.
\end{Remark}
%
%
%
%
%
%
%
%
%
%
%
%
%
%
%
%
%
%
%
%
%
%
%
%
%
%
%
%
%
%
%
%
%
%
%
%
%
%
%
%
\appendix
%
%
%
%
\section{Explicit formula for the Whittaker periods on $G=\mathrm{GSp}_2$}
\label{appendix A}
Here we shall prove Theorem~\ref{gsp whittaker}.

Let $(\pi, V_{\pi})$ be an irreducible cuspidal globally generic 
automorphic representation of $G\left(\mA\right)$.
Then Soudry~\cite{So} has shown that the theta lift  of $\pi$ to $\mathrm{GSO}_{3,3}$ is non-zero and globally generic.
We may divide into two cases
according to whether 
the theta lift  of $\pi$ to $\mathrm{GSO}_{3,3}$ is cuspidal 
or not.

Suppose that the theta lift of $\pi$ to $\mathrm{GSO}_{3,3}$ is cuspidal.
Since $\mathrm{PGSO}_{3,3}\simeq \mathrm{PGL}_4$
and the explicit formula for the Whittaker periods
on $\mathrm{GL}_n$
is known by Lapid and Mao~\cite{LM},
the arguments in \ref{s: pf wh gso} and \ref{ss: local pull-back computation},
which are used to obtain \eqref{e:gso whittaker}
in Theorem~\ref{gso whittaker}
from \eqref{e:gsp whittaker},
work mutatis mutandis
to obtain \eqref{e:gsp whittaker}
from the Lapid--Mao formula in the case of $\mathrm{GL}_4$.

Suppose that the theta lift of $\pi$ to $\mathrm{GSO}_{3,3}$ is not cuspidal.
Then the theta lift of $\pi$ to $\mathrm{GSO}_{2,2}$
is non-zero and cuspidal.

Thus here we give a proof of Theorem~\ref{gsp whittaker} 
only in the case when $\pi$ is a theta lift from 
$\mathrm{GSO}_{2,2}$.
Recall that $\mathrm{PGSO}_{2,2}\simeq
\mathrm{PGL}_2\times \mathrm{PGL}_2$.
Our argument is similar to the one for 
\cite[Theorem~4.3]{Liu2}. 
 Indeed
 we shall prove \eqref{e:gsp whittaker}
 by  pushing forward the 
Lapid--Mao formula for $\mathrm{GSO}_{2,2}$ to $G$.
%
%
%
%
\subsection{Global pull-back computation}
%
%
%
Let $\left(X,\left<\,,\,\right>\right)$ be the $4$ dimensional symplectic space as in \ref{sp4 so42} and 
let $\left\{x_1,x_2,x_{-1},x_{-2}\right\}$  be the standard basis of $X$ 
given by \eqref{sb for X}.

Let $Y=F^4$ be an  orthogonal space with a non-degenerate symmetric bilinear form defined by
\[
(v_1, v_2) = {}^{t}v_1 J_4 v_2
\quad\text{for $v_1,v_2 \in Y$}
\] 
where $J_4$ is given by \eqref{d: J_m}.
We take a standard basis $\left\{y_{-2},y_{-1},y_1,y_2\right\}$
of $Y = F^{4}$ given by 
\[
y_{-2} = {}^{t}(1, 0, 0, 0), \quad y_{-1} = {}^{t}(0, 1, 0, 0),\quad
y_{1} = {}^{t}(0, 0, 1, 0), \quad y_{2} = {}^{t}(0, 0, 0, 1).
\]
We note that  $\left( y_i , y_{-j}\right)=\delta_{ij}$
for $1\le i,j\le 2$.

Put $Z=X \otimes Y$.
Then $Z$ is naturally a symplectic space over $F$.
We take a polarization $Z = Z_+ \oplus Z_-$ where
\[
Z_\pm=X_\pm\otimes Y
\]
and $X_\pm=F\cdot x_{\pm 1}+F\cdot x_{\pm 2}$.
Here all the double signs correspond.
When  $z_+=x_1\otimes a_1+x_2\otimes a_2\in Z_+\left(\mA\right)$ where
$a_1,a_2\in Y$, 
we write $z_+=\left(a_1,a_2\right)$
and $\phi\left(z_+\right)=\phi\left(a_1,a_2\right)$
for $\phi\in\mathcal S\left(Z_+\left(\mA\right)\right)$.
%

Let $N_{2,2}$ denote the group of upper triangular unipotent 
matrices of $\mathrm{GO}_{2,2}$, i.e.
\[
N_{2,2}\left(F\right) = \left\{ \begin{pmatrix} 1&x&y&-xy\\ 0&1&0&-y\\ 0&0&1&-x\\0&0&0&1\end{pmatrix} | \, x, y \in F \right\}.
\]
We define a non-degenerate character $\psi_{2,2}$ of $N_{2,2}(\mA)$ by
\[
\psi_{2,2} \begin{pmatrix} 1&x&y&-xy\\ 0&1&0&-y\\ 0&0&1&-x\\0&0&0&1\end{pmatrix} = \psi(x+y).
\]
Then for a cusp form $f$ on $\mathrm{GSO}_{2,2}\left(\mA\right)$, we define its Whittaker period $W_{2,2}(f)$ by
\[
W_{2,2}(f) = \int_{N_{2,2}(F) \backslash N_{2,2}(\mA)} f\left(n\right) 
\,\psi_{2,2}\left(n\right)^{-1} \, dn.
\]

The following identity is stated in \cite[p.113]{GRS97} but without a proof.
Though
it is shown by  an argument similar
to  the one for \cite[Proposition~2.6]{GRS97}, 
here we give a proof  for the convenience of the reader.
%
%
%
\begin{proposition}
\label{GRS identity}
Let $\varphi$ be a cusp form  on $\mathrm{GO}_{2,2}\left(\mA\right)$.
For $\phi \in \mathcal{S}(Z(\mA)_+)$, 
let $\Theta_\psi\left(\varphi,\phi\right)$  (resp.
$\theta_\psi\left(\varphi,\phi\right)$) be the theta
lift of $\sigma$ (resp. the restriction of $\varphi$ to
$\mathrm{GSO}_{2,2}\left(\mA\right)$) to $G\left(\mA\right)$.

Then we have
\begin{equation}
\label{GO22 to GSp4}
W_{\psi_{U_G}}(\Theta_\psi(\varphi, \phi)) = \int_{N_0(\mA) \backslash \mathrm{O}_{2,2}(\mA)} \phi(g^{-1}(y_{-2}, y_{-1}+y_1)) W_{\psi_{2,2}}(\sigma(g) \varphi) \, dg
\end{equation}
where $N_0$ denotes the unipotent subgroup
\[
N_0 = \left\{ \begin{pmatrix}1&x&-x&x^2\\ 0&1&0&-x\\ 0&0&1&x\\ 0&0&0&1 \end{pmatrix}\right\},
\]
which is the stabilizer of $y_{-2}$ and $y_{-1}+y_1$.

Similarly
we have 
\begin{equation}
\label{4.2 Liu2}
W_{\psi_{U_G}}(\theta_\psi(\varphi, \phi)) = \int_{N_0(\mA) \backslash \mathrm{SO}_{2,2}(\mA)} \phi(g^{-1}(y_{-2}, y_{-1}+y_1)) W_{\psi_{2,2}}(\sigma(g) \varphi) \, dg.
\end{equation}
\end{proposition}
%
%
%
\begin{proof}
Since the proofs are similar, we prove 
only \eqref{GO22 to GSp4}.
From the definition of the theta lift, we may write 
\begin{multline*}
\int_{N(F) \backslash N(\mA)} \Theta_\psi(\varphi, \phi)
\left(ug\right) \psi_{U_G}(u)^{-1} \,du
\\
= \int_{\mathrm{O}_{2,2}(F) \backslash \mathrm{O}_{2,2}(\mA)} \sum_{(a_1, a_2) \in \mathcal{X}} \omega_{\psi}(g, h)
\,\phi(a_1, a_2) \varphi(h) \, dh
\end{multline*}
where
\[
\mathcal{X} = \left\{ (a_1, a_2) \in Y(F)^2 : \begin{pmatrix} ( a_1, a_1) & ( a_1, a_2 ) \\ ( a_2, a_1)&( a_2, a_2 ) \end{pmatrix}
= \begin{pmatrix} 0&0\\ 0&1\end{pmatrix}\right\}.
\]
Then as in \cite[Lemma~1]{Fu}, only $\left(a_1,a_2\right)\in 
\mathcal X$
such that $a_1$ and $a_2$ are linearly independent contributes in
 the above sum.
Thus, by Witt's theorem, we may rewrite the above integral as
\begin{align*}
 &\int_{\mathrm{O}_{2,2}(F) \backslash \mathrm{O}_{2,2}(\mA)} \sum_{ \gamma \in N_0(F) \backslash \mathrm{O}_{2,2}(F)} \omega_{\psi}(g, h)
 \phi( \gamma^{-1} y_{-2}, \gamma^{-1} (y_{-1}+y_1))
 \, \varphi(h) \, dh
 \\
=&  \int_{\mathrm{O}_{2,2}(F) \backslash \mathrm{O}_{2,2}(\mA)} \sum_{ \gamma \in N_0(F) \backslash \mathrm{O}_{2,2}(F)} \omega_{\psi}(g, \gamma h)
 \phi(y_{-2}, y_{-1}+y_1) 
 \,\varphi(h) \, dh
 \\
 =
  &\int_{N_0(F) \backslash \mathrm{O}_{2,2}(\mA)} \omega_{\psi}(g, h)
 \phi( y_{-2}, y_{-1}+y_1)
 \, \varphi(h) \, dh
 \\
 =
 &\int_{N_0(\mA) \backslash \mathrm{O}_{2,2}(\mA)}  \int_{N_0(F) \backslash N_0(\mA)} \omega_{\psi}(g, h)
 \phi( y_{-2}, y_{-1}+y_1)
 \, \varphi(n h) \, dn \, dh.
\end{align*}
Thus by \eqref{d: U_G} we have
\begin{multline}\label{e: whittaker 1}
W_{\psi_{U_G}}(\Theta_\psi\left(\varphi,\phi\right)) 
= \int_{N_0(\mA) \backslash \mathrm{O}_{2,2}(\mA)}  \int_{N_2(F) \backslash N_2(\mA)} \int_{N_0(F) \backslash N_0(\mA)} 
\\
\omega_{\psi}(m(u)g, h)
 \phi( y_{-2}, y_{-1}+y_1) \varphi(nh) 
  \psi_{U_G}\left(m(u)\right)^{-1} \, dh \, du.
\end{multline}
Here we have
\[
\omega_{\psi}(m(u)g, h)
 \phi( y_{-2}, y_{-1}+y_1) =\omega_{\psi}(g, m_0(u)h)
 \phi( y_{-2}, y_{-1}+y_1) 
 \]
 where
 $m_0(u) =
 \begin{pmatrix}1&\frac{a}{2}&\frac{a}{2}&\frac{a^2}{4}\\ 0&1&0&-\frac{a}{2}\\ 0&0&1&-\frac{a}{2}\\ 0&0&0&1 \end{pmatrix}$
 for $u=\begin{pmatrix}1&a\\0&1\end{pmatrix}$,
 since $\psi_{U_G}(m(u))^{-1} = \psi(-a)$.
By noting the decomposition
\[
 \begin{pmatrix} 1&x&y&-xy\\ 0&1&0&-y\\ 0&0&1&-x\\0&0&0&1\end{pmatrix} 
=
\begin{pmatrix}1&\frac{x+y}{2}&\frac{x+y}{2}&\frac{(x+y)^2}{4}\\ 0&1&0&-\frac{x+y}{2}\\ 0&0&1&-\frac{x+y}{2}\\ 0&0&0&1 \end{pmatrix}
\begin{pmatrix}1&\frac{x-y}{2}&-\frac{x-y}{2}&\frac{(x-y)^2}{4}\\ 0&1&0&-\frac{x-y}{2}\\ 0&0&1&-\frac{x-y}{2}\\ 0&0&0&1 \end{pmatrix},
\]
the required identity \eqref{GO22 to GSp4}
 follows from \eqref{e: whittaker 1}.
\end{proof}
Recall the exact sequence 
\[
1 \rightarrow \mathrm{GSO}_{2,2} \rightarrow \mathrm{GO}_{2,2} \rightarrow \mu_2 \rightarrow 1.
\]
Hence we have 
\[
\Theta_\psi(\varphi, \phi)(g) = \int_{\mu_2(F) \backslash \mu_2(\mA)} \theta_\psi(\varphi^\varepsilon: \phi^\varepsilon)(g) \, d\varepsilon
\]
where $\varphi^\varepsilon = \sigma(\varepsilon)\varphi$ and $\phi^\varepsilon =\omega_\psi(\varepsilon) \phi$. 
Thus we have 
\[
\left|W_{\psi_{U_G}}(\Theta_\psi(\varphi, \phi))\right|^2 = \int_{\mu_2(F) \backslash \mu_2(\mA)} 
\mathbb{W}_{\psi_{U_G}}(\theta_\psi(\varphi^\varepsilon, \phi^\varepsilon)) \, d\varepsilon
\]
where
\[
\mathbb{W}_{\psi_{U_G}}(\theta_\psi(\varphi^\varepsilon, \phi^\varepsilon))
= \int_{\mu_2(F) \backslash \mu_2(\mA)} W_{\psi_{U_G}}(\theta_\psi(\varphi^\varepsilon, \phi^\varepsilon))\,
\overline{W_{\psi_{U_G}}(\theta_\psi(\varphi, \phi))} \, d\varepsilon.
\]
%
%
%
%
%
\subsection{Lapid-Mao formula}
Let us recall the Lapid-Mao formula in the $\mathrm{GL}_2$ case.
Let $(\tau, V_\tau)$ denote an irreducible cuspidal unitary automorphic representation of $\mathrm{GL}_2(\mA)$.
Then for $f \in V_\tau$, its Whittaker period is defined  by 
\[
W_2(f) = \int_{F \backslash \mA} f \begin{pmatrix}1&x\\ 0&1 \end{pmatrix} \psi(-x) \, dx
\]
with the Tamagawa measure $dx = \prod dx_v$.
Let $v$ be a place of $F$. For $f_v \in \tau_v$ and $\widetilde{f}_v \in \overline{\tau}_v$, by \cite{Liu2} (see also \cite[Section~2]{LM} ), 
we may define 
\[
\mathcal{W}_2(f_v, \widetilde{f}_v) = \int_{F}^{st} \mathcal{B}_{\tau_v} (\tau_v(x_v)f_v, \widetilde{f}_v) \psi_v(-x_v) \, dx_v.
\]
Put 
\[
\mathcal{W}_2^\natural(f_v, \widetilde{f}_v) =\frac{L(1, \tau_v, \mathrm{Ad})}{\zeta_{F_v}(2)}\mathcal{W}_2(f_v, \widetilde{f}_v)
\]
which is equal to $1$ at almost all places $v$ by \cite[Proposition~2.14]{LM}.
Let us define
\[
\langle f, f \rangle = \int_{\mA^\times \mathrm{GL}_2(F) \backslash \mathrm{GL}_2(\mA)} |f(g)|^2 \,dg
\]
where $dg$ is the Tamagawa measure.
We note that 
$\mathrm{Vol}\left(\mA^\times \mathrm{GL}_2(F)\backslash \mathrm{GL}_2(\mA),
dg\right)=2$.
Further, let us take a 
local $\mathrm{GL}_2(F_v)$-invariant pairing $\langle \,,\, \rangle_v$ on $\tau_v \times \tau_v$ such that $\langle f, f \rangle =\prod \langle f_v, f_v \rangle_v$. Then by \cite[Theorem~4.1]{LM}, we have
\begin{equation}
\label{LM GL2}
|W_2(f)|^2 =\frac{1}{2} \cdot \frac{\zeta_F(2)}{L(1, \tau, \mathrm{Ad})}\, \prod  \mathcal{W}_2^\natural(f_v, \overline{f}_v).
\end{equation}
for a factorizable vector $f = \otimes f_v \in V_\tau$.
%
%
%
%
\subsection{Local pull-back computation}
We fix a place $v$ of $F$ which will be suppressed from the notation in this section.
Further, we simply write $X(F)$ by $X$ for any object $X$
defined over $F$.
Let $\sigma$ be an irreducible tempered representation of $\mathrm{GO}_{2,2}$
such that its big theta lift $\Theta(\sigma)$ to $H$ is non-zero.
Because of the Howe duality proved by Howe~\cite{Ho1}, 
Waldspurger~\cite{Wa} and Gan-Takeda~\cite{GT}, combined
with Roberts~\cite{Rob},
$\Theta(\sigma)$ has a unique irreducible quotient, which we denote by $\pi$.
Put $R= \{(g, h) \in G \times \mathrm{GO}_{2,2} : \lambda(g)=\nu(h)\}.$ Then we have a 
unique $R$-equivariant map
\[
\theta : \omega_{\psi} \otimes \sigma  \rightarrow \pi.
\] 
Let $\mathcal{B}_\omega : \omega_\psi \otimes \overline{\omega_\psi} \rightarrow \mC$ be the 
canonical bilinear pairing defined by 
\[
\mathcal{B}_\omega(\phi, \tilde{\phi})=\int_{V^2} \phi(x) \widetilde{\phi}(x) \, dx.
\]
By \cite[Lemma~5.6]{GI0}, the pairing 
$\mathcal{Z} : (\sigma \otimes \overline{\sigma}) \otimes (\omega_\psi \otimes \overline{\omega_\psi}) \rightarrow \mC$, defined as
\[
\mathcal{Z}(\varphi, \widetilde{\varphi}, \phi, \widetilde{\phi})
=\frac{\zeta_F(2) \zeta_F(4)}{L(1, \sigma, \mathrm{std})} \int_{\mathrm{O}_{2,2}} \mathcal{B}_\omega(\omega_\psi(h)\phi, \widetilde{\phi}) 
\langle \sigma(h)\varphi, \widetilde{\varphi} \rangle \, dh,
\]
which converges absolutely by \cite[Lemma~3.19]{Liu2}, gives a pairing $\mathcal{B}_\pi : \pi \otimes \overline{\pi} \rightarrow \mC$
by 
\[
\mathcal{B}_\pi(\theta(\varphi, \phi), \theta(\widetilde{\varphi}, \widetilde{\phi}))=\mathcal{Z}(\varphi, \widetilde{\varphi}, \phi, \widetilde{\phi}).
\]
\begin{proposition}
We write $y_0=(y_{-2}, y_{-1}+y_1)$. For any $u \in N_{2}$,
\begin{multline*}
\left(\frac{\zeta_F(2) \zeta_F(4)}{L(1, \sigma, \mathrm{std})} \right)^{-1}
\int_{N_H}^{st} \mathcal{B}_\pi (\pi(n m(u)) \theta(\varphi, \phi),\theta(\widetilde{\varphi}, \widetilde{\phi}) ) \psi_{U_H}(n)^{-1} \, dn
\\
=
\int_{\mathrm{O}_{2,2}} \int_{N_0 \backslash \mathrm{SO}_{2,2}}
\left( \omega_{\psi}(g, m(u))\phi \right)(y_0) \overline{\widetilde{\phi}(h^{-1} \cdot y_0)} \langle \sigma(g) \varphi, \sigma(h) \widetilde{\varphi} \rangle \, dg \, dh.
\end{multline*}
\end{proposition}
Let us define 
\[
\mathcal{W}_{\psi_{U_H}}(f_1, f_2))
= \int_{U_H}^{st} \mathcal{B}_\pi(\pi(u)f_1, f_2) \psi_{U_H}^{-1}(u) \, du.
\]
Take the measure $dh_0 =2 dh|_{\mathrm{SO}_{2,2}}$. Then 
\begin{multline*}
\left(\frac{\zeta_F(2) \zeta_F(4)}{L(1, \sigma, \mathrm{std})} \right)^{-1}
\mathcal{W}_{\psi_{U_H}}(\theta(\varphi, \phi), \theta(\widetilde{\varphi}, \widetilde{\phi}))
\\
=\int_{N_2}^{st} \int_{\mathrm{O}_{2,2}} \int_{N_0 \backslash \mathrm{SO}_{2,2}}
\left( \omega_{\psi}(g, m(u))\phi \right)(y_0) \overline{\widetilde{\phi}(h^{-1} \cdot y_0)} 
\\
\times\langle \sigma(g) \varphi, \sigma(h) \widetilde{\varphi} \rangle  dg \, dh \,
du.
\end{multline*}
By an argument  similar to the one for  \cite[Section~3.4.2]{FM2} and \cite[Section~5.4]{FM3}, we see that this is equal to
\[
\int_{N_0 \backslash \mathrm{O}_{2,2}} \int_{N_0 \backslash \mathrm{SO}_{2,2}} \int_{N_{2,2}}^{st}
\left( \omega_{\psi}(g, m(u))\phi \right)(y_0) \overline{\widetilde{\phi}(h^{-1} \cdot y_0)} \langle \sigma(g) \varphi, \sigma(h) \widetilde{\varphi} \rangle  dg \, dh \,
du.
\]
Further, it is equal to
\begin{multline}
\label{4.6 Liu2}
\sum_{\varepsilon=\pm1} \int_{N_0 \backslash \mathrm{SO}_{2,2}} \int_{N_0 \backslash \mathrm{SO}_{2,2}} \int_{N_{2,2}}^{st}
\left( \omega_{\psi}(g, m(u))\phi^\varepsilon \right)(y_0) \overline{\widetilde{\phi}(h^{-1} \cdot y_0)}
\\
\times
 \langle \sigma(g) \varphi^\varepsilon, \sigma(h) \widetilde{\varphi} \rangle \, dg \, dh \,
du\\
=\sum_{\varepsilon=\pm1} \int_{N_0 \backslash \mathrm{SO}_{2,2}} \int_{N_0 \backslash \mathrm{SO}_{2,2}} \phi^\varepsilon(g^{-1} \cdot y_0)
\widetilde{\phi}(h^{-1} \cdot y_0)
\mathcal{W}_{2,2}(\sigma(g)\varphi^\varepsilon, \sigma(h)\widetilde{\varphi}) \, dg \, dh
\end{multline}
where we define 
\[
\mathcal{W}_{2,2}(\varphi_1, \varphi_2): = \int_{N_{2,2}}^{st} \langle \sigma(u)\varphi_1, \varphi_2 \rangle \psi_{2,2}^{-1}(u) \,du
\quad\text{for}\quad
 \varphi_i \in V_\sigma.
\]
Let us introduce a measure
$
d^\prime h = \zeta_F(2)^2 dh
$.
Then we get 
\begin{multline*}
\mathcal{W}_{\psi_{U_H}}^\natural(\theta(\varphi, \phi), \theta(\widetilde{\varphi}, \widetilde{\phi}))= 
\sum_{\varepsilon=\pm1} \int_{N_0 \backslash \mathrm{SO}_{2,2}} \int_{N_0 \backslash \mathrm{SO}_{2,2}} \phi^\varepsilon(g^{-1} \cdot y_0)
\widetilde{\phi}(h^{-1} \cdot y_0)
\\
\times
\mathcal{W}_{2,2}^\natural(\sigma(g)\varphi^\varepsilon, \sigma(h)\widetilde{\varphi}) \, dg \, d^\prime h.
\end{multline*}
Here
\[
\mathcal{W}_{2,2}^\natural(\sigma(g)\varphi^\varepsilon, \sigma(h)\widetilde{\varphi})
= \frac{L(1, \sigma_1, \mathrm{Ad}) L(1, \sigma_2, \mathrm{Ad})}{\zeta_F(2)^2}  \mathcal{W}_{2,2}(\sigma(g)\varphi^\varepsilon, \sigma(h)\widetilde{\varphi}).
\]
%
%
%
%
\subsection{Proof of Theorem~\ref{gsp whittaker}}
Let $(\sigma, V_{\sigma})$ be an irreducible cuspidal automorphic representation of the group $\mathrm{GO}_{2,2}(\mA)$.
Suppose that $\sigma$ is induced by the representation $\sigma_1 \boxtimes \sigma_2$ of 
$\mathrm{GL}_2(\mA) \times \mathrm{GL}_2(\mA)$.
For $f = f_1 \otimes f_2 \in V_{\sigma_1} \otimes V_{\sigma_2}$,
we have
\[
W_{U_H}(f) =  \int_{F \backslash \mA} f_1 \left( \begin{pmatrix}1&x\\ &1 \end{pmatrix}h_1\right) \psi(-x) \, dx
\int_{F \backslash \mA} f_2 \left( \begin{pmatrix}1&x\\ &1 \end{pmatrix} h_2\right) \psi(-x) \, dx
\]
for $h=(h_1, h_2) \in \mathrm{SO}_{2,2}(\mA)$.
Moreover, for any place $v$ of $F$, we have
\[
\mathcal{W}_{2,2}^\natural(\varphi_v, \widetilde{\varphi}_v) 
= \mathcal{W}_2^\natural(\varphi_{1,v}, \widetilde{\varphi}_{1,v})\mathcal{W}_2^\natural(\varphi_{2,v}, \widetilde{\varphi}_{2,v})
\]
with $\varphi_v = (\varphi_{1,v}, \varphi_{2,v})$ and $\widetilde{\varphi}_v = (\widetilde{\varphi}_{1,v}, \widetilde{\varphi}_{2,v})$.
Then by \eqref{4.2 Liu2} and the 
Lapid-Mao  formula \eqref{LM GL2}, we obtain
\begin{multline*}
\mathbb{W}_{\psi_{U_H}}(\theta_\psi(\varphi^\varepsilon, \phi^\varepsilon)) =
\frac{1}{4} \frac{\zeta_F(2)^2}{L(1, \sigma_1, \mathrm{Ad})L(1, \sigma_2, \mathrm{Ad})}
\\
\times
 \int_{\mu_2(F) \backslash \mu_2(\mA)} 
 \prod_v 
\int \int_{(N_{0}(F_v) \backslash \mathrm{SO}_{2,2})^2}  
\left(\prod_{\alpha=1, 2} \mathcal{W}_2^\natural((\sigma(g_v)\varphi_{v}^\varepsilon)_\alpha, (\overline{\sigma}(h_v)\overline{\varphi}_{v})_\alpha) \right)
\\
\times
\phi_{v}^\varepsilon (g_v^{-1} \cdot y_0)\overline{\phi}_{v} (h_v^{-1} \cdot y_0) \, dg \, dh
\\
=\frac{1}{4} \frac{\zeta_F(2)^2}{L(1, \sigma_1, \mathrm{Ad})L(1, \sigma_2, \mathrm{Ad})}
\int_{\mu_2(F) \backslash \mu_2(\mA)} \prod_v 
\int \int_{(N_{0}(F_v) \backslash \mathrm{SO}_{2,2})^2}  
\\
\mathcal{W}_{2,2}^\natural(\sigma(g_v)\varphi_v, \overline{\sigma}_v(h_v)\overline{\varphi}_v) 
\phi_{v}^\varepsilon (g_v^{-1} \cdot y_0)\overline{\phi}_{v} (h_v^{-1} \cdot y_0) \, dg \, dh.
\end{multline*}
By \eqref{4.6 Liu2}, this is equal to
\[
\frac{1}{4} \frac{\zeta_F(2) \zeta_F(4)}{L(1, \sigma_1, \mathrm{Ad})L(1, \sigma_2, \mathrm{Ad})}\prod \mathcal{W}_{\psi_{U_H}}^\natural(\theta(\varphi_v, \phi_v), \theta(\overline{\varphi}_v, \overline{\phi}_v)),
\]
and thus this completes our proof of Theorem~\ref{gsp whittaker}.
%
%
%
%
%
%
%
%
%
%
%
%
%
%
%
%
%
%
%
%
\section{Explicit computation of local Bessel periods at the
real place}
\label{s:e comp}
The goal of this appendix  is to 
compute explicitly the local Bessel periods at the
real place and to 
complete our proof of Theorem~\ref{t: vector valued boecherer}.
In this section, we use the same notation as in
Section~\ref{GBC}.

For a newform $\varPhi\in S_\varrho\left(\Gamma_0\left(N\right)\right)$ in Theorem~\ref{t: vector valued boecherer},
we define a scalar valued automorphic form $\phi_{\varPhi,S}$
on $ G\left(\mathbb A\right)$ by
\begin{equation}
\label{def vector adelize}
\phi_{\varPhi,S}\left(g\right)=
\left(\varphi_\varPhi\left(g\right),
Q_{S,\varrho}\right)_{2r}
\quad\text{for $g\in G\left(\mathbb A\right)$},
\end{equation}
where $\varphi_\varPhi$ is the 
adelization of $\varPhi$ given by \eqref{e: vector valued}
and $Q_{S,\varrho}$  by \eqref{e: def of Q}.
We note that by the argument in \cite[3.2]{DPSS}, $\phi_{\varPhi,S}$ is a factorizable vector
$\phi_{\varPhi, S}=\otimes_v\,\phi_{\varPhi, S, v}$.
For a place $v$ of $\mathbb Q$, we define $J_v$ by
\begin{equation}\label{e: j infty}
J_v= \frac{\alpha_v^{\natural}
\left(\phi_{\varPhi,S,v}, \phi_{\varPhi,S,v}\right)}{
\langle \phi_{\varPhi, S,v}, \phi_{\varPhi, S,v} \rangle_v}.
\end{equation}
It is clear that $J_v$ remains invariant under
replacing $\phi_{\varPhi, S,v}$ by
its non-zero scalar multiple.
Further, we put
\begin{equation}\label{e: def of c}
\mathcal C= C_\xi\cdot
\frac{\zeta_\mQ\left(2\right)\, \zeta_\mQ\left(4\right)}{L(1, \chi_{E})}
\end{equation}
with the Haar measure constant $C_\xi$ defined by \eqref{C_{S_D}}.
Then the following identity holds.
%
%
%
\begin{theorem}
\label{arch comp}
\begin{equation}
\label{e:arch comp}
C\left(Q_{S,\varrho}\right)
\mathcal C J_\infty=\frac{2^{4k+6r-1}e^{-4\pi\,\mathrm{tr}\left(S\right)}}{
D_E}.
\end{equation}
Recall that $C\left(Q_{S,\varrho}\right)$ is defined by
\eqref{e: norm constant} for $v^\prime=Q_{S,\varrho}$.
\end{theorem}
%
\begin{Remark}
\label{scalar rem}
In the scalar valued case, i.e. $r=0$, the explicit computation of 
$J_\infty$ is done in  Dickson~et al.~\cite[3.5]{DPSS} using 
the explicit formula for  matrix coefficients 
when $k \geq 3$.
Meanwhile Hsieh and Yamana~\cite[Proposition~5.7]{HY}
compute $J_ \infty$ in a different way when $k \geq 2$, 
based on Shimura's work on confluent hypergeometric functions.
\end{Remark}
%
%

We note that the left hand side of \eqref{e:arch comp}
depends only on the archimedean representation
$\pi\left(\varPhi\right)_\infty$
and the vector $\phi_{\varPhi, S,\infty}$.
Thus our strategy is to first obtain an explicit formula \eqref{e: formula 1}
for
the Bessel periods of vector valued Yoshida lifts
by
combining the results in Hsieh and Namikawa~\cite{HN1,HN2}, Chida and Hsieh~\cite{CH}, Martin and
Whitehouse~\cite{MW},
and, then to evaluate $C\left(Q_{S,\varrho}\right)
\mathcal C J_\infty$ by singling out the real place contribution,
comparing \eqref{e: formula 1} with
\eqref{e: main identity}.
%
%
%
%
%
%
%
\subsection{Explicit formula for Bessel periods of Yoshida lifts}
%
%
%
%
For a prime number $p$, let
\[
\Gamma_0^{(1)}\left(p\right)=
\left\{\begin{pmatrix}a&b\\c&d\end{pmatrix}\in
\mathrm{SL}_2\left(\mathbb Z\right):
c\equiv 0\pmod{p}\right\}
\]
and  $S_k\left(\Gamma_0^{(1)}\left(p\right)\right)$
 the space of cusp forms of weight $k$ with respect to 
$\Gamma_0^{(1)}\left(p\right)$.

In order to insure what follows  to be non-vacuous, 
first we shall prove the following technical lemma.
%
\begin{lemma}\label{l: simultaneous non-vanishing}
Let $k_1$ and $k_2$ be integers with $k_1 \geq k_2\ge0$.
Then there is a constant $N=N(k_1, k_2, E) \in \mR$ such that 
for any prime $p > N$,
there exist  distinct normalized newforms
$f_i\in S_{2k_i+2}\left(\Gamma_0^{(1)}\left(p\right)\right)$ 
for $i=1,2$ satisfying the condition:
\begin{equation}\label{e: eigenvalue condition}
\text{the Atkin-Lehner eigenvalues of $f_i$ at $p$ for
$i=1,2$ coincide.}
\end{equation}
\end{lemma}
%
%
%
\begin{proof}
We divide into the following two cases:
\begin{subequations}\label{e: parities}
\begin{equation}\label{e: case 1}
k_1\equiv k_2\pmod{2};
\end{equation}
\begin{equation}\label{e: case 3}
k_1+1\equiv k_2\equiv 0\pmod{2}.
\end{equation}
\end{subequations}

Suppose that \eqref{e: case 1} holds.
Then by Iwaniec, Luo and Sarnak~\cite[Corollary 2.14]{ILS},
there is a constant $N(k_1, k_2)$ such that, for any
 prime $p > N(k_1, k_2)$,
there exist distinct normalized newforms $f_i\in
S_{2k_i+2}\left(\Gamma_0^{(1)}\left(p\right)\right)$ for $i=1,2$
such that
\[
\varepsilon\left(1\slash 2,\pi_1\right)=\varepsilon\left(1\slash 2,\pi_2\right)
\]
where $\pi_i$ denotes the automorphic representation of 
$\mathrm{GL}_2\left(\mA\right)$ corresponding to 
$f_i$ for $i=1,2$.
Since $\pi_i$ is unramified at all prime numbers different from $p$,
we have
\[
\left(-1\right)^{k_1+1}\cdot\varepsilon_p\left(1\slash 2,\pi_1\right)=
\left(-1\right)^{k_2+1}\cdot\varepsilon_p\left(1\slash 2,\pi_2\right).
\]
Hence 
$\varepsilon_p\left(1\slash 2,\pi_1\right)=\varepsilon_p\left(1\slash 2,\pi_2\right)$
by \eqref{e: case 1}.
Then by the relationship between the local $\varepsilon$-factor at $p$
and the Atkin-Lehner eigenvalue at $p$ (e.g. \cite[4.4]{HN2}),
we see that \eqref{e: eigenvalue condition} holds.

Suppose that \eqref{e: case 3} holds. 
Then by Michel and Ramakrishnan~\cite[Theorem~3]{MR} or Ramakrishnan and Rogawski~\cite[Corollary~B]{RR},
there exists a constant $N_1=N_1(k_1, E)$ such that for any
 prime $p > N_1$, there exists a normalized newform $f_1\in
S_{2k_1+2}\left(\Gamma_0^{(1)}\left(p\right)\right)$
such that
\[
L\left(1\slash 2,\pi_1\right)L\left(1\slash 2,\pi_1\times\chi_E\right)\ne 0.
\]
In particular, $\varepsilon\left(1\slash 2,\pi_1\right) =1$, and thus as in the previous case, we have
\[
\left(-1\right)^{k_1+1}\cdot\varepsilon_p\left(1\slash 2,\pi_1\right) = 1.
\]
Moreover, by \cite[Corollary 2.14]{ILS}, there exists a constant $N_2 = N_2(k_2)$ such that 
for any prime $p > N_2$, there exists a  normalized newform
$f_2\in
S_{2k_2+2}\left(\Gamma_0^{(1)}\left(p\right)\right)$
such that
\[
\varepsilon\left(1\slash 2,\pi_2\right) =-1.
\]
Then by taking the constant $N$ to be  $\mathrm{max}(N_1, N_2)$,
the condition
\eqref{e: eigenvalue condition} holds  by the same argument as above.
\end{proof}
%
%
%
%
\subsubsection{Vector valued Yoshida lift}
As for the Yoshida lifting, we refer the details to 
our main references
Hsieh and Namikawa~\cite{HN1,HN2}.

Let $k_1$ and $k_2$ be integers with $k_1\ge k_2\ge 0$.
Then 
by Lemma~\ref{l: simultaneous non-vanishing},
we may take a prime number $p$ satisfying the condition:
\begin{equation}\label{e: inert condition}
\text{$p$ is odd, and inert and unramified in $E$}
\end{equation}
and may take distinct normalized newforms
$f_i\in
S_{2k_i+2}\left(\Gamma_0^{(1)}\left(p\right)\right)$ ($i=1,2$)
satisfying the condition \eqref{e: eigenvalue condition}.

For a non-negative integer $r$, we denote by 
$\left(\tau_r,\mathcal W_r\right)$
the  representation 
$\left(\varrho,V_\varrho\right)$ of $\gl_2\left(\mathbb C\right)$
where $\varrho=\varrho_{\left(r,-r\right)}$, i.e.
$\tau_r=\mathrm{Sym}^{2r}\otimes \det^{-r}$.
We note that the action of the center  of $\gl_2\left(\mathbb C\right)$
on $\mathcal W_r$ by $\tau_r$ is trivial
and the pairing $\left(\, ,\,\right)_{2r}$ is
$\gl_2\left(\mathbb C\right)$-invariant by \eqref{e: equivariance}.
Let $p$ be a prime number and
$D=D_{p,\infty}$ the unique division quaternion algebra over $\mathbb Q$
which ramifies precisely at $p$ and $\infty$.
Let $\mathcal O_D$ be the maximal order of $D$
specified as in \cite[3.2]{HN1}
and we put $\hat{\mathcal O}_D=\mathcal O_D\otimes_{\mathbb Z}
\hat{\mathbb Z}$.
%
%
%
\begin{Definition}
$\mathcal A_r\left(D^\times\left(\mathbb A\right),
\hat{\mathcal O}_D\right)$,
the space of automorphic forms of weight $r$
and level $\hat{\mathcal O}_D$ on $D^\times\left(\mathbb A\right)$
is a space of functions $\mathbf{g}:D^\times\left(\mathbb A\right)
\to \mathcal W_r$
satisfying
\[
\mathbf{g}\left(z\gamma h u\right)=\tau_r\left(h_\infty\right)^{-1}
\mathbf{g}\left(h_f\right)
\]
for $z\in \mathbb A^\times$,
$\gamma\in D^\times\left(\mathbb Q\right)$,
$u\in \hat{\mathcal O}_D^\times$ and $h=\left(h_\infty, h_f\right)
\in D^\times\left(\mathbb R\right)\times D^\times\left(\mathbb A_f\right)$.
\end{Definition}
%
%
%
%
For $i=1,2$, let $\pi_i$ be the irreducible cuspidal automorphic representation
of $\gl_2\left(\mathbb A\right)$ corresponding to $f_i$.
Let $\pi^D_i$ be the Jacquet-Langlands transfer of $\pi_i$
to $D^\times\left(\mathbb A\right)$.
We denote by 
$\mathcal A_{k_i}\left(D^\times\left(\mathbb A\right),
\hat{\mathcal O}_D\right)\left[\pi_i^D\right]$
the $\pi_i^D$-isotypic subspace of $\mathcal A_{k_i}\left(D^\times\left(\mathbb A\right),
\hat{\mathcal O}_D\right)$.
Then $\mathcal A_{k_i}\left(D^\times\left(\mathbb A\right),
\hat{\mathcal O}_D\right)\left[\pi_i^D\right]$ has a subspace of newforms, 
which is one dimensional.
Let us take newforms $\mathbf{f}_i\in
\mathcal A_{k_i}\left(D^\times\left(\mathbb A\right),
\hat{\mathcal O}_D\right)\left[\pi_i^D\right]$ for $i=1,2$
and fix.
Then to the pair $\mathbf{f}=\left(\mathbf{f}_1,\mathbf{f}_2\right)$,
Hsieh and Namikawa~\cite[3.7]{HN1}
associate the \emph{Yoshida lift} $\theta_{\mathbf{f}}$,
a $V_\varrho$-valued cuspidal automorphic form on 
$ G\left(\mathbb A\right)$
where $\varrho=\varrho_\kappa$ with
\[
\kappa=\left(k_1+k_2+2, k_1-k_2+2\right)\in\mathbb L.
\]
%
The \emph{classical Yoshida lift}
$\theta^\ast_{\mathbf{f}}\in S_\varrho\left(\Gamma_0\left(p\right)\right)$
is also attached to $\mathbf{f}$ in \cite[3.7]{HN1}
so that
$\theta_{\mathbf{f}}$ is obtained from $\theta^\ast_{\mathbf{f}}$
by the adelization procedure in \eqref{e: vector valued}.
%
\subsubsection{Bessel periods of Yoshida lifts}
Let $\phi_{\mathbf{f},S}$ denote
  a scalar valued automorphic form attached
  to  $\theta^\ast_{\mathbf{f}}$ as in \eqref{def vector adelize}.
Hsieh and Namikawa evaluated the Bessel periods
of $\phi_{\mathbf{f},S}$ 
in \cite{HN1}.

First we remark that by \cite[Theorem~5.3]{HN1}, 
for any sufficiently large  prime number $q$ which is different from
$p$, 
we may take a character $\Lambda_0$ of $\mathbb A_E^\times$ satisfying:
\begin{subequations}\label{e: Lambda}
\begin{equation}\label{e: Lambda 1}
L\left(1\slash 2,\pi_1\otimes\mathcal{AI} \left(\Lambda_0 \right)\right)
L\left(1\slash 2,\pi_2\otimes\mathcal{AI} \left(\Lambda_0^{-1} \right)\right)\ne 0;
\end{equation}
\begin{equation}\label{e: Lambda 2}
\text{the conductor of $\Lambda_0$ is  $q^m\mathcal O_E$
where  $m>0$};
\end{equation}
\begin{equation}\label{e: Lambda 3}
\text{$\Lambda_0\mid_{\mathbb A^\times}$ is trivial};
\end{equation}
\begin{equation}\label{e: Lambda 4}
\text{$\Lambda_{0,\infty}$ is trivial.}
\end{equation}
\end{subequations}
 
 Then \cite[Proposition~4.7]{HN1} yields the following
 formula.
 %
\begin{lemma}\label{l: bessel for Yoshida}
 We have
 \begin{equation}\label{e: Bessel 1}
 B_{S, \Lambda_0, \psi}\left(\phi_{\mathbf{f},S}\right)=
 q^{2m}\cdot\left(-2\sqrt{-1}\,\right)^{k_1+k_2}\cdot
 e^{-2\pi\,\mathrm{Tr}\left(S\right)}\cdot
 \prod_{i=1}^2\,
 P\left(\mathbf{f}_i, \Lambda_0^{\alpha_i},1_2\right)
 \end{equation}
 where $\alpha_i=\left(-1\right)^{i+1}$ and
 \[
 P\left(\mathbf{f}_i, \Lambda_0^{\alpha_i},1_2\right)=
 \int_{E^\times \mathbb A^\times\backslash \mathbb A_E^\times}
 \left(\left(XY\right)^{k_i}, \mathbf{f}_i\left(t\right)\right)_{2k_i}
 \cdot
 \Lambda_0^{\alpha_i}\left(t\right)\, dt.
 \]
 \end{lemma}
 %
 From \eqref{e: Bessel 1}, we have
 \begin{equation}\label{e: Bessel 2}
 \left| B_{S, \Lambda_0, \psi}\left(\phi_{\mathbf{f},S}\right)\right|^2=
 q^{4m}\cdot 2^{2(k_1+k_2)}\cdot
 e^{-4\pi\,\mathrm{tr}\left(S\right)}\cdot
 \prod_{i=1}^2
 \,
 \left|P\left(\mathbf{f}_i, \Lambda_0^{\alpha_i},1_2\right)\right|^2.
 \end{equation}
 Since $p$ is odd and inert in $E$, we may evaluate
 the right hand side of \eqref{e: Bessel 2} by
 Martin and Whitehouse~\cite{MW}.
 Namely  the following formula holds by
 \cite[Theorem~4.1]{MW}.
 %
 \begin{lemma}\label{l: MW}
 We have
\begin{multline}\label{e: Bessel 3}
\frac{\left| P(\mathbf{f}_i, \Lambda_0^{\alpha_i}, 1_2) \right|^2}{
\left\|\phi_{\mathbf{f}_i} \right\|^2} =\frac{1}{4}\cdot  \frac{\xi(2)}{\zeta_{\mQ_p}(2)}\cdot
\frac{L\left(1 \slash 2, \pi_i \otimes \mathcal{AI}\left(\Lambda_0^{\alpha_i}\right)\right)}{L\left(1, \pi_i, \mathrm{Ad}\right)} \cdot \left(1+p^{-1}\right)^{-1}
\\
\times
\frac{\Gamma\left(2k_i+2\right)}{ 2\,q^m\, \pi \,D_E^{1\slash 2} \,\,
\Gamma\left(k_i+1\right)^2}
\end{multline}
where $\xi(s)$ denotes the complete Riemann zeta function,
$\phi_{\mathbf{f}_i}$  the scalar valued automorphic form on $D^\times\left(\mathbb A
\right)$ defined by
\[
\phi_{\mathbf{f}_i}\left(h\right)=
\left(\left(XY\right)^{k_i},\mathbf{f}_i\left(h\right)\right)_{2k_i}
\quad\text{for $h\in D^\times\left(\mathbb A\right)$}
\]
and
\[
\left\|\phi_{\mathbf{f}_i} \right\|^2=
\int_{\mathbb A^\times D^\times\left(\mathbb Q\right)\backslash
D^\times\left(\mathbb A\right)}
\left|\phi_{\mathbf{f}_i}\left(h\right) \right|^2\, dh.
\]
Here $dh$ is the Tamagawa measure on $\mA^\times \backslash D^\times\left(\mathbb A\right)$,
and thus 
\[
\mathrm{Vol}(\mathbb A^\times D^\times\left(\mathbb Q, \right)\backslash
D^\times\left(\mathbb A\right),dh)=2.
\]
 \end{lemma}
 \begin{Remark}
 The factor $\frac{1}{4}$ in \eqref{e: Bessel 3}
 originates from the difference of measures between the one used
  here and the one in \cite{MW}.
 \end{Remark}
 %
 %
 %
 In order to utilize the explicit inner product formula for vector valued
 Yoshida lifts in Hsieh and Namikawa~\cite{HN2}, we need the following lemma.
 %
 \begin{lemma}\label{l: inner product change}
 Let us define an inner product $\langle\mathbf{f}_i,\mathbf{f}_i\rangle$
 for $i=1,2$ by
 \begin{equation}
 \label{finite sum PI}
 \langle\mathbf{f}_i,\mathbf{f}_i\rangle=
 \sum_{a}\,
 \langle\mathbf{f}_i\left(a\right),
 \mathbf{f}_i\left(a\right)\rangle_{\tau_{k_i}}
 \cdot
 \frac{1}{\# \,\Gamma_a}
 \end{equation}
  where 
  $\left<\,,\,\right>_{\tau_{k_i}}$ is defined by 
  \eqref{e: def of hermitian},
  $a$ runs over  double coset representatives of
 $D^\times\left(\mathbb Q\right)\backslash
 D^\times\left(\mathbb A_f\right)\slash
 \hat{\mathcal O}_D^\times$ and
 $\Gamma_a=\left(a\,\hat{\mathcal O}_D^\times \, a^{-1}\cap D^\times\left(\mathbb Q
 \right)\right)\slash \left\{\pm 1\right\}$.
 
 Then for $i=1,2$, we have
 \begin{equation}\label{e: little inner product}
 \left\|\phi_{\mathbf{f}_i}\right\|^2=
 2^3\cdot 3\cdot p^{-1}\left(1-p^{-1}\right)^{-1}\cdot
 \frac{\Gamma\left(k_i+1\right)^2}{\Gamma\left(2k_i+1\right)}
 \cdot\frac{1}{\left(2k_i+1\right)^2}\cdot
 \langle\mathbf{f}_i,\mathbf{f}_i\rangle.
 \end{equation}
 \end{lemma}
 %
 \begin{proof}
 Since $\left\|\phi_{\mathbf{f}_i}\right\|^2=\left\|\pi_i^D\left(h_\infty\right)
 \phi_{\mathbf{f}_i}\right\|^2$ for $h_\infty\in D^\times\left(\mathbb R\right)$,
 we have
\begin{multline*}
\left\|\phi_{\mathbf{f}_i}\right\|^2
=\frac{1}{
\mathrm{Vol}\left(\mathbb R^\times\backslash D^\times\left(\mathbb R\right),
dh_\infty\right)}
\\
\times
\int_{\mathbb R^\times\backslash D^\times\left(\mathbb R\right)}
\int_{\mathbb A^\times D^\times\left(\mathbb Q\right)
\backslash D^\times\left(\mathbb A\right)}
\left|\phi_{\mathbf{f}_i}\left(hh_\infty\right)\right|^2 dh\,dh_\infty.
\end{multline*}
By interchanging the order of integration, we have
\begin{multline*}
\left\|\phi_{\mathbf{f}_i}\right\|^2
=
\frac{1}{
\mathrm{Vol}\left(\mathbb R^\times\backslash D^\times\left(\mathbb R\right),
dh_\infty\right)}
\\
\times
\int_{\mathbb A^\times D^\times\left(\mathbb Q\right)
\backslash D^\times\left(\mathbb A\right)}
\int_{\mathbb R^\times\backslash D^\times\left(\mathbb R\right)}
\left|\phi_{\mathbf{f}_i}\left(hh_\infty\right)\right|^2 dh_\infty\,dh.
\end{multline*}
Here the Schur orthogonality implies
\begin{multline*}
\frac{1}{
\mathrm{Vol}\left(\mathbb R^\times\backslash D^\times\left(\mathbb R\right),
dh_\infty\right)}
\int_{\mathbb R^\times\backslash D^\times\left(\mathbb R\right)}
\left|
\left(\left(XY\right)^{k_i},\mathbf{f}_i\left(hh_\infty\right)\right)_{2k_i}
\right|^2
\,dh_\infty
\\
=d_i^{-1}
\cdot
\left(\left(XY\right)^{k_i},\left(XY\right)^{k_i}\right)_{2k_i}
\cdot
\left(\mathbf{f}_i\left(h\right),\overline{\mathbf{f}_i\left(h\right)}\right)_{2k_i}
\end{multline*}
where $d_i=\dim \mathrm{Sym}^{2k_i}=2k_i+1$
and 
$
\left(\left(XY\right)^{k_i},\left(XY\right)^{k_i}\right)_{2k_i}
=\left(-1\right)^{k_i}
\begin{pmatrix}2k_i\\k_i\end{pmatrix}^{-1}$.
Hence
\[
\left\| \phi_{\mathbf{f}_i} \right\|^2
=\begin{pmatrix}2k_i\\ k_i \end{pmatrix}^{-1} (2k_i+1)^{-1}  
\int_{\mA^\times D^\times(\mQ) \backslash D^\times(\mA)}
\left(\mathbf{f}_i\left(h\right), \overline{\mathbf{f}_i\left(h\right)} \right)_{2k_i}
 \, dh.
\]
By \cite[Lemma~6]{HN1}, we have
\begin{multline}
\label{PI difference}
 \int_{\mA^\times D^\times(\mQ) \backslash D^\times(\mA)}
\left( \mathbf{f}_i\left(h\right), \overline{\mathbf{f}_i\left(h\right)} \right)_{2k_i}
 \, dh
 \\
=\frac{(-1)^{k_i}}{2k_i+1}
\int_{\mA^\times D^\times(\mQ) \backslash D^\times(\mA)}
 \langle\mathbf{f}_i\left(h\right), \mathbf{f}_i\left(h\right)\rangle_{\tau_{k_i}}
 \, dh.
\end{multline}
Finally by Chida and Hsieh~\cite[(3.10)]{CH} with 
the following Remark~\ref{2 power miss CH},
we obtain \eqref{e: little inner product}.
 \end{proof}
 %
 %
 %
 %
 %
 %
 \begin{Remark}
 \label{2 power miss CH}
 In \cite{CH}, the Eichler mass formula
 is used to express the right hand side of \eqref{PI difference}
 in terms of the inner product defined by \eqref{finite sum PI}.
 There is a typo in the Eichler mass formula
 in \cite[p.103]{CH}. 
 The right hand side of the formula quoted there should
 be multiplied by $2$.
  \end{Remark}
 %
 %
 %
 Let us recall the inner product formula for $\theta^\ast_{\mathbf{f}}$
 by Hsieh and Namikawa~\cite[Theorem~A]{HN2}.
 %
 \begin{proposition}\label{p: inner product classical}
 We have
 \begin{multline}\label{e: inner product classical}
 \frac{
 \langle\theta_{\mathbf{f}}^\ast, \theta_{\mathbf{f}}^\ast\rangle_{
 \varrho
 }}{
 \langle\mathbf{f}_1,\mathbf{f}_1\rangle\langle\mathbf{f}_2,\mathbf{f}_2\rangle}
 \\
 =L\left(1,\pi_1\times\pi_2\right)\cdot 
 \frac{2^{-\left(2k_1+6\right)}}{\left(2k_1+1\right)\left(2k_2+1\right)}
 \cdot \frac{1}{p^2\left(1+p^{-1}\right)\left(1+p^{-2}\right)}.
 \end{multline}
 Here $\langle\theta_{\mathbf{f}}^\ast, 
 \theta_{\mathbf{f}}^\ast\rangle_{
  \varrho
 }$
 is given by
 \[
\langle \theta_{\mathbf{f}}^\ast, \theta_{\mathbf{f}}^\ast \rangle_{
\varrho
}
= 
\frac{1}{[\mathrm{Sp}_2(\mZ) : \Gamma_0\left(p\right)]} \int_{\Gamma_0\left(p\right) \backslash \mathfrak{H}_2} 
\langle\theta_{\mathbf{f}}^\ast(Z), \theta_{\mathbf{f}}^\ast(Z)
\rangle_{\varrho}
 \left(\det Y\right)^{k_1-k_2-1}
 \, dX\, dY
\]
with $\varrho=\varrho_\kappa$
where $\kappa=\left(k_1+k_2+2,k_1-k_2+2\right)$.
 \end{proposition}
 %
 %
 %
 Thus by combining \eqref{e: Bessel 2}, \eqref{e: Bessel 3},
 \eqref{e: little inner product} and \eqref{e: inner product classical},
 we have
 \begin{multline}
\label{e: formula 1}
\frac{\left|B_{S, \Lambda_0, \psi}\left(\phi_{\mathbf{f},S}\right)\right|^2}{
\langle\theta_{\mathbf{f}}^\ast, \theta_{\mathbf{f}}^\ast \rangle_{
\varrho
}}
= \frac{2^{4k_1 + 2k_2+ 5} e^{-4\pi \,\mathrm{tr}\left(S\right) } }{D_E} 
\cdot
2\left(1+p^{-1}\right)\left(1+p^{-2}\right) \cdot q^{2m}
\\
\times
\frac{L\left(1 \slash 2, \pi_1 \otimes \mathcal{AI}\left(\Lambda_0\right)\right) 
L\left(1 \slash 2, \pi_2 \otimes \mathcal{AI}\left(\Lambda_0^{-1}\right)\right) }{
L\left(1, \pi_1, \mathrm{Ad}\right) L\left(1, \pi_2, \mathrm{Ad}\right) 
L\left(1, \pi_1 \times \pi_2\right)}.
\end{multline}
Here we note that the both sides of \eqref{e: formula 1}
are non-zero due to  the conditions
\eqref{e: eigenvalue condition}
and \eqref{e: Lambda}.

%
%
%
%
\subsection{Proof of Theorem~\ref{arch comp}}
Since  the  Ichino-Ikeda type formula has been proved for Yoshida lifts
by Liu
\cite[Theorem~4.3]{Liu2},
the computations in Dickson et al.~\cite{DPSS} implies
\begin{multline}
\label{e: formula 2}
\frac{\left|B_{S,\Lambda_0, \psi}\left(\phi_{\mathbf{f},S}\right)\right|^2}{
\langle\phi_{\mathbf{f},S},\phi_{\mathbf{f},S}
 \rangle
}
= \frac{\mathcal CJ_\infty}{2^2} 
\cdot
2\left(1+p^{-1}\right)\left(1+p^{-2}\right) \cdot J_q 
\\
\times
\frac{L\left(1 \slash 2, \pi_1 \otimes \mathcal{AI}\left(\Lambda_0\right)\right) 
L\left(1 \slash 2, \pi_2 \otimes \mathcal{AI}\left(\Lambda_0^{-1}\right)\right) }{
L\left(1, \pi_1, \mathrm{Ad}\right) L\left(1, \pi_2, \mathrm{Ad}\right) 
L\left(1, \pi_1 \times \pi_2\right)}.
\end{multline}
Thus in order to evaluate $J_\infty$, we need to determine $J_q$.

Here we use a scalar valued Yoshida lift to evaluate $J_q$.
First we recall that \eqref{e:arch comp} holds in the scalar valued case, i.e.
when $k_2=0$, 
as we noted in Remark~\ref{scalar rem}.
By Lemma~\ref{l: simultaneous non-vanishing},
when $q$ is large enough, there also exist
distinct normalized newforms $f_1^\prime\in
S_{2k_1+2}\left(\Gamma_0^{(1)}\left(p\right)\right)$
and $f_2^\prime\in S_{2}\left(\Gamma_0^{(1)}\left(p\right)\right)$
satisfying the condition \eqref{e: eigenvalue condition},
and, a character $\lambda_0^\prime$ of $\mathbb A_E^\times$
satisfying the conditions \eqref{e: Lambda} for $\pi_i^\prime$
$\left(i=1,2\right)$ where $\pi_i^\prime$
is the automorphic representation of 
$\mathrm{GL}_2\left(\mA\right)$.
Define $\mathbf{f}^\prime$ similarly for $\pi_1^\prime$
and $\pi_2^\prime$.
Since \eqref{e:arch comp} is valid in the scalar valued case,
we have
\begin{multline*}
\frac{\left|B_{S, \Lambda_0, \psi}\left(\phi_{\mathbf{f}^\prime,S}\right)\right|^2}{
\langle\phi_{\mathbf{f}^\prime,S},\phi_{\mathbf{f}^\prime,S}
 \rangle
}
=
\frac{2^{4k_1+5}e^{-4\pi\,\mathrm{tr}\left(S\right)}}{
D_E}
\cdot
C\left(Q_{S,\varrho_{\left(k_1, k_1\right)}}\right)^{-1}
 \\
\cdot
2\left(1+p^{-1}\right)\left(1+p^{-2}\right) \cdot J_q \cdot
\frac{L\left(1 \slash 2, \pi_1^\prime
 \otimes \mathcal{AI}\left(\Lambda_0^\prime\right)\right) 
L\left(1 \slash 2, \pi_2^\prime \otimes \mathcal{AI}\left(\Lambda_0^{\prime\, -1}\right)\right) }{
L\left(1, \pi_1^\prime, \mathrm{Ad}\right) L\left(1, \pi_2^\prime, \mathrm{Ad}\right) 
L\left(1, \pi_1^\prime \times \pi_2^\prime\right)}.
\end{multline*}
We note that $J_q$ here is the same as the one in
\eqref{e: formula 2}.
Then by comparing the formula above with 
\eqref{e: formula 1} for $\mathbf{f}^\prime$ and $\Lambda_0^\prime$,
we have $J_q=q^{2m}$.

Finally
 by comparing \eqref{e: formula 1} with 
 \eqref{e: formula 2} substituting  $J_q=q^{2m}$,
 we have
 \begin{equation}\label{e: j_infty}
 C\left(Q_{S,\varrho}\right)\mathcal CJ_\infty=\frac{2^{4k_1+2k_2+7}e^{-4\pi\,
\mathrm{tr}\left(S\right)}}{D_E}
\end{equation}
in the general case.

For $\varPhi$ in Theorem~\ref{t: vector valued boecherer},
a scalar valued automorphic form $\phi_{\varPhi,S}$
defined by
\[
\label{phi_Phi S}
\phi_{\varPhi,S}\left(g\right)=
\left(\varphi_\varPhi\left(g\right),
Q_{S,\varrho}\right)_{2r}
\quad\text{for $g\in G\left(\mathbb A\right)$}
\]
is factorizable, i.e. 
$\phi_{\varPhi, S}=\otimes_v\,\phi_{\varPhi, S,v}$.
Let us choose
$k_1$ and $k_2$ so that
\[
\left(2r+k,k\right)=\left(k_1+k_2+2,k_1-k_2+2\right),
\quad
\text{i.e.}\quad
k_1 =r+k-2,  \,\, k_2=r.
\]
Then for $\phi_{\mathbf{f},S}=\otimes_v
\,\phi_{\mathbf{f},S,v}$ in \eqref{e: formula 2},
the archimedean factor $\phi_{\mathbf{f},S,\infty}$
is a non-zero scalar multiple of $\phi_{\varPhi, S,\infty}$.
Thus \eqref{e:arch comp} follows from
\eqref{e: j_infty}.
\qed
%
%
%
%
\subsection{Proof of Theorem~\ref{t: vector valued boecherer}}
Let us complete our proof of Theorem~\ref{t: vector valued boecherer}.
%
By Theorem~\ref{ref ggp}, we have
\begin{equation}\label{e: final1}
\frac{\left| B_{S,\Lambda, \psi}\left(\phi_{\varPhi,S}\right)\right|^2}{
\langle \phi_{\varPhi,S},\phi_{\varPhi,S}\rangle}
=\frac{\mathcal CJ_\infty}{2^{c-3}}
\cdot
\frac{L\left(1\slash 2,\pi\left(\varPhi\right) \times \mathcal{AI} \left(\Lambda \right)\right)}{
L\left(1,\pi\left(\varPhi\right),\mathrm{Ad}\right)}
\cdot \prod_{p | N} J_p
\end{equation}
where $c$ is as stated in Theorem~\ref{t: vector valued boecherer}.
By \eqref{e: scalar vs vector} and \eqref{e: scalar bs1},
we have
\[
B_{S,\psi, \Lambda}\left(\phi_{\varPhi,S}\right)=2 \cdot e^{-2\pi\,\mathrm{tr}\left(S\right)}
\cdot
\mathcal B_\Lambda \left(\varPhi;E\right).
\]
Since $\langle \phi_{\varPhi,S},\phi_{\varPhi,S}\rangle
=C\left(Q_{S,\varrho}\right)\cdot
\langle \varPhi,\varPhi\rangle_\varrho$
by Lemma~\ref{l: norm constant}, we have
\begin{equation}\label{e: final2}
\frac{\left| \mathcal B_\Lambda \left(\varPhi;E\right)\right|^2}{
\langle \varPhi,\varPhi\rangle_\varrho}=
\frac{\left| B_{S,\Lambda, \psi}\left(\phi_{\varPhi,S}\right)\right|^2}{
\langle \phi_{\varPhi,S},\phi_{\varPhi,S}\rangle}
\cdot 2^{-2}e^{4\pi\,\mathrm{tr}\left(S\right)}
C\left(Q_{S,\varrho}\right).
\end{equation}
Thus by combining
\eqref{e: final1}, \eqref{e: final2} and \eqref{e:arch comp}, the
identity \eqref{e: vector valued boecherer} holds.
%
%
%
%
%
%
%
%
%
%
%
%
%
%
%
%
%
%
%
%
\section{Meromorphic continuation of $L$-functions for $\mathrm{SO}(5) \times \mathrm{SO}(2)$}\label{appendix c}
As we remarked in Remark~\ref{L-fct def rem}, 
here we show the meromorphic continuation of 
$L^S(s, \pi \times \mathcal{AI}(\Lambda))$
in Theorem~\ref{ggp SO},
when $\mathcal{AI} \left(\Lambda \right)$ is cuspidal and $S$ is a sufficiently large finite 
set of places of $F$ containing all archimedean places.
The following theorem  clearly suffices.
%
%
%
\begin{theorem}
\label{thm mero}
Let $\pi$ (resp. $\tau$) be an irreducible unitary cuspidal automorphic representation $\pi$ of $G_D(\mA)$ (resp. $\mathrm{GL}_2(\mA)$) 
with a trivial central character.
Then $L^S(s, \pi \times \tau)$ has a meromorphic continuation to $\mC$ and it is holomorphic at $s= \frac{1}{2}$
for a sufficiently large finite set $S$ of places of $F$ containing all archimedean places.
\end{theorem}
%
%
%
When $D$ is split, then $G_D\simeq G$ and
the theorem follows from Arthur~\cite{Ar}.
Hence from now on we assume that $D$ is non-split.

By \cite{JSLi}, for some $\xi$ and $\Lambda$,
$\pi$ has the $\left(\xi, \Lambda, \psi\right)$-Bessel period.
Thus we may use the the integral representation 
of the $L$-function for $G_D\times \mathrm{GL}_2$ introduced in \cite{Mo3}.
Then the meromorphic continuation of the Siegel Eisenstein series
on $\mathrm{GU}_{3,3}$, which is used in the integral
representation  is known
by the main theorem of Tan~\cite{Tan}
(see also \cite[Proposition~3.6.2]{PSS}).
Hence by the standard argument,
our theorem is reduced to the analysis of the local zeta integrals.
Meanwhile the non-archimedean local integrals are already
studied in
\cite[Lemma~5.1]{Mo3}.
Hence it suffices for us to investigate the archimedean ones.
Since the case when $E_v$ is a quadratic extension
field of $F_v$ is similar to, and indeed simpler than,
 the split case, here we only consider the split case.
%
%
%

Let us briefly recall our local zeta integral (see \cite[(28)]{Mo3}).
Let $v$ be an archimedean place of $F$.
Since we consider the split case, $D_v$ is split
and we may assume that $G_{D}\left(F_v\right)=G\left(F_v\right)
=\mathrm{GSp}_2\left(F_v\right)$
and $\xi=\begin{pmatrix}1&0\\0&-1\end{pmatrix}$.
Then we have
\[
T_\xi\left(F_v\right) = 
\left\{g\in\mathrm{GL}_2\left(F\right)\mid
{}^tg\xi g=\det\left(g\right)\xi \right\}
=
\left\{ \begin{pmatrix} x&y\\ y&x\end{pmatrix} \in \mathrm{GL}_2(F) \right\}.
\]

In what follows, we omit the subscript $v$ from any object
in order to simplify the notation.
Let $\Lambda$ be a unitary character of $F^\times$. Then we regard $\Lambda$ as a character of $T_\xi(F)$ by 
\[
\Lambda  \begin{pmatrix} x&y\\ y&x\end{pmatrix} = \Lambda \left(\frac{x+y}{x-y} \right)
\quad\text{for $\begin{pmatrix} x&y\\ y&x\end{pmatrix}
\in T_S\left(F\right)$}.
\]
For a non-trivial
character $\psi$ of $F$, let $\mathcal{B}_{\xi, \Lambda, \psi}(\pi)$ denote the $(\xi, \Lambda, \psi)$-Bessel model of $\pi$, 
i.e. the space 
of functions $B:G(F)\to \mathbb C$ such that 
\[
B(tug) = \Lambda(t) \psi_\xi(u) B(g) \quad 
\text{for $t \in T_\xi(F)$, $u \in N\left(F\right)$
and $g\in G\left(F\right)$},
\]
which affords $\pi$ by the right regular representation.
Let $\mathcal{W}(\tau)$ denote the Whittaker model of $\pi$, i.e. 
the space of functions $W:\mathrm{GL}_2(F)\to\mathbb C$ such that
\[
W\left( \begin{pmatrix} 1&x\\ 0&1\end{pmatrix}g\right) = \psi(- x) W(g)\quad
\text{for $x\in F$ and $g\in \mathrm{GL}_2\left(F\right)$},
\]
which affords $\tau$ by the right translation.
Let $G_0\left(F\right)= \mathrm{GL}_2(F)  \times G(F)$
and we regard $G$ as a subgroup of $\mathrm{GL}_6(F)$ by the embedding 
\[
\iota : G_0  \ni
 \left(\begin{pmatrix}a&b\\ c&d\end{pmatrix}, \begin{pmatrix}A&B\\ C&C\end{pmatrix}  \right)
 \hookrightarrow
 \begin{pmatrix}a&0&b&0\\ 0&A&0&B\\ c&0&d&0\\ 0&C&0&D\end{pmatrix} \in \mathrm{GL}_6(F).
\]
Let us define a  subgroup $H_0$ of $G_0$ by 
\[
H_0(F) = \left\{ \nu(h)\left(\begin{pmatrix}1&\mathrm{tr}(\xi X)\\0 &1 \end{pmatrix}, \begin{pmatrix}h&0\\ 0&\det h \cdot {}^{t}h^{-1} \end{pmatrix}
\begin{pmatrix}1_2&X\\ 0&1_2 \end{pmatrix}  \right) 
\mid X={}^{t}X,  h \in T_\xi\left(F\right)\right\}
\]
where
\[
\nu(h) = x-y \quad \text{for
$ h= \begin{pmatrix} x&y\\y&x\end{pmatrix}\in T_\xi\left(F\right)$}.
\]
Let $P_3$ be the maximal parabolic subgroup of $\mathrm{GL}_6$ defined by
\[
P_3= \left\{ \begin{pmatrix} h_1&X\\0&h_2\end{pmatrix} 
: h_1,h_2\in\mathrm{GL}_3
\right\}.
\]
Then we consider a principal series representation
\[
I(\Lambda, s) =\left\{ f_s : \mathrm{GL}_6(F) \rightarrow \mC 
\mid
 f_s \left( \left(\begin{smallmatrix} h_1&X\\0&h_2\end{smallmatrix}
 \right) h \right) =  \Lambda\left(\frac{\det h_1}{\det h_2}\right)
\left|\frac{\det h_1}{\det h_2} \right|^{3s+\frac{3}{2}}f_s(h) \right\}.
\]
For $f_s \in I\left(\Lambda, s\right) $,
$B \in \mathcal{B}_{\xi, \Lambda, \psi}\left(\pi\right)$ and $W \in \mathcal{W}\left(\tau\right)$, our local zeta integral $Z(f_s, B, W)$ 
is given by
\[
Z(f_s, B, W) = \int_{Z_0\left(F\right)H_0\left(F\right) 
\backslash G_0\left(F\right)} f_s\left(\theta_0 \,
\iota\left(g_1, g_2\right)\right) B(g_2)W(g_1) \, dg_1\,dg_2
\]
where $Z_0$ denote the center of $G_0$ and 
\[
\theta_{0}
=
\begin{pmatrix}
0&0&0&0&0&-1\\
0&1&0&0&0&0\\
1&0&0&0&0&0\\
1&-1&1&0&0&0\\
0&0&0&0&1&-1\\
0&0&0&1&0&-1
\end{pmatrix}.
\]
As explained above, Theorem~\ref{thm mero}
follows by the standard argument
if we prove the following lemma.
%
%
%
\begin{lemma}\label{meromorphy lemma}
Let $s_0$ be an arbitrary point in  $\mC$.
Then we may choose
$f_s, B$ and $W$ so that $Z(f_s, B, W)$ has a meromorphic continuation to $\mC$
and is holomorphic and non-zero at $s=s_0$.
\end{lemma}
%
%
%
\begin{proof}
For $\varphi \in C_c^\infty(\mathrm{GL}_6(F))$, we may 
define $P_s[\varphi]  \in I(\Lambda, s)$ by
\begin{multline*}
P_s[\varphi](h) = \int_{\mathrm{GL}_3(F)} \int_{\mathrm{GL}_3(F)} \int_{\mathrm{Mat}_{3 \times 3}(F)} 
\varphi \left(\begin{pmatrix}h_1 &0\\ 0&h_2 \end{pmatrix} 
\begin{pmatrix}1_3 &X\\ &1_3 \end{pmatrix} h \right) 
\\
\times
\left|\frac{\det h_1}{\det h_2} \right|^{-3s+\frac{3}{2}}
\Lambda\left(\frac{\det h_1}{\det h_2}\right)^{-1}
  \, dh_1 \, dh_2 \, dX.
\end{multline*}
In what follows we construct  $\varphi$ of a special form,
whose support is contained 
in the open double coset $P_3\left(F\right) \theta_0 G_0\left(F\right)$
in $\mathrm{GL}_6\left(F\right)$.

Let $B_0$ be the group of upper triangular matrices in $\mathrm{GL}_2$,
and, $P_0$  the mirabolic subgroup of $\mathrm{GL}_2$, i.e.
\[
P_0\left(F\right) = \left\{  \begin{pmatrix} a&b\\ 0&1\end{pmatrix} 
\mid a \in F^\times, \, b \in F\right\}.
\]
We define a subgroup $M_0$ of $G$ by
\[
M_0\left(F\right)=\left\{  \begin{pmatrix}h&0\\0& \lambda \cdot {}^{t}h^{-1} \end{pmatrix} \mid \lambda \in F^\times, h \in B_0\left(F\right) \right\}
\]
and $M = \iota\left(P_0, M_0\right)$.
Then by the Iwasawa decomposition for $G_0\left(F\right)$
and the inclusion
\begin{equation}
\label{intersection}
H_0\left(F\right) \subset G_0\left(F\right)
 \cap \theta_0^{-1}P_3\left(F\right)\theta_0  ,
\end{equation}
we have
\[
P_3\left(F\right) \theta_0 G_0\left(F\right) = 
P_3\left(F\right) \theta_0 M\left(F\right) K_0
\]
where $K_0$ is a maximal compact subgroup of $G_0\left(F\right)$.
We take $K_0=\iota\left(K_1,K_2\right)$
where $K_1$ (resp. $K_2$) is a maximal compact
subgroup of $\mathrm{GL}_2\left(F\right)$
(resp. $G\left(F\right)$).
By direct computations, we see that 
\[
\begin{cases}
\theta_0 \,N\left(F\right)\, \theta_0^{-1} \cap P_3\left(F\right) =\{ 1_6 \};
\\
\theta_0 \,M\left(F\right)\, \theta_0^{-1} \cap P_3\left(F\right)
 =\theta_0 \,A\left(F\right)\, \theta_0^{-1};
 \\
 \theta_0 \,K_0\,\theta_0^{-1} \cap P_3\left(F\right) =\{ 1_6 \},
 \end{cases}
\]
where
\[
A\left(F\right) =\left\{ \begin{pmatrix}a\cdot 1_3 &\\&1_3 \end{pmatrix} : a \in F^\times \right\}.
\]
Let us define subgroups $T_0$, $N_0$ of $G_0$ by
\begin{align*}
T_0\left(F\right)& = \left\{ \iota\left( \begin{pmatrix}a&\\ &1 \end{pmatrix}, \begin{pmatrix}x&&&\\ &y&&\\&&\lambda x^{-1}&\\ &&&\lambda y^{-1} \end{pmatrix} \right)
: x, y, \lambda \in F^\times 
\right\};
\\
N_0\left(F\right) &= \left\{ \iota \left( \begin{pmatrix}1&x\\ &1 \end{pmatrix}, \begin{pmatrix}1&y&&\\ &1&&\\&&1&\\ &&-y&1 \end{pmatrix}\right) : x, y \in F\right\}.
\end{align*}

Then for
$\varphi_1 \in C_c^\infty\left(N_0\left(F\right)\right)$, $\varphi_2 \in C_c^\infty\left(T_0\left(F\right)\right)$, $\varphi_3, \varphi_4 \in C_c^\infty\left(\mathrm{GL}_3\left(F\right)\right)$, 
$\varphi_5  \in C_c^\infty\left(\mathrm{Mat}_{3 \times 3}\left(F\right)
\right)$ and $\varphi_6 \in C_c^\infty(K_0)$, 
we may construct  $\varphi^\prime\in C_c^\infty(\mathrm{GL}_6(F))$, whose support is contained in 
$P_3\left(F\right) \theta_0 G_0\left(F\right)$, by
\begin{multline*}
\varphi^\prime \left(\begin{pmatrix}h_1&0\\ 0&h_2 \end{pmatrix}\begin{pmatrix}1_3&X\\ 0&1_3 \end{pmatrix} \theta_0\, n_0\, t_0\, k \right)
\\
=\varphi_6(k) \varphi_3(h_1) \varphi_4(h_2) \varphi_5(X)  \varphi_1(n_0 ) \int_{A\left(F\right)} \varphi_2(t_0\, a) \, d^\times a
\end{multline*}
where  $n_0 \in N_0\left(F\right)$, $t_0 \in T_0\left(F\right)$ 
and $k \in K_0$.

Then the local zeta integral
$Z(P_s[\varphi^\prime], B, W)$ is written as
\begin{multline*}
Z(P_s[\varphi^\prime], B, W)=\int \varphi^\prime \left(\begin{pmatrix}h_1 &0\\ 0&h_2 \end{pmatrix} \begin{pmatrix}1_3 &X\\ &1_3 \end{pmatrix} 
 \iota(n_{0,1}, n_{0,2})\iota(t_{0,1}, t_{0,2}) \iota(k_1, k_2) \right) 
\\
\times
\left|\frac{\det h_1}{\det h_2} \right|^{-3s+\frac{3}{2}}
\Lambda\left(\frac{\det h_1}{\det h_2}\right) ^{-1}
W(n_{0,1} t_{0,1} k_1) B(n_{0,2} t_{0, 2}k_2)
\, dh_1 \, dh_2 \, dX \, dn_0 \, dt_0\, dk
\\
=
\int \varphi_6(\iota(k_1, k_2)) \varphi_3(h_1) \varphi_4(h_2) \varphi_5(X)  \varphi_1(n_0 ) \varphi_2(t_0a)
\left|\frac{\det h_1}{\det h_2} \right|^{-3s+\frac{3}{2}}
\Lambda\left(\frac{\det h_1}{\det h_2}\right) ^{-1} 
\\
\qquad\qquad\times W(n_{0,1} t_{0,1} k_1) B(n_{0,2} t_{0, 2}k_2) 
 \, d^\times a \, dh_1 \, dh_2 \, dX \, dn_0 \, dt_0\, dk
\\
=
\int \varphi_6(\iota(k_1, k_2)) \varphi_3(h_1) \varphi_4(h_2) \varphi_5(X)  \varphi_1(n_0 ) \varphi_2(t_0)
\left|\frac{\det h_1}{\det h_2} \right|^{-3s+\frac{3}{2}}
\Lambda\left(\frac{\det h_1}{\det h_2}\right) ^{-1}  
\\
\times \Lambda(\lambda)
|\lambda|^{3s-\frac{9}{2}}\,W \left(n_{0,1} \begin{pmatrix}\lambda&0\\ 0&1 \end{pmatrix} t_{0,1} k_1 \right) B \left(n_{0,2} \begin{pmatrix} \lambda \cdot 1_2&0\\ 0&1_2 \end{pmatrix} t_{0, 2}k_2 \right) 
 \\
  d^\times \lambda \, dh_1 \, dh_2 \, dX \, dn_0 \, dt_0\, dk
\end{multline*}
where we write $n_0 =  \iota(n_{0,1}, n_{0,2}) \in N_0\left(F\right)$, $t_0 =  \iota(t_{0,1}, t_{0,2}) \in T_0\left(F\right)$ and 
$k = \iota(k_1, k_2) \in K_0$. 
Since we may vary $\varphi_i$ ($1\le i\le 6$), our assertion
in Lemma~\ref{meromorphy lemma} follows from the 
same assertion for
the integral
\begin{equation}
\label{mero proof e:last}
\int_{F^\times} \Lambda(\lambda)|\lambda|^{3s-\frac{9}{2}}\,B \begin{pmatrix}\lambda \cdot 1_2&0\\0 &1_2 \end{pmatrix}
\, W \begin{pmatrix}\lambda&0\\ 0&1 \end{pmatrix} d^\times \lambda.
\end{equation}

For any $\phi \in C_c^\infty(F^\times)$,
there exists $W_\phi\in W\left(\tau\right)$
such that $W_\phi\begin{pmatrix}a&0\\0&1\end{pmatrix}=
\phi\left(a\right)$
by the theory of Kirillov model for $\mathrm{GL}_2(\mR)$ by Jacquet~\cite[Proposition~5]{Jac} and 
for $\mathrm{GL}_2(\mC)$ by Kemarsky~\cite[Theorem~1]{Kem}.
Thus our assertion clearly holds for the integral
\eqref{mero proof e:last}.
\end{proof}
%
%
%
%
%
%
%
%
%
%

%
%
%
%
%
%
%
%
%
%
%
%
%
%
%
%
%
%
%
%
\begin{theindex}
\item $M_D, N_D$  \pageref{d:N_D}
\item $G_D$ \pageref{e: G_D}
\item $G$ \pageref{gsp}
\item $G_D^1$ \pageref{G_D^1}
\item $G^1$ \pageref{G^1}
   \item $T_{\xi}$ \pageref{e: T_xi}
   \item $G_D^+$ \pageref{e: G^+_D}
   \item $B_{\xi, \Lambda, \psi}$ \pageref{e: dependency}
   \item $\mathcal{AI}(\Lambda)$ \pageref{e: partial non-vanishing}
   \item $\mathcal{B}_\Lambda\left(\varPhi , E\right) $ \pageref{mathcal B Phi E}
   \item $J_m$ \pageref{d: J_m}
         \item $\mathrm{GO}_{n+2, n}, \mathrm{GSO}_{n+2, n}$ \pageref{d: GO_n+2,n}
   \item $\mathrm{GU}_{3, D}, \mathrm{GSU}_{3, D}$ \pageref{def of gu_3,D}
   \item $\mathrm{GU}_{4, \varepsilon}$ \pageref{e: gu(2,2) or gu(3,1)}
   \item $\Phi_D$ \pageref{acc isom1}
   \item $\Phi$ \pageref{acc isom2}
   \item $M$, $N$ \pageref{d:N}
   \item $T_S$ \pageref{T_S}
      \item $B_{S, \Lambda, \psi}$ \pageref{Beesel def gsp}
   \item $M_{3, D}$ $N_{3, D}$ \pageref{d:N3,D}
      \item $M_{X, D}$ \pageref{M_X D}
      \item $\mathcal{B}_{X, \chi, \psi}$ \pageref{Besse def gsud}
   \item $\mathcal{B}_{X, \chi, \psi}^D$ \pageref{Besse def gsud}
   \item $M_{4,2}$ $N_{4,2}$ \pageref{d:N4,2}
\item $M_{X}$ \pageref{M_X}
      \item $\alpha_{\chi, \psi_N}(\phi, \phi^\prime)$ \pageref{e: local integral 1}
         \item Type I-A, Type I-B \pageref{type}
      \item $W^{\psi_U}$ \pageref{W U}
   \item $\mathcal{W}^{\psi_U}$ \pageref{mathcal W psi U}
   \item $\mathcal{W}^\circ_{G, v}, \mathcal{W}_{G, v}$ \pageref{mathcal W_G}
      \item $\mathcal{L}_v^\circ(\phi_v, f_v), \mathcal{L}_v(\phi_v, f_v)$ \pageref{mathcal L}
         \item $W_{\psi_{U_G}}$ \pageref{W U_G}
      \item $\phi_{\Phi, S}$ \pageref{phi_Phi S}
\end{theindex}
%
%
%
%
%
%
%
%
%
%
%

\begin{thebibliography}{99}
\bibitem{AB}
J. Adams and D. Barbasch, 
\emph{Reductive dual pair correspondence for complex groups.}
J. Funct. Anal. \textbf{132} (1995), no. 1, 1--42. 
\bibitem{AC}
J. Arthur and L. Clozel, 
\emph{Simple Algebras, Base Change and the Advanced Theory of the Trace Formula.} 
Ann. Math. Studies \textbf{120} (1989), Princeton, NJ.
\bibitem{Ar}
J. Arthur, 
\emph{The endoscopic classification of representations. Orthogonal and symplectic groups.}
Amer. Math. Soc. Colloq. Publ.  \textbf{61}, xviii+590 pp.
Amer. Math. Soc., Providence, RI,  2013.
\bibitem{At}
H. Atobe, 
\emph{On the uniqueness of generic representations in an $L$-packet.} 
Int. Math. Res. Not. IMRN 2017, no. 23, 7051--7068.
\bibitem{AG}
H. Atobe and W. T. Gan,
\emph{Local theta correspondence of tempered representations and Langlands parameters.}
Invent. Math. \textbf{210} (2017), no. 2, 341--415. 
\bibitem{BP1}
R.  Beuzart-Plessis,
\emph{La conjecture locale de Gross-Prasad pour les repr\'{e}sentations temp\'{e}r\'{e}es des groupes unitaires.} 
M\'{e}m. Soc. Math. Fr. (N.S.) 2016, no. 149, vii+191 pp.
\bibitem{BP2}
R.  Beuzart-Plessis,
\emph{A local trace formula for the Gan-Gross-Prasad conjecture for unitary groups: the Archimedean case. }
Ast\'{e}risque No. \textbf{418} (2020), ix + 305 pp. 
\bibitem{BPC}
R. Beuzart-Plessis and P.-H. Chaudouard,
\emph{The global Gan-Gross-Prasad conjecture for unitary groups. II. From Eisenstein series to Bessel periods.}
Preprint, 
arXiv:2302.12331
\bibitem{BPCZ}
R. Beuzart-Plessis, P.-H. Chaudouard and M. Zydor, 
\emph{The global Gan-Gross-Prasad conjecture for unitary groups: the endoscopic case.}
Publ. Math. Inst. Hautes \'{E}tudes Sci. \textbf{135} (2022), 183--336.
\bibitem{BPLZZ}
R. Beuzart-Plessis, Y. Liu, W. Zhang, X. Zhu,
\emph{Isolation of cuspidal spectrum, with application to the Gan-Gross-Prasad conjecture.}
Ann. of Math. (2) \textbf{194} (2021), no. 2, 519--584.
\bibitem{Bl}
D. Blasius, 
\emph{Hilbert modular forms and the Ramanujan conjecture.} 
Noncommutative Geometry and Number Theory, Aspects Math. E37, Vieweg, Wiesbaden 2006, 35--56.
\bibitem{Blo}
V. Blomer,
\emph{Spectral summation formula for $\mathrm{GSp}(4)$ and moments of spinor $L$-functions.} 
J. Eur. Math. Soc. (JEMS) \textbf{21} (2019), no. 6, 1751--1774.
\bibitem{Bo}
S. B\"{o}cherer, 
\emph{Bemerkungen \"uber die Dirichletreihen von Koecher und 
Maa\ss.}
Mathematica Gottingensis, G\"ottingen, 
vol. \textbf{68}, p. 36 (1986).
\bibitem{CFK}
Y. Cai, S. Friedberg and E. Kaplan,
\emph{Doubling constructions: local and global theory, with an application to global functoriality for non-generic cuspidal representations.}
Preprint,
arXiv:1802.02637
\bibitem{Car}
A.  Caraiani, 
\emph{Local-global compatibility and the action of monodromy on nearby cycles.} 
Duke Math. J. \textbf{161} (2012), no. 12, 2311--2413.
\bibitem{PSC}
P.-S. Chan,
\emph{Invariant representations of GSp(2) under tensor product with a quadratic character.}
Mem. Amer. Math. Soc. \textbf{204} (2010), no.957, vi+172 pp.
\bibitem{CI}
S.-Y. Chen and A. Ichino, 
\emph{On Petersson norms of generic cusp forms and special values of adjoint $L$-functions for $\mathrm{GSp}_4$.}
Amer. J. Math. \textbf{145} (2023), 899--993.
\bibitem{CH}
M. Chida,  and  M.-L.. Hsieh,
\emph{Special values of anticyclotomic $L$-functions for modular forms.}
J. reine angew. Math., \textbf{741} (2018), 87--131.
\bibitem{CKPSS}
J. W. Cogdell, H. H. Kim, I. I. Piatetski-Shapiro, and F. Shahidi, 
\emph{Functoriality for the classical groups.} Publ. Math. Inst. Hautes \'{E}tudes Sci. \textbf{99} (2004), 163--233.
\bibitem{Co}
A. Corbett, 
\emph{A proof of the refined Gan-Gross-Prasad conjecture for non-endoscopic Yoshida lifts.} 
Forum Math. \textbf{29} (2017), no. 1, 59--90. 
\bibitem{DPSS}
M. Dickson, A. Pitale, A. Saha and R. Schmidt,
\emph{Explicit refinements of B\"{o}cherer's conjecture for Siegel modular forms of squarefree level.} 
J. Math. Soc. Japan \textbf{72} (2020), no. 1, 251--301. 
\bibitem{Dummigan}
N. Dummigan,
\emph{Congruences of Saito-Kurokwa lifts and denominators of central special
$L$-values.}
Glasg. Math. J. \textbf{64} (2022), no. 2, 504--525. 
\bibitem{Fu}
M. Furusawa,
\emph{On the theta lift from $\mathrm{SO}_{2n+1}$ to $\widetilde{\mathrm{Sp}}_n$.}
J. Reine Angew. Math. \textbf{466} (1995), 87--110.
\bibitem{FuMa}
M. Furusawa and K. Martin, 
\emph{On central critical values of the degree four L-functions for $\mathrm{GSp}(4)$: the fundamental lemma.II.} 
Amer. J. Math. \textbf{133} (2011), 197--233.
\bibitem{FuMaS}
M. Furusawa and K. Martin, 
\emph{On central critical values of the degree four L-functions for $\mathrm{GSp}(4)$: the fundamental lemma.III.} 
Memoirs of the AMS, Vol. 225, No. 1057 (2013), x+134pp. 
\bibitem{FM0}
M. Furusawa and K. Morimoto,
\emph{Shalika periods on $\mathrm{GU}(2, 2)$.} Proc. Amer. Math.
Soc. \textbf{141} (2013), no. 12, 4125--4137.
\bibitem{FM1}
M. Furusawa and K. Morimoto,
\emph{On special Bessel periods and the Gross--Prasad conjecture for $\mathrm{SO}(2n+1) \times \mathrm{SO}(2)$.} 
Math. Ann. \textbf{368} (2017), no. 1-2, 561--586.
\bibitem{FM2}
M. Furusawa and K. Morimoto,
\emph{Refined global Gross-Prasad conjecture on special Bessel periods and B\"ocherer's conjecture.}
J. Eur. Math. Soc. (JEMS) \textbf{23}, 1295--1331 (2021).
\bibitem{FM3}
M. Furusawa and K. Morimoto,
\emph{On the Gan-Gross-Prasad conjecture and its
refinement for $\left(\mathrm{U}\left(2n\right), \mathrm{U}\left(1\right)\right)$.}
Preprint,
arXiv:2205.09471
\bibitem{FS}
M. Furusawa and J. Shalika. 
\emph{On central critical values of the degree four L-functions for $\mathrm{GSp}(4)$: the fundamental lemma.} 
Mem. Amer. Math. Soc. \textbf{164} (2003), no. 782, x+139 pp.
\bibitem{Gan}
W. T. Gan, 
\emph{The Saito-Kurokawa space of $\mathrm{PGSp}_4$ and its transfer to inner forms.}
In: Eisenstein series and applications, Progr. Math.
\textbf{258}, pp. 87--123.
Birkh\"auser Boston, Boston, MA (2008).
\bibitem{GGP}
W. T. Gan, B. Gross and D. Prasad,
\emph{Symplectic local root numbers, central critical L values, and restriction problems in the representation theory of classical groups.} 
Sur les conjectures de Gross et Prasad. I. 
Ast\'{e}risque No. 346 (2012), 1--109.
\bibitem{GSun}
W. T. Gan and B. Sun,
\emph{The Howe duality conjecture: quaternionic case.} 
Representation theory, number theory, and invariant theory, 175--192, Progr. Math., 323, Birkh\"{a}user/Springer, Cham, 2017.
\bibitem{GT10}
W. T. Gan and S. Takeda,
\emph{On Shalika periods and a theorem of Jacquet-Martin.}
Amer. J. Math. \textbf{132} (2010), no.2, 475--528.
\bibitem{GT11}
W. T. Gan and S. Takeda,
\emph{The local Langlands conjecture for $\mathrm{GSp}(4)$.}
Ann. of Math. (2) \textbf{173} (2011), no. 3, 1841--1882.
\bibitem{GT0}
W. T. Gan and S. Takeda,
\emph{Theta correspondences for $\mathrm{GSp}(4)$.} 
Represent. Theory \textbf{15} (2011), 670--718.
\bibitem{GT}
W. T. Gan and S. Takeda,
\emph{A proof of the Howe duality conjecture.} 
J. Amer. Math. Soc. \textbf{29} (2016), no. 2, 473--493.
\bibitem{GaTan}
W. T. Gan and W. Tantono,
\emph{The local Langlands conjecture for $\mathrm{GSp}(4)$, II: The case of inner forms.}
Amer. J. Math. \textbf{136} (2014), no. 3, 761--805. 
\bibitem{GI0}
W. T. Gan and A. Ichino,
\emph{On endoscopy and the refined Gross-Prasad conjecture for $(\mathrm{SO}_5,\mathrm{SO}_4)$.} 
J. Inst. Math. Jussieu \textbf{10} (2011), no. 2, 235--324.
\bibitem{GI1}
W. T. Gan and A. Ichino,
\emph{Formal degrees and local theta correspondence.} 
Invent. Math. \textbf{195} (2014), no. 3, 509--672.
\bibitem{GQT}
W. T. Gan, Y. Qiu and S. Takeda,
\emph{The regularized Siegel-Weil formula (the second term identity) and the Rallis inner product formula.} 
Invent. Math. \textbf{198} (2014), no. 3, 739--831.
\bibitem{GRS97}
D. Ginzburg, S. Rallis and D. Soudry,
\emph{Periods, poles of L-functions and symplectic-orthogonal theta lifts.} 
J. Reine Angew. Math. \textbf{487} (1997), 85--114. 
\bibitem{GRS}
D. Ginzburg, S. Rallis and D. Soudry,
\emph{The descent map from automorphic representations of GL(n) to classical groups.} 
World Scientific Publishing Co. Pte. Ltd., Hackensack, NJ, 2011. x+339 pp. 
\bibitem{GP1}
B. Gross and D. Prasad, 
\emph{On the decomposition of a representation of $\mathrm{SO}_n$ when restricted to $\mathrm{SO}_{n-1}$.} 
Canad. J. Math. \textbf{44}, 974--1002 (1992).
\bibitem{GP2}
B. Gross and D. Prasad, 
\emph{On irreducible representations of $\mathrm{SO}_{2n+1} \times \mathrm{SO}_{2m}$.} 
Canad. J. Math. \textbf{46}, 930--950 (1994).
\bibitem{HK}
M. Harris and S. Kudla,
\textit{Arithmetic automorphic forms for the nonholomorphic discrete series of ${\rm GSp}(2)$.}  
Duke Math. J.  \textbf{66}  (1992),  no. 1, 59--121. 
\bibitem{HST}
M. Harris, D.Soudry, R.Taylor, 
\emph{$\ell$-adic representations associated to modular forms over imaginary quadratic fields. I. Lifting to $\mathrm{GSp}_4(\mQ)$.} 
Invent. Math. \textbf{112} (1993), no. 2, 377--411.
\bibitem{Ha}
N. Harris, \emph{The refined Gross-Prasad conjecture for unitary groups.} 
Int. Math. Res. Not. IMRN (2014), no. 2, 303--389.
\bibitem{HII1}
K. Hiraga, A. Ichino and T. Ikeda,
\emph{Formal degrees and adjoint $\gamma$-factors.} 
J. Amer. Math. Soc. \textbf{21} (2008), no. 1, 283--304.
\bibitem{HII2}
K. Hiraga, A. Ichino and T. Ikeda,
\emph{Correction to: ``Formal degrees and adjoint $\gamma$-factors''.} 
J. Amer. Math. Soc. \textbf{21} (2008), no. 4, 1211--1213.
\bibitem{HiSa}
K. Hiraga and H. Saito, 
\emph{On $L$-packets for inner forms of $\mathrm{SL}_n$.} 
Mem. Amer. Math. Soc. \textbf{215} (2012), no. 1013, vi+97 pp.
\bibitem{Ho1}
R. Howe
\emph{Transcending classical invariant theory.}
J. Amer. Math. Soc. \textbf{2}, 535--552 (1989).
\bibitem{HN1}
M.-L. Hsieh and K. Namikawa,
\emph{Bessel periods and the non-vanishing of Yoshida lifts modulo a prime.}
Math. Z. \textbf{285}, 851--878 (2017).
\bibitem{HN2}
M.-L. Hsieh and K. Namikawa,
\emph{Inner product formula for Yoshida lifts.} 
Ann. Math. Qu\'{e}. \textbf{42} (2018), no. 2, 215--253.
 \bibitem{HY}
 M.-L. Hsieh and S. Yamana,
\emph{Bessel periods and anticyclotomic p-adic spinor $L$-functions.}
To appear in Trans. Amer. Math. Soc.,
DOI: 10.1090/tran/9143
\bibitem{Ich2}
A. Ichino, 
\emph{Trilinear forms and the central values of triple product L-functions.} 
Duke Math. J. Volume \textbf{145}, Number 2 (2008), 281--307.
\bibitem{II}
A. Ichino and T. Ikeda,
\emph{On the periods of automorphic forms on special orthogonal groups
and the Gross-Prasad conjecture.}
Geom. Funct. Anal. \textbf{19}, 1378--1425 (2010).
\bibitem{IP}
A. Ichino and K. Prasanna,
\emph{Periods of quaternionic Shimura varieties. I.}
Contemp. Math., \textbf{762} American Mathematical Society, [Providence], RI, [2021], vi+214 pp.
\bibitem{Ishimoto}
H. Ishimoto,
\emph{The endoscopic classification of representations of non-quasi-split odd special orthogonal groups.}
Preprint,  arXiv:2301.12143
\bibitem{ILS}
H. Iwaniec, W. Luo and P. Sarnak,
\emph{Low Lying zeros of families of $L$-functions.}
Inst. Hautes \'{E}tudes Sci. Publ. Math.  \textbf{91} (2000),
55--131 (2000).
\bibitem{Jac}
H. Jacquet, 
\emph{Distinction by the quasi-split unitary group,} 
Isr. J. Math. \textbf{178} (1) (2010) 269--324.
\bibitem{JSZ}
D. Jiang, B. Sun, C.-B. Zhu, 
\emph{Uniqueness of Bessel models: the Archimedean case.} 
Geom. Funct. Anal. \textbf{20} (2010), no. 3, 690--709.
\bibitem{JZ}
D. Jiang and L. Zhang,
\emph{Arthur parameters and cuspidal automorphic modules of classical groups.}
 Ann. of Math. (2) \textbf{191} (2020), no. 3, 739--827.
\bibitem{Jo}
A. Jorza, 
\emph{Galois representations for holomorphic Siegel modular forms.} 
Math. Ann. \textbf{355} (2013), no. 1, 381--400.
\bibitem{KMSW}
T. Kaletha, A. Minguez, S. W. Shin, and P. J. White,
\emph{Endoscopic classification of representations: Inner forms of unitary groups.}
Preprint,  arXiv:1409.3731
\bibitem{Kem}
A. Kemarsky, 
\emph{A note on the Kirillov model for representations of $\mathrm{GL}_n(\mC)$.} 
C. R. Math. Acad. Sci. Paris \textbf{353} (2015), no. 7, 579--582.
\bibitem{Ku}
S. Kudla,
\emph{Splitting metaplectic covers of dual reductive pairs.} 
Israel J. Math. \textbf{87} (1994), 361--401.
\bibitem{Ku2}
S. Kudla,
\emph{On the local theta-correspondence. }
Invent. Math. \textbf{83} (1986), no. 2, 229--255.
\bibitem{Kutz}
P. Kutzko,
 \emph{The Langlands conjecture for $\mathrm{Gl}_2$ of a local field.} 
 Ann. of Math. (2) \textbf{112} (1980), no. 2, 381--412.
\bibitem{Lan}
R. P. Langlands, 
\emph{On the classification of irreducible representations of real algebraic groups, Representation theory and harmonic analysis on semisimple Lie groups.}
101--170, Math. Surveys Monogr., 31, Amer. Math. Soc.,
Providence, RI, 1989.
\bibitem{LM}
E. Lapid and Z. Mao,
\emph{A conjecture on Whittaker-Fourier coefficients of cusp forms.} 
J. Number Theory \textbf{146} (2015), 448--505.
\bibitem{LM17}
E. Lapid and Z. Mao,
\emph{On an analogue of the Ichino-Ikeda conjecture for Whittaker coefficients on the metaplectic group.}
Algebra Number Theory \textbf{11} (2017), no.3, 713--765.
\bibitem{LR}
E. Lapid and S. Rallis,
\emph{On the local factors of representations of classical groups.} In: Automorphic
representations, L-functions and applications: progress and prospects, Ohio State Univ. Math. Res.
Inst. Publ. 11, pp. 309--359. de Gruyter, Berlin (2005)
\bibitem{JSLi}
J.-S. Li, 
\emph{Nonexistence of singular cusp forms.}
Compositio Math. \textbf{83} (1992), no. 1, 43--51. 
\bibitem{LPTZ}
J.-S. Li,  A. Paul, E.-C. Tan and C.-B. Zhu,
\emph{The explicit duality correspondence of $(\mathrm{Sp}(p,q),\mathrm{O}^\ast(2n))$.}
J. Funct. Anal. \textbf{200} (2003), no. 1, 71--100. 
\bibitem{Liu2}
Y. Liu,
\emph{Refined Gan-Gross-Prasad conjecture for Bessel periods.}
J. Reine Angew. Math. \textbf{717} (2016) 133--194.
\bibitem{Luo}
Z. Luo,
\emph{A local trace formula for the local Gan-Gross-Prasad conjecture
for special orthogonal groups.}
Preprint, arXiv:2009.13947
\bibitem{MW}
K. Martin and D. Whitehouse, 
\emph{Central values and toric periods for $\gl\left(2\right)$.}
Int. Math. Res. Not. IMRN 2009, 141--191.
\bibitem{MR}
P. Michel and D. Ramakrishnan,
\emph{Consequences of the Gross-Zagier formulae: stability of average 
$L$-values, subconvexity, and non-vanishing mod $p$.}
Number theory, analysis and geometry, 437--459, Springer, New York, 2012.
\bibitem{Moe}
C. M\oe glin, 
\emph{Correspondance de Howe pour les paires reductives duales: quelques calculs dans le cas archim\'{e}dien.}
J. Funct. Anal. \textbf{85} (1989), no. 1, 1--85. 
\bibitem{MVW}
C. M\oe glin, M.-F. Vigneras, J.-L. Waldspurger,
\emph{Correspondances de Howe sur un corps $p$-adique.} 
Lecture Notes in Mathematics, 1291. Springer-Verlag, Berlin, 1987. viii+163 pp.
\bibitem{Mok}
C. P. Mok,
\emph{Endoscopic classification of representations of quasi-split unitary groups.} 
Mem. Amer. Math. Soc. \textbf{235} (2015), no. 1108, vi+248 pp.
\bibitem{Mo}
K. Morimoto,
\emph{On the theta correspondence for $(\mathrm{GSp}(4),\mathrm{GSO}(4,2))$ and Shalika periods.} 
Represent. Theory \textbf{18} (2014), 28--87.
\bibitem{Mo3}
K. Morimoto,
\emph{On  $L$-functions for quaternion unitary groups of degree $2$ and $\mathrm{GL}(2)$ 
(with an Appendix by M. Furusawa and A. Ichino).}  Int. Math. Res. Not. IMRN 2014, no. 7, 1729--1832. 
\bibitem{Pau}
A. Paul, 
\emph{Howe correspondence for real unitary groups.} 
J. Funct. Anal. \textbf{159} (1998), no. 2, 384--431.
\bibitem{Pau2}
A. Paul, 
\emph{Howe correspondence for real unitary groups. II.} 
Proc. Amer. Math. Soc. \textbf{128} (2000), no. 10, 3129--3136.
\bibitem{Pau3}
A. Paul, 
\emph{On the Howe correspondence for symplectic--orthogonal dual pairs.}
J. Funct. Anal. \textbf{228} (2005), no. 2, 270--310. 
\bibitem{PSR}
I. I. Piatetski-Shapiro and S. Rallis, 
\emph{$L$-functions for the classical groups.} 
in Explicit constructions of automorphic L-functions, Lecture Notes in Mathematics, Volume 1254, 1--52.
\bibitem{PSS}
A. Pitale, A. Saha and R. Schmidt,
\emph{Transfer of Siegel cusp forms of degree 2.}
Mem. Amer. Math. Soc. (2014), \textbf{232} (1090).
\bibitem{PSS2}
A. Pitale, A. Saha and R. Schmidt,
\emph{Simple supercuspidal representations of $\mathrm{GSp}_4$
and test vectors.}
Preprint, arXiv:2302.05148
\bibitem{PR}
D. Prasad and D. Ramakrishnan,
\emph{On the global root numbers of $\mathrm{GL}(n) \times \mathrm{GL}(m)$.}
Automorphic forms, automorphic representations, and arithmetic (Fort Worth, TX, 1996), 311--330, 
Proc. Sympos. Pure Math., 66, Part 2, Amer. Math. Soc., Providence, RI, 1999. 
\bibitem{PT}
D. Prasad and R. Takloo-Bighash,
\emph{Bessel models for $\mathrm{GSp}(4)$.}
J. Reine Angew. Math. \textbf{655} (2011), 189--243.
\bibitem{RS}
A. Raghuram and M. Sarnobat,
\emph{Cohomological representations and functorial transfer from classical groups.} 
In Cohomology of arithmetic group, 157-176, Springer Proc. Math. Stat., 245, Springer Cham., 2018. 
\bibitem{RR}
D. Ramakrishnan and J. Rogawski, 
\emph{Average values of modular $L$-series via the relative trace formula. }
Pure Appl. Math. Q. \textbf{1},
Special Issue: In memory of Armand Borel. Part 3, 701--735 (2005).
\bibitem{Rob}
B. Roberts, 
\emph{The theta correspondence for similitudes.}
Israel J. Math. \textbf{94} (1996), 285--317,
\bibitem{Ro}
B. Roberts,
\emph{Global L-packets for $\mathrm{GSp}(2)$ and theta lifts.} 
Doc. Math. \textbf{6} (2001), 247--314.
\bibitem{Saha}
A. Saha,
\emph{A relation between multiplicity one and B\"{o}cherer's conjecture.}
Ramanujan J. (2014), \textbf{33} (2): 263--268.
\bibitem{Sa}
A. Saha,
\emph{On ratios of Petersson norms for Yoshida lifts.}
Forum Math. \textbf{27}, 2361--2412 
\bibitem{SS}
A. Saha and R. Schmidt,
\emph{Yoshida lifts and simultaneous non-vanishing
of dihedral twists of modular $L$-functions.}
J. London Math. Soc. \textbf{88}, 251--270.
\bibitem{Sch1}
R. Schmidt,
\emph{Iwahori-spherical representations of $\mathrm{GSp}\left(4\right)$
and Siegel modular forms of degree $2$ with square-free level.}
J. Math. Soc. Japan \textbf{57} (2005), no. 1, 259--293.
\bibitem{Sha0}
F. Shahidi, 
\emph{On certain $L$-functions.} 
Amer. J. Math. \textbf{103} (1981) 297--355.
\bibitem{Satake}
I. Satake,
\emph{Some remarks to the preceding paper of Tsukamoto.}
J. Math. Soc. Japan \textbf{13} (1961), 401--409.
\bibitem{So}
D. Soudry,
\emph{A uniqueness theorem for representations of $\mathrm{GSO}(6)$ and 
the strong multiplicity one theorem for generic representations of $\mathrm{GSp}(4)$.} 
Israel J. Math. \textbf{58} (1987), no. 3, 257--287. 
\bibitem{Su}
T. Sugano,
\emph{On holomorphic cusp forms on quaternion unitary group
of degree $2$.}
J. Fac. Sci. Univ. Tokyo Sect \textrm{I}A Math. \textbf{31}, 521--568 (1985)
\bibitem{SZ}
B. Sun and C.-B. Zhu, 
\emph{Conservation relations for local theta correspondence.}
J. Amer. Math. Soc. \textbf{28} (2015), no. 4, 939--983. 
\bibitem{Tak}
S. Takeda,
\emph{Some local-global non-vanishing results of theta lifts for symplectic-orthogonal dual pairs.} 
J. Reine Angew. Math. \textbf{657} (2011), 81--111.
\bibitem{Tan}
V. Tan,
\emph{Poles of Siegel Eisenstein series on $\mathrm{U}(n, n)$.} 
Canad. J. Math. \textbf{51} (1999), no. 1, 164--175.
\bibitem{Tsukamoto}
T. Tsukamoto, 
\emph{On the local theory of quaternionic anti-hermitian forms.}
J. Math. Soc. Japan \textbf{13} (1961), 387--400.
\bibitem{Var}
S. Varma, 
\emph{On descent and the generic packet conjecture.}
Forum Math. \textbf{29} (2017), no. 1, 111--155.
\bibitem{Vo}
 D. Vogan, 
 \emph{Gel'fand-Kirillov dimension for Harish-Chandra modules.} 
 Invent. Math. \textbf{48} (1978), no. 1, 75--98. 
\bibitem{Waibel}
F. Waibel,
\emph{Moments of spinor $L$-functions and symplectic Kloosterman sums.}
Q. J. Math. \textbf{70} (2019), no. 4, 1411--1436.
\bibitem{Wal}
J.-L. Waldspurger, 
\emph{Sur les valeurs de certaines fonctions $L$ automorphes en leur centre de sym\'{e}trie.} 
Compos. Math. \textbf{54} (1985), no. 2, 173--242.
\bibitem{Wa}
J.-L. Waldspurger, 
\emph{D\'emonstration d'une conjecture de dualit\'e de
Howe dans le cas $p$-adique, $p\ne 2$.}
In: Festschrift in honor of I. I. Piatetski-Shapiro
on the occasion of his sixtieth birthday, Part \textrm{I}
(Ramat Aviv, 1989), Israel Math. Conf. Proc.,
vol.~2, pp. 267--324. 
Weizmann, Jerusalem (1990).
\bibitem{Wa2}
J.-L. Waldspurger, 
\emph{Une formule int\'{e}grale reli\'{e}e \`{a} la conjecture locale de Gross-Prasad,
2e partie: extension aux repr\'{e}sentations temp\'{e}r\'{e}es.}
 Sur les conjectures de Gross et
Prasad. \textrm{I}.  Ast\'{e}risque \textbf{346}, 171-312 (2012).
\bibitem{Wal12}
J.-L. Waldspurger,
\emph{La conjecture locale de Gross-Prasad pour les repr\'{e}sentations temp\'{e}r\'{e}es des groupes sp\'{e}ciaux orthogonaux.} 
Sur les conjectures de Gross et Prasad. II  Ast\'{e}risque \textbf{347}, 103--165  (2012)
\bibitem{We}
R. Weissauer, 
\emph{Endoscopy for $\mathrm{GSp}(4)$ and the cohomology of Siegel modular threefolds.} 
Lecture Notes in Mathematics, vol. 1968, pp. xviii+368. Springer, Berlin (2009).
 \bibitem{Xue2}
H. Xue,
\emph{Refined global Gan-Gross-Prasad conjecture for Fourier-Jacobi periods on symplectic groups.} 
Compos. Math. \textbf{153} (2017), no. 1, 68--131.
\bibitem{Xuea}
H. Xue, 
\emph{Bessel models for real unitary groups: the tempered case.}
Duke Math. J. \textbf{172} (2023), no. 5, 995--1031
\bibitem{Yam2}
S. Yamana,
\emph{The Siegel-Weil formula for unitary groups.} 
Pacific J. Math. \textbf{264} (2013) 235--257.
\bibitem{Yam}
S. Yamana,
\emph{$L$-functions and theta correspondence for classical groups.}
Invent. Math. \textbf{196} (2014), no. 3, 651--732. 
\bibitem{CZ}
Z. Zhang,
\emph{A note on the local theta correspondence for unitary similitude dual pairs.}
J. Number Theory \textbf{133} (2013), no.9, 3057--3064.
\end{thebibliography}
\end{document}